\documentclass[12pt]{amsart}

\usepackage[english]{babel}
\usepackage{times,bm,amsfonts,amsmath,amssymb,dsfont} 
\usepackage{graphicx}

\usepackage[babel=true]{csquotes}

\makeatletter
\def\section{\@startsection{section}{1}\z@{.9\linespacing\@plus\linespacing}%
  {.7\linespacing} {\fontsize{13}{15}\selectfont\scshape\centering}}
\def\paragraph{\@startsection{paragraph}{4}%
  \z@{0.3em}{-.5em}%
  {$\bullet$ \ \normalfont\itshape}}
\makeatother

\newtheorem{theorem}{Theorem}[section]
\newtheorem{prop}[theorem]{Proposition}
\newtheorem{lem}[theorem]{Lemma}

\theoremstyle{definition}

\newtheorem{notation}[theorem]{Notation}

\theoremstyle{remark}
\newtheorem{remark}[theorem]{Remark}

\makeatletter

\@addtoreset{equation}{section}
\makeatother

\usepackage{color}

\definecolor{gr}{rgb}   {0.,   0.69,   0.23 }
\definecolor{bl}{rgb}   {0.,   0.5,   1. }
\definecolor{mg}{rgb}   {0.85,  0.,    0.85}
\definecolor{yl}{rgb}   {0.8,  0.7,   0.}

\definecolor{webred}{rgb}{0.75,0,0}
\definecolor{webgreen}{rgb}{0,0.75,0}
\usepackage[citecolor=webgreen,colorlinks=true,linkcolor=webred]{hyperref}

\newcommand{\la}{\lambda}
\newcommand{\eps}{\varepsilon}
\newcommand{\Om}{\Omega}
\newcommand{\dr}{\partial}

\newcommand{\N}{\mathbb{N}}
\newcommand{\R}{\mathbb{R}}

\newcommand{\BO}{\mathsf{BO}}

\newcommand{\lef}{\mathsf{lef}}
\newcommand{\ri}{\mathsf{rig}}

\newcommand{\ess}{\mathsf{ess}}
\newcommand{\dis}{\mathsf{dis}}
\newcommand{\app}{\mathsf{app}}

\newcommand{\dist}{\mathrm{dist}}
\newcommand{\inff}{\mathop{\operatorname{\vphantom{p}inf}}}

\newcommand{\toy}{\mathsf{toy}}
\newcommand{\Tri}{\mathsf{Tri}}
\newcommand{\Rec}{\mathsf{Rec}}
\newcommand{\Gui}{\mathsf{Gui}}
\newcommand{\Hst}{\mathsf{Hst}}
\newcommand{\Stlef}{\mathsf{Hlef}}
\newcommand{\Stri}{\mathsf{Hrig}}
\newcommand{\sing}{\mathsf{sing}}

\newcommand{\Dir}{\mathsf{Dir}}
\newcommand{\Mix}{\mathsf{Mix}}

\newcommand{\D}{\mathcal{D}}
\renewcommand{\H}{\mathcal{H}}
\newcommand{\OO}{\mathcal{O}}

\renewcommand{\L}{\mathcal{L}}
\newcommand{\cN}{\mathcal{N}}

\newcommand{\di}{\displaystyle}

\newcommand{\A}{\mathsf{A}}

\newcommand{\Dom}{\mathsf{Dom}}

\newcommand\got[1]{{\bm{\mathfrak{#1}}}}
\newcommand\EN{\got{E}_{N_0}(h)}
\newcommand\gS{\mathfrak{S}}

\newcommand{\Id} {\mathrm{Id}}

\title[Plane waveguides with corners in the small angle limit]
{\large Plane waveguides with corners in the\\[1ex] small angle limit}
\author[Monique Dauge and Nicolas Raymond]{\small Monique Dauge and Nicolas Raymond}
\date{\today}

\address{Laboratoire IRMAR, UMR 6625 du CNRS,
Campus de Beaulieu
35042 Rennes cedex, France}

\email{monique.dauge@univ-rennes1.fr}
\urladdr{http://perso.univ-rennes1.fr/monique.dauge/}
\email{nicolas.raymond@univ-rennes1.fr}
\urladdr{http://perso.univ-rennes1.fr/nicolas.raymond/}


\subjclass{35}

\begin{document}
%
%
%
%
\ifx\figforTeXisloaded\relax \else\global\let\figforTeXisloaded=\relax\fi
\message{version 1.8.4}
\catcode`\@=11
\ifx\ctr@ln@m\undefined\else%
    \immediate\write16{*** Fig4TeX WARNING : \string\ctr@ln@m\space already defined.}\fi
\def\ctr@ln@m#1{\ifx#1\undefined\else%
    \immediate\write16{*** Fig4TeX WARNING : \string#1 already defined.}\fi}
\ctr@ln@m\ctr@ld@f
\def\ctr@ld@f#1#2{\ctr@ln@m#2#1#2}
\ctr@ld@f\def\ctr@ln@w#1#2{\ctr@ln@m#2\csname#1\endcsname#2}
{\catcode`\/=0 \catcode`/\=12 /ctr@ld@f/gdef/BS@{\}}
\ctr@ld@f\def\ctr@lcsn@m#1{\expandafter\ifx\csname#1\endcsname\relax\else%
    \immediate\write16{*** Fig4TeX WARNING : \BS@\expandafter\string#1\space already defined.}\fi}
\ctr@ld@f\edef\colonc@tcode{\the\catcode`\:}
\ctr@ld@f\edef\semicolonc@tcode{\the\catcode`\;}
\ctr@ld@f\def\t@stc@tcodech@nge{{\let\c@tcodech@nged=\z@%
    \ifnum\colonc@tcode=\the\catcode`\:\else\let\c@tcodech@nged=\@ne\fi%
    \ifnum\semicolonc@tcode=\the\catcode`\;\else\let\c@tcodech@nged=\@ne\fi%
    \ifx\c@tcodech@nged\@ne%
    \immediate\write16{}
    \immediate\write16{!!!=============================================================!!!}
    \immediate\write16{ Fig4TeX WARNING :}
    \immediate\write16{ The category code of some characters has been changed, which will}
    \immediate\write16{ result in an error (message "Runaway argument?").}
    \immediate\write16{ This probably comes from another package that changed the category}
    \immediate\write16{ code after Fig4TeX was loaded. If that proves to be exact, the}
    \immediate\write16{ solution is to exchange the loading commands on top of your file}
    \immediate\write16{ so that Fig4TeX is loaded last. For example, in LaTeX, we should}
    \immediate\write16{ say :}
    \immediate\write16{\BS@ usepackage[french]{babel}}
    \immediate\write16{\BS@ usepackage{fig4tex}}
    \immediate\write16{!!!=============================================================!!!}
    \immediate\write16{}
    \fi}}
\ctr@ld@f\def\FigforTeX{F\kern-.05em i\kern-.05em g\kern-.1em\raise-.14em\hbox{4}\kern-.19em\TeX}
\ctr@ln@w{newdimen}\epsil@n\epsil@n=0.00005pt
\ctr@ln@w{newdimen}\Cepsil@n\Cepsil@n=0.005pt
\ctr@ln@w{newdimen}\dcq@\dcq@=254pt
\ctr@ln@w{newdimen}\PI@\PI@=3.141592pt
\ctr@ln@w{newdimen}\DemiPI@deg\DemiPI@deg=90pt
\ctr@ln@w{newdimen}\PI@deg\PI@deg=180pt
\ctr@ln@w{newdimen}\DePI@deg\DePI@deg=360pt
\ctr@ld@f\chardef\t@n=10
\ctr@ld@f\chardef\c@nt=100
\ctr@ld@f\chardef\@lxxiv=74
\ctr@ld@f\chardef\@xci=91
\ctr@ld@f\mathchardef\@nMnCQn=9949
\ctr@ld@f\chardef\@vi=6
\ctr@ld@f\chardef\@xxx=30
\ctr@ld@f\chardef\@lvi=56
\ctr@ld@f\chardef\@@lxxi=71
\ctr@ld@f\chardef\@lxxxv=85
\ctr@ld@f\mathchardef\@@mmmmlxviii=4068
\ctr@ld@f\mathchardef\@ccclx=360
\ctr@ld@f\mathchardef\@dccxx=720
\ctr@ln@w{newcount}\p@rtent \ctr@ln@w{newcount}\f@ctech \ctr@ln@w{newcount}\result@tent
\ctr@ln@w{newdimen}\v@lmin \ctr@ln@w{newdimen}\v@lmax \ctr@ln@w{newdimen}\v@leur
\ctr@ln@w{newdimen}\result@t\ctr@ln@w{newdimen}\result@@t
\ctr@ln@w{newdimen}\mili@u \ctr@ln@w{newdimen}\c@rre \ctr@ln@w{newdimen}\delt@
\ctr@ld@f\def\degT@rd{0.017453 }  
\ctr@ld@f\def\rdT@deg{57.295779 } 
\ctr@ln@m\v@leurseule
{\catcode`p=12 \catcode`t=12 \gdef\v@leurseule#1pt{#1}}
\ctr@ld@f\def\repdecn@mb#1{\expandafter\v@leurseule\the#1\space}
\ctr@ld@f\def\arct@n#1(#2,#3){{\v@lmin=#2\v@lmax=#3%
    \maxim@m{\mili@u}{-\v@lmin}{\v@lmin}\maxim@m{\c@rre}{-\v@lmax}{\v@lmax}%
    \delt@=\mili@u\m@ech\mili@u%
    \ifdim\c@rre>\@nMnCQn\mili@u\divide\v@lmax\tw@\c@lATAN\v@leur(\z@,\v@lmax)
    \else%
    \maxim@m{\mili@u}{-\v@lmin}{\v@lmin}\maxim@m{\c@rre}{-\v@lmax}{\v@lmax}%
    \m@ech\c@rre%
    \ifdim\mili@u>\@nMnCQn\c@rre\divide\v@lmin\tw@
    \maxim@m{\mili@u}{-\v@lmin}{\v@lmin}\c@lATAN\v@leur(\mili@u,\z@)%
    \else\c@lATAN\v@leur(\delt@,\v@lmax)\fi\fi%
    \ifdim\v@lmin<\z@\v@leur=-\v@leur\ifdim\v@lmax<\z@\advance\v@leur-\PI@%
    \else\advance\v@leur\PI@\fi\fi%
    \global\result@t=\v@leur}#1=\result@t}
\ctr@ld@f\def\m@ech#1{\ifdim#1>1.646pt\divide\mili@u\t@n\divide\c@rre\t@n\m@ech#1\fi}
\ctr@ld@f\def\c@lATAN#1(#2,#3){{\v@lmin=#2\v@lmax=#3\v@leur=\z@\delt@=\tw@ pt%
    \un@iter{0.785398}{\v@lmax<}%
    \un@iter{0.463648}{\v@lmax<}%
    \un@iter{0.244979}{\v@lmax<}%
    \un@iter{0.124355}{\v@lmax<}%
    \un@iter{0.062419}{\v@lmax<}%
    \un@iter{0.031240}{\v@lmax<}%
    \un@iter{0.015624}{\v@lmax<}%
    \un@iter{0.007812}{\v@lmax<}%
    \un@iter{0.003906}{\v@lmax<}%
    \un@iter{0.001953}{\v@lmax<}%
    \un@iter{0.000976}{\v@lmax<}%
    \un@iter{0.000488}{\v@lmax<}%
    \un@iter{0.000244}{\v@lmax<}%
    \un@iter{0.000122}{\v@lmax<}%
    \un@iter{0.000061}{\v@lmax<}%
    \un@iter{0.000030}{\v@lmax<}%
    \un@iter{0.000015}{\v@lmax<}%
    \global\result@t=\v@leur}#1=\result@t}
\ctr@ld@f\def\un@iter#1#2{%
    \divide\delt@\tw@\edef\dpmn@{\repdecn@mb{\delt@}}%
    \mili@u=\v@lmin%
    \ifdim#2\z@%
      \advance\v@lmin-\dpmn@\v@lmax\advance\v@lmax\dpmn@\mili@u%
      \advance\v@leur-#1pt%
    \else%
      \advance\v@lmin\dpmn@\v@lmax\advance\v@lmax-\dpmn@\mili@u%
      \advance\v@leur#1pt%
    \fi}
\ctr@ld@f\def\c@ssin#1#2#3{\expandafter\ifx\csname COS@\number#3\endcsname\relax\c@lCS{#3pt}%
    \expandafter\xdef\csname COS@\number#3\endcsname{\repdecn@mb\result@t}%
    \expandafter\xdef\csname SIN@\number#3\endcsname{\repdecn@mb\result@@t}\fi%
    \edef#1{\csname COS@\number#3\endcsname}\edef#2{\csname SIN@\number#3\endcsname}}
\ctr@ld@f\def\c@lCS#1{{\mili@u=#1\p@rtent=\@ne%
    \relax\ifdim\mili@u<\z@\red@ng<-\else\red@ng>+\fi\f@ctech=\p@rtent%
    \relax\ifdim\mili@u<\z@\mili@u=-\mili@u\f@ctech=-\f@ctech\fi\c@@lCS}}
\ctr@ld@f\def\c@@lCS{\v@lmin=\mili@u\c@rre=-\mili@u\advance\c@rre\DemiPI@deg\v@lmax=\c@rre%
    \mili@u\@@lxxi\mili@u\divide\mili@u\@@mmmmlxviii%
    \edef\v@larg{\repdecn@mb{\mili@u}}\mili@u=-\v@larg\mili@u%
    \edef\v@lmxde{\repdecn@mb{\mili@u}}%
    \c@rre\@@lxxi\c@rre\divide\c@rre\@@mmmmlxviii%
    \edef\v@largC{\repdecn@mb{\c@rre}}\c@rre=-\v@largC\c@rre%
    \edef\v@lmxdeC{\repdecn@mb{\c@rre}}%
    \fctc@s\mili@u\v@lmin\global\result@t\p@rtent\v@leur%
    \let\t@mp=\v@larg\let\v@larg=\v@largC\let\v@largC=\t@mp%
    \let\t@mp=\v@lmxde\let\v@lmxde=\v@lmxdeC\let\v@lmxdeC=\t@mp%
    \fctc@s\c@rre\v@lmax\global\result@@t\f@ctech\v@leur}
\ctr@ld@f\def\fctc@s#1#2{\v@leur=#1\relax\ifdim#2<\@lxxxv\p@\cosser@h\else\sinser@t\fi}
\ctr@ld@f\def\cosser@h{\advance\v@leur\@lvi\p@\divide\v@leur\@lvi%
    \v@leur=\v@lmxde\v@leur\advance\v@leur\@xxx\p@%
    \v@leur=\v@lmxde\v@leur\advance\v@leur\@ccclx\p@%
    \v@leur=\v@lmxde\v@leur\advance\v@leur\@dccxx\p@\divide\v@leur\@dccxx}
\ctr@ld@f\def\sinser@t{\v@leur=\v@lmxdeC\p@\advance\v@leur\@vi\p@%
    \v@leur=\v@largC\v@leur\divide\v@leur\@vi}
\ctr@ld@f\def\red@ng#1#2{\relax\ifdim\mili@u#1#2\DemiPI@deg\advance\mili@u#2-\PI@deg%
    \p@rtent=-\p@rtent\red@ng#1#2\fi}
\ctr@ld@f\def\pr@c@lCS#1#2#3{\ctr@lcsn@m{COS@\number#3 }%
    \expandafter\xdef\csname COS@\number#3\endcsname{#1}%
    \expandafter\xdef\csname SIN@\number#3\endcsname{#2}}
\pr@c@lCS{1}{0}{0}
\pr@c@lCS{0.7071}{0.7071}{45}\pr@c@lCS{0.7071}{-0.7071}{-45}
\pr@c@lCS{0}{1}{90}          \pr@c@lCS{0}{-1}{-90}
\pr@c@lCS{-1}{0}{180}        \pr@c@lCS{-1}{0}{-180}
\pr@c@lCS{0}{-1}{270}        \pr@c@lCS{0}{1}{-270}
\ctr@ld@f\def\invers@#1#2{{\v@leur=#2\maxim@m{\v@lmax}{-\v@leur}{\v@leur}%
    \f@ctech=\@ne\m@inv@rs%
    \multiply\v@leur\f@ctech\edef\v@lv@leur{\repdecn@mb{\v@leur}}%
    \p@rtentiere{\p@rtent}{\v@leur}\v@lmin=\p@\divide\v@lmin\p@rtent%
    \inv@rs@\multiply\v@lmax\f@ctech\global\result@t=\v@lmax}#1=\result@t}
\ctr@ld@f\def\m@inv@rs{\ifdim\v@lmax<\p@\multiply\v@lmax\t@n\multiply\f@ctech\t@n\m@inv@rs\fi}
\ctr@ld@f\def\inv@rs@{\v@lmax=-\v@lmin\v@lmax=\v@lv@leur\v@lmax%
    \advance\v@lmax\tw@ pt\v@lmax=\repdecn@mb{\v@lmin}\v@lmax%
    \delt@=\v@lmax\advance\delt@-\v@lmin\ifdim\delt@<\z@\delt@=-\delt@\fi%
    \ifdim\delt@>\epsil@n\v@lmin=\v@lmax\inv@rs@\fi}
\ctr@ld@f\def\minim@m#1#2#3{\relax\ifdim#2<#3#1=#2\else#1=#3\fi}
\ctr@ld@f\def\maxim@m#1#2#3{\relax\ifdim#2>#3#1=#2\else#1=#3\fi}
\ctr@ld@f\def\p@rtentiere#1#2{#1=#2\divide#1by65536 }
\ctr@ld@f\def\r@undint#1#2{{\v@leur=#2\divide\v@leur\t@n\p@rtentiere{\p@rtent}{\v@leur}%
    \v@leur=\p@rtent pt\global\result@t=\t@n\v@leur}#1=\result@t}
\ctr@ld@f\def\sqrt@#1#2{{\v@leur=#2%
    \minim@m{\v@lmin}{\p@}{\v@leur}\maxim@m{\v@lmax}{\p@}{\v@leur}%
    \f@ctech=\@ne\m@sqrt@\sqrt@@%
    \mili@u=\v@lmin\advance\mili@u\v@lmax\divide\mili@u\tw@\multiply\mili@u\f@ctech%
    \global\result@t=\mili@u}#1=\result@t}
\ctr@ld@f\def\m@sqrt@{\ifdim\v@leur>\dcq@\divide\v@leur\c@nt\v@lmax=\v@leur%
    \multiply\f@ctech\t@n\m@sqrt@\fi}
\ctr@ld@f\def\sqrt@@{\mili@u=\v@lmin\advance\mili@u\v@lmax\divide\mili@u\tw@%
    \c@rre=\repdecn@mb{\mili@u}\mili@u%
    \ifdim\c@rre<\v@leur\v@lmin=\mili@u\else\v@lmax=\mili@u\fi%
    \delt@=\v@lmax\advance\delt@-\v@lmin\ifdim\delt@>\epsil@n\sqrt@@\fi}
\ctr@ld@f\def\extrairelepremi@r#1\de#2{\expandafter\lepremi@r#2@#1#2}
\ctr@ld@f\def\lepremi@r#1,#2@#3#4{\def#3{#1}\def#4{#2}\ignorespaces}
\ctr@ld@f\def\@cfor#1:=#2\do#3{%
  \edef\@fortemp{#2}%
  \ifx\@fortemp\empty\else\@cforloop#2,\@nil,\@nil\@@#1{#3}\fi}
\ctr@ln@m\@nextwhile
\ctr@ld@f\def\@cforloop#1,#2\@@#3#4{%
  \def#3{#1}%
  \ifx#3\Fig@nnil\let\@nextwhile=\Fig@fornoop\else#4\relax\let\@nextwhile=\@cforloop\fi%
  \@nextwhile#2\@@#3{#4}}

\ctr@ld@f\def\@ecfor#1:=#2\do#3{%
  \def\@@cfor{\@cfor#1:=}%
  \edef\@@@cfor{#2}%
  \expandafter\@@cfor\@@@cfor\do{#3}}
\ctr@ld@f\def\Fig@nnil{\@nil}
\ctr@ld@f\def\Fig@fornoop#1\@@#2#3{}
\ctr@ln@m\list@@rg
\ctr@ld@f\def\trtlis@rg#1#2{\def\list@@rg{#1}%
    \@ecfor\p@rv@l:=\list@@rg\do{\expandafter#2\p@rv@l|}}
\ctr@ln@w{newbox}\b@xvisu
\ctr@ln@w{newtoks}\let@xte
\ctr@ln@w{newif}\ifitis@K
\ctr@ln@w{newcount}\s@mme
\ctr@ln@w{newcount}\l@mbd@un \ctr@ln@w{newcount}\l@mbd@de
\ctr@ln@w{newcount}\superc@ntr@l\superc@ntr@l=\@ne        
\ctr@ln@w{newcount}\typec@ntr@l\typec@ntr@l=\superc@ntr@l 
\ctr@ln@w{newdimen}\v@lX  \ctr@ln@w{newdimen}\v@lY  \ctr@ln@w{newdimen}\v@lZ
\ctr@ln@w{newdimen}\v@lXa \ctr@ln@w{newdimen}\v@lYa \ctr@ln@w{newdimen}\v@lZa
\ctr@ln@w{newdimen}\unit@\unit@=\p@ 
\ctr@ld@f\def\unit@util{pt}
\ctr@ld@f\def\ptT@ptps{0.996264 }
\ctr@ld@f\def\ptpsT@pt{1.00375 }
\ctr@ld@f\def\ptT@unit@{1} 
\ctr@ld@f\def\setunit@#1{\def\unit@util{#1}\setunit@@#1:\invers@{\result@t}{\unit@}%
    \edef\ptT@unit@{\repdecn@mb\result@t}}
\ctr@ld@f\def\setunit@@#1#2:{\ifcat#1a\unit@=\@ne#1#2\else\unit@=#1#2\fi}
\ctr@ld@f\def\d@fm@cdim#1#2{{\v@leur=#2\v@leur=\ptT@unit@\v@leur\xdef#1{\repdecn@mb\v@leur}}}
\ctr@ln@w{newif}\ifBdingB@x\BdingB@xtrue
\ctr@ln@w{newdimen}\c@@rdXmin \ctr@ln@w{newdimen}\c@@rdYmin  
\ctr@ln@w{newdimen}\c@@rdXmax \ctr@ln@w{newdimen}\c@@rdYmax
\ctr@ld@f\def\b@undb@x#1#2{\ifBdingB@x%
    \relax\ifdim#1<\c@@rdXmin\global\c@@rdXmin=#1\fi%
    \relax\ifdim#2<\c@@rdYmin\global\c@@rdYmin=#2\fi%
    \relax\ifdim#1>\c@@rdXmax\global\c@@rdXmax=#1\fi%
    \relax\ifdim#2>\c@@rdYmax\global\c@@rdYmax=#2\fi\fi}
\ctr@ld@f\def\b@undb@xP#1{{\Figg@tXY{#1}\b@undb@x{\v@lX}{\v@lY}}}
\ctr@ld@f\def\ellBB@x#1;#2,#3(#4,#5,#6){{\s@uvc@ntr@l\et@tellBB@x%
    \setc@ntr@l{2}\figptell-2::#1;#2,#3(#4,#6)\b@undb@xP{-2}%
    \figptell-2::#1;#2,#3(#5,#6)\b@undb@xP{-2}%
    \c@ssin{\C@}{\S@}{#6}\v@lmin=\C@ pt\v@lmax=\S@ pt%
    \mili@u=#3\v@lmin\delt@=#2\v@lmax\arct@n\v@leur(\delt@,\mili@u)%
    \mili@u=-#3\v@lmax\delt@=#2\v@lmin\arct@n\c@rre(\delt@,\mili@u)%
    \v@leur=\rdT@deg\v@leur\advance\v@leur-\DePI@deg%
    \c@rre=\rdT@deg\c@rre\advance\c@rre-\DePI@deg%
    \v@lmin=#4pt\v@lmax=#5pt%
    \loop\ifdim\v@leur<\v@lmax\ifdim\v@leur>\v@lmin%
    \edef\@ngle{\repdecn@mb\v@leur}\figptell-2::#1;#2,#3(\@ngle,#6)%
    \b@undb@xP{-2}\fi\advance\v@leur\PI@deg\repeat%
    \loop\ifdim\c@rre<\v@lmax\ifdim\c@rre>\v@lmin%
    \edef\@ngle{\repdecn@mb\c@rre}\figptell-2::#1;#2,#3(\@ngle,#6)%
    \b@undb@xP{-2}\fi\advance\c@rre\PI@deg\repeat%
    \resetc@ntr@l\et@tellBB@x}\ignorespaces}
\ctr@ld@f\def\initb@undb@x{\c@@rdXmin=\maxdimen\c@@rdYmin=\maxdimen%
    \c@@rdXmax=-\maxdimen\c@@rdYmax=-\maxdimen}
\ctr@ld@f\def\c@ntr@lnum#1{%
    \relax\ifnum\typec@ntr@l=\@ne%
    \ifnum#1<\z@%
    \immediate\write16{*** Forbidden point number (#1). Abort.}\end\fi\fi%
    \set@bjc@de{#1}}
\ctr@ln@m\objc@de
\ctr@ld@f\def\set@bjc@de#1{\edef\objc@de{@BJ\ifnum#1<\z@ M\romannumeral-#1\else\romannumeral#1\fi}}
\s@mme=\m@ne\loop\ifnum\s@mme>-19
  \set@bjc@de{\s@mme}\ctr@lcsn@m\objc@de\ctr@lcsn@m{\objc@de T}
\advance\s@mme\m@ne\repeat
\s@mme=\@ne\loop\ifnum\s@mme<6
  \set@bjc@de{\s@mme}\ctr@lcsn@m\objc@de\ctr@lcsn@m{\objc@de T}
\advance\s@mme\@ne\repeat
\ctr@ld@f\def\setc@ntr@l#1{\ifnum\superc@ntr@l>#1\typec@ntr@l=\superc@ntr@l%
    \else\typec@ntr@l=#1\fi}
\ctr@ld@f\def\resetc@ntr@l#1{\global\superc@ntr@l=#1\setc@ntr@l{#1}}
\ctr@ld@f\def\s@uvc@ntr@l#1{\edef#1{\the\superc@ntr@l}}
\ctr@ln@m\c@lproscal
\ctr@ld@f\def\c@lproscalDD#1[#2,#3]{{\Figg@tXY{#2}%
    \edef\Xu@{\repdecn@mb{\v@lX}}\edef\Yu@{\repdecn@mb{\v@lY}}\Figg@tXY{#3}%
    \global\result@t=\Xu@\v@lX\global\advance\result@t\Yu@\v@lY}#1=\result@t}
\ctr@ld@f\def\c@lproscalTD#1[#2,#3]{{\Figg@tXY{#2}\edef\Xu@{\repdecn@mb{\v@lX}}%
    \edef\Yu@{\repdecn@mb{\v@lY}}\edef\Zu@{\repdecn@mb{\v@lZ}}%
    \Figg@tXY{#3}\global\result@t=\Xu@\v@lX\global\advance\result@t\Yu@\v@lY%
    \global\advance\result@t\Zu@\v@lZ}#1=\result@t}
\ctr@ld@f\def\c@lprovec#1{%
    \det@rmC\v@lZa(\v@lX,\v@lY,\v@lmin,\v@lmax)%
    \det@rmC\v@lXa(\v@lY,\v@lZ,\v@lmax,\v@leur)%
    \det@rmC\v@lYa(\v@lZ,\v@lX,\v@leur,\v@lmin)%
    \Figv@ctCreg#1(\v@lXa,\v@lYa,\v@lZa)}
\ctr@ld@f\def\det@rm#1[#2,#3]{{\Figg@tXY{#2}\Figg@tXYa{#3}%
    \delt@=\repdecn@mb{\v@lX}\v@lYa\advance\delt@-\repdecn@mb{\v@lY}\v@lXa%
    \global\result@t=\delt@}#1=\result@t}
\ctr@ld@f\def\det@rmC#1(#2,#3,#4,#5){{\global\result@t=\repdecn@mb{#2}#5%
    \global\advance\result@t-\repdecn@mb{#3}#4}#1=\result@t}
\ctr@ld@f\def\getredf@ctDD#1(#2,#3){{\maxim@m{\v@lXa}{-#2}{#2}\maxim@m{\v@lYa}{-#3}{#3}%
    \maxim@m{\v@lXa}{\v@lXa}{\v@lYa}
    \ifdim\v@lXa>\@xci pt\divide\v@lXa\@xci%
    \p@rtentiere{\p@rtent}{\v@lXa}\advance\p@rtent\@ne\else\p@rtent=\@ne\fi%
    \global\result@tent=\p@rtent}#1=\result@tent\ignorespaces}
\ctr@ld@f\def\getredf@ctTD#1(#2,#3,#4){{\maxim@m{\v@lXa}{-#2}{#2}\maxim@m{\v@lYa}{-#3}{#3}%
    \maxim@m{\v@lZa}{-#4}{#4}\maxim@m{\v@lXa}{\v@lXa}{\v@lYa}%
    \maxim@m{\v@lXa}{\v@lXa}{\v@lZa}
    \ifdim\v@lXa>\@lxxiv pt\divide\v@lXa\@lxxiv%
    \p@rtentiere{\p@rtent}{\v@lXa}\advance\p@rtent\@ne\else\p@rtent=\@ne\fi%
    \global\result@tent=\p@rtent}#1=\result@tent\ignorespaces}
\ctr@ld@f\def\FigptintercircB@zDD#1:#2:#3,#4[#5,#6,#7,#8]{{\s@uvc@ntr@l\et@tfigptintercircB@zDD%
    \setc@ntr@l{2}\figvectPDD-1[#5,#8]\Figg@tXY{-1}\getredf@ctDD\f@ctech(\v@lX,\v@lY)%
    \mili@u=#4\unit@\divide\mili@u\f@ctech\c@rre=\repdecn@mb{\mili@u}\mili@u%
    \figptBezierDD-5::#3[#5,#6,#7,#8]%
    \v@lmin=#3\p@\v@lmax=\v@lmin\advance\v@lmax0.1\p@%
    \loop\edef\T@{\repdecn@mb{\v@lmax}}\figptBezierDD-2::\T@[#5,#6,#7,#8]%
    \figvectPDD-1[-5,-2]\n@rmeucCDD{\delt@}{-1}\ifdim\delt@<\c@rre\v@lmin=\v@lmax%
    \advance\v@lmax0.1\p@\repeat%
    \loop\mili@u=\v@lmin\advance\mili@u\v@lmax%
    \divide\mili@u\tw@\edef\T@{\repdecn@mb{\mili@u}}\figptBezierDD-2::\T@[#5,#6,#7,#8]%
    \figvectPDD-1[-5,-2]\n@rmeucCDD{\delt@}{-1}\ifdim\delt@>\c@rre\v@lmax=\mili@u%
    \else\v@lmin=\mili@u\fi\v@leur=\v@lmax\advance\v@leur-\v@lmin%
    \ifdim\v@leur>\epsil@n\repeat\figptcopyDD#1:#2/-2/%
    \resetc@ntr@l\et@tfigptintercircB@zDD}\ignorespaces}
\ctr@ln@m\figptinterlines
\ctr@ld@f\def\inters@cDD#1:#2[#3,#4;#5,#6]{{\s@uvc@ntr@l\et@tinters@cDD%
    \setc@ntr@l{2}\vecunit@{-1}{#4}\vecunit@{-2}{#6}%
    \Figg@tXY{-1}\setc@ntr@l{1}\Figg@tXYa{#3}%
    \edef\A@{\repdecn@mb{\v@lX}}\edef\B@{\repdecn@mb{\v@lY}}%
    \v@lmin=\B@\v@lXa\advance\v@lmin-\A@\v@lYa%
    \Figg@tXYa{#5}\setc@ntr@l{2}\Figg@tXY{-2}%
    \edef\C@{\repdecn@mb{\v@lX}}\edef\D@{\repdecn@mb{\v@lY}}%
    \v@lmax=\D@\v@lXa\advance\v@lmax-\C@\v@lYa%
    \delt@=\A@\v@lY\advance\delt@-\B@\v@lX%
    \invers@{\v@leur}{\delt@}\edef\v@ldelta{\repdecn@mb{\v@leur}}%
    \v@lXa=\A@\v@lmax\advance\v@lXa-\C@\v@lmin%
    \v@lYa=\B@\v@lmax\advance\v@lYa-\D@\v@lmin%
    \v@lXa=\v@ldelta\v@lXa\v@lYa=\v@ldelta\v@lYa%
    \setc@ntr@l{1}\Figp@intregDD#1:{#2}(\v@lXa,\v@lYa)%
    \resetc@ntr@l\et@tinters@cDD}\ignorespaces}
\ctr@ld@f\def\inters@cTD#1:#2[#3,#4;#5,#6]{{\s@uvc@ntr@l\et@tinters@cTD%
    \setc@ntr@l{2}\figvectNVTD-1[#4,#6]\figvectNVTD-2[#6,-1]\figvectPTD-1[#3,#5]%
    \r@pPSTD\v@leur[-2,-1,#4]\edef\v@lcoef{\repdecn@mb{\v@leur}}%
    \figpttraTD#1:{#2}=#3/\v@lcoef,#4/\resetc@ntr@l\et@tinters@cTD}\ignorespaces}
\ctr@ld@f\def\r@pPSTD#1[#2,#3,#4]{{\Figg@tXY{#2}\edef\Xu@{\repdecn@mb{\v@lX}}%
    \edef\Yu@{\repdecn@mb{\v@lY}}\edef\Zu@{\repdecn@mb{\v@lZ}}%
    \Figg@tXY{#3}\v@lmin=\Xu@\v@lX\advance\v@lmin\Yu@\v@lY\advance\v@lmin\Zu@\v@lZ%
    \Figg@tXY{#4}\v@lmax=\Xu@\v@lX\advance\v@lmax\Yu@\v@lY\advance\v@lmax\Zu@\v@lZ%
    \invers@{\v@leur}{\v@lmax}\global\result@t=\repdecn@mb{\v@leur}\v@lmin}%
    #1=\result@t}
\ctr@ln@m\n@rminf
\ctr@ld@f\def\n@rminfDD#1#2{{\Figg@tXY{#2}\maxim@m{\v@lX}{\v@lX}{-\v@lX}%
    \maxim@m{\v@lY}{\v@lY}{-\v@lY}\maxim@m{\global\result@t}{\v@lX}{\v@lY}}%
    #1=\result@t}
\ctr@ld@f\def\n@rminfTD#1#2{{\Figg@tXY{#2}\maxim@m{\v@lX}{\v@lX}{-\v@lX}%
    \maxim@m{\v@lY}{\v@lY}{-\v@lY}\maxim@m{\v@lZ}{\v@lZ}{-\v@lZ}%
    \maxim@m{\v@lX}{\v@lX}{\v@lY}\maxim@m{\global\result@t}{\v@lX}{\v@lZ}}%
    #1=\result@t}
\ctr@ld@f\def\n@rmeucCDD#1#2{\Figg@tXY{#2}\divide\v@lX\f@ctech\divide\v@lY\f@ctech%
    #1=\repdecn@mb{\v@lX}\v@lX\v@lX=\repdecn@mb{\v@lY}\v@lY\advance#1\v@lX}
\ctr@ld@f\def\n@rmeucCTD#1#2{\Figg@tXY{#2}%
    \divide\v@lX\f@ctech\divide\v@lY\f@ctech\divide\v@lZ\f@ctech%
    #1=\repdecn@mb{\v@lX}\v@lX\v@lX=\repdecn@mb{\v@lY}\v@lY\advance#1\v@lX%
    \v@lX=\repdecn@mb{\v@lZ}\v@lZ\advance#1\v@lX}
\ctr@ln@m\n@rmeucSV
\ctr@ld@f\def\n@rmeucSVDD#1#2{{\Figg@tXY{#2}%
    \v@lXa=\repdecn@mb{\v@lX}\v@lX\v@lYa=\repdecn@mb{\v@lY}\v@lY%
    \advance\v@lXa\v@lYa\sqrt@{\global\result@t}{\v@lXa}}#1=\result@t}
\ctr@ld@f\def\n@rmeucSVTD#1#2{{\Figg@tXY{#2}\v@lXa=\repdecn@mb{\v@lX}\v@lX%
    \v@lYa=\repdecn@mb{\v@lY}\v@lY\v@lZa=\repdecn@mb{\v@lZ}\v@lZ%
    \advance\v@lXa\v@lYa\advance\v@lXa\v@lZa\sqrt@{\global\result@t}{\v@lXa}}#1=\result@t}
\ctr@ln@m\n@rmeuc
\ctr@ld@f\def\n@rmeucDD#1#2{{\Figg@tXY{#2}\getredf@ctDD\f@ctech(\v@lX,\v@lY)%
    \divide\v@lX\f@ctech\divide\v@lY\f@ctech%
    \v@lXa=\repdecn@mb{\v@lX}\v@lX\v@lYa=\repdecn@mb{\v@lY}\v@lY%
    \advance\v@lXa\v@lYa\sqrt@{\global\result@t}{\v@lXa}%
    \global\multiply\result@t\f@ctech}#1=\result@t}
\ctr@ld@f\def\n@rmeucTD#1#2{{\Figg@tXY{#2}\getredf@ctTD\f@ctech(\v@lX,\v@lY,\v@lZ)%
    \divide\v@lX\f@ctech\divide\v@lY\f@ctech\divide\v@lZ\f@ctech%
    \v@lXa=\repdecn@mb{\v@lX}\v@lX%
    \v@lYa=\repdecn@mb{\v@lY}\v@lY\v@lZa=\repdecn@mb{\v@lZ}\v@lZ%
    \advance\v@lXa\v@lYa\advance\v@lXa\v@lZa\sqrt@{\global\result@t}{\v@lXa}%
    \global\multiply\result@t\f@ctech}#1=\result@t}
\ctr@ln@m\vecunit@
\ctr@ld@f\def\vecunit@DD#1#2{{\Figg@tXY{#2}\getredf@ctDD\f@ctech(\v@lX,\v@lY)%
    \divide\v@lX\f@ctech\divide\v@lY\f@ctech%
    \Figv@ctCreg#1(\v@lX,\v@lY)\n@rmeucSV{\v@lYa}{#1}%
    \invers@{\v@lXa}{\v@lYa}\edef\v@lv@lXa{\repdecn@mb{\v@lXa}}%
    \v@lX=\v@lv@lXa\v@lX\v@lY=\v@lv@lXa\v@lY%
    \Figv@ctCreg#1(\v@lX,\v@lY)\multiply\v@lYa\f@ctech\global\result@t=\v@lYa}}
\ctr@ld@f\def\vecunit@TD#1#2{{\Figg@tXY{#2}\getredf@ctTD\f@ctech(\v@lX,\v@lY,\v@lZ)%
    \divide\v@lX\f@ctech\divide\v@lY\f@ctech\divide\v@lZ\f@ctech%
    \Figv@ctCreg#1(\v@lX,\v@lY,\v@lZ)\n@rmeucSV{\v@lYa}{#1}%
    \invers@{\v@lXa}{\v@lYa}\edef\v@lv@lXa{\repdecn@mb{\v@lXa}}%
    \v@lX=\v@lv@lXa\v@lX\v@lY=\v@lv@lXa\v@lY\v@lZ=\v@lv@lXa\v@lZ%
    \Figv@ctCreg#1(\v@lX,\v@lY,\v@lZ)\multiply\v@lYa\f@ctech\global\result@t=\v@lYa}}
\ctr@ld@f\def\vecunitC@TD[#1,#2]{\Figg@tXYa{#1}\Figg@tXY{#2}%
    \advance\v@lX-\v@lXa\advance\v@lY-\v@lYa\advance\v@lZ-\v@lZa\c@lvecunitTD}
\ctr@ld@f\def\vecunitCV@TD#1{\Figg@tXY{#1}\c@lvecunitTD}
\ctr@ld@f\def\c@lvecunitTD{\getredf@ctTD\f@ctech(\v@lX,\v@lY,\v@lZ)%
    \divide\v@lX\f@ctech\divide\v@lY\f@ctech\divide\v@lZ\f@ctech%
    \v@lXa=\repdecn@mb{\v@lX}\v@lX%
    \v@lYa=\repdecn@mb{\v@lY}\v@lY\v@lZa=\repdecn@mb{\v@lZ}\v@lZ%
    \advance\v@lXa\v@lYa\advance\v@lXa\v@lZa\sqrt@{\v@lYa}{\v@lXa}%
    \invers@{\v@lXa}{\v@lYa}\edef\v@lv@lXa{\repdecn@mb{\v@lXa}}%
    \v@lX=\v@lv@lXa\v@lX\v@lY=\v@lv@lXa\v@lY\v@lZ=\v@lv@lXa\v@lZ}
\ctr@ln@m\figgetangle
\ctr@ld@f\def\figgetangleDD#1[#2,#3,#4]{\ifps@cri{\s@uvc@ntr@l\et@tfiggetangleDD\setc@ntr@l{2}%
    \figvectPDD-1[#2,#3]\figvectPDD-2[#2,#4]\vecunit@{-1}{-1}%
    \c@lproscalDD\delt@[-2,-1]\figvectNVDD-1[-1]\c@lproscalDD\v@leur[-2,-1]%
    \arct@n\v@lmax(\delt@,\v@leur)\v@lmax=\rdT@deg\v@lmax%
    \ifdim\v@lmax<\z@\advance\v@lmax\DePI@deg\fi\xdef#1{\repdecn@mb{\v@lmax}}%
    \resetc@ntr@l\et@tfiggetangleDD}\ignorespaces\fi}
\ctr@ld@f\def\figgetangleTD#1[#2,#3,#4,#5]{\ifps@cri{\s@uvc@ntr@l\et@tfiggetangleTD\setc@ntr@l{2}%
    \figvectPTD-1[#2,#3]\figvectPTD-2[#2,#5]\figvectNVTD-3[-1,-2]%
    \figvectPTD-2[#2,#4]\figvectNVTD-4[-3,-1]%
    \vecunit@{-1}{-1}\c@lproscalTD\delt@[-2,-1]\c@lproscalTD\v@leur[-2,-4]%
    \arct@n\v@lmax(\delt@,\v@leur)\v@lmax=\rdT@deg\v@lmax%
    \ifdim\v@lmax<\z@\advance\v@lmax\DePI@deg\fi\xdef#1{\repdecn@mb{\v@lmax}}%
    \resetc@ntr@l\et@tfiggetangleTD}\ignorespaces\fi}    
\ctr@ld@f\def\figgetdist#1[#2,#3]{\ifps@cri{\s@uvc@ntr@l\et@tfiggetdist\setc@ntr@l{2}%
    \figvectP-1[#2,#3]\n@rmeuc{\v@lX}{-1}\v@lX=\ptT@unit@\v@lX\xdef#1{\repdecn@mb{\v@lX}}%
    \resetc@ntr@l\et@tfiggetdist}\ignorespaces\fi}
\ctr@ld@f\def\Figg@tT#1{\c@ntr@lnum{#1}%
    {\expandafter\expandafter\expandafter\extr@ctT\csname\objc@de\endcsname:%
     \ifnum\B@@ltxt=\z@\ptn@me{#1}\else\csname\objc@de T\endcsname\fi}}
\ctr@ld@f\def\extr@ctT#1,#2,#3/#4:{\def\B@@ltxt{#3}}
\ctr@ld@f\def\Figg@tXY#1{\c@ntr@lnum{#1}%
    \expandafter\expandafter\expandafter\extr@ctC\csname\objc@de\endcsname:}
\ctr@ln@m\extr@ctC
\ctr@ld@f\def\extr@ctCDD#1/#2,#3,#4:{\v@lX=#2\v@lY=#3}
\ctr@ld@f\def\extr@ctCTD#1/#2,#3,#4:{\v@lX=#2\v@lY=#3\v@lZ=#4}
\ctr@ld@f\def\Figg@tXYa#1{\c@ntr@lnum{#1}%
    \expandafter\expandafter\expandafter\extr@ctCa\csname\objc@de\endcsname:}
\ctr@ln@m\extr@ctCa
\ctr@ld@f\def\extr@ctCaDD#1/#2,#3,#4:{\v@lXa=#2\v@lYa=#3}
\ctr@ld@f\def\extr@ctCaTD#1/#2,#3,#4:{\v@lXa=#2\v@lYa=#3\v@lZa=#4}
\ctr@ln@m\t@xt@
\ctr@ld@f\def\figinit#1{\t@stc@tcodech@nge\initpr@lim\Figinit@#1,:\initpss@ttings\ignorespaces}
\ctr@ld@f\def\Figinit@#1,#2:{\setunit@{#1}\def\t@xt@{#2}\ifx\t@xt@\empty\else\Figinit@@#2:\fi}
\ctr@ld@f\def\Figinit@@#1#2:{\if#12 \else\Figs@tproj{#1}\initTD@\fi}
\ctr@ln@w{newif}\ifTr@isDim
\ctr@ld@f\def\UnD@fined{UNDEFINED}
\ctr@ld@f\def\ifundefined#1{\expandafter\ifx\csname#1\endcsname\relax}
\ctr@ln@m\@utoFN
\ctr@ln@m\@utoFInDone
\ctr@ln@m\disob@unit
\ctr@ld@f\def\initpr@lim{\initb@undb@x\figsetmark{}\figsetptname{$A_{##1}$}\def\Sc@leFact{1}%
    \initDD@\figsetroundcoord{yes}\ps@critrue\expandafter\setupd@te\defaultupdate:%
    \edef\disob@unit{\UnD@fined}\edef\t@rgetpt{\UnD@fined}\gdef\@utoFInDone{1}\gdef\@utoFN{0}}
\ctr@ld@f\def\initDD@{\Tr@isDimfalse%
    \ifPDFm@ke%
     \let\Ps@rcerc=\Ps@rcercBz%
     \let\Ps@rell=\Ps@rellBz%
    \fi
    \let\c@lDCUn=\c@lDCUnDD%
    \let\c@lDCDeux=\c@lDCDeuxDD%
    \let\c@ldefproj=\relax%
    \let\c@lproscal=\c@lproscalDD%
    \let\c@lprojSP=\relax%
    \let\extr@ctC=\extr@ctCDD%
    \let\extr@ctCa=\extr@ctCaDD%
    \let\extr@ctCF=\extr@ctCFDD%
    \let\Figp@intreg=\Figp@intregDD%
    \let\Figpts@xes=\Figpts@xesDD%
    \let\n@rmeucSV=\n@rmeucSVDD\let\n@rmeuc=\n@rmeucDD\let\n@rminf=\n@rminfDD%
    \let\pr@dMatV=\pr@dMatVDD%
    \let\ps@xes=\ps@xesDD%
    \let\vecunit@=\vecunit@DD%
    \let\figcoord=\figcoordDD%
    \let\figgetangle=\figgetangleDD%
    \let\figpt=\figptDD%
    \let\figptBezier=\figptBezierDD%
    \let\figptbary=\figptbaryDD%
    \let\figptcirc=\figptcircDD%
    \let\figptcircumcenter=\figptcircumcenterDD%
    \let\figptcopy=\figptcopyDD%
    \let\figptcurvcenter=\figptcurvcenterDD%
    \let\figptell=\figptellDD%
    \let\figptendnormal=\figptendnormalDD%
    \let\figptinterlineplane=\figptinterlineplaneDD%
    \let\figptinterlines=\inters@cDD%
    \let\figptorthocenter=\figptorthocenterDD%
    \let\figptorthoprojline=\figptorthoprojlineDD%
    \let\figptorthoprojplane=\figptorthoprojplaneDD%
    \let\figptrot=\figptrotDD%
    \let\figptscontrol=\figptscontrolDD%
    \let\figptsintercirc=\figptsintercircDD%
    \let\figptsinterlinell=\figptsinterlinellDD%
    \let\figptsorthoprojline=\figptsorthoprojlineDD%
    \let\figptorthoprojplane=\figptorthoprojplaneDD%
    \let\figptsrot=\figptsrotDD%
    \let\figptssym=\figptssymDD%
    \let\figptstra=\figptstraDD%
    \let\figptsym=\figptsymDD%
    \let\figpttraC=\figpttraCDD%
    \let\figpttra=\figpttraDD%
    \let\figptvisilimSL=\figptvisilimSLDD%
    \let\figsetobdist=\figsetobdistDD%
    \let\figsettarget=\figsettargetDD%
    \let\figsetview=\figsetviewDD%
    \let\figvectDBezier=\figvectDBezierDD%
    \let\figvectN=\figvectNDD%
    \let\figvectNV=\figvectNVDD%
    \let\figvectP=\figvectPDD%
    \let\figvectU=\figvectUDD%
    \let\psarccircP=\psarccircPDD%
    \let\psarccirc=\psarccircDD%
    \let\psarcell=\psarcellDD%
    \let\psarcellPA=\psarcellPADD%
    \let\psarrowBezier=\psarrowBezierDD%
    \let\psarrowcircP=\psarrowcircPDD%
    \let\psarrowcirc=\psarrowcircDD%
    \let\psarrowhead=\psarrowheadDD%
    \let\psarrow=\psarrowDD%
    \let\psBezier=\psBezierDD%
    \let\pscirc=\pscircDD%
    \let\pscurve=\pscurveDD%
    \let\psnormal=\psnormalDD%
    }
\ctr@ld@f\def\initTD@{\Tr@isDimtrue\initb@undb@xTD\newt@rgetptfalse\newdis@bfalse%
    \let\c@lDCUn=\c@lDCUnTD%
    \let\c@lDCDeux=\c@lDCDeuxTD%
    \let\c@ldefproj=\c@ldefprojTD%
    \let\c@lproscal=\c@lproscalTD%
    \let\extr@ctC=\extr@ctCTD%
    \let\extr@ctCa=\extr@ctCaTD%
    \let\extr@ctCF=\extr@ctCFTD%
    \let\Figp@intreg=\Figp@intregTD%
    \let\Figpts@xes=\Figpts@xesTD%
    \let\n@rmeucSV=\n@rmeucSVTD\let\n@rmeuc=\n@rmeucTD\let\n@rminf=\n@rminfTD%
    \let\pr@dMatV=\pr@dMatVTD%
    \let\ps@xes=\ps@xesTD%
    \let\vecunit@=\vecunit@TD%
    \let\figcoord=\figcoordTD%
    \let\figgetangle=\figgetangleTD%
    \let\figpt=\figptTD%
    \let\figptBezier=\figptBezierTD%
    \let\figptbary=\figptbaryTD%
    \let\figptcirc=\figptcircTD%
    \let\figptcircumcenter=\figptcircumcenterTD%
    \let\figptcopy=\figptcopyTD%
    \let\figptcurvcenter=\figptcurvcenterTD%
    \let\figptinterlineplane=\figptinterlineplaneTD%
    \let\figptinterlines=\inters@cTD%
    \let\figptorthocenter=\figptorthocenterTD%
    \let\figptorthoprojline=\figptorthoprojlineTD%
    \let\figptorthoprojplane=\figptorthoprojplaneTD%
    \let\figptrot=\figptrotTD%
    \let\figptscontrol=\figptscontrolTD%
    \let\figptsintercirc=\figptsintercircTD%
    \let\figptsorthoprojline=\figptsorthoprojlineTD%
    \let\figptsorthoprojplane=\figptsorthoprojplaneTD%
    \let\figptsrot=\figptsrotTD%
    \let\figptssym=\figptssymTD%
    \let\figptstra=\figptstraTD%
    \let\figptsym=\figptsymTD%
    \let\figpttraC=\figpttraCTD%
    \let\figpttra=\figpttraTD%
    \let\figptvisilimSL=\figptvisilimSLTD%
    \let\figsetobdist=\figsetobdistTD%
    \let\figsettarget=\figsettargetTD%
    \let\figsetview=\figsetviewTD%
    \let\figvectDBezier=\figvectDBezierTD%
    \let\figvectN=\figvectNTD%
    \let\figvectNV=\figvectNVTD%
    \let\figvectP=\figvectPTD%
    \let\figvectU=\figvectUTD%
    \let\psarccircP=\psarccircPTD%
    \let\psarccirc=\psarccircTD%
    \let\psarcell=\psarcellTD%
    \let\psarcellPA=\psarcellPATD%
    \let\psarrowBezier=\psarrowBezierTD%
    \let\psarrowcircP=\psarrowcircPTD%
    \let\psarrowcirc=\psarrowcircTD%
    \let\psarrowhead=\psarrowheadTD%
    \let\psarrow=\psarrowTD%
    \let\psBezier=\psBezierTD%
    \let\pscirc=\pscircTD%
    \let\pscurve=\pscurveTD%
    }
\ctr@ld@f\def\un@v@ilable#1{\immediate\write16{*** The macro #1 is not available in the current context.}}
\ctr@ld@f\def\figinsert#1{{\def\t@xt@{#1}\relax%
    \ifx\t@xt@\empty\ifnum\@utoFInDone>\z@\Figinsert@\DefGIfilen@me,:\fi%
    \else\expandafter\FiginsertNu@#1 :\fi}\ignorespaces}
\ctr@ld@f\def\FiginsertNu@#1 #2:{\def\t@xt@{#1}\relax\ifx\t@xt@\empty\def\t@xt@{#2}%
    \ifx\t@xt@\empty\ifnum\@utoFInDone>\z@\Figinsert@\DefGIfilen@me,:\fi%
    \else\FiginsertNu@#2:\fi\else\expandafter\FiginsertNd@#1 #2:\fi}
\ctr@ld@f\def\FiginsertNd@#1#2:{\ifcat#1a\Figinsert@#1#2,:\else%
    \ifnum\@utoFInDone>\z@\Figinsert@\DefGIfilen@me,#1#2,:\fi\fi}
\ctr@ln@m\Sc@leFact
\ctr@ld@f\def\Figinsert@#1,#2:{\def\t@xt@{#2}\ifx\t@xt@\empty\xdef\Sc@leFact{1}\else%
    \X@rgdeux@#2\xdef\Sc@leFact{\@rgdeux}\fi%
    \Figdisc@rdLTS{#1}{\t@xt@}\@psfgetbb{\t@xt@}%
    \v@lX=\@psfllx\p@\v@lX=\ptpsT@pt\v@lX\v@lX=\Sc@leFact\v@lX%
    \v@lY=\@psflly\p@\v@lY=\ptpsT@pt\v@lY\v@lY=\Sc@leFact\v@lY%
    \b@undb@x{\v@lX}{\v@lY}%
    \v@lX=\@psfurx\p@\v@lX=\ptpsT@pt\v@lX\v@lX=\Sc@leFact\v@lX%
    \v@lY=\@psfury\p@\v@lY=\ptpsT@pt\v@lY\v@lY=\Sc@leFact\v@lY%
    \b@undb@x{\v@lX}{\v@lY}%
    \ifPDFm@ke\Figinclud@PDF{\t@xt@}{\Sc@leFact}\else%
    \v@lX=\c@nt pt\v@lX=\Sc@leFact\v@lX\edef\F@ct{\repdecn@mb{\v@lX}}%
    \ifx\TeXturesonMacOSltX\special{postscriptfile #1 vscale=\F@ct\space hscale=\F@ct}%
    \else\includegraphics{#1}\fi\fi%
    \message{[\t@xt@]}\ignorespaces}
\ctr@ld@f\def\Figdisc@rdLTS#1#2{\expandafter\Figdisc@rdLTS@#1 :#2}
\ctr@ld@f\def\Figdisc@rdLTS@#1 #2:#3{\def#3{#1}\relax\ifx#3\empty\expandafter\Figdisc@rdLTS@#2:#3\fi}
\ctr@ld@f\def\figinsertE#1{\FiginsertE@#1,:\ignorespaces}
\ctr@ld@f\def\FiginsertE@#1,#2:{{\def\t@xt@{#2}\ifx\t@xt@\empty\xdef\Sc@leFact{1}\else%
    \X@rgdeux@#2\xdef\Sc@leFact{\@rgdeux}\fi%
    \Figdisc@rdLTS{#1}{\t@xt@}\pdfximage{\t@xt@}%
    \setbox\Gb@x=\hbox{\pdfrefximage\pdflastximage}%
    \v@lX=\z@\v@lY=-\Sc@leFact\dp\Gb@x\b@undb@x{\v@lX}{\v@lY}%
    \advance\v@lX\Sc@leFact\wd\Gb@x\advance\v@lY\Sc@leFact\dp\Gb@x%
    \advance\v@lY\Sc@leFact\ht\Gb@x\b@undb@x{\v@lX}{\v@lY}%
    \v@lX=\Sc@leFact\wd\Gb@x\pdfximage width \v@lX {\t@xt@}%
    \rlap{\pdfrefximage\pdflastximage}\message{[\t@xt@]}}\ignorespaces}
\ctr@ld@f\def\X@rgdeux@#1,{\edef\@rgdeux{#1}}
\ctr@ln@m\figpt
\ctr@ld@f\def\figptDD#1:#2(#3,#4){\ifps@cri\c@ntr@lnum{#1}%
    {\v@lX=#3\unit@\v@lY=#4\unit@\Fig@dmpt{#2}{\z@}}\ignorespaces\fi}
\ctr@ld@f\def\Fig@dmpt#1#2{\def\t@xt@{#1}\ifx\t@xt@\empty\def\B@@ltxt{\z@}%
    \else\expandafter\gdef\csname\objc@de T\endcsname{#1}\def\B@@ltxt{\@ne}\fi%
    \expandafter\xdef\csname\objc@de\endcsname{\ifitis@vect@r\C@dCl@svect%
    \else\C@dCl@spt\fi,\z@,\B@@ltxt/\the\v@lX,\the\v@lY,#2}}
\ctr@ld@f\def\C@dCl@spt{P}
\ctr@ld@f\def\C@dCl@svect{V}
\ctr@ln@m\c@@rdYZ
\ctr@ln@m\c@@rdY
\ctr@ld@f\def\figptTD#1:#2(#3,#4){\ifps@cri\c@ntr@lnum{#1}%
    \def\c@@rdYZ{#4,0,0}\extrairelepremi@r\c@@rdY\de\c@@rdYZ%
    \extrairelepremi@r\c@@rdZ\de\c@@rdYZ%
    {\v@lX=#3\unit@\v@lY=\c@@rdY\unit@\v@lZ=\c@@rdZ\unit@\Fig@dmpt{#2}{\the\v@lZ}%
    \b@undb@xTD{\v@lX}{\v@lY}{\v@lZ}}\ignorespaces\fi}
\ctr@ln@m\Figp@intreg
\ctr@ld@f\def\Figp@intregDD#1:#2(#3,#4){\c@ntr@lnum{#1}%
    {\result@t=#4\v@lX=#3\v@lY=\result@t\Fig@dmpt{#2}{\z@}}\ignorespaces}
\ctr@ld@f\def\Figp@intregTD#1:#2(#3,#4){\c@ntr@lnum{#1}%
    \def\c@@rdYZ{#4,\z@,\z@}\extrairelepremi@r\c@@rdY\de\c@@rdYZ%
    \extrairelepremi@r\c@@rdZ\de\c@@rdYZ%
    {\v@lX=#3\v@lY=\c@@rdY\v@lZ=\c@@rdZ\Fig@dmpt{#2}{\the\v@lZ}%
    \b@undb@xTD{\v@lX}{\v@lY}{\v@lZ}}\ignorespaces}
\ctr@ln@m\figptBezier
\ctr@ld@f\def\figptBezierDD#1:#2:#3[#4,#5,#6,#7]{\ifps@cri{\s@uvc@ntr@l\et@tfigptBezierDD%
    \FigptBezier@#3[#4,#5,#6,#7]\Figp@intregDD#1:{#2}(\v@lX,\v@lY)%
    \resetc@ntr@l\et@tfigptBezierDD}\ignorespaces\fi}
\ctr@ld@f\def\figptBezierTD#1:#2:#3[#4,#5,#6,#7]{\ifps@cri{\s@uvc@ntr@l\et@tfigptBezierTD%
    \FigptBezier@#3[#4,#5,#6,#7]\Figp@intregTD#1:{#2}(\v@lX,\v@lY,\v@lZ)%
    \resetc@ntr@l\et@tfigptBezierTD}\ignorespaces\fi}
\ctr@ld@f\def\FigptBezier@#1[#2,#3,#4,#5]{\setc@ntr@l{2}%
    \edef\T@{#1}\v@leur=\p@\advance\v@leur-#1pt\edef\UNmT@{\repdecn@mb{\v@leur}}%
    \figptcopy-4:/#2/\figptcopy-3:/#3/\figptcopy-2:/#4/\figptcopy-1:/#5/%
    \l@mbd@un=-4 \l@mbd@de=-\thr@@\p@rtent=\m@ne\c@lDecast%
    \l@mbd@un=-4 \l@mbd@de=-\thr@@\p@rtent=-\tw@\c@lDecast%
    \l@mbd@un=-4 \l@mbd@de=-\thr@@\p@rtent=-\thr@@\c@lDecast\Figg@tXY{-4}}
\ctr@ln@m\c@lDCUn
\ctr@ld@f\def\c@lDCUnDD#1#2{\Figg@tXY{#1}\v@lX=\UNmT@\v@lX\v@lY=\UNmT@\v@lY%
    \Figg@tXYa{#2}\advance\v@lX\T@\v@lXa\advance\v@lY\T@\v@lYa%
    \Figp@intregDD#1:(\v@lX,\v@lY)}
\ctr@ld@f\def\c@lDCUnTD#1#2{\Figg@tXY{#1}\v@lX=\UNmT@\v@lX\v@lY=\UNmT@\v@lY\v@lZ=\UNmT@\v@lZ%
    \Figg@tXYa{#2}\advance\v@lX\T@\v@lXa\advance\v@lY\T@\v@lYa\advance\v@lZ\T@\v@lZa%
    \Figp@intregTD#1:(\v@lX,\v@lY,\v@lZ)}
\ctr@ld@f\def\c@lDecast{\relax\ifnum\l@mbd@un<\p@rtent\c@lDCUn{\l@mbd@un}{\l@mbd@de}%
    \advance\l@mbd@un\@ne\advance\l@mbd@de\@ne\c@lDecast\fi}
\ctr@ld@f\def\figptmap#1:#2=#3/#4/#5/{\ifps@cri{\s@uvc@ntr@l\et@tfigptmap%
    \setc@ntr@l{2}\figvectP-1[#4,#3]\Figg@tXY{-1}%
    \pr@dMatV/#5/\figpttra#1:{#2}=#4/1,-1/%
    \resetc@ntr@l\et@tfigptmap}\ignorespaces\fi}
\ctr@ln@m\pr@dMatV
\ctr@ld@f\def\pr@dMatVDD/#1,#2;#3,#4/{\v@lXa=#1\v@lX\advance\v@lXa#2\v@lY%
    \v@lYa=#3\v@lX\advance\v@lYa#4\v@lY\Figv@ctCreg-1(\v@lXa,\v@lYa)}
\ctr@ld@f\def\pr@dMatVTD/#1,#2,#3;#4,#5,#6;#7,#8,#9/{%
    \v@lXa=#1\v@lX\advance\v@lXa#2\v@lY\advance\v@lXa#3\v@lZ%
    \v@lYa=#4\v@lX\advance\v@lYa#5\v@lY\advance\v@lYa#6\v@lZ%
    \v@lZa=#7\v@lX\advance\v@lZa#8\v@lY\advance\v@lZa#9\v@lZ%
    \Figv@ctCreg-1(\v@lXa,\v@lYa,\v@lZa)}
\ctr@ln@m\figptbary
\ctr@ld@f\def\figptbaryDD#1:#2[#3;#4]{\ifps@cri{\edef\list@num{#3}\extrairelepremi@r\p@int\de\list@num%
    \s@mme=\z@\@ecfor\c@ef:=#4\do{\advance\s@mme\c@ef}%
    \edef\listec@ef{#4,0}\extrairelepremi@r\c@ef\de\listec@ef%
    \Figg@tXY{\p@int}\divide\v@lX\s@mme\divide\v@lY\s@mme%
    \multiply\v@lX\c@ef\multiply\v@lY\c@ef%
    \@ecfor\p@int:=\list@num\do{\extrairelepremi@r\c@ef\de\listec@ef%
           \Figg@tXYa{\p@int}\divide\v@lXa\s@mme\divide\v@lYa\s@mme%
           \multiply\v@lXa\c@ef\multiply\v@lYa\c@ef%
           \advance\v@lX\v@lXa\advance\v@lY\v@lYa}%
    \Figp@intregDD#1:{#2}(\v@lX,\v@lY)}\ignorespaces\fi}
\ctr@ld@f\def\figptbaryTD#1:#2[#3;#4]{\ifps@cri{\edef\list@num{#3}\extrairelepremi@r\p@int\de\list@num%
    \s@mme=\z@\@ecfor\c@ef:=#4\do{\advance\s@mme\c@ef}%
    \edef\listec@ef{#4,0}\extrairelepremi@r\c@ef\de\listec@ef%
    \Figg@tXY{\p@int}\divide\v@lX\s@mme\divide\v@lY\s@mme\divide\v@lZ\s@mme%
    \multiply\v@lX\c@ef\multiply\v@lY\c@ef\multiply\v@lZ\c@ef%
    \@ecfor\p@int:=\list@num\do{\extrairelepremi@r\c@ef\de\listec@ef%
           \Figg@tXYa{\p@int}\divide\v@lXa\s@mme\divide\v@lYa\s@mme\divide\v@lZa\s@mme%
           \multiply\v@lXa\c@ef\multiply\v@lYa\c@ef\multiply\v@lZa\c@ef%
           \advance\v@lX\v@lXa\advance\v@lY\v@lYa\advance\v@lZ\v@lZa}%
    \Figp@intregTD#1:{#2}(\v@lX,\v@lY,\v@lZ)}\ignorespaces\fi}
\ctr@ld@f\def\figptbaryR#1:#2[#3;#4]{\ifps@cri{%
    \v@leur=\z@\@ecfor\c@ef:=#4\do{\maxim@m{\v@lmax}{\c@ef pt}{-\c@ef pt}%
    \ifdim\v@lmax>\v@leur\v@leur=\v@lmax\fi}%
    \ifdim\v@leur<\p@\f@ctech=\@M\else\ifdim\v@leur<\t@n\p@\f@ctech=\@m\else%
    \ifdim\v@leur<\c@nt\p@\f@ctech=\c@nt\else\ifdim\v@leur<\@m\p@\f@ctech=\t@n\else%
    \f@ctech=\@ne\fi\fi\fi\fi%
    \def\listec@ef{0}%
    \@ecfor\c@ef:=#4\do{\sc@lec@nvRI{\c@ef pt}\edef\listec@ef{\listec@ef,\the\s@mme}}%
    \extrairelepremi@r\c@ef\de\listec@ef\figptbary#1:#2[#3;\listec@ef]}\ignorespaces\fi}
\ctr@ld@f\def\sc@lec@nvRI#1{\v@leur=#1\p@rtentiere{\s@mme}{\v@leur}\advance\v@leur-\s@mme\p@%
    \multiply\v@leur\f@ctech\p@rtentiere{\p@rtent}{\v@leur}%
    \multiply\s@mme\f@ctech\advance\s@mme\p@rtent}
\ctr@ln@m\figptcirc
\ctr@ld@f\def\figptcircDD#1:#2:#3;#4(#5){\ifps@cri{\s@uvc@ntr@l\et@tfigptcircDD%
    \c@lptellDD#1:{#2}:#3;#4,#4(#5)\resetc@ntr@l\et@tfigptcircDD}\ignorespaces\fi}
\ctr@ld@f\def\figptcircTD#1:#2:#3,#4,#5;#6(#7){\ifps@cri{\s@uvc@ntr@l\et@tfigptcircTD%
    \setc@ntr@l{2}\c@lExtAxes#3,#4,#5(#6)\figptellP#1:{#2}:#3,-4,-5(#7)%
    \resetc@ntr@l\et@tfigptcircTD}\ignorespaces\fi}
\ctr@ln@m\figptcircumcenter
\ctr@ld@f\def\figptcircumcenterDD#1:#2[#3,#4,#5]{\ifps@cri{\s@uvc@ntr@l\et@tfigptcircumcenterDD%
    \setc@ntr@l{2}\figvectNDD-5[#3,#4]\figptbaryDD-3:[#3,#4;1,1]%
                  \figvectNDD-6[#4,#5]\figptbaryDD-4:[#4,#5;1,1]%
    \resetc@ntr@l{2}\inters@cDD#1:{#2}[-3,-5;-4,-6]%
    \resetc@ntr@l\et@tfigptcircumcenterDD}\ignorespaces\fi}
\ctr@ld@f\def\figptcircumcenterTD#1:#2[#3,#4,#5]{\ifps@cri{\s@uvc@ntr@l\et@tfigptcircumcenterTD%
    \setc@ntr@l{2}\figvectNTD-1[#3,#4,#5]%
    \figvectPTD-3[#3,#4]\figvectNVTD-5[-1,-3]\figptbaryTD-3:[#3,#4;1,1]%
    \figvectPTD-4[#4,#5]\figvectNVTD-6[-1,-4]\figptbaryTD-4:[#4,#5;1,1]%
    \resetc@ntr@l{2}\inters@cTD#1:{#2}[-3,-5;-4,-6]%
    \resetc@ntr@l\et@tfigptcircumcenterTD}\ignorespaces\fi}
\ctr@ln@m\figptcopy
\ctr@ld@f\def\figptcopyDD#1:#2/#3/{\ifps@cri{\Figg@tXY{#3}%
    \Figp@intregDD#1:{#2}(\v@lX,\v@lY)}\ignorespaces\fi}
\ctr@ld@f\def\figptcopyTD#1:#2/#3/{\ifps@cri{\Figg@tXY{#3}%
    \Figp@intregTD#1:{#2}(\v@lX,\v@lY,\v@lZ)}\ignorespaces\fi}
\ctr@ln@m\figptcurvcenter
\ctr@ld@f\def\figptcurvcenterDD#1:#2:#3[#4,#5,#6,#7]{\ifps@cri{\s@uvc@ntr@l\et@tfigptcurvcenterDD%
    \setc@ntr@l{2}\c@lcurvradDD#3[#4,#5,#6,#7]\edef\Sprim@{\repdecn@mb{\result@t}}%
    \figptBezierDD-1::#3[#4,#5,#6,#7]\figpttraDD#1:{#2}=-1/\Sprim@,-5/%
    \resetc@ntr@l\et@tfigptcurvcenterDD}\ignorespaces\fi}
\ctr@ld@f\def\figptcurvcenterTD#1:#2:#3[#4,#5,#6,#7]{\ifps@cri{\s@uvc@ntr@l\et@tfigptcurvcenterTD%
    \setc@ntr@l{2}\figvectDBezierTD -5:1,#3[#4,#5,#6,#7]%
    \figvectDBezierTD -6:2,#3[#4,#5,#6,#7]\vecunit@TD{-5}{-5}%
    \edef\Sprim@{\repdecn@mb{\result@t}}\figvectNVTD-1[-6,-5]%
    \figvectNVTD-5[-5,-1]\c@lproscalTD\v@leur[-6,-5]%
    \invers@{\v@leur}{\v@leur}\v@leur=\Sprim@\v@leur\v@leur=\Sprim@\v@leur%
    \figptBezierTD-1::#3[#4,#5,#6,#7]\edef\Sprim@{\repdecn@mb{\v@leur}}%
    \figpttraTD#1:{#2}=-1/\Sprim@,-5/\resetc@ntr@l\et@tfigptcurvcenterTD}\ignorespaces\fi}
\ctr@ld@f\def\c@lcurvradDD#1[#2,#3,#4,#5]{{\figvectDBezierDD -5:1,#1[#2,#3,#4,#5]%
    \figvectDBezierDD -6:2,#1[#2,#3,#4,#5]\vecunit@DD{-5}{-5}%
    \edef\Sprim@{\repdecn@mb{\result@t}}\figvectNVDD-5[-5]\c@lproscalDD\v@leur[-6,-5]%
    \invers@{\v@leur}{\v@leur}\v@leur=\Sprim@\v@leur\v@leur=\Sprim@\v@leur%
    \global\result@t=\v@leur}}
\ctr@ln@m\figptell
\ctr@ld@f\def\figptellDD#1:#2:#3;#4,#5(#6,#7){\ifps@cri{\s@uvc@ntr@l\et@tfigptell%
    \c@lptellDD#1::#3;#4,#5(#6)\figptrotDD#1:{#2}=#1/#3,#7/%
    \resetc@ntr@l\et@tfigptell}\ignorespaces\fi}
\ctr@ld@f\def\c@lptellDD#1:#2:#3;#4,#5(#6){\c@ssin{\C@}{\S@}{#6}\v@lmin=\C@ pt\v@lmax=\S@ pt%
    \v@lmin=#4\v@lmin\v@lmax=#5\v@lmax%
    \edef\Xc@mp{\repdecn@mb{\v@lmin}}\edef\Yc@mp{\repdecn@mb{\v@lmax}}%
    \setc@ntr@l{2}\figvectC-1(\Xc@mp,\Yc@mp)\figpttraDD#1:{#2}=#3/1,-1/}
\ctr@ld@f\def\figptellP#1:#2:#3,#4,#5(#6){\ifps@cri{\s@uvc@ntr@l\et@tfigptellP%
    \setc@ntr@l{2}\figvectP-1[#3,#4]\figvectP-2[#3,#5]%
    \v@leur=#6pt\c@lptellP{#3}{-1}{-2}\figptcopy#1:{#2}/-3/%
    \resetc@ntr@l\et@tfigptellP}\ignorespaces\fi}
\ctr@ln@m\@ngle
\ctr@ld@f\def\c@lptellP#1#2#3{\edef\@ngle{\repdecn@mb\v@leur}\c@ssin{\C@}{\S@}{\@ngle}%
    \figpttra-3:=#1/\C@,#2/\figpttra-3:=-3/\S@,#3/}
\ctr@ln@m\figptendnormal
\ctr@ld@f\def\figptendnormalDD#1:#2:#3,#4[#5,#6]{\ifps@cri{\s@uvc@ntr@l\et@tfigptendnormal%
    \Figg@tXYa{#5}\Figg@tXY{#6}%
    \advance\v@lX-\v@lXa\advance\v@lY-\v@lYa%
    \setc@ntr@l{2}\Figv@ctCreg-1(\v@lX,\v@lY)\vecunit@{-1}{-1}\Figg@tXY{-1}%
    \delt@=#3\unit@\maxim@m{\delt@}{\delt@}{-\delt@}\edef\l@ngueur{\repdecn@mb{\delt@}}%
    \v@lX=\l@ngueur\v@lX\v@lY=\l@ngueur\v@lY%
    \delt@=\p@\advance\delt@-#4pt\edef\l@ngueur{\repdecn@mb{\delt@}}%
    \figptbaryR-1:[#5,#6;#4,\l@ngueur]\Figg@tXYa{-1}%
    \advance\v@lXa\v@lY\advance\v@lYa-\v@lX%
    \setc@ntr@l{1}\Figp@intregDD#1:{#2}(\v@lXa,\v@lYa)\resetc@ntr@l\et@tfigptendnormal}%
    \ignorespaces\fi}
\ctr@ld@f\def\figptexcenter#1:#2[#3,#4,#5]{\ifps@cri{\let@xte={-}%
    \Figptexinsc@nter#1:#2[#3,#4,#5]}\ignorespaces\fi}
\ctr@ld@f\def\figptincenter#1:#2[#3,#4,#5]{\ifps@cri{\let@xte={}%
    \Figptexinsc@nter#1:#2[#3,#4,#5]}\ignorespaces\fi}
\ctr@ld@f\let\figptinscribedcenter=\figptincenter
\ctr@ld@f\def\Figptexinsc@nter#1:#2[#3,#4,#5]{%
    \figgetdist\LA@[#4,#5]\figgetdist\LB@[#3,#5]\figgetdist\LC@[#3,#4]%
    \figptbaryR#1:{#2}[#3,#4,#5;\the\let@xte\LA@,\LB@,\LC@]}
\ctr@ln@m\figptinterlineplane
\ctr@ld@f\def\figptinterlineplaneDD{\un@v@ilable{figptinterlineplane}}
\ctr@ld@f\def\figptinterlineplaneTD#1:#2[#3,#4;#5,#6]{\ifps@cri{\s@uvc@ntr@l\et@tfigptinterlineplane%
    \setc@ntr@l{2}\figvectPTD-1[#3,#5]\vecunit@TD{-2}{#6}%
    \r@pPSTD\v@leur[-2,-1,#4]\edef\v@lcoef{\repdecn@mb{\v@leur}}%
    \figpttraTD#1:{#2}=#3/\v@lcoef,#4/\resetc@ntr@l\et@tfigptinterlineplane}\ignorespaces\fi}
\ctr@ln@m\figptorthocenter
\ctr@ld@f\def\figptorthocenterDD#1:#2[#3,#4,#5]{\ifps@cri{\s@uvc@ntr@l\et@tfigptorthocenterDD%
    \setc@ntr@l{2}\figvectNDD-3[#3,#4]\figvectNDD-4[#4,#5]%
    \resetc@ntr@l{2}\inters@cDD#1:{#2}[#5,-3;#3,-4]%
    \resetc@ntr@l\et@tfigptorthocenterDD}\ignorespaces\fi}
\ctr@ld@f\def\figptorthocenterTD#1:#2[#3,#4,#5]{\ifps@cri{\s@uvc@ntr@l\et@tfigptorthocenterTD%
    \setc@ntr@l{2}\figvectNTD-1[#3,#4,#5]%
    \figvectPTD-2[#3,#4]\figvectNVTD-3[-1,-2]%
    \figvectPTD-2[#4,#5]\figvectNVTD-4[-1,-2]%
    \resetc@ntr@l{2}\inters@cTD#1:{#2}[#5,-3;#3,-4]%
    \resetc@ntr@l\et@tfigptorthocenterTD}\ignorespaces\fi}
\ctr@ln@m\figptorthoprojline
\ctr@ld@f\def\figptorthoprojlineDD#1:#2=#3/#4,#5/{\ifps@cri{\s@uvc@ntr@l\et@tfigptorthoprojlineDD%
    \setc@ntr@l{2}\figvectPDD-3[#4,#5]\figvectNVDD-4[-3]\resetc@ntr@l{2}%
    \inters@cDD#1:{#2}[#3,-4;#4,-3]\resetc@ntr@l\et@tfigptorthoprojlineDD}\ignorespaces\fi}
\ctr@ld@f\def\figptorthoprojlineTD#1:#2=#3/#4,#5/{\ifps@cri{\s@uvc@ntr@l\et@tfigptorthoprojlineTD%
    \setc@ntr@l{2}\figvectPTD-1[#4,#3]\figvectPTD-2[#4,#5]\vecunit@TD{-2}{-2}%
    \c@lproscalTD\v@leur[-1,-2]\edef\v@lcoef{\repdecn@mb{\v@leur}}%
    \figpttraTD#1:{#2}=#4/\v@lcoef,-2/\resetc@ntr@l\et@tfigptorthoprojlineTD}\ignorespaces\fi}
\ctr@ln@m\figptorthoprojplane
\ctr@ld@f\def\figptorthoprojplaneDD{\un@v@ilable{figptorthoprojplane}}
\ctr@ld@f\def\figptorthoprojplaneTD#1:#2=#3/#4,#5/{\ifps@cri{\s@uvc@ntr@l\et@tfigptorthoprojplane%
    \setc@ntr@l{2}\figvectPTD-1[#3,#4]\vecunit@TD{-2}{#5}%
    \c@lproscalTD\v@leur[-1,-2]\edef\v@lcoef{\repdecn@mb{\v@leur}}%
    \figpttraTD#1:{#2}=#3/\v@lcoef,-2/\resetc@ntr@l\et@tfigptorthoprojplane}\ignorespaces\fi}
\ctr@ld@f\def\figpthom#1:#2=#3/#4,#5/{\ifps@cri{\s@uvc@ntr@l\et@tfigpthom%
    \setc@ntr@l{2}\figvectP-1[#4,#3]\figpttra#1:{#2}=#4/#5,-1/%
    \resetc@ntr@l\et@tfigpthom}\ignorespaces\fi}
\ctr@ln@m\figptrot
\ctr@ld@f\def\figptrotDD#1:#2=#3/#4,#5/{\ifps@cri{\s@uvc@ntr@l\et@tfigptrotDD%
    \c@ssin{\C@}{\S@}{#5}\setc@ntr@l{2}\figvectPDD-1[#4,#3]\Figg@tXY{-1}%
    \v@lXa=\C@\v@lX\advance\v@lXa-\S@\v@lY%
    \v@lYa=\S@\v@lX\advance\v@lYa\C@\v@lY%
    \Figv@ctCreg-1(\v@lXa,\v@lYa)\figpttraDD#1:{#2}=#4/1,-1/%
    \resetc@ntr@l\et@tfigptrotDD}\ignorespaces\fi}
\ctr@ld@f\def\figptrotTD#1:#2=#3/#4,#5,#6/{\ifps@cri{\s@uvc@ntr@l\et@tfigptrotTD%
    \c@ssin{\C@}{\S@}{#5}%
    \setc@ntr@l{2}\figptorthoprojplaneTD-3:=#4/#3,#6/\figvectPTD-2[-3,#3]%
    \n@rmeucTD\v@leur{-2}\ifdim\v@leur<\Cepsil@n\Figg@tXYa{#3}\else%
    \edef\v@lcoef{\repdecn@mb{\v@leur}}\figvectNVTD-1[#6,-2]%
    \Figg@tXYa{-1}\v@lXa=\v@lcoef\v@lXa\v@lYa=\v@lcoef\v@lYa\v@lZa=\v@lcoef\v@lZa%
    \v@lXa=\S@\v@lXa\v@lYa=\S@\v@lYa\v@lZa=\S@\v@lZa\Figg@tXY{-2}%
    \advance\v@lXa\C@\v@lX\advance\v@lYa\C@\v@lY\advance\v@lZa\C@\v@lZ%
    \Figg@tXY{-3}\advance\v@lXa\v@lX\advance\v@lYa\v@lY\advance\v@lZa\v@lZ\fi%
    \Figp@intregTD#1:{#2}(\v@lXa,\v@lYa,\v@lZa)\resetc@ntr@l\et@tfigptrotTD}\ignorespaces\fi}
\ctr@ln@m\figptsym
\ctr@ld@f\def\figptsymDD#1:#2=#3/#4,#5/{\ifps@cri{\s@uvc@ntr@l\et@tfigptsymDD%
    \resetc@ntr@l{2}\figptorthoprojlineDD-5:=#3/#4,#5/\figvectPDD-2[#3,-5]%
    \figpttraDD#1:{#2}=#3/2,-2/\resetc@ntr@l\et@tfigptsymDD}\ignorespaces\fi}
\ctr@ld@f\def\figptsymTD#1:#2=#3/#4,#5/{\ifps@cri{\s@uvc@ntr@l\et@tfigptsymTD%
    \resetc@ntr@l{2}\figptorthoprojplaneTD-3:=#3/#4,#5/\figvectPTD-2[#3,-3]%
    \figpttraTD#1:{#2}=#3/2,-2/\resetc@ntr@l\et@tfigptsymTD}\ignorespaces\fi}
\ctr@ln@m\figpttra
\ctr@ld@f\def\figpttraDD#1:#2=#3/#4,#5/{\ifps@cri{\Figg@tXYa{#5}\v@lXa=#4\v@lXa\v@lYa=#4\v@lYa%
    \Figg@tXY{#3}\advance\v@lX\v@lXa\advance\v@lY\v@lYa%
    \Figp@intregDD#1:{#2}(\v@lX,\v@lY)}\ignorespaces\fi}
\ctr@ld@f\def\figpttraTD#1:#2=#3/#4,#5/{\ifps@cri{\Figg@tXYa{#5}\v@lXa=#4\v@lXa\v@lYa=#4\v@lYa%
    \v@lZa=#4\v@lZa\Figg@tXY{#3}\advance\v@lX\v@lXa\advance\v@lY\v@lYa%
    \advance\v@lZ\v@lZa\Figp@intregTD#1:{#2}(\v@lX,\v@lY,\v@lZ)}\ignorespaces\fi}
\ctr@ln@m\figpttraC
\ctr@ld@f\def\figpttraCDD#1:#2=#3/#4,#5/{\ifps@cri{\v@lXa=#4\unit@\v@lYa=#5\unit@%
    \Figg@tXY{#3}\advance\v@lX\v@lXa\advance\v@lY\v@lYa%
    \Figp@intregDD#1:{#2}(\v@lX,\v@lY)}\ignorespaces\fi}
\ctr@ld@f\def\figpttraCTD#1:#2=#3/#4,#5,#6/{\ifps@cri{\v@lXa=#4\unit@\v@lYa=#5\unit@\v@lZa=#6\unit@%
    \Figg@tXY{#3}\advance\v@lX\v@lXa\advance\v@lY\v@lYa\advance\v@lZ\v@lZa%
    \Figp@intregTD#1:{#2}(\v@lX,\v@lY,\v@lZ)}\ignorespaces\fi}
\ctr@ld@f\def\figptsaxes#1:#2(#3){\ifps@cri{\an@lys@xes#3,:\ifx\t@xt@\empty%
    \ifTr@isDim\Figpts@xes#1:#2(0,#3,0,#3,0,#3)\else\Figpts@xes#1:#2(0,#3,0,#3)\fi%
    \else\Figpts@xes#1:#2(#3)\fi}\ignorespaces\fi}
\ctr@ln@m\Figpts@xes
\ctr@ld@f\def\Figpts@xesDD#1:#2(#3,#4,#5,#6){%
    \s@mme=#1\figpttraC\the\s@mme:$x$=#2/#4,0/%
    \advance\s@mme\@ne\figpttraC\the\s@mme:$y$=#2/0,#6/}
\ctr@ld@f\def\Figpts@xesTD#1:#2(#3,#4,#5,#6,#7,#8){%
    \s@mme=#1\figpttraC\the\s@mme:$x$=#2/#4,0,0/%
    \advance\s@mme\@ne\figpttraC\the\s@mme:$y$=#2/0,#6,0/%
    \advance\s@mme\@ne\figpttraC\the\s@mme:$z$=#2/0,0,#8/}
\ctr@ld@f\def\figptsmap#1=#2/#3/#4/{\ifps@cri{\s@uvc@ntr@l\et@tfigptsmap%
    \setc@ntr@l{2}\def\list@num{#2}\s@mme=#1%
    \@ecfor\p@int:=\list@num\do{\figvectP-1[#3,\p@int]\Figg@tXY{-1}%
    \pr@dMatV/#4/\figpttra\the\s@mme:=#3/1,-1/\advance\s@mme\@ne}%
    \resetc@ntr@l\et@tfigptsmap}\ignorespaces\fi}
\ctr@ln@m\figptscontrol
\ctr@ld@f\def\figptscontrolDD#1[#2,#3,#4,#5]{\ifps@cri{\s@uvc@ntr@l\et@tfigptscontrolDD\setc@ntr@l{2}%
    \v@lX=\z@\v@lY=\z@\Figtr@nptDD{-5}{#2}\Figtr@nptDD{2}{#5}%
    \divide\v@lX\@vi\divide\v@lY\@vi%
    \Figtr@nptDD{3}{#3}\Figtr@nptDD{-1.5}{#4}\Figp@intregDD-1:(\v@lX,\v@lY)%
    \v@lX=\z@\v@lY=\z@\Figtr@nptDD{2}{#2}\Figtr@nptDD{-5}{#5}%
    \divide\v@lX\@vi\divide\v@lY\@vi\Figtr@nptDD{-1.5}{#3}\Figtr@nptDD{3}{#4}%
    \s@mme=#1\advance\s@mme\@ne\Figp@intregDD\the\s@mme:(\v@lX,\v@lY)%
    \figptcopyDD#1:/-1/\resetc@ntr@l\et@tfigptscontrolDD}\ignorespaces\fi}
\ctr@ld@f\def\figptscontrolTD#1[#2,#3,#4,#5]{\ifps@cri{\s@uvc@ntr@l\et@tfigptscontrolTD\setc@ntr@l{2}%
    \v@lX=\z@\v@lY=\z@\v@lZ=\z@\Figtr@nptTD{-5}{#2}\Figtr@nptTD{2}{#5}%
    \divide\v@lX\@vi\divide\v@lY\@vi\divide\v@lZ\@vi%
    \Figtr@nptTD{3}{#3}\Figtr@nptTD{-1.5}{#4}\Figp@intregTD-1:(\v@lX,\v@lY,\v@lZ)%
    \v@lX=\z@\v@lY=\z@\v@lZ=\z@\Figtr@nptTD{2}{#2}\Figtr@nptTD{-5}{#5}%
    \divide\v@lX\@vi\divide\v@lY\@vi\divide\v@lZ\@vi\Figtr@nptTD{-1.5}{#3}\Figtr@nptTD{3}{#4}%
    \s@mme=#1\advance\s@mme\@ne\Figp@intregTD\the\s@mme:(\v@lX,\v@lY,\v@lZ)%
    \figptcopyTD#1:/-1/\resetc@ntr@l\et@tfigptscontrolTD}\ignorespaces\fi}
\ctr@ld@f\def\Figtr@nptDD#1#2{\Figg@tXYa{#2}\v@lXa=#1\v@lXa\v@lYa=#1\v@lYa%
    \advance\v@lX\v@lXa\advance\v@lY\v@lYa}
\ctr@ld@f\def\Figtr@nptTD#1#2{\Figg@tXYa{#2}\v@lXa=#1\v@lXa\v@lYa=#1\v@lYa\v@lZa=#1\v@lZa%
    \advance\v@lX\v@lXa\advance\v@lY\v@lYa\advance\v@lZ\v@lZa}
\ctr@ld@f\def\figptscontrolcurve#1,#2[#3]{\ifps@cri{\s@uvc@ntr@l\et@tfigptscontrolcurve%
    \def\list@num{#3}\extrairelepremi@r\Ak@\de\list@num%
    \extrairelepremi@r\Ai@\de\list@num\extrairelepremi@r\Aj@\de\list@num%
    \s@mme=#1\figptcopy\the\s@mme:/\Ai@/%
    \setc@ntr@l{2}\figvectP -1[\Ak@,\Aj@]%
    \@ecfor\Ak@:=\list@num\do{\advance\s@mme\@ne\figpttra\the\s@mme:=\Ai@/\curv@roundness,-1/%
       \figvectP -1[\Ai@,\Ak@]\advance\s@mme\@ne\figpttra\the\s@mme:=\Aj@/-\curv@roundness,-1/%
       \advance\s@mme\@ne\figptcopy\the\s@mme:/\Aj@/%
       \edef\Ai@{\Aj@}\edef\Aj@{\Ak@}}\advance\s@mme-#1\divide\s@mme\thr@@%
       \xdef#2{\the\s@mme}%
    \resetc@ntr@l\et@tfigptscontrolcurve}\ignorespaces\fi}
\ctr@ln@m\figptsintercirc
\ctr@ld@f\def\figptsintercircDD#1[#2,#3;#4,#5]{\ifps@cri{\s@uvc@ntr@l\et@tfigptsintercircDD%
    \setc@ntr@l{2}\let\c@lNVintc=\c@lNVintcDD\Figptsintercirc@#1[#2,#3;#4,#5]%
    \resetc@ntr@l\et@tfigptsintercircDD}\ignorespaces\fi}
\ctr@ld@f\def\figptsintercircTD#1[#2,#3;#4,#5;#6]{\ifps@cri{\s@uvc@ntr@l\et@tfigptsintercircTD%
    \setc@ntr@l{2}\let\c@lNVintc=\c@lNVintcTD\vecunitC@TD[#2,#6]%
    \Figv@ctCreg-3(\v@lX,\v@lY,\v@lZ)\Figptsintercirc@#1[#2,#3;#4,#5]%
    \resetc@ntr@l\et@tfigptsintercircTD}\ignorespaces\fi}
\ctr@ld@f\def\Figptsintercirc@#1[#2,#3;#4,#5]{\figvectP-1[#2,#4]%
    \vecunit@{-1}{-1}\delt@=\result@t\f@ctech=\result@tent%
    \s@mme=#1\advance\s@mme\@ne\figptcopy#1:/#2/\figptcopy\the\s@mme:/#4/%
    \ifdim\delt@=\z@\else%
    \v@lmin=#3\unit@\v@lmax=#5\unit@\v@leur=\v@lmin\advance\v@leur\v@lmax%
    \ifdim\v@leur>\delt@%
    \v@leur=\v@lmin\advance\v@leur-\v@lmax\maxim@m{\v@leur}{\v@leur}{-\v@leur}%
    \ifdim\v@leur<\delt@%
    \divide\v@lmin\f@ctech\divide\v@lmax\f@ctech\divide\delt@\f@ctech%
    \v@lmin=\repdecn@mb{\v@lmin}\v@lmin\v@lmax=\repdecn@mb{\v@lmax}\v@lmax%
    \invers@{\v@leur}{\delt@}\advance\v@lmax-\v@lmin%
    \v@lmax=-\repdecn@mb{\v@leur}\v@lmax\advance\delt@\v@lmax\delt@=.5\delt@%
    \v@lmax=\delt@\multiply\v@lmax\f@ctech%
    \edef\t@ille{\repdecn@mb{\v@lmax}}\figpttra-2:=#2/\t@ille,-1/%
    \delt@=\repdecn@mb{\delt@}\delt@\advance\v@lmin-\delt@%
    \sqrt@{\v@leur}{\v@lmin}\multiply\v@leur\f@ctech\edef\t@ille{\repdecn@mb{\v@leur}}%
    \c@lNVintc\figpttra#1:=-2/-\t@ille,-1/\figpttra\the\s@mme:=-2/\t@ille,-1/\fi\fi\fi}
\ctr@ld@f\def\c@lNVintcDD{\Figg@tXY{-1}\Figv@ctCreg-1(-\v@lY,\v@lX)} 
\ctr@ld@f\def\c@lNVintcTD{{\Figg@tXY{-3}\v@lmin=\v@lX\v@lmax=\v@lY\v@leur=\v@lZ%
    \Figg@tXY{-1}\c@lprovec{-3}\vecunit@{-3}{-3}
    \Figg@tXY{-1}\v@lmin=\v@lX\v@lmax=\v@lY%
    \v@leur=\v@lZ\Figg@tXY{-3}\c@lprovec{-1}}} 
\ctr@ln@m\figptsinterlinell
\ctr@ld@f\def\figptsinterlinellDD#1[#2,#3,#4,#5;#6,#7]{\ifps@cri{\s@uvc@ntr@l\et@tfigptsinterlinellDD%
    \figptcopy#1:/#6/\s@mme=#1\advance\s@mme\@ne\figptcopy\the\s@mme:/#7/%
    \v@lmin=#3\unit@\v@lmax=#4\unit@
    \setc@ntr@l{2}\figptbaryDD-4:[#6,#7;1,1]\figptsrotDD-3=-4,#7/#2,-#5/
    \Figg@tXY{-3}\Figg@tXYa{#2}\advance\v@lX-\v@lXa\advance\v@lY-\v@lYa
    \figvectP-1[-3,-2]\Figg@tXYa{-1}\figvectP-3[-4,#7]\Figptsint@rLE{#1}
    \resetc@ntr@l\et@tfigptsinterlinellDD}\ignorespaces\fi}
\ctr@ld@f\def\figptsinterlinellP#1[#2,#3,#4;#5,#6]{\ifps@cri{\s@uvc@ntr@l\et@tfigptsinterlinellP%
    \figptcopy#1:/#5/\s@mme=#1\advance\s@mme\@ne\figptcopy\the\s@mme:/#6/\setc@ntr@l{2}%
    \figvectP-1[#2,#3]\vecunit@{-1}{-1}\v@lmin=\result@t
    \figvectP-2[#2,#4]\vecunit@{-2}{-2}\v@lmax=\result@t
    \figptbary-4:[#5,#6;1,1]
    \figvectP-3[#2,-4]\c@lproscal\v@lX[-3,-1]\c@lproscal\v@lY[-3,-2]
    \figvectP-3[-4,#6]\c@lproscal\v@lXa[-3,-1]\c@lproscal\v@lYa[-3,-2]
    \Figptsint@rLE{#1}\resetc@ntr@l\et@tfigptsinterlinellP}\ignorespaces\fi}
\ctr@ld@f\def\Figptsint@rLE#1{%
    \getredf@ctDD\f@ctech(\v@lmin,\v@lmax)%
    \getredf@ctDD\p@rtent(\v@lX,\v@lY)\ifnum\p@rtent>\f@ctech\f@ctech=\p@rtent\fi%
    \getredf@ctDD\p@rtent(\v@lXa,\v@lYa)\ifnum\p@rtent>\f@ctech\f@ctech=\p@rtent\fi%
    \divide\v@lmin\f@ctech\divide\v@lmax\f@ctech\divide\v@lX\f@ctech\divide\v@lY\f@ctech%
    \divide\v@lXa\f@ctech\divide\v@lYa\f@ctech%
    \c@rre=\repdecn@mb\v@lXa\v@lmax\mili@u=\repdecn@mb\v@lYa\v@lmin%
    \getredf@ctDD\f@ctech(\c@rre,\mili@u)%
    \c@rre=\repdecn@mb\v@lX\v@lmax\mili@u=\repdecn@mb\v@lY\v@lmin%
    \getredf@ctDD\p@rtent(\c@rre,\mili@u)\ifnum\p@rtent>\f@ctech\f@ctech=\p@rtent\fi%
    \divide\v@lmin\f@ctech\divide\v@lmax\f@ctech\divide\v@lX\f@ctech\divide\v@lY\f@ctech%
    \divide\v@lXa\f@ctech\divide\v@lYa\f@ctech%
    \v@lmin=\repdecn@mb{\v@lmin}\v@lmin\v@lmax=\repdecn@mb{\v@lmax}\v@lmax%
    \edef\G@xde{\repdecn@mb\v@lmin}\edef\P@xde{\repdecn@mb\v@lmax}%
    \c@rre=-\v@lmax\v@leur=\repdecn@mb\v@lY\v@lY\advance\c@rre\v@leur\c@rre=\G@xde\c@rre%
    \v@leur=\repdecn@mb\v@lX\v@lX\v@leur=\P@xde\v@leur\advance\c@rre\v@leur
    \v@lmin=\repdecn@mb\v@lYa\v@lmin\v@lmax=\repdecn@mb\v@lXa\v@lmax%
    \mili@u=\repdecn@mb\v@lX\v@lmax\advance\mili@u\repdecn@mb\v@lY\v@lmin
    \v@lmax=\repdecn@mb\v@lXa\v@lmax\advance\v@lmax\repdecn@mb\v@lYa\v@lmin
    \ifdim\v@lmax>\epsil@n%
    \maxim@m{\v@leur}{\c@rre}{-\c@rre}\maxim@m{\v@lmin}{\mili@u}{-\mili@u}%
    \maxim@m{\v@leur}{\v@leur}{\v@lmin}\maxim@m{\v@lmin}{\v@lmax}{-\v@lmax}%
    \maxim@m{\v@leur}{\v@leur}{\v@lmin}\p@rtentiere{\p@rtent}{\v@leur}\advance\p@rtent\@ne%
    \divide\c@rre\p@rtent\divide\mili@u\p@rtent\divide\v@lmax\p@rtent%
    \delt@=\repdecn@mb{\mili@u}\mili@u\v@leur=\repdecn@mb{\v@lmax}\c@rre%
    \advance\delt@-\v@leur\ifdim\delt@<\z@\else\sqrt@\delt@\delt@%
    \invers@\v@lmax\v@lmax\edef\Uns@rAp{\repdecn@mb\v@lmax}%
    \v@leur=-\mili@u\advance\v@leur-\delt@\v@leur=\Uns@rAp\v@leur%
    \edef\t@ille{\repdecn@mb\v@leur}\figpttra#1:=-4/\t@ille,-3/\s@mme=#1\advance\s@mme\@ne%
    \v@leur=-\mili@u\advance\v@leur\delt@\v@leur=\Uns@rAp\v@leur%
    \edef\t@ille{\repdecn@mb\v@leur}\figpttra\the\s@mme:=-4/\t@ille,-3/\fi\fi}
\ctr@ln@m\figptsorthoprojline
\ctr@ld@f\def\figptsorthoprojlineDD#1=#2/#3,#4/{\ifps@cri{\s@uvc@ntr@l\et@tfigptsorthoprojlineDD%
    \setc@ntr@l{2}\figvectPDD-3[#3,#4]\figvectNVDD-4[-3]\resetc@ntr@l{2}%
    \def\list@num{#2}\s@mme=#1\@ecfor\p@int:=\list@num\do{%
    \inters@cDD\the\s@mme:[\p@int,-4;#3,-3]\advance\s@mme\@ne}%
    \resetc@ntr@l\et@tfigptsorthoprojlineDD}\ignorespaces\fi}
\ctr@ld@f\def\figptsorthoprojlineTD#1=#2/#3,#4/{\ifps@cri{\s@uvc@ntr@l\et@tfigptsorthoprojlineTD%
    \setc@ntr@l{2}\figvectPTD-2[#3,#4]\vecunit@TD{-2}{-2}%
    \def\list@num{#2}\s@mme=#1\@ecfor\p@int:=\list@num\do{%
    \figvectPTD-1[#3,\p@int]\c@lproscalTD\v@leur[-1,-2]%
    \edef\v@lcoef{\repdecn@mb{\v@leur}}\figpttraTD\the\s@mme:=#3/\v@lcoef,-2/%
    \advance\s@mme\@ne}\resetc@ntr@l\et@tfigptsorthoprojlineTD}\ignorespaces\fi}
\ctr@ln@m\figptsorthoprojplane
\ctr@ld@f\def\figptsorthoprojplaneDD{\un@v@ilable{figptsorthoprojplane}}
\ctr@ld@f\def\figptsorthoprojplaneTD#1=#2/#3,#4/{\ifps@cri{\s@uvc@ntr@l\et@tfigptsorthoprojplane%
    \setc@ntr@l{2}\vecunit@TD{-2}{#4}%
    \def\list@num{#2}\s@mme=#1\@ecfor\p@int:=\list@num\do{\figvectPTD-1[\p@int,#3]%
    \c@lproscalTD\v@leur[-1,-2]\edef\v@lcoef{\repdecn@mb{\v@leur}}%
    \figpttraTD\the\s@mme:=\p@int/\v@lcoef,-2/\advance\s@mme\@ne}%
    \resetc@ntr@l\et@tfigptsorthoprojplane}\ignorespaces\fi}
\ctr@ld@f\def\figptshom#1=#2/#3,#4/{\ifps@cri{\s@uvc@ntr@l\et@tfigptshom%
    \setc@ntr@l{2}\def\list@num{#2}\s@mme=#1%
    \@ecfor\p@int:=\list@num\do{\figvectP-1[#3,\p@int]%
    \figpttra\the\s@mme:=#3/#4,-1/\advance\s@mme\@ne}%
    \resetc@ntr@l\et@tfigptshom}\ignorespaces\fi}
\ctr@ln@m\figptsrot
\ctr@ld@f\def\figptsrotDD#1=#2/#3,#4/{\ifps@cri{\s@uvc@ntr@l\et@tfigptsrotDD%
    \c@ssin{\C@}{\S@}{#4}\setc@ntr@l{2}\def\list@num{#2}\s@mme=#1%
    \@ecfor\p@int:=\list@num\do{\figvectPDD-1[#3,\p@int]\Figg@tXY{-1}%
    \v@lXa=\C@\v@lX\advance\v@lXa-\S@\v@lY%
    \v@lYa=\S@\v@lX\advance\v@lYa\C@\v@lY%
    \Figv@ctCreg-1(\v@lXa,\v@lYa)\figpttraDD\the\s@mme:=#3/1,-1/\advance\s@mme\@ne}%
    \resetc@ntr@l\et@tfigptsrotDD}\ignorespaces\fi}
\ctr@ld@f\def\figptsrotTD#1=#2/#3,#4,#5/{\ifps@cri{\s@uvc@ntr@l\et@tfigptsrotTD%
    \c@ssin{\C@}{\S@}{#4}%
    \setc@ntr@l{2}\def\list@num{#2}\s@mme=#1%
    \@ecfor\p@int:=\list@num\do{\figptorthoprojplaneTD-3:=#3/\p@int,#5/%
    \figvectPTD-2[-3,\p@int]%
    \figvectNVTD-1[#5,-2]\n@rmeucTD\v@leur{-2}\edef\v@lcoef{\repdecn@mb{\v@leur}}%
    \Figg@tXYa{-1}\v@lXa=\v@lcoef\v@lXa\v@lYa=\v@lcoef\v@lYa\v@lZa=\v@lcoef\v@lZa%
    \v@lXa=\S@\v@lXa\v@lYa=\S@\v@lYa\v@lZa=\S@\v@lZa\Figg@tXY{-2}%
    \advance\v@lXa\C@\v@lX\advance\v@lYa\C@\v@lY\advance\v@lZa\C@\v@lZ%
    \Figg@tXY{-3}\advance\v@lXa\v@lX\advance\v@lYa\v@lY\advance\v@lZa\v@lZ%
    \Figp@intregTD\the\s@mme:(\v@lXa,\v@lYa,\v@lZa)\advance\s@mme\@ne}%
    \resetc@ntr@l\et@tfigptsrotTD}\ignorespaces\fi}
\ctr@ln@m\figptssym
\ctr@ld@f\def\figptssymDD#1=#2/#3,#4/{\ifps@cri{\s@uvc@ntr@l\et@tfigptssymDD%
    \setc@ntr@l{2}\figvectPDD-3[#3,#4]\Figg@tXY{-3}\Figv@ctCreg-4(-\v@lY,\v@lX)%
    \resetc@ntr@l{2}\def\list@num{#2}\s@mme=#1%
    \@ecfor\p@int:=\list@num\do{\inters@cDD-5:[#3,-3;\p@int,-4]\figvectPDD-2[\p@int,-5]%
    \figpttraDD\the\s@mme:=\p@int/2,-2/\advance\s@mme\@ne}%
    \resetc@ntr@l\et@tfigptssymDD}\ignorespaces\fi}
\ctr@ld@f\def\figptssymTD#1=#2/#3,#4/{\ifps@cri{\s@uvc@ntr@l\et@tfigptssymTD%
    \setc@ntr@l{2}\vecunit@TD{-2}{#4}\def\list@num{#2}\s@mme=#1%
    \@ecfor\p@int:=\list@num\do{\figvectPTD-1[\p@int,#3]%
    \c@lproscalTD\v@leur[-1,-2]\v@leur=2\v@leur\edef\v@lcoef{\repdecn@mb{\v@leur}}%
    \figpttraTD\the\s@mme:=\p@int/\v@lcoef,-2/\advance\s@mme\@ne}%
    \resetc@ntr@l\et@tfigptssymTD}\ignorespaces\fi}
\ctr@ln@m\figptstra
\ctr@ld@f\def\figptstraDD#1=#2/#3,#4/{\ifps@cri{\Figg@tXYa{#4}\v@lXa=#3\v@lXa\v@lYa=#3\v@lYa%
    \def\list@num{#2}\s@mme=#1\@ecfor\p@int:=\list@num\do{\Figg@tXY{\p@int}%
    \advance\v@lX\v@lXa\advance\v@lY\v@lYa%
    \Figp@intregDD\the\s@mme:(\v@lX,\v@lY)\advance\s@mme\@ne}}\ignorespaces\fi}
\ctr@ld@f\def\figptstraTD#1=#2/#3,#4/{\ifps@cri{\Figg@tXYa{#4}\v@lXa=#3\v@lXa\v@lYa=#3\v@lYa%
    \v@lZa=#3\v@lZa\def\list@num{#2}\s@mme=#1\@ecfor\p@int:=\list@num\do{\Figg@tXY{\p@int}%
    \advance\v@lX\v@lXa\advance\v@lY\v@lYa\advance\v@lZ\v@lZa%
    \Figp@intregTD\the\s@mme:(\v@lX,\v@lY,\v@lZ)\advance\s@mme\@ne}}\ignorespaces\fi}
\ctr@ln@m\figptvisilimSL
\ctr@ld@f\def\figptvisilimSLDD{\un@v@ilable{figptvisilimSL}}
\ctr@ld@f\def\figptvisilimSLTD#1:#2[#3,#4;#5,#6]{\ifps@cri{\s@uvc@ntr@l\et@tfigptvisilimSLTD%
    \setc@ntr@l{2}\figvectP-1[#3,#4]\n@rminf{\delt@}{-1}%
    \ifcase\curr@ntproj\v@lX=\cxa@\p@\v@lY=-\p@\v@lZ=\cxb@\p@
    \Figv@ctCreg-2(\v@lX,\v@lY,\v@lZ)\figvectP-3[#5,#6]\figvectNV-1[-2,-3]%
    \or\figvectP-1[#5,#6]\vecunitCV@TD{-1}\v@lmin=\v@lX\v@lmax=\v@lY
    \v@leur=\v@lZ\v@lX=\cza@\p@\v@lY=\czb@\p@\v@lZ=\czc@\p@\c@lprovec{-1}%
    \or\c@ley@pt{-2}\figvectN-1[#5,#6,-2]\fi
    \edef\Ai@{#3}\edef\Aj@{#4}\figvectP-2[#5,\Ai@]\c@lproscal\v@leur[-1,-2]%
    \ifdim\v@leur>\z@\p@rtent=\@ne\else\p@rtent=\m@ne\fi%
    \figvectP-2[#5,\Aj@]\c@lproscal\v@leur[-1,-2]%
    \ifdim\p@rtent\v@leur>\z@\figptcopy#1:#2/#3/%
    \message{*** \BS@ figptvisilimSL: points are on the same side.}\else%
    \figptcopy-3:/#3/\figptcopy-4:/#4/%
    \loop\figptbary-5:[-3,-4;1,1]\figvectP-2[#5,-5]\c@lproscal\v@leur[-1,-2]%
    \ifdim\p@rtent\v@leur>\z@\figptcopy-3:/-5/\else\figptcopy-4:/-5/\fi%
    \divide\delt@\tw@\ifdim\delt@>\epsil@n\repeat%
    \figptbary#1:#2[-3,-4;1,1]\fi\resetc@ntr@l\et@tfigptvisilimSLTD}\ignorespaces\fi}
\ctr@ld@f\def\c@ley@pt#1{\t@stp@r\ifitis@K\v@lX=\cza@\p@\v@lY=\czb@\p@\v@lZ=\czc@\p@%
    \Figv@ctCreg-1(\v@lX,\v@lY,\v@lZ)\Figp@intreg-2:(\wd\Bt@rget,\ht\Bt@rget,\dp\Bt@rget)%
    \figpttra#1:=-2/-\disob@intern,-1/\else\end\fi}
\ctr@ld@f\def\t@stp@r{\itis@Ktrue\ifnewt@rgetpt\else\itis@Kfalse%
    \message{*** \BS@ figptvisilimXX: target point undefined.}\fi\ifnewdis@b\else%
    \itis@Kfalse\message{*** \BS@ figptvisilimXX: observation distance undefined.}\fi%
    \ifitis@K\else\message{*** This macro must be called after \BS@ psbeginfig or after
    having set the missing parameter(s) with \BS@ figset proj()}\fi}
\ctr@ld@f\def\figscan#1(#2,#3){{\s@uvc@ntr@l\et@tfigscan\@psfgetbb{#1}\if@psfbbfound\else%
    \def\@psfllx{0}\def\@psflly{20}\def\@psfurx{540}\def\@psfury{640}\fi\figscan@{#2}{#3}%
    \resetc@ntr@l\et@tfigscan}\ignorespaces}
\ctr@ld@f\def\figscan@#1#2{%
    \unit@=\@ne bp\setc@ntr@l{2}\figsetmark{}%
    \def\minst@p{20pt}%
    \v@lX=\@psfllx\p@\v@lX=\Sc@leFact\v@lX\r@undint\v@lX\v@lX%
    \v@lY=\@psflly\p@\v@lY=\Sc@leFact\v@lY\ifdim\v@lY>\z@\r@undint\v@lY\v@lY\fi%
    \delt@=\@psfury\p@\delt@=\Sc@leFact\delt@%
    \advance\delt@-\v@lY\v@lXa=\@psfurx\p@\v@lXa=\Sc@leFact\v@lXa\v@leur=\minst@p%
    \edef\valv@lY{\repdecn@mb{\v@lY}}\edef\LgTr@it{\the\delt@}%
    \loop\ifdim\v@lX<\v@lXa\edef\valv@lX{\repdecn@mb{\v@lX}}%
    \figptDD -1:(\valv@lX,\valv@lY)\figwriten -1:\hbox{\vrule height\LgTr@it}(0)%
    \ifdim\v@leur<\minst@p\else\figsetmark{\raise-8bp\hbox{$\scriptscriptstyle\triangle$}}%
    \figwrites -1:\@ffichnb{0}{\valv@lX}(6)\v@leur=\z@\figsetmark{}\fi%
    \advance\v@leur#1pt\advance\v@lX#1pt\repeat%
    \def\minst@p{10pt}%
    \v@lX=\@psfllx\p@\v@lX=\Sc@leFact\v@lX\ifdim\v@lX>\z@\r@undint\v@lX\v@lX\fi%
    \v@lY=\@psflly\p@\v@lY=\Sc@leFact\v@lY\r@undint\v@lY\v@lY%
    \delt@=\@psfurx\p@\delt@=\Sc@leFact\delt@%
    \advance\delt@-\v@lX\v@lYa=\@psfury\p@\v@lYa=\Sc@leFact\v@lYa\v@leur=\minst@p%
    \edef\valv@lX{\repdecn@mb{\v@lX}}\edef\LgTr@it{\the\delt@}%
    \loop\ifdim\v@lY<\v@lYa\edef\valv@lY{\repdecn@mb{\v@lY}}%
    \figptDD -1:(\valv@lX,\valv@lY)\figwritee -1:\vbox{\hrule width\LgTr@it}(0)%
    \ifdim\v@leur<\minst@p\else\figsetmark{$\triangleright$\kern4bp}%
    \figwritew -1:\@ffichnb{0}{\valv@lY}(6)\v@leur=\z@\figsetmark{}\fi%
    \advance\v@leur#2pt\advance\v@lY#2pt\repeat}
\ctr@ld@f\let\figscanI=\figscan
\ctr@ld@f\def\figscan@E#1(#2,#3){{\s@uvc@ntr@l\et@tfigscan@E%
    \Figdisc@rdLTS{#1}{\t@xt@}\pdfximage{\t@xt@}%
    \setbox\Gb@x=\hbox{\pdfrefximage\pdflastximage}%
    \edef\@psfllx{0}\v@lY=-\dp\Gb@x\edef\@psflly{\repdecn@mb{\v@lY}}%
    \edef\@psfurx{\repdecn@mb{\wd\Gb@x}}%
    \v@lY=\dp\Gb@x\advance\v@lY\ht\Gb@x\edef\@psfury{\repdecn@mb{\v@lY}}%
    \figscan@{#2}{#3}\resetc@ntr@l\et@tfigscan@E}\ignorespaces}
\ctr@ld@f\def\figshowpts[#1,#2]{{\figsetmark{$\bullet$}\figsetptname{\bf ##1}%
    \p@rtent=#2\relax\ifnum\p@rtent<\z@\p@rtent=\z@\fi%
    \s@mme=#1\relax\ifnum\s@mme<\z@\s@mme=\z@\fi%
    \loop\ifnum\s@mme<\p@rtent\pt@rvect{\s@mme}%
    \ifitis@K\figwriten{\the\s@mme}:(4pt)\fi\advance\s@mme\@ne\repeat%
    \pt@rvect{\s@mme}\ifitis@K\figwriten{\the\s@mme}:(4pt)\fi}\ignorespaces}
\ctr@ld@f\def\pt@rvect#1{\set@bjc@de{#1}%
    \expandafter\expandafter\expandafter\inqpt@rvec\csname\objc@de\endcsname:}
\ctr@ld@f\def\inqpt@rvec#1#2:{\if#1\C@dCl@spt\itis@Ktrue\else\itis@Kfalse\fi}
\ctr@ld@f\def\figshowsettings{{%
    \immediate\write16{====================================================================}%
    \immediate\write16{ Current settings about:}%
    \immediate\write16{ --- GENERAL ---}%
    \immediate\write16{Scale factor and Unit = \unit@util\space (\the\unit@)
     \space -> \BS@ figinit{ScaleFactorUnit}}%
    \immediate\write16{Update mode = \ifpsupdatem@de yes\else no\fi
     \space-> \BS@ psset(update=yes/no) or \BS@ pssetdefault(update=yes/no)}%
    \immediate\write16{ --- PRINTING ---}%
    \immediate\write16{Implicit point name = \ptn@me{i} \space-> \BS@ figsetptname{Name}}%
    \immediate\write16{Point marker = \the\c@nsymb \space -> \BS@ figsetmark{Mark}}%
    \immediate\write16{Print rounded coordinates = \ifr@undcoord yes\else no\fi
     \space-> \BS@ figsetroundcoord{yes/no}}%
    \immediate\write16{ --- GRAPHICAL (general) ---}%
    \immediate\write16{First-level (or primary) settings:}%
    \immediate\write16{ Color = \curr@ntcolor \space-> \BS@ psset(color=ColorDefinition)}%
    \immediate\write16{ Filling mode = \iffillm@de yes\else no\fi
     \space-> \BS@ psset(fillmode=yes/no)}%
    \immediate\write16{ Line join = \curr@ntjoin \space-> \BS@ psset(join=miter/round/bevel)}%
    \immediate\write16{ Line style = \curr@ntdash \space-> \BS@ psset(dash=Index/Pattern)}%
    \immediate\write16{ Line width = \curr@ntwidth
     \space-> \BS@ psset(width=real in PostScript units)}%
    \immediate\write16{Second-level (or secondary) settings:}%
    \immediate\write16{ Color = \sec@ndcolor \space-> \BS@ psset second(color=ColorDefinition)}%
    \immediate\write16{ Line style = \curr@ntseconddash
     \space-> \BS@ psset second(dash=Index/Pattern)}%
    \immediate\write16{ Line width = \curr@ntsecondwidth
     \space-> \BS@ psset second(width=real in PostScript units)}%
    \immediate\write16{Third-level (or ternary) settings:}%
    \immediate\write16{ Color = \th@rdcolor \space-> \BS@ psset third(color=ColorDefinition)}%
    \immediate\write16{ --- GRAPHICAL (specific) ---}%
    \immediate\write16{Arrow-head:}%
    \immediate\write16{ (half-)Angle = \@rrowheadangle
     \space-> \BS@ psset arrowhead(angle=real in degrees)}%
    \immediate\write16{ Filling mode = \if@rrowhfill yes\else no\fi
     \space-> \BS@ psset arrowhead(fillmode=yes/no)}%
    \immediate\write16{ "Outside" = \if@rrowhout yes\else no\fi
     \space-> \BS@ psset arrowhead(out=yes/no)}%
    \immediate\write16{ Length = \@rrowheadlength
     \if@rrowratio\space(not active)\else\space(active)\fi
     \space-> \BS@ psset arrowhead(length=real in user coord.)}%
    \immediate\write16{ Ratio = \@rrowheadratio
     \if@rrowratio\space(active)\else\space(not active)\fi
     \space-> \BS@ psset arrowhead(ratio=real in [0,1])}%
    \immediate\write16{Curve: Roundness = \curv@roundness
     \space-> \BS@ psset curve(roundness=real in [0,0.5])}%
    \immediate\write16{Mesh: Diagonal = \c@ntrolmesh
     \space-> \BS@ psset mesh(diag=integer in {-1,0,1})}%
    \immediate\write16{Flow chart:}%
    \immediate\write16{ Arrow position = \@rrowp@s
     \space-> \BS@ psset flowchart(arrowposition=real in [0,1])}%
    \immediate\write16{ Arrow reference point = \ifcase\@rrowr@fpt start\else end\fi
     \space-> \BS@ psset flowchart(arrowrefpt = start/end)}%
    \immediate\write16{ Line type = \ifcase\fclin@typ@ curve\else polygon\fi
     \space-> \BS@ psset flowchart(line=polygon/curve)}%
    \immediate\write16{ Padding = (\Xp@dd, \Yp@dd)
     \space-> \BS@ psset flowchart(padding = real in user coord.)}%
    \immediate\write16{\space\space\space\space(or
     \BS@ psset flowchart(xpadding=real, ypadding=real) )}%
    \immediate\write16{ Radius = \fclin@r@d
     \space-> \BS@ psset flowchart(radius=positive real in user coord.)}%
    \immediate\write16{ Shape = \fcsh@pe
     \space-> \BS@ psset flowchart(shape = rectangle, ellipse or lozenge)}%
    \immediate\write16{ Thickness = \thickn@ss
     \space-> \BS@ psset flowchart(thickness = real in user coord.)}%
    \ifTr@isDim%
    \immediate\write16{ --- 3D to 2D PROJECTION ---}%
    \immediate\write16{Projection : \typ@proj \space-> \BS@ figinit{ScaleFactorUnit, ProjType}}%
    \immediate\write16{Longitude (psi) = \v@lPsi \space-> \BS@ figset proj(psi=real in degrees)}%
    \ifcase\curr@ntproj\immediate\write16{Depth coeff. (Lambda)
     \space = \v@lTheta \space-> \BS@ figset proj(lambda=real in [0,1])}%
    \else\immediate\write16{Latitude (theta)
     \space = \v@lTheta \space-> \BS@ figset proj(theta=real in degrees)}%
    \fi%
    \ifnum\curr@ntproj=\tw@%
    \immediate\write16{Observation distance = \disob@unit
     \space-> \BS@ figset proj(dist=real in user coord.)}%
    \immediate\write16{Target point = \t@rgetpt \space-> \BS@ figset proj(targetpt=pt number)}%
     \v@lX=\ptT@unit@\wd\Bt@rget\v@lY=\ptT@unit@\ht\Bt@rget\v@lZ=\ptT@unit@\dp\Bt@rget%
    \immediate\write16{ Its coordinates are
     (\repdecn@mb{\v@lX}, \repdecn@mb{\v@lY}, \repdecn@mb{\v@lZ})}%
    \fi%
    \fi%
    \immediate\write16{====================================================================}%
    \ignorespaces}}
\ctr@ln@w{newif}\ifitis@vect@r
\ctr@ld@f\def\figvectC#1(#2,#3){{\itis@vect@rtrue\figpt#1:(#2,#3)}\ignorespaces}
\ctr@ld@f\def\Figv@ctCreg#1(#2,#3){{\itis@vect@rtrue\Figp@intreg#1:(#2,#3)}\ignorespaces}
\ctr@ln@m\figvectDBezier
\ctr@ld@f\def\figvectDBezierDD#1:#2,#3[#4,#5,#6,#7]{\ifps@cri{\s@uvc@ntr@l\et@tfigvectDBezierDD%
    \FigvectDBezier@#2,#3[#4,#5,#6,#7]\v@lX=\c@ef\v@lX\v@lY=\c@ef\v@lY%
    \Figv@ctCreg#1(\v@lX,\v@lY)\resetc@ntr@l\et@tfigvectDBezierDD}\ignorespaces\fi}
\ctr@ld@f\def\figvectDBezierTD#1:#2,#3[#4,#5,#6,#7]{\ifps@cri{\s@uvc@ntr@l\et@tfigvectDBezierTD%
    \FigvectDBezier@#2,#3[#4,#5,#6,#7]\v@lX=\c@ef\v@lX\v@lY=\c@ef\v@lY\v@lZ=\c@ef\v@lZ%
    \Figv@ctCreg#1(\v@lX,\v@lY,\v@lZ)\resetc@ntr@l\et@tfigvectDBezierTD}\ignorespaces\fi}
\ctr@ld@f\def\FigvectDBezier@#1,#2[#3,#4,#5,#6]{\setc@ntr@l{2}%
    \edef\T@{#2}\v@leur=\p@\advance\v@leur-#2pt\edef\UNmT@{\repdecn@mb{\v@leur}}%
    \ifnum#1=\tw@\def\c@ef{6}\else\def\c@ef{3}\fi%
    \figptcopy-4:/#3/\figptcopy-3:/#4/\figptcopy-2:/#5/\figptcopy-1:/#6/%
    \l@mbd@un=-4 \l@mbd@de=-\thr@@\p@rtent=\m@ne\c@lDecast%
    \ifnum#1=\tw@\c@lDCDeux{-4}{-3}\c@lDCDeux{-3}{-2}\c@lDCDeux{-4}{-3}\else%
    \l@mbd@un=-4 \l@mbd@de=-\thr@@\p@rtent=-\tw@\c@lDecast%
    \c@lDCDeux{-4}{-3}\fi\Figg@tXY{-4}}
\ctr@ln@m\c@lDCDeux
\ctr@ld@f\def\c@lDCDeuxDD#1#2{\Figg@tXY{#2}\Figg@tXYa{#1}%
    \advance\v@lX-\v@lXa\advance\v@lY-\v@lYa\Figp@intregDD#1:(\v@lX,\v@lY)}
\ctr@ld@f\def\c@lDCDeuxTD#1#2{\Figg@tXY{#2}\Figg@tXYa{#1}\advance\v@lX-\v@lXa%
    \advance\v@lY-\v@lYa\advance\v@lZ-\v@lZa\Figp@intregTD#1:(\v@lX,\v@lY,\v@lZ)}
\ctr@ln@m\figvectN
\ctr@ld@f\def\figvectNDD#1[#2,#3]{\ifps@cri{\Figg@tXYa{#2}\Figg@tXY{#3}%
    \advance\v@lX-\v@lXa\advance\v@lY-\v@lYa%
    \Figv@ctCreg#1(-\v@lY,\v@lX)}\ignorespaces\fi}
\ctr@ld@f\def\figvectNTD#1[#2,#3,#4]{\ifps@cri{\vecunitC@TD[#2,#4]\v@lmin=\v@lX\v@lmax=\v@lY%
    \v@leur=\v@lZ\vecunitC@TD[#2,#3]\c@lprovec{#1}}\ignorespaces\fi}
\ctr@ln@m\figvectNV
\ctr@ld@f\def\figvectNVDD#1[#2]{\ifps@cri{\Figg@tXY{#2}\Figv@ctCreg#1(-\v@lY,\v@lX)}\ignorespaces\fi}
\ctr@ld@f\def\figvectNVTD#1[#2,#3]{\ifps@cri{\vecunitCV@TD{#3}\v@lmin=\v@lX\v@lmax=\v@lY%
    \v@leur=\v@lZ\vecunitCV@TD{#2}\c@lprovec{#1}}\ignorespaces\fi}
\ctr@ln@m\figvectP
\ctr@ld@f\def\figvectPDD#1[#2,#3]{\ifps@cri{\Figg@tXYa{#2}\Figg@tXY{#3}%
    \advance\v@lX-\v@lXa\advance\v@lY-\v@lYa%
    \Figv@ctCreg#1(\v@lX,\v@lY)}\ignorespaces\fi}
\ctr@ld@f\def\figvectPTD#1[#2,#3]{\ifps@cri{\Figg@tXYa{#2}\Figg@tXY{#3}%
    \advance\v@lX-\v@lXa\advance\v@lY-\v@lYa\advance\v@lZ-\v@lZa%
    \Figv@ctCreg#1(\v@lX,\v@lY,\v@lZ)}\ignorespaces\fi}
\ctr@ln@m\figvectU
\ctr@ld@f\def\figvectUDD#1[#2]{\ifps@cri{\n@rmeuc\v@leur{#2}\invers@\v@leur\v@leur%
    \delt@=\repdecn@mb{\v@leur}\unit@\edef\v@ldelt@{\repdecn@mb{\delt@}}%
    \Figg@tXY{#2}\v@lX=\v@ldelt@\v@lX\v@lY=\v@ldelt@\v@lY%
    \Figv@ctCreg#1(\v@lX,\v@lY)}\ignorespaces\fi}
\ctr@ld@f\def\figvectUTD#1[#2]{\ifps@cri{\n@rmeuc\v@leur{#2}\invers@\v@leur\v@leur%
    \delt@=\repdecn@mb{\v@leur}\unit@\edef\v@ldelt@{\repdecn@mb{\delt@}}%
    \Figg@tXY{#2}\v@lX=\v@ldelt@\v@lX\v@lY=\v@ldelt@\v@lY\v@lZ=\v@ldelt@\v@lZ%
    \Figv@ctCreg#1(\v@lX,\v@lY,\v@lZ)}\ignorespaces\fi}
\ctr@ld@f\def\figvisu#1#2#3{\c@ldefproj\initb@undb@x\xdef\figforTeXFigno{\figforTeXnextFigno}%
    \s@mme=\figforTeXnextFigno\advance\s@mme\@ne\xdef\figforTeXnextFigno{\number\s@mme}%
    \setbox\b@xvisu=\hbox{\ifnum\@utoFN>\z@\figinsert{}\gdef\@utoFInDone{0}\fi\ignorespaces#3}%
    \gdef\@utoFInDone{1}\gdef\@utoFN{0}%
    \v@lXa=-\c@@rdYmin\v@lYa=\c@@rdYmax\advance\v@lYa-\c@@rdYmin%
    \v@lX=\c@@rdXmax\advance\v@lX-\c@@rdXmin%
    \setbox#1=\hbox{#2}\v@lY=-\v@lX\maxim@m{\v@lX}{\v@lX}{\wd#1}%
    \advance\v@lY\v@lX\divide\v@lY\tw@\advance\v@lY-\c@@rdXmin%
    \setbox#1=\vbox{\parindent0mm\hsize=\v@lX\vskip\v@lYa%
    \rlap{\hskip\v@lY\smash{\raise\v@lXa\box\b@xvisu}}%
    \def\t@xt@{#2}\ifx\t@xt@\empty\else\medskip\centerline{#2}\fi}\wd#1=\v@lX}
\ctr@ld@f\def\figDecrementFigno{{\xdef\figforTeXnextFigno{\figforTeXFigno}%
    \s@mme=\figforTeXFigno\advance\s@mme\m@ne\xdef\figforTeXFigno{\number\s@mme}}}
\ctr@ln@w{newbox}\Bt@rget\setbox\Bt@rget=\null
\ctr@ln@w{newbox}\BminTD@\setbox\BminTD@=\null
\ctr@ln@w{newbox}\BmaxTD@\setbox\BmaxTD@=\null
\ctr@ln@w{newif}\ifnewt@rgetpt\ctr@ln@w{newif}\ifnewdis@b
\ctr@ld@f\def\b@undb@xTD#1#2#3{%
    \relax\ifdim#1<\wd\BminTD@\global\wd\BminTD@=#1\fi%
    \relax\ifdim#2<\ht\BminTD@\global\ht\BminTD@=#2\fi%
    \relax\ifdim#3<\dp\BminTD@\global\dp\BminTD@=#3\fi%
    \relax\ifdim#1>\wd\BmaxTD@\global\wd\BmaxTD@=#1\fi%
    \relax\ifdim#2>\ht\BmaxTD@\global\ht\BmaxTD@=#2\fi%
    \relax\ifdim#3>\dp\BmaxTD@\global\dp\BmaxTD@=#3\fi}
\ctr@ld@f\def\c@ldefdisob{{\ifdim\wd\BminTD@<\maxdimen\v@leur=\wd\BmaxTD@\advance\v@leur-\wd\BminTD@%
    \delt@=\ht\BmaxTD@\advance\delt@-\ht\BminTD@\maxim@m{\v@leur}{\v@leur}{\delt@}%
    \delt@=\dp\BmaxTD@\advance\delt@-\dp\BminTD@\maxim@m{\v@leur}{\v@leur}{\delt@}%
    \v@leur=5\v@leur\else\v@leur=800pt\fi\c@ldefdisob@{\v@leur}}}
\ctr@ln@m\disob@intern
\ctr@ln@m\disob@
\ctr@ln@m\divf@ctproj
\ctr@ld@f\def\c@ldefdisob@#1{{\v@leur=#1\ifdim\v@leur<\p@\v@leur=800pt\fi%
    \xdef\disob@intern{\repdecn@mb{\v@leur}}%
    \delt@=\ptT@unit@\v@leur\xdef\disob@unit{\repdecn@mb{\delt@}}%
    \f@ctech=\@ne\loop\ifdim\v@leur>\t@n pt\divide\v@leur\t@n\multiply\f@ctech\t@n\repeat%
    \xdef\disob@{\repdecn@mb{\v@leur}}\xdef\divf@ctproj{\the\f@ctech}}%
    \global\newdis@btrue}
\ctr@ln@m\t@rgetpt
\ctr@ld@f\def\c@ldeft@rgetpt{\newt@rgetpttrue\def\t@rgetpt{CenterBoundBox}{%
    \delt@=\wd\BmaxTD@\advance\delt@-\wd\BminTD@\divide\delt@\tw@%
    \v@leur=\wd\BminTD@\advance\v@leur\delt@\global\wd\Bt@rget=\v@leur%
    \delt@=\ht\BmaxTD@\advance\delt@-\ht\BminTD@\divide\delt@\tw@%
    \v@leur=\ht\BminTD@\advance\v@leur\delt@\global\ht\Bt@rget=\v@leur%
    \delt@=\dp\BmaxTD@\advance\delt@-\dp\BminTD@\divide\delt@\tw@%
    \v@leur=\dp\BminTD@\advance\v@leur\delt@\global\dp\Bt@rget=\v@leur}}
\ctr@ln@m\c@ldefproj
\ctr@ld@f\def\c@ldefprojTD{\ifnewt@rgetpt\else\c@ldeft@rgetpt\fi\ifnewdis@b\else\c@ldefdisob\fi}
\ctr@ld@f\def\c@lprojcav{
    \v@lZa=\cxa@\v@lY\advance\v@lX\v@lZa%
    \v@lZa=\cxb@\v@lY\v@lY=\v@lZ\advance\v@lY\v@lZa\ignorespaces}
\ctr@ln@m\v@lcoef
\ctr@ld@f\def\c@lprojrea{
    \advance\v@lX-\wd\Bt@rget\advance\v@lY-\ht\Bt@rget\advance\v@lZ-\dp\Bt@rget%
    \v@lZa=\cza@\v@lX\advance\v@lZa\czb@\v@lY\advance\v@lZa\czc@\v@lZ%
    \divide\v@lZa\divf@ctproj\advance\v@lZa\disob@ pt\invers@{\v@lZa}{\v@lZa}%
    \v@lZa=\disob@\v@lZa\edef\v@lcoef{\repdecn@mb{\v@lZa}}%
    \v@lXa=\cxa@\v@lX\advance\v@lXa\cxb@\v@lY\v@lXa=\v@lcoef\v@lXa%
    \v@lY=\cyb@\v@lY\advance\v@lY\cya@\v@lX\advance\v@lY\cyc@\v@lZ%
    \v@lY=\v@lcoef\v@lY\v@lX=\v@lXa\ignorespaces}
\ctr@ld@f\def\c@lprojort{
    \v@lXa=\cxa@\v@lX\advance\v@lXa\cxb@\v@lY%
    \v@lY=\cyb@\v@lY\advance\v@lY\cya@\v@lX\advance\v@lY\cyc@\v@lZ%
    \v@lX=\v@lXa\ignorespaces}
\ctr@ld@f\def\Figptpr@j#1:#2/#3/{{\Figg@tXY{#3}\superc@lprojSP%
    \Figp@intregDD#1:{#2}(\v@lX,\v@lY)}\ignorespaces}
\ctr@ln@m\figsetobdist
\ctr@ld@f\def\figsetobdistDD{\un@v@ilable{figsetobdist}}
\ctr@ld@f\def\figsetobdistTD(#1){{\ifcurr@ntPS%
    \immediate\write16{*** \BS@ figsetobdist is ignored inside a
     \BS@ psbeginfig-\BS@ psendfig block.}%
    \else\v@leur=#1\unit@\c@ldefdisob@{\v@leur}\fi}\ignorespaces}
\ctr@ln@m\c@lprojSP
\ctr@ln@m\curr@ntproj
\ctr@ln@m\typ@proj
\ctr@ln@m\superc@lprojSP
\ctr@ld@f\def\Figs@tproj#1{%
    \if#13 \d@faultproj\else\if#1c\d@faultproj%
    \else\if#1o\xdef\curr@ntproj{1}\xdef\typ@proj{orthogonal}%
         \figsetviewTD(\def@ultpsi,\def@ulttheta)%
         \global\let\c@lprojSP=\c@lprojort\global\let\superc@lprojSP=\c@lprojort%
    \else\if#1r\xdef\curr@ntproj{2}\xdef\typ@proj{realistic}%
         \figsetviewTD(\def@ultpsi,\def@ulttheta)%
         \global\let\c@lprojSP=\c@lprojrea\global\let\superc@lprojSP=\c@lprojrea%
    \else\d@faultproj\message{*** Unknown projection. Cavalier projection assumed.}%
    \fi\fi\fi\fi}
\ctr@ld@f\def\d@faultproj{\xdef\curr@ntproj{0}\xdef\typ@proj{cavalier}\figsetviewTD(\def@ultpsi,0.5)%
         \global\let\c@lprojSP=\c@lprojcav\global\let\superc@lprojSP=\c@lprojcav}
\ctr@ln@m\figsettarget
\ctr@ld@f\def\figsettargetDD{\un@v@ilable{figsettarget}}
\ctr@ld@f\def\figsettargetTD[#1]{{\ifcurr@ntPS%
    \immediate\write16{*** \BS@ figsettarget is ignored inside a
     \BS@ psbeginfig-\BS@ psendfig block.}%
    \else\global\newt@rgetpttrue\xdef\t@rgetpt{#1}\Figg@tXY{#1}\global\wd\Bt@rget=\v@lX%
    \global\ht\Bt@rget=\v@lY\global\dp\Bt@rget=\v@lZ\fi}\ignorespaces}
\ctr@ln@m\figsetview
\ctr@ld@f\def\figsetviewDD{\un@v@ilable{figsetview}}
\ctr@ld@f\def\figsetviewTD(#1){\ifcurr@ntPS%
     \immediate\write16{*** \BS@ figsetview is ignored inside a
     \BS@ psbeginfig-\BS@ psendfig block.}\else\Figsetview@#1,:\fi\ignorespaces}
\ctr@ld@f\def\Figsetview@#1,#2:{{\xdef\v@lPsi{#1}\def\t@xt@{#2}%
    \ifx\t@xt@\empty\def\@rgdeux{\v@lTheta}\else\X@rgdeux@#2\fi%
    \c@ssin{\costhet@}{\sinthet@}{#1}\v@lmin=\costhet@ pt\v@lmax=\sinthet@ pt%
    \ifcase\curr@ntproj%
    \v@leur=\@rgdeux\v@lmin\xdef\cxa@{\repdecn@mb{\v@leur}}%
    \v@leur=\@rgdeux\v@lmax\xdef\cxb@{\repdecn@mb{\v@leur}}\v@leur=\@rgdeux pt%
    \relax\ifdim\v@leur>\p@\message{*** Lambda too large ! See \BS@ figset proj() !}\fi%
    \else%
    \v@lmax=-\v@lmax\xdef\cxa@{\repdecn@mb{\v@lmax}}\xdef\cxb@{\costhet@}%
    \ifx\t@xt@\empty\edef\@rgdeux{\def@ulttheta}\fi\c@ssin{\C@}{\S@}{\@rgdeux}%
    \v@lmax=-\S@ pt%
    \v@leur=\v@lmax\v@leur=\costhet@\v@leur\xdef\cya@{\repdecn@mb{\v@leur}}%
    \v@leur=\v@lmax\v@leur=\sinthet@\v@leur\xdef\cyb@{\repdecn@mb{\v@leur}}%
    \xdef\cyc@{\C@}\v@lmin=-\C@ pt%
    \v@leur=\v@lmin\v@leur=\costhet@\v@leur\xdef\cza@{\repdecn@mb{\v@leur}}%
    \v@leur=\v@lmin\v@leur=\sinthet@\v@leur\xdef\czb@{\repdecn@mb{\v@leur}}%
    \xdef\czc@{\repdecn@mb{\v@lmax}}\fi%
    \xdef\v@lTheta{\@rgdeux}}}
\ctr@ld@f\def\def@ultpsi{40}
\ctr@ld@f\def\def@ulttheta{25}
\ctr@ln@m\l@debut
\ctr@ln@m\n@mref
\ctr@ld@f\def\figset#1(#2){\def\t@xt@{#1}\ifx\t@xt@\empty\trtlis@rg{#2}{\Figsetwr@te}
    \else\keln@mde#1|%
    \def\n@mref{pr}\ifx\l@debut\n@mref\ifcurr@ntPS
     \immediate\write16{*** \BS@ figset proj(...) is ignored inside a
     \BS@ psbeginfig-\BS@ psendfig block.}\else\trtlis@rg{#2}{\Figsetpr@j}\fi\else%
    \def\n@mref{wr}\ifx\l@debut\n@mref\trtlis@rg{#2}{\Figsetwr@te}\else
    \immediate\write16{*** Unknown keyword: \BS@ figset #1(...)}%
    \fi\fi\fi\ignorespaces}
\ctr@ld@f\def\Figsetpr@j#1=#2|{\keln@mtr#1|%
    \def\n@mref{dep}\ifx\l@debut\n@mref\Figsetd@p{#2}\else
    \def\n@mref{dis}\ifx\l@debut\n@mref%
     \ifnum\curr@ntproj=\tw@\figsetobdist(#2)\else\Figset@rr\fi\else
    \def\n@mref{lam}\ifx\l@debut\n@mref\Figsetd@p{#2}\else
    \def\n@mref{lat}\ifx\l@debut\n@mref\Figsetth@{#2}\else
    \def\n@mref{lon}\ifx\l@debut\n@mref\figsetview(#2)\else
    \def\n@mref{psi}\ifx\l@debut\n@mref\figsetview(#2)\else
    \def\n@mref{tar}\ifx\l@debut\n@mref%
     \ifnum\curr@ntproj=\tw@\figsettarget[#2]\else\Figset@rr\fi\else
    \def\n@mref{the}\ifx\l@debut\n@mref\Figsetth@{#2}\else
    \immediate\write16{*** Unknown attribute: \BS@ figset proj(..., #1=...).}%
    \fi\fi\fi\fi\fi\fi\fi\fi}
\ctr@ld@f\def\Figsetd@p#1{\ifnum\curr@ntproj=\z@\figsetview(\v@lPsi,#1)\else\Figset@rr\fi}
\ctr@ld@f\def\Figsetth@#1{\ifnum\curr@ntproj=\z@\Figset@rr\else\figsetview(\v@lPsi,#1)\fi}
\ctr@ld@f\def\Figset@rr{\message{*** \BS@ figset proj(): Attribute "\n@mref" ignored, incompatible
    with current projection}}
\ctr@ld@f\def\initb@undb@xTD{\wd\BminTD@=\maxdimen\ht\BminTD@=\maxdimen\dp\BminTD@=\maxdimen%
    \wd\BmaxTD@=-\maxdimen\ht\BmaxTD@=-\maxdimen\dp\BmaxTD@=-\maxdimen}
\ctr@ln@w{newbox}\Gb@x      
\ctr@ln@w{newbox}\Gb@xSC    
\ctr@ln@w{newtoks}\c@nsymb  
\ctr@ln@w{newif}\ifr@undcoord\ctr@ln@w{newif}\ifunitpr@sent
\ctr@ld@f\def\unssqrttw@{0.707106 }
\ctr@ld@f\def\figAst{\raise-1.15ex\hbox{$\ast$}}
\ctr@ld@f\def\figBullet{\raise-1.15ex\hbox{$\bullet$}}
\ctr@ld@f\def\figCirc{\raise-1.15ex\hbox{$\circ$}}
\ctr@ld@f\def\figDiamond{\raise-1.15ex\hbox{$\diamond$}}%
\ctr@ld@f\def\boxit#1#2{\leavevmode\hbox{\vrule\vbox{\hrule\vglue#1%
    \vtop{\hbox{\kern#1{#2}\kern#1}\vglue#1\hrule}}\vrule}}
\ctr@ld@f\def\centertext#1#2{\vbox{\hsize#1\parindent0cm%
    \leftskip=0pt plus 1fil\rightskip=0pt plus 1fil\parfillskip=0pt{#2}}}
\ctr@ld@f\def\lefttext#1#2{\vbox{\hsize#1\parindent0cm\rightskip=0pt plus 1fil#2}}
\ctr@ld@f\def\c@nterpt{\ignorespaces%
    \kern-.5\wd\Gb@xSC%
    \raise-.5\ht\Gb@xSC\rlap{\hbox{\raise.5\dp\Gb@xSC\hbox{\copy\Gb@xSC}}}%
    \kern .5\wd\Gb@xSC\ignorespaces}
\ctr@ld@f\def\b@undb@xSC#1#2{{\v@lXa=#1\v@lYa=#2%
    \v@leur=\ht\Gb@xSC\advance\v@leur\dp\Gb@xSC%
    \advance\v@lXa-.5\wd\Gb@xSC\advance\v@lYa-.5\v@leur\b@undb@x{\v@lXa}{\v@lYa}%
    \advance\v@lXa\wd\Gb@xSC\advance\v@lYa\v@leur\b@undb@x{\v@lXa}{\v@lYa}}}
\ctr@ln@m\Dist@n
\ctr@ln@m\l@suite
\ctr@ld@f\def\@keldist#1#2{\edef\Dist@n{#2}\y@tiunit{\Dist@n}%
    \ifunitpr@sent#1=\Dist@n\else#1=\Dist@n\unit@\fi}
\ctr@ld@f\def\y@tiunit#1{\unitpr@sentfalse\expandafter\y@tiunit@#1:}
\ctr@ld@f\def\y@tiunit@#1#2:{\ifcat#1a\unitpr@senttrue\else\def\l@suite{#2}%
    \ifx\l@suite\empty\else\y@tiunit@#2:\fi\fi}
\ctr@ln@m\figcoord
\ctr@ld@f\def\figcoordDD#1{{\v@lX=\ptT@unit@\v@lX\v@lY=\ptT@unit@\v@lY%
    \ifr@undcoord\ifcase#1\v@leur=0.5pt\or\v@leur=0.05pt\or\v@leur=0.005pt%
    \or\v@leur=0.0005pt\else\v@leur=\z@\fi%
    \ifdim\v@lX<\z@\advance\v@lX-\v@leur\else\advance\v@lX\v@leur\fi%
    \ifdim\v@lY<\z@\advance\v@lY-\v@leur\else\advance\v@lY\v@leur\fi\fi%
    (\@ffichnb{#1}{\repdecn@mb{\v@lX}},\ifmmode\else\thinspace\fi%
    \@ffichnb{#1}{\repdecn@mb{\v@lY}})}}
\ctr@ld@f\def\@ffichnb#1#2{{\def\@@ffich{\@ffich#1(}\edef\n@mbre{#2}%
    \expandafter\@@ffich\n@mbre)}}
\ctr@ld@f\def\@ffich#1(#2.#3){{#2\ifnum#1>\z@.\fi\def\dig@ts{#3}\s@mme=\z@%
    \loop\ifnum\s@mme<#1\expandafter\@ffichdec\dig@ts:\advance\s@mme\@ne\repeat}}
\ctr@ld@f\def\@ffichdec#1#2:{\relax#1\def\dig@ts{#20}}
\ctr@ld@f\def\figcoordTD#1{{\v@lX=\ptT@unit@\v@lX\v@lY=\ptT@unit@\v@lY\v@lZ=\ptT@unit@\v@lZ%
    \ifr@undcoord\ifcase#1\v@leur=0.5pt\or\v@leur=0.05pt\or\v@leur=0.005pt%
    \or\v@leur=0.0005pt\else\v@leur=\z@\fi%
    \ifdim\v@lX<\z@\advance\v@lX-\v@leur\else\advance\v@lX\v@leur\fi%
    \ifdim\v@lY<\z@\advance\v@lY-\v@leur\else\advance\v@lY\v@leur\fi%
    \ifdim\v@lZ<\z@\advance\v@lZ-\v@leur\else\advance\v@lZ\v@leur\fi\fi%
    (\@ffichnb{#1}{\repdecn@mb{\v@lX}},\ifmmode\else\thinspace\fi%
     \@ffichnb{#1}{\repdecn@mb{\v@lY}},\ifmmode\else\thinspace\fi%
     \@ffichnb{#1}{\repdecn@mb{\v@lZ}})}}
\ctr@ld@f\def\figsetroundcoord#1{\expandafter\Figsetr@undcoord#1:\ignorespaces}
\ctr@ld@f\def\Figsetr@undcoord#1#2:{\if#1n\r@undcoordfalse\else\r@undcoordtrue\fi}
\ctr@ld@f\def\Figsetwr@te#1=#2|{\keln@mun#1|%
    \def\n@mref{m}\ifx\l@debut\n@mref\figsetmark{#2}\else
    \immediate\write16{*** Unknown attribute: \BS@ figset (..., #1=...)}%
    \fi}
\ctr@ld@f\def\figsetmark#1{\c@nsymb={#1}\setbox\Gb@xSC=\hbox{\the\c@nsymb}\ignorespaces}
\ctr@ln@m\ptn@me
\ctr@ld@f\def\figsetptname#1{\def\ptn@me##1{#1}\ignorespaces}
\ctr@ld@f\def\FigWrit@L#1:#2(#3,#4){\ignorespaces\@keldist\v@leur{#3}\@keldist\delt@{#4}%
    \C@rp@r@m\def\list@num{#1}\@ecfor\p@int:=\list@num\do{\FigWrit@pt{\p@int}{#2}}}
\ctr@ld@f\def\FigWrit@pt#1#2{\FigWp@r@m{#1}{#2}\Vc@rrect\figWp@si%
    \ifdim\wd\Gb@xSC>\z@\b@undb@xSC{\v@lX}{\v@lY}\fi\figWBB@x}
\ctr@ld@f\def\FigWp@r@m#1#2{\Figg@tXY{#1}%
    \setbox\Gb@x=\hbox{\def\t@xt@{#2}\ifx\t@xt@\empty\Figg@tT{#1}\else#2\fi}\c@lprojSP}
\ctr@ld@f\let\Vc@rrect=\relax
\ctr@ld@f\let\C@rp@r@m=\relax
\ctr@ld@f\def\figwrite[#1]#2{{\ignorespaces\def\list@num{#1}\@ecfor\p@int:=\list@num\do{%
    \setbox\Gb@x=\hbox{\def\t@xt@{#2}\ifx\t@xt@\empty\Figg@tT{\p@int}\else#2\fi}%
    \Figwrit@{\p@int}}}\ignorespaces}
\ctr@ld@f\def\Figwrit@#1{\Figg@tXY{#1}\c@lprojSP%
    \rlap{\kern\v@lX\raise\v@lY\hbox{\unhcopy\Gb@x}}\v@leur=\v@lY%
    \advance\v@lY\ht\Gb@x\b@undb@x{\v@lX}{\v@lY}\advance\v@lX\wd\Gb@x%
    \v@lY=\v@leur\advance\v@lY-\dp\Gb@x\b@undb@x{\v@lX}{\v@lY}}
\ctr@ld@f\def\figwritec[#1]#2{{\ignorespaces\def\list@num{#1}%
    \@ecfor\p@int:=\list@num\do{\Figwrit@c{\p@int}{#2}}}\ignorespaces}
\ctr@ld@f\def\Figwrit@c#1#2{\FigWp@r@m{#1}{#2}%
    \rlap{\kern\v@lX\raise\v@lY\hbox{\rlap{\kern-.5\wd\Gb@x%
    \raise-.5\ht\Gb@x\hbox{\raise.5\dp\Gb@x\hbox{\unhcopy\Gb@x}}}}}%
    \v@leur=\ht\Gb@x\advance\v@leur\dp\Gb@x%
    \advance\v@lX-.5\wd\Gb@x\advance\v@lY-.5\v@leur\b@undb@x{\v@lX}{\v@lY}%
    \advance\v@lX\wd\Gb@x\advance\v@lY\v@leur\b@undb@x{\v@lX}{\v@lY}}
\ctr@ld@f\def\figwritep[#1]{{\ignorespaces\def\list@num{#1}\setbox\Gb@x=\hbox{\c@nterpt}%
    \@ecfor\p@int:=\list@num\do{\Figwrit@{\p@int}}}\ignorespaces}
\ctr@ld@f\def\figwritew#1:#2(#3){\figwritegcw#1:{#2}(#3,0pt)}
\ctr@ld@f\def\figwritee#1:#2(#3){\figwritegce#1:{#2}(#3,0pt)}
\ctr@ld@f\def\figwriten#1:#2(#3){{\def\Vc@rrect{\v@lZ=\v@leur\advance\v@lZ\dp\Gb@x}%
    \Figwrit@NS#1:{#2}(#3)}\ignorespaces}
\ctr@ld@f\def\figwrites#1:#2(#3){{\def\Vc@rrect{\v@lZ=-\v@leur\advance\v@lZ-\ht\Gb@x}%
    \Figwrit@NS#1:{#2}(#3)}\ignorespaces}
\ctr@ld@f\def\Figwrit@NS#1:#2(#3){\let\figWp@si=\FigWp@siNS\let\figWBB@x=\FigWBB@xNS%
    \FigWrit@L#1:{#2}(#3,0pt)}
\ctr@ld@f\def\FigWp@siNS{\rlap{\kern\v@lX\raise\v@lY\hbox{\rlap{\kern-.5\wd\Gb@x%
    \raise\v@lZ\hbox{\unhcopy\Gb@x}}\c@nterpt}}}
\ctr@ld@f\def\FigWBB@xNS{\advance\v@lY\v@lZ%
    \advance\v@lY-\dp\Gb@x\advance\v@lX-.5\wd\Gb@x\b@undb@x{\v@lX}{\v@lY}%
    \advance\v@lY\ht\Gb@x\advance\v@lY\dp\Gb@x%
    \advance\v@lX\wd\Gb@x\b@undb@x{\v@lX}{\v@lY}}
\ctr@ld@f\def\figwritenw#1:#2(#3){{\let\figWp@si=\FigWp@sigW\let\figWBB@x=\FigWBB@xgWE%
    \def\C@rp@r@m{\v@leur=\unssqrttw@\v@leur\delt@=\v@leur%
    \ifdim\delt@=\z@\delt@=\epsil@n\fi}\let@xte={-}\FigWrit@L#1:{#2}(#3,0pt)}\ignorespaces}
\ctr@ld@f\def\figwritesw#1:#2(#3){{\let\figWp@si=\FigWp@sigW\let\figWBB@x=\FigWBB@xgWE%
    \def\C@rp@r@m{\v@leur=\unssqrttw@\v@leur\delt@=-\v@leur%
    \ifdim\delt@=\z@\delt@=-\epsil@n\fi}\let@xte={-}\FigWrit@L#1:{#2}(#3,0pt)}\ignorespaces}
\ctr@ld@f\def\figwritene#1:#2(#3){{\let\figWp@si=\FigWp@sigE\let\figWBB@x=\FigWBB@xgWE%
    \def\C@rp@r@m{\v@leur=\unssqrttw@\v@leur\delt@=\v@leur%
    \ifdim\delt@=\z@\delt@=\epsil@n\fi}\let@xte={}\FigWrit@L#1:{#2}(#3,0pt)}\ignorespaces}
\ctr@ld@f\def\figwritese#1:#2(#3){{\let\figWp@si=\FigWp@sigE\let\figWBB@x=\FigWBB@xgWE%
    \def\C@rp@r@m{\v@leur=\unssqrttw@\v@leur\delt@=-\v@leur%
    \ifdim\delt@=\z@\delt@=-\epsil@n\fi}\let@xte={}\FigWrit@L#1:{#2}(#3,0pt)}\ignorespaces}
\ctr@ld@f\def\figwritegw#1:#2(#3,#4){{\let\figWp@si=\FigWp@sigW\let\figWBB@x=\FigWBB@xgWE%
    \let@xte={-}\FigWrit@L#1:{#2}(#3,#4)}\ignorespaces}
\ctr@ld@f\def\figwritege#1:#2(#3,#4){{\let\figWp@si=\FigWp@sigE\let\figWBB@x=\FigWBB@xgWE%
    \let@xte={}\FigWrit@L#1:{#2}(#3,#4)}\ignorespaces}
\ctr@ld@f\def\FigWp@sigW{\v@lXa=\z@\v@lYa=\ht\Gb@x\advance\v@lYa\dp\Gb@x%
    \ifdim\delt@>\z@\relax%
    \rlap{\kern\v@lX\raise\v@lY\hbox{\rlap{\kern-\wd\Gb@x\kern-\v@leur%
          \raise\delt@\hbox{\raise\dp\Gb@x\hbox{\unhcopy\Gb@x}}}\c@nterpt}}%
    \else\ifdim\delt@<\z@\relax\v@lYa=-\v@lYa%
    \rlap{\kern\v@lX\raise\v@lY\hbox{\rlap{\kern-\wd\Gb@x\kern-\v@leur%
          \raise\delt@\hbox{\raise-\ht\Gb@x\hbox{\unhcopy\Gb@x}}}\c@nterpt}}%
    \else\v@lXa=-.5\v@lYa%
    \rlap{\kern\v@lX\raise\v@lY\hbox{\rlap{\kern-\wd\Gb@x\kern-\v@leur%
          \raise-.5\ht\Gb@x\hbox{\raise.5\dp\Gb@x\hbox{\unhcopy\Gb@x}}}\c@nterpt}}%
    \fi\fi}
\ctr@ld@f\def\FigWp@sigE{\v@lXa=\z@\v@lYa=\ht\Gb@x\advance\v@lYa\dp\Gb@x%
    \ifdim\delt@>\z@\relax%
    \rlap{\kern\v@lX\raise\v@lY\hbox{\c@nterpt\kern\v@leur%
          \raise\delt@\hbox{\raise\dp\Gb@x\hbox{\unhcopy\Gb@x}}}}%
    \else\ifdim\delt@<\z@\relax\v@lYa=-\v@lYa%
    \rlap{\kern\v@lX\raise\v@lY\hbox{\c@nterpt\kern\v@leur%
          \raise\delt@\hbox{\raise-\ht\Gb@x\hbox{\unhcopy\Gb@x}}}}%
    \else\v@lXa=-.5\v@lYa%
    \rlap{\kern\v@lX\raise\v@lY\hbox{\c@nterpt\kern\v@leur%
          \raise-.5\ht\Gb@x\hbox{\raise.5\dp\Gb@x\hbox{\unhcopy\Gb@x}}}}%
    \fi\fi}
\ctr@ld@f\def\FigWBB@xgWE{\advance\v@lY\delt@%
    \advance\v@lX\the\let@xte\v@leur\advance\v@lY\v@lXa\b@undb@x{\v@lX}{\v@lY}%
    \advance\v@lX\the\let@xte\wd\Gb@x\advance\v@lY\v@lYa\b@undb@x{\v@lX}{\v@lY}}
\ctr@ld@f\def\figwritegcw#1:#2(#3,#4){{\let\figWp@si=\FigWp@sigcW\let\figWBB@x=\FigWBB@xgcWE%
    \let@xte={-}\FigWrit@L#1:{#2}(#3,#4)}\ignorespaces}
\ctr@ld@f\def\figwritegce#1:#2(#3,#4){{\let\figWp@si=\FigWp@sigcE\let\figWBB@x=\FigWBB@xgcWE%
    \let@xte={}\FigWrit@L#1:{#2}(#3,#4)}\ignorespaces}
\ctr@ld@f\def\FigWp@sigcW{\rlap{\kern\v@lX\raise\v@lY\hbox{\rlap{\kern-\wd\Gb@x\kern-\v@leur%
     \raise-.5\ht\Gb@x\hbox{\raise\delt@\hbox{\raise.5\dp\Gb@x\hbox{\unhcopy\Gb@x}}}}%
     \c@nterpt}}}
\ctr@ld@f\def\FigWp@sigcE{\rlap{\kern\v@lX\raise\v@lY\hbox{\c@nterpt\kern\v@leur%
    \raise-.5\ht\Gb@x\hbox{\raise\delt@\hbox{\raise.5\dp\Gb@x\hbox{\unhcopy\Gb@x}}}}}}
\ctr@ld@f\def\FigWBB@xgcWE{\v@lZ=\ht\Gb@x\advance\v@lZ\dp\Gb@x%
    \advance\v@lX\the\let@xte\v@leur\advance\v@lY\delt@\advance\v@lY.5\v@lZ%
    \b@undb@x{\v@lX}{\v@lY}%
    \advance\v@lX\the\let@xte\wd\Gb@x\advance\v@lY-\v@lZ\b@undb@x{\v@lX}{\v@lY}}
\ctr@ld@f\def\figwritebn#1:#2(#3){{\def\Vc@rrect{\v@lZ=\v@leur}\Figwrit@NS#1:{#2}(#3)}\ignorespaces}
\ctr@ld@f\def\figwritebs#1:#2(#3){{\def\Vc@rrect{\v@lZ=-\v@leur}\Figwrit@NS#1:{#2}(#3)}\ignorespaces}
\ctr@ld@f\def\figwritebw#1:#2(#3){{\let\figWp@si=\FigWp@sibW\let\figWBB@x=\FigWBB@xbWE%
    \let@xte={-}\FigWrit@L#1:{#2}(#3,0pt)}\ignorespaces}
\ctr@ld@f\def\figwritebe#1:#2(#3){{\let\figWp@si=\FigWp@sibE\let\figWBB@x=\FigWBB@xbWE%
    \let@xte={}\FigWrit@L#1:{#2}(#3,0pt)}\ignorespaces}
\ctr@ld@f\def\FigWp@sibW{\rlap{\kern\v@lX\raise\v@lY\hbox{\rlap{\kern-\wd\Gb@x\kern-\v@leur%
          \hbox{\unhcopy\Gb@x}}\c@nterpt}}}
\ctr@ld@f\def\FigWp@sibE{\rlap{\kern\v@lX\raise\v@lY\hbox{\c@nterpt\kern\v@leur%
          \hbox{\unhcopy\Gb@x}}}}
\ctr@ld@f\def\FigWBB@xbWE{\v@lZ=\ht\Gb@x\advance\v@lZ\dp\Gb@x%
    \advance\v@lX\the\let@xte\v@leur\advance\v@lY\ht\Gb@x\b@undb@x{\v@lX}{\v@lY}%
    \advance\v@lX\the\let@xte\wd\Gb@x\advance\v@lY-\v@lZ\b@undb@x{\v@lX}{\v@lY}}
\ctr@ln@w{newread}\frf@g  \ctr@ln@w{newwrite}\fwf@g
\ctr@ln@w{newif}\ifcurr@ntPS
\ctr@ln@w{newif}\ifps@cri
\ctr@ln@w{newif}\ifUse@llipse
\ctr@ln@w{newif}\ifpsdebugmode \psdebugmodefalse 
\ctr@ln@w{newif}\ifPDFm@ke
\ifx\pdfliteral\undefined\else\ifnum\pdfoutput>\z@\PDFm@ketrue\fi\fi
\ctr@ld@f\def\initPDF@rDVI{%
\ifPDFm@ke
 \let\figscan=\figscan@E
 \let\newGr@FN=\newGr@FNPDF
 \ctr@ld@f\def\c@mcurveto{c}
 \ctr@ld@f\def\c@mfill{f}
 \ctr@ld@f\def\c@mgsave{q}
 \ctr@ld@f\def\c@mgrestore{Q}
 \ctr@ld@f\def\c@mlineto{l}
 \ctr@ld@f\def\c@mmoveto{m}
 \ctr@ld@f\def\c@msetgray{g}     \ctr@ld@f\def\c@msetgrayStroke{G}
 \ctr@ld@f\def\c@msetcmykcolor{k}\ctr@ld@f\def\c@msetcmykcolorStroke{K}
 \ctr@ld@f\def\c@msetrgbcolor{rg}\ctr@ld@f\def\c@msetrgbcolorStroke{RG}
 \ctr@ld@f\def\d@fprimarC@lor{\curr@ntcolor\space\curr@ntcolorc@md%
               \space\curr@ntcolor\space\curr@ntcolorc@mdStroke}
 \ctr@ld@f\def\d@fsecondC@lor{\sec@ndcolor\space\sec@ndcolorc@md%
               \space\sec@ndcolor\space\sec@ndcolorc@mdStroke}
 \ctr@ld@f\def\d@fthirdC@lor{\th@rdcolor\space\th@rdcolorc@md%
              \space\th@rdcolor\space\th@rdcolorc@mdStroke}
 \ctr@ld@f\def\c@msetdash{d}
 \ctr@ld@f\def\c@msetlinejoin{j}
 \ctr@ld@f\def\c@msetlinewidth{w}
 \ctr@ld@f\def\f@gclosestroke{\immediate\write\fwf@g{s}}
 \ctr@ld@f\def\f@gfill{\immediate\write\fwf@g{\fillc@md}}
 \ctr@ld@f\def\f@gnewpath{}
 \ctr@ld@f\def\f@gstroke{\immediate\write\fwf@g{S}}
\else
 \let\figinsertE=\figinsert
 \let\newGr@FN=\newGr@FNDVI
 \ctr@ld@f\def\c@mcurveto{curveto}
 \ctr@ld@f\def\c@mfill{fill}
 \ctr@ld@f\def\c@mgsave{gsave}
 \ctr@ld@f\def\c@mgrestore{grestore}
 \ctr@ld@f\def\c@mlineto{lineto}
 \ctr@ld@f\def\c@mmoveto{moveto}
 \ctr@ld@f\def\c@msetgray{setgray}          \ctr@ld@f\def\c@msetgrayStroke{}
 \ctr@ld@f\def\c@msetcmykcolor{setcmykcolor}\ctr@ld@f\def\c@msetcmykcolorStroke{}
 \ctr@ld@f\def\c@msetrgbcolor{setrgbcolor}  \ctr@ld@f\def\c@msetrgbcolorStroke{}
 \ctr@ld@f\def\d@fprimarC@lor{\curr@ntcolor\space\curr@ntcolorc@md}
 \ctr@ld@f\def\d@fsecondC@lor{\sec@ndcolor\space\sec@ndcolorc@md}
 \ctr@ld@f\def\d@fthirdC@lor{\th@rdcolor\space\th@rdcolorc@md}
 \ctr@ld@f\def\c@msetdash{setdash}
 \ctr@ld@f\def\c@msetlinejoin{setlinejoin}
 \ctr@ld@f\def\c@msetlinewidth{setlinewidth}
 \ctr@ld@f\def\f@gclosestroke{\immediate\write\fwf@g{closepath\space stroke}}
 \ctr@ld@f\def\f@gfill{\immediate\write\fwf@g{\fillc@md}}
 \ctr@ld@f\def\f@gnewpath{\immediate\write\fwf@g{newpath}}
 \ctr@ld@f\def\f@gstroke{\immediate\write\fwf@g{stroke}}
\fi}
\ctr@ld@f\def\c@pypsfile#1#2{\c@pyfil@{\immediate\write#1}{#2}}
\ctr@ld@f\def\Figinclud@PDF#1#2{\openin\frf@g=#1\pdfliteral{q #2 0 0 #2 0 0 cm}%
    \c@pyfil@{\pdfliteral}{\frf@g}\pdfliteral{Q}\closein\frf@g}
\ctr@ln@w{newif}\ifmored@ta
\ctr@ln@m\bl@nkline
\ctr@ld@f\def\c@pyfil@#1#2{\def\bl@nkline{\par}{\catcode`\%=12
    \loop\ifeof#2\mored@tafalse\else\mored@tatrue\immediate\read#2 to\tr@c
    \ifx\tr@c\bl@nkline\else#1{\tr@c}\fi\fi\ifmored@ta\repeat}}
\ctr@ld@f\def\keln@mun#1#2|{\def\l@debut{#1}\def\l@suite{#2}}
\ctr@ld@f\def\keln@mde#1#2#3|{\def\l@debut{#1#2}\def\l@suite{#3}}
\ctr@ld@f\def\keln@mtr#1#2#3#4|{\def\l@debut{#1#2#3}\def\l@suite{#4}}
\ctr@ld@f\def\keln@mqu#1#2#3#4#5|{\def\l@debut{#1#2#3#4}\def\l@suite{#5}}
\ctr@ld@f\let\@psffilein=\frf@g 
\ctr@ln@w{newif}\if@psffileok    
\ctr@ln@w{newif}\if@psfbbfound   
\ctr@ln@w{newif}\if@psfverbose   
\@psfverbosetrue
\ctr@ln@m\@psfllx \ctr@ln@m\@psflly
\ctr@ln@m\@psfurx \ctr@ln@m\@psfury
\ctr@ln@m\resetcolonc@tcode
\ctr@ld@f\def\@psfgetbb#1{\global\@psfbbfoundfalse%
\global\def\@psfllx{0}\global\def\@psflly{0}%
\global\def\@psfurx{30}\global\def\@psfury{30}%
\openin\@psffilein=#1\relax
\ifeof\@psffilein\errmessage{I couldn't open #1, will ignore it}\else
   \edef\resetcolonc@tcode{\catcode`\noexpand\:\the\catcode`\:\relax}%
   {\@psffileoktrue \chardef\other=12
    \def\do##1{\catcode`##1=\other}\dospecials \catcode`\ =10 \resetcolonc@tcode
    \loop
       \read\@psffilein to \@psffileline
       \ifeof\@psffilein\@psffileokfalse\else
          \expandafter\@psfaux\@psffileline:. \\%
       \fi
   \if@psffileok\repeat
   \if@psfbbfound\else
    \if@psfverbose\message{No bounding box comment in #1; using defaults}\fi\fi
   }\closein\@psffilein\fi}%
\ctr@ln@m\@psfbblit
\ctr@ln@m\@psfpercent
{\catcode`\%=12 \global\let\@psfpercent=
\ctr@ln@m\@psfaux
\long\def\@psfaux#1#2:#3\\{\ifx#1\@psfpercent
   \def\testit{#2}\ifx\testit\@psfbblit
      \@psfgrab #3 . . . \\%
      \@psffileokfalse
      \global\@psfbbfoundtrue
   \fi\else\ifx#1\par\else\@psffileokfalse\fi\fi}%
\ctr@ld@f\def\@psfempty{}%
\ctr@ld@f\def\@psfgrab #1 #2 #3 #4 #5\\{%
\global\def\@psfllx{#1}\ifx\@psfllx\@psfempty
      \@psfgrab #2 #3 #4 #5 .\\\else
   \global\def\@psflly{#2}%
   \global\def\@psfurx{#3}\global\def\@psfury{#4}\fi}%
\ctr@ld@f\def\PSwrit@cmd#1#2#3{{\Figg@tXY{#1}\c@lprojSP\b@undb@x{\v@lX}{\v@lY}%
    \v@lX=\ptT@ptps\v@lX\v@lY=\ptT@ptps\v@lY%
    \immediate\write#3{\repdecn@mb{\v@lX}\space\repdecn@mb{\v@lY}\space#2}}}
\ctr@ld@f\def\PSwrit@cmdS#1#2#3#4#5{{\Figg@tXY{#1}\c@lprojSP\b@undb@x{\v@lX}{\v@lY}%
    \global\result@t=\v@lX\global\result@@t=\v@lY%
    \v@lX=\ptT@ptps\v@lX\v@lY=\ptT@ptps\v@lY%
    \immediate\write#3{\repdecn@mb{\v@lX}\space\repdecn@mb{\v@lY}\space#2}}%
    \edef#4{\the\result@t}\edef#5{\the\result@@t}}
\ctr@ld@f\def\psaltitude#1[#2,#3,#4]{{\ifcurr@ntPS\ifps@cri%
    \PSc@mment{psaltitude Square Dim=#1, Triangle=[#2 / #3,#4]}%
    \s@uvc@ntr@l\et@tpsaltitude\resetc@ntr@l{2}\figptorthoprojline-5:=#2/#3,#4/%
    \figvectP -1[#3,#4]\n@rminf{\v@leur}{-1}\vecunit@{-3}{-1}%
    \figvectP -1[-5,#3]\n@rminf{\v@lmin}{-1}\figvectP -2[-5,#4]\n@rminf{\v@lmax}{-2}%
    \ifdim\v@lmin<\v@lmax\s@mme=#3\else\v@lmax=\v@lmin\s@mme=#4\fi%
    \figvectP -4[-5,#2]\vecunit@{-4}{-4}\delt@=#1\unit@%
    \edef\t@ille{\repdecn@mb{\delt@}}\figpttra-1:=-5/\t@ille,-3/%
    \figptstra-3=-5,-1/\t@ille,-4/\psline[#2,-5]\s@uvdash{\typ@dash}%
    \pssetdash{\defaultdash}\psline[-1,-2,-3]\pssetdash{\typ@dash}%
    \ifdim\v@leur<\v@lmax\Pss@tsecondSt\psline[-5,\the\s@mme]\Psrest@reSt\fi%
    \PSc@mment{End psaltitude}\resetc@ntr@l\et@tpsaltitude\fi\fi}}
\ctr@ld@f\def\Ps@rcerc#1;#2(#3,#4){\ellBB@x#1;#2,#2(#3,#4,0)%
    \f@gnewpath{\delt@=#2\unit@\delt@=\ptT@ptps\delt@%
    \BdingB@xfalse%
    \PSwrit@cmd{#1}{\repdecn@mb{\delt@}\space #3\space #4\space arc}{\fwf@g}}}
\ctr@ln@m\psarccirc
\ctr@ld@f\def\psarccircDD#1;#2(#3,#4){\ifcurr@ntPS\ifps@cri%
    \PSc@mment{psarccircDD Center=#1 ; Radius=#2 (Ang1=#3, Ang2=#4)}%
    \iffillm@de\Ps@rcerc#1;#2(#3,#4)%
    \f@gfill%
    \else\Ps@rcerc#1;#2(#3,#4)\f@gstroke\fi%
    \PSc@mment{End psarccircDD}\fi\fi}
\ctr@ld@f\def\psarccircTD#1,#2,#3;#4(#5,#6){{\ifcurr@ntPS\ifps@cri\s@uvc@ntr@l\et@tpsarccircTD%
    \PSc@mment{psarccircTD Center=#1,P1=#2,P2=#3 ; Radius=#4 (Ang1=#5, Ang2=#6)}%
    \setc@ntr@l{2}\c@lExtAxes#1,#2,#3(#4)\psarcellPATD#1,-4,-5(#5,#6)%
    \PSc@mment{End psarccircTD}\resetc@ntr@l\et@tpsarccircTD\fi\fi}}
\ctr@ld@f\def\c@lExtAxes#1,#2,#3(#4){%
    \figvectPTD-5[#1,#2]\vecunit@{-5}{-5}\figvectNTD-4[#1,#2,#3]\vecunit@{-4}{-4}%
    \figvectNVTD-3[-4,-5]\delt@=#4\unit@\edef\r@yon{\repdecn@mb{\delt@}}%
    \figpttra-4:=#1/\r@yon,-5/\figpttra-5:=#1/\r@yon,-3/}
\ctr@ln@m\psarccircP
\ctr@ld@f\def\psarccircPDD#1;#2[#3,#4]{{\ifcurr@ntPS\ifps@cri\s@uvc@ntr@l\et@tpsarccircPDD%
    \PSc@mment{psarccircPDD Center=#1; Radius=#2, [P1=#3, P2=#4]}%
    \Ps@ngleparam#1;#2[#3,#4]\ifdim\v@lmin>\v@lmax\advance\v@lmax\DePI@deg\fi%
    \edef\@ngdeb{\repdecn@mb{\v@lmin}}\edef\@ngfin{\repdecn@mb{\v@lmax}}%
    \psarccirc#1;\r@dius(\@ngdeb,\@ngfin)%
    \PSc@mment{End psarccircPDD}\resetc@ntr@l\et@tpsarccircPDD\fi\fi}}
\ctr@ld@f\def\psarccircPTD#1;#2[#3,#4,#5]{{\ifcurr@ntPS\ifps@cri\s@uvc@ntr@l\et@tpsarccircPTD%
    \PSc@mment{psarccircPTD Center=#1; Radius=#2, [P1=#3, P2=#4, P3=#5]}%
    \setc@ntr@l{2}\c@lExtAxes#1,#3,#5(#2)\psarcellPP#1,-4,-5[#3,#4]%
    \PSc@mment{End psarccircPTD}\resetc@ntr@l\et@tpsarccircPTD\fi\fi}}
\ctr@ld@f\def\Ps@ngleparam#1;#2[#3,#4]{\setc@ntr@l{2}%
    \figvectPDD-1[#1,#3]\vecunit@{-1}{-1}\Figg@tXY{-1}\arct@n\v@lmin(\v@lX,\v@lY)%
    \figvectPDD-2[#1,#4]\vecunit@{-2}{-2}\Figg@tXY{-2}\arct@n\v@lmax(\v@lX,\v@lY)%
    \v@lmin=\rdT@deg\v@lmin\v@lmax=\rdT@deg\v@lmax%
    \v@leur=#2pt\maxim@m{\mili@u}{-\v@leur}{\v@leur}%
    \edef\r@dius{\repdecn@mb{\mili@u}}}
\ctr@ld@f\def\Ps@rcercBz#1;#2(#3,#4){\Ps@rellBz#1;#2,#2(#3,#4,0)}
\ctr@ld@f\def\Ps@rellBz#1;#2,#3(#4,#5,#6){%
    \ellBB@x#1;#2,#3(#4,#5,#6)\BdingB@xfalse%
    \c@lNbarcs{#4}{#5}\v@leur=#4pt\setc@ntr@l{2}\figptell-13::#1;#2,#3(#4,#6)%
    \f@gnewpath\PSwrit@cmd{-13}{\c@mmoveto}{\fwf@g}%
    \s@mme=\z@\bcl@rellBz#1;#2,#3(#6)\BdingB@xtrue}
\ctr@ld@f\def\bcl@rellBz#1;#2,#3(#4){\relax%
    \ifnum\s@mme<\p@rtent\advance\s@mme\@ne%
    \advance\v@leur\delt@\edef\@ngle{\repdecn@mb\v@leur}\figptell-14::#1;#2,#3(\@ngle,#4)%
    \advance\v@leur\delt@\edef\@ngle{\repdecn@mb\v@leur}\figptell-15::#1;#2,#3(\@ngle,#4)%
    \advance\v@leur\delt@\edef\@ngle{\repdecn@mb\v@leur}\figptell-16::#1;#2,#3(\@ngle,#4)%
    \figptscontrolDD-18[-13,-14,-15,-16]%
    \PSwrit@cmd{-18}{}{\fwf@g}\PSwrit@cmd{-17}{}{\fwf@g}%
    \PSwrit@cmd{-16}{\c@mcurveto}{\fwf@g}%
    \figptcopyDD-13:/-16/\bcl@rellBz#1;#2,#3(#4)\fi}
\ctr@ld@f\def\Ps@rell#1;#2,#3(#4,#5,#6){\ellBB@x#1;#2,#3(#4,#5,#6)%
    \f@gnewpath{\v@lmin=#2\unit@\v@lmin=\ptT@ptps\v@lmin%
    \v@lmax=#3\unit@\v@lmax=\ptT@ptps\v@lmax\BdingB@xfalse%
    \PSwrit@cmd{#1}%
    {#6\space\repdecn@mb{\v@lmin}\space\repdecn@mb{\v@lmax}\space #4\space #5\space ellipse}{\fwf@g}}%
    \global\Use@llipsetrue}
\ctr@ln@m\psarcell
\ctr@ld@f\def\psarcellDD#1;#2,#3(#4,#5,#6){{\ifcurr@ntPS\ifps@cri%
    \PSc@mment{psarcellDD Center=#1 ; XRad=#2, YRad=#3 (Ang1=#4, Ang2=#5, Inclination=#6)}%
    \iffillm@de\Ps@rell#1;#2,#3(#4,#5,#6)%
    \f@gfill%
    \else\Ps@rell#1;#2,#3(#4,#5,#6)\f@gstroke\fi%
    \PSc@mment{End psarcellDD}\fi\fi}}
\ctr@ld@f\def\psarcellTD#1;#2,#3(#4,#5,#6){{\ifcurr@ntPS\ifps@cri\s@uvc@ntr@l\et@tpsarcellTD%
    \PSc@mment{psarcellTD Center=#1 ; XRad=#2, YRad=#3 (Ang1=#4, Ang2=#5, Inclination=#6)}%
    \setc@ntr@l{2}\figpttraC -8:=#1/#2,0,0/\figpttraC -7:=#1/0,#3,0/%
    \figvectC -4(0,0,1)\figptsrot -8=-8,-7/#1,#6,-4/\psarcellPATD#1,-8,-7(#4,#5)%
    \PSc@mment{End psarcellTD}\resetc@ntr@l\et@tpsarcellTD\fi\fi}}
\ctr@ln@m\psarcellPA
\ctr@ld@f\def\psarcellPADD#1,#2,#3(#4,#5){{\ifcurr@ntPS\ifps@cri\s@uvc@ntr@l\et@tpsarcellPADD%
    \PSc@mment{psarcellPADD Center=#1,PtAxis1=#2,PtAxis2=#3 (Ang1=#4, Ang2=#5)}%
    \setc@ntr@l{2}\figvectPDD-1[#1,#2]\vecunit@DD{-1}{-1}\v@lX=\ptT@unit@\result@t%
    \edef\XR@d{\repdecn@mb{\v@lX}}\Figg@tXY{-1}\arct@n\v@lmin(\v@lX,\v@lY)%
    \v@lmin=\rdT@deg\v@lmin\edef\Inclin@{\repdecn@mb{\v@lmin}}%
    \figgetdist\YR@d[#1,#3]\psarcellDD#1;\XR@d,\YR@d(#4,#5,\Inclin@)%
    \PSc@mment{End psarcellPADD}\resetc@ntr@l\et@tpsarcellPADD\fi\fi}}
\ctr@ld@f\def\psarcellPATD#1,#2,#3(#4,#5){{\ifcurr@ntPS\ifps@cri\s@uvc@ntr@l\et@tpsarcellPATD%
    \PSc@mment{psarcellPATD Center=#1,PtAxis1=#2,PtAxis2=#3 (Ang1=#4, Ang2=#5)}%
    \iffillm@de\Ps@rellPATD#1,#2,#3(#4,#5)%
    \f@gfill%
    \else\Ps@rellPATD#1,#2,#3(#4,#5)\f@gstroke\fi%
    \PSc@mment{End psarcellPATD}\resetc@ntr@l\et@tpsarcellPATD\fi\fi}}
\ctr@ld@f\def\Ps@rellPATD#1,#2,#3(#4,#5){\let\c@lprojSP=\relax%
    \setc@ntr@l{2}\figvectPTD-1[#1,#2]\figvectPTD-2[#1,#3]\c@lNbarcs{#4}{#5}%
    \v@leur=#4pt\c@lptellP{#1}{-1}{-2}\Figptpr@j-5:/-3/%
    \f@gnewpath\PSwrit@cmdS{-5}{\c@mmoveto}{\fwf@g}{\X@un}{\Y@un}%
    \edef\C@nt@r{#1}\s@mme=\z@\bcl@rellPATD}
\ctr@ld@f\def\bcl@rellPATD{\relax%
    \ifnum\s@mme<\p@rtent\advance\s@mme\@ne%
    \advance\v@leur\delt@\c@lptellP{\C@nt@r}{-1}{-2}\Figptpr@j-4:/-3/%
    \advance\v@leur\delt@\c@lptellP{\C@nt@r}{-1}{-2}\Figptpr@j-6:/-3/%
    \advance\v@leur\delt@\c@lptellP{\C@nt@r}{-1}{-2}\Figptpr@j-3:/-3/%
    \v@lX=\z@\v@lY=\z@\Figtr@nptDD{-5}{-5}\Figtr@nptDD{2}{-3}%
    \divide\v@lX\@vi\divide\v@lY\@vi%
    \Figtr@nptDD{3}{-4}\Figtr@nptDD{-1.5}{-6}\v@lmin=\v@lX\v@lmax=\v@lY%
    \v@lX=\z@\v@lY=\z@\Figtr@nptDD{2}{-5}\Figtr@nptDD{-5}{-3}%
    \divide\v@lX\@vi\divide\v@lY\@vi\Figtr@nptDD{-1.5}{-4}\Figtr@nptDD{3}{-6}%
    \BdingB@xfalse%
    \Figp@intregDD-4:(\v@lmin,\v@lmax)\PSwrit@cmdS{-4}{}{\fwf@g}{\X@de}{\Y@de}%
    \Figp@intregDD-4:(\v@lX,\v@lY)\PSwrit@cmdS{-4}{}{\fwf@g}{\X@tr}{\Y@tr}%
    \BdingB@xtrue\PSwrit@cmdS{-3}{\c@mcurveto}{\fwf@g}{\X@qu}{\Y@qu}%
    \B@zierBB@x{1}{\Y@un}(\X@un,\X@de,\X@tr,\X@qu)%
    \B@zierBB@x{2}{\X@un}(\Y@un,\Y@de,\Y@tr,\Y@qu)%
    \edef\X@un{\X@qu}\edef\Y@un{\Y@qu}\figptcopyDD-5:/-3/\bcl@rellPATD\fi}
\ctr@ld@f\def\c@lNbarcs#1#2{%
    \delt@=#2pt\advance\delt@-#1pt\maxim@m{\v@lmax}{\delt@}{-\delt@}%
    \v@leur=\v@lmax\divide\v@leur45 \p@rtentiere{\p@rtent}{\v@leur}\advance\p@rtent\@ne%
    \s@mme=\p@rtent\multiply\s@mme\thr@@\divide\delt@\s@mme}
\ctr@ld@f\def\psarcellPP#1,#2,#3[#4,#5]{{\ifcurr@ntPS\ifps@cri\s@uvc@ntr@l\et@tpsarcellPP%
    \PSc@mment{psarcellPP Center=#1,PtAxis1=#2,PtAxis2=#3 [Point1=#4, Point2=#5]}%
    \setc@ntr@l{2}\figvectP-2[#1,#3]\vecunit@{-2}{-2}\v@lmin=\result@t%
    \invers@{\v@lmax}{\v@lmin}%
    \figvectP-1[#1,#2]\vecunit@{-1}{-1}\v@leur=\result@t%
    \v@leur=\repdecn@mb{\v@lmax}\v@leur\edef\AsB@{\repdecn@mb{\v@leur}}
    \c@lAngle{#1}{#4}{\v@lmin}\edef\@ngdeb{\repdecn@mb{\v@lmin}}%
    \c@lAngle{#1}{#5}{\v@lmax}\ifdim\v@lmin>\v@lmax\advance\v@lmax\DePI@deg\fi%
    \edef\@ngfin{\repdecn@mb{\v@lmax}}\psarcellPA#1,#2,#3(\@ngdeb,\@ngfin)%
    \PSc@mment{End psarcellPP}\resetc@ntr@l\et@tpsarcellPP\fi\fi}}
\ctr@ld@f\def\c@lAngle#1#2#3{\figvectP-3[#1,#2]%
    \c@lproscal\delt@[-3,-1]\c@lproscal\v@leur[-3,-2]%
    \v@leur=\AsB@\v@leur\arct@n#3(\delt@,\v@leur)#3=\rdT@deg#3}
\ctr@ln@w{newif}\if@rrowratio\@rrowratiotrue
\ctr@ln@w{newif}\if@rrowhfill
\ctr@ln@w{newif}\if@rrowhout
\ctr@ld@f\def\Psset@rrowhe@d#1=#2|{\keln@mun#1|%
    \def\n@mref{a}\ifx\l@debut\n@mref\pssetarrowheadangle{#2}\else
    \def\n@mref{f}\ifx\l@debut\n@mref\pssetarrowheadfill{#2}\else
    \def\n@mref{l}\ifx\l@debut\n@mref\pssetarrowheadlength{#2}\else
    \def\n@mref{o}\ifx\l@debut\n@mref\pssetarrowheadout{#2}\else
    \def\n@mref{r}\ifx\l@debut\n@mref\pssetarrowheadratio{#2}\else
    \immediate\write16{*** Unknown attribute: \BS@ psset arrowhead(..., #1=...)}%
    \fi\fi\fi\fi\fi}
\ctr@ln@m\@rrowheadangle
\ctr@ln@m\C@AHANG \ctr@ln@m\S@AHANG \ctr@ln@m\UNSS@N
\ctr@ld@f\def\pssetarrowheadangle#1{\edef\@rrowheadangle{#1}{\c@ssin{\C@}{\S@}{#1}%
    \xdef\C@AHANG{\C@}\xdef\S@AHANG{\S@}\v@lmax=\S@ pt%
    \invers@{\v@leur}{\v@lmax}\maxim@m{\v@leur}{\v@leur}{-\v@leur}%
    \xdef\UNSS@N{\the\v@leur}}}
\ctr@ld@f\def\pssetarrowheadfill#1{\expandafter\set@rrowhfill#1:}
\ctr@ld@f\def\set@rrowhfill#1#2:{\if#1n\@rrowhfillfalse\else\@rrowhfilltrue\fi}
\ctr@ld@f\def\pssetarrowheadout#1{\expandafter\set@rrowhout#1:}
\ctr@ld@f\def\set@rrowhout#1#2:{\if#1n\@rrowhoutfalse\else\@rrowhouttrue\fi}
\ctr@ln@m\@rrowheadlength
\ctr@ld@f\def\pssetarrowheadlength#1{\edef\@rrowheadlength{#1}\@rrowratiofalse}
\ctr@ln@m\@rrowheadratio
\ctr@ld@f\def\pssetarrowheadratio#1{\edef\@rrowheadratio{#1}\@rrowratiotrue}
\ctr@ln@m\defaultarrowheadlength
\ctr@ld@f\def\psresetarrowhead{%
    \pssetarrowheadangle{\defaultarrowheadangle}%
    \pssetarrowheadfill{\defaultarrowheadfill}%
    \pssetarrowheadout{\defaultarrowheadout}%
    \pssetarrowheadratio{\defaultarrowheadratio}%
    \d@fm@cdim\defaultarrowheadlength{\defaulth@rdahlength}
    \pssetarrowheadlength{\defaultarrowheadlength}}
\ctr@ld@f\def\defaultarrowheadratio{0.1}
\ctr@ld@f\def\defaultarrowheadangle{20}
\ctr@ld@f\def\defaultarrowheadfill{no}
\ctr@ld@f\def\defaultarrowheadout{no}
\ctr@ld@f\def\defaulth@rdahlength{8pt}
\ctr@ln@m\psarrow
\ctr@ld@f\def\psarrowDD[#1,#2]{{\ifcurr@ntPS\ifps@cri\s@uvc@ntr@l\et@tpsarrow%
    \PSc@mment{psarrowDD [Pt1,Pt2]=[#1,#2]}\pssetfillmode{no}%
    \psarrowheadDD[#1,#2]\setc@ntr@l{2}\psline[#1,-3]%
    \PSc@mment{End psarrowDD}\resetc@ntr@l\et@tpsarrow\fi\fi}}
\ctr@ld@f\def\psarrowTD[#1,#2]{{\ifcurr@ntPS\ifps@cri\s@uvc@ntr@l\et@tpsarrowTD%
    \PSc@mment{psarrowTD [Pt1,Pt2]=[#1,#2]}\resetc@ntr@l{2}%
    \Figptpr@j-5:/#1/\Figptpr@j-6:/#2/\let\c@lprojSP=\relax\psarrowDD[-5,-6]%
    \PSc@mment{End psarrowTD}\resetc@ntr@l\et@tpsarrowTD\fi\fi}}
\ctr@ln@m\psarrowhead
\ctr@ld@f\def\psarrowheadDD[#1,#2]{{\ifcurr@ntPS\ifps@cri\s@uvc@ntr@l\et@tpsarrowheadDD%
    \if@rrowhfill\def\@hangle{-\@rrowheadangle}\else\def\@hangle{\@rrowheadangle}\fi%
    \if@rrowratio%
    \if@rrowhout\def\@hratio{-\@rrowheadratio}\else\def\@hratio{\@rrowheadratio}\fi%
    \PSc@mment{psarrowheadDD Ratio=\@hratio, Angle=\@hangle, [Pt1,Pt2]=[#1,#2]}%
    \Ps@rrowhead\@hratio,\@hangle[#1,#2]%
    \else%
    \if@rrowhout\def\@hlength{-\@rrowheadlength}\else\def\@hlength{\@rrowheadlength}\fi%
    \PSc@mment{psarrowheadDD Length=\@hlength, Angle=\@hangle, [Pt1,Pt2]=[#1,#2]}%
    \Ps@rrowheadfd\@hlength,\@hangle[#1,#2]%
    \fi%
    \PSc@mment{End psarrowheadDD}\resetc@ntr@l\et@tpsarrowheadDD\fi\fi}}
\ctr@ld@f\def\psarrowheadTD[#1,#2]{{\ifcurr@ntPS\ifps@cri\s@uvc@ntr@l\et@tpsarrowheadTD%
    \PSc@mment{psarrowheadTD [Pt1,Pt2]=[#1,#2]}\resetc@ntr@l{2}%
    \Figptpr@j-5:/#1/\Figptpr@j-6:/#2/\let\c@lprojSP=\relax\psarrowheadDD[-5,-6]%
    \PSc@mment{End psarrowheadTD}\resetc@ntr@l\et@tpsarrowheadTD\fi\fi}}
\ctr@ld@f\def\Ps@rrowhead#1,#2[#3,#4]{\v@leur=#1\p@\maxim@m{\v@leur}{\v@leur}{-\v@leur}%
    \ifdim\v@leur>\Cepsil@n{
    \PSc@mment{ps@rrowhead Ratio=#1, Angle=#2, [Pt1,Pt2]=[#3,#4]}\v@leur=\UNSS@N%
    \v@leur=\curr@ntwidth\v@leur\v@leur=\ptpsT@pt\v@leur\delt@=.5\v@leur
    \setc@ntr@l{2}\figvectPDD-3[#4,#3]%
    \Figg@tXY{-3}\v@lX=#1\v@lX\v@lY=#1\v@lY\Figv@ctCreg-3(\v@lX,\v@lY)%
    \vecunit@{-4}{-3}\mili@u=\result@t%
    \ifdim#2pt>\z@\v@lXa=-\C@AHANG\delt@%
     \edef\c@ef{\repdecn@mb{\v@lXa}}\figpttraDD-3:=-3/\c@ef,-4/\fi%
    \edef\c@ef{\repdecn@mb{\delt@}}%
    \v@lXa=\mili@u\v@lXa=\C@AHANG\v@lXa%
    \v@lYa=\ptpsT@pt\p@\v@lYa=\curr@ntwidth\v@lYa\v@lYa=\sDcc@ngle\v@lYa%
    \advance\v@lXa-\v@lYa\gdef\sDcc@ngle{0}%
    \ifdim\v@lXa>\v@leur\edef\c@efendpt{\repdecn@mb{\v@leur}}%
    \else\edef\c@efendpt{\repdecn@mb{\v@lXa}}\fi%
    \Figg@tXY{-3}\v@lmin=\v@lX\v@lmax=\v@lY%
    \v@lXa=\C@AHANG\v@lmin\v@lYa=\S@AHANG\v@lmax\advance\v@lXa\v@lYa%
    \v@lYa=-\S@AHANG\v@lmin\v@lX=\C@AHANG\v@lmax\advance\v@lYa\v@lX%
    \setc@ntr@l{1}\Figg@tXY{#4}\advance\v@lX\v@lXa\advance\v@lY\v@lYa%
    \setc@ntr@l{2}\Figp@intregDD-2:(\v@lX,\v@lY)%
    \v@lXa=\C@AHANG\v@lmin\v@lYa=-\S@AHANG\v@lmax\advance\v@lXa\v@lYa%
    \v@lYa=\S@AHANG\v@lmin\v@lX=\C@AHANG\v@lmax\advance\v@lYa\v@lX%
    \setc@ntr@l{1}\Figg@tXY{#4}\advance\v@lX\v@lXa\advance\v@lY\v@lYa%
    \setc@ntr@l{2}\Figp@intregDD-1:(\v@lX,\v@lY)%
    \ifdim#2pt<\z@\fillm@detrue\psline[-2,#4,-1]
    \else\figptstraDD-3=#4,-2,-1/\c@ef,-4/\psline[-2,-3,-1]\fi
    \ifdim#1pt>\z@\figpttraDD-3:=#4/\c@efendpt,-4/\else\figptcopyDD-3:/#4/\fi%
    \PSc@mment{End ps@rrowhead}}\fi}
\ctr@ld@f\def\sDcc@ngle{0}
\ctr@ld@f\def\Ps@rrowheadfd#1,#2[#3,#4]{{%
    \PSc@mment{ps@rrowheadfd Length=#1, Angle=#2, [Pt1,Pt2]=[#3,#4]}%
    \setc@ntr@l{2}\figvectPDD-1[#3,#4]\n@rmeucDD{\v@leur}{-1}\v@leur=\ptT@unit@\v@leur%
    \invers@{\v@leur}{\v@leur}\v@leur=#1\v@leur\edef\R@tio{\repdecn@mb{\v@leur}}%
    \Ps@rrowhead\R@tio,#2[#3,#4]\PSc@mment{End ps@rrowheadfd}}}
\ctr@ln@m\psarrowBezier
\ctr@ld@f\def\psarrowBezierDD[#1,#2,#3,#4]{{\ifcurr@ntPS\ifps@cri\s@uvc@ntr@l\et@tpsarrowBezierDD%
    \PSc@mment{psarrowBezierDD Control points=#1,#2,#3,#4}\setc@ntr@l{2}%
    \if@rrowratio\c@larclengthDD\v@leur,10[#1,#2,#3,#4]\else\v@leur=\z@\fi%
    \Ps@rrowB@zDD\v@leur[#1,#2,#3,#4]%
    \PSc@mment{End psarrowBezierDD}\resetc@ntr@l\et@tpsarrowBezierDD\fi\fi}}
\ctr@ld@f\def\psarrowBezierTD[#1,#2,#3,#4]{{\ifcurr@ntPS\ifps@cri\s@uvc@ntr@l\et@tpsarrowBezierTD%
    \PSc@mment{psarrowBezierTD Control points=#1,#2,#3,#4}\resetc@ntr@l{2}%
    \Figptpr@j-7:/#1/\Figptpr@j-8:/#2/\Figptpr@j-9:/#3/\Figptpr@j-10:/#4/%
    \let\c@lprojSP=\relax\ifnum\curr@ntproj<\tw@\psarrowBezierDD[-7,-8,-9,-10]%
    \else\f@gnewpath\PSwrit@cmd{-7}{\c@mmoveto}{\fwf@g}%
    \if@rrowratio\c@larclengthDD\mili@u,10[-7,-8,-9,-10]\else\mili@u=\z@\fi%
    \p@rtent=\NBz@rcs\advance\p@rtent\m@ne\subB@zierTD\p@rtent[#1,#2,#3,#4]%
    \f@gstroke%
    \advance\v@lmin\p@rtent\delt@
    \v@leur=\v@lmin\advance\v@leur0.33333 \delt@\edef\unti@rs{\repdecn@mb{\v@leur}}%
    \v@leur=\v@lmin\advance\v@leur0.66666 \delt@\edef\deti@rs{\repdecn@mb{\v@leur}}%
    \figptcopyDD-8:/-10/\c@lsubBzarc\unti@rs,\deti@rs[#1,#2,#3,#4]%
    \figptcopyDD-8:/-4/\figptcopyDD-9:/-3/\Ps@rrowB@zDD\mili@u[-7,-8,-9,-10]\fi%
    \PSc@mment{End psarrowBezierTD}\resetc@ntr@l\et@tpsarrowBezierTD\fi\fi}}
\ctr@ld@f\def\c@larclengthDD#1,#2[#3,#4,#5,#6]{{\p@rtent=#2\figptcopyDD-5:/#3/%
    \delt@=\p@\divide\delt@\p@rtent\c@rre=\z@\v@leur=\z@\s@mme=\z@%
    \loop\ifnum\s@mme<\p@rtent\advance\s@mme\@ne\advance\v@leur\delt@%
    \edef\T@{\repdecn@mb{\v@leur}}\figptBezierDD-6::\T@[#3,#4,#5,#6]%
    \figvectPDD-1[-5,-6]\n@rmeucDD{\mili@u}{-1}\advance\c@rre\mili@u%
    \figptcopyDD-5:/-6/\repeat\global\result@t=\ptT@unit@\c@rre}#1=\result@t}
\ctr@ld@f\def\Ps@rrowB@zDD#1[#2,#3,#4,#5]{{\pssetfillmode{no}%
    \if@rrowratio\delt@=\@rrowheadratio#1\else\delt@=\@rrowheadlength pt\fi%
    \v@leur=\C@AHANG\delt@\edef\R@dius{\repdecn@mb{\v@leur}}%
    \FigptintercircB@zDD-5::0,\R@dius[#5,#4,#3,#2]%
    \pssetarrowheadlength{\repdecn@mb{\delt@}}\psarrowheadDD[-5,#5]%
    \let\n@rmeuc=\n@rmeucDD\figgetdist\R@dius[#5,-3]%
    \FigptintercircB@zDD-6::0,\R@dius[#5,#4,#3,#2]%
    \figptBezierDD-5::0.33333[#5,#4,#3,#2]\figptBezierDD-3::0.66666[#5,#4,#3,#2]%
    \figptscontrolDD-5[-6,-5,-3,#2]\psBezierDD1[-6,-5,-4,#2]}}
\ctr@ln@m\psarrowcirc
\ctr@ld@f\def\psarrowcircDD#1;#2(#3,#4){{\ifcurr@ntPS\ifps@cri\s@uvc@ntr@l\et@tpsarrowcircDD%
    \PSc@mment{psarrowcircDD Center=#1 ; Radius=#2 (Ang1=#3,Ang2=#4)}%
    \pssetfillmode{no}\Pscirc@rrowhead#1;#2(#3,#4)%
    \setc@ntr@l{2}\figvectPDD -4[#1,-3]\vecunit@{-4}{-4}%
    \Figg@tXY{-4}\arct@n\v@lmin(\v@lX,\v@lY)%
    \v@lmin=\rdT@deg\v@lmin\v@leur=#4pt\advance\v@leur-\v@lmin%
    \maxim@m{\v@leur}{\v@leur}{-\v@leur}%
    \ifdim\v@leur>\DemiPI@deg\relax\ifdim\v@lmin<#4pt\advance\v@lmin\DePI@deg%
    \else\advance\v@lmin-\DePI@deg\fi\fi\edef\ar@ngle{\repdecn@mb{\v@lmin}}%
    \ifdim#3pt<#4pt\psarccirc#1;#2(#3,\ar@ngle)\else\psarccirc#1;#2(\ar@ngle,#3)\fi%
    \PSc@mment{End psarrowcircDD}\resetc@ntr@l\et@tpsarrowcircDD\fi\fi}}
\ctr@ld@f\def\psarrowcircTD#1,#2,#3;#4(#5,#6){{\ifcurr@ntPS\ifps@cri\s@uvc@ntr@l\et@tpsarrowcircTD%
    \PSc@mment{psarrowcircTD Center=#1,P1=#2,P2=#3 ; Radius=#4 (Ang1=#5, Ang2=#6)}%
    \resetc@ntr@l{2}\c@lExtAxes#1,#2,#3(#4)\let\c@lprojSP=\relax%
    \figvectPTD-11[#1,-4]\figvectPTD-12[#1,-5]\c@lNbarcs{#5}{#6}%
    \if@rrowratio\v@lmax=\degT@rd\v@lmax\edef\D@lpha{\repdecn@mb{\v@lmax}}\fi%
    \advance\p@rtent\m@ne\mili@u=\z@%
    \v@leur=#5pt\c@lptellP{#1}{-11}{-12}\Figptpr@j-9:/-3/%
    \f@gnewpath\PSwrit@cmdS{-9}{\c@mmoveto}{\fwf@g}{\X@un}{\Y@un}%
    \edef\C@nt@r{#1}\s@mme=\z@\bcl@rcircTD\f@gstroke%
    \advance\v@leur\delt@\c@lptellP{#1}{-11}{-12}\Figptpr@j-5:/-3/%
    \advance\v@leur\delt@\c@lptellP{#1}{-11}{-12}\Figptpr@j-6:/-3/%
    \advance\v@leur\delt@\c@lptellP{#1}{-11}{-12}\Figptpr@j-10:/-3/%
    \figptscontrolDD-8[-9,-5,-6,-10]%
    \if@rrowratio\c@lcurvradDD0.5[-9,-8,-7,-10]\advance\mili@u\result@t%
    \maxim@m{\mili@u}{\mili@u}{-\mili@u}\mili@u=\ptT@unit@\mili@u%
    \mili@u=\D@lpha\mili@u\advance\p@rtent\@ne\divide\mili@u\p@rtent\fi%
    \Ps@rrowB@zDD\mili@u[-9,-8,-7,-10]%
    \PSc@mment{End psarrowcircTD}\resetc@ntr@l\et@tpsarrowcircTD\fi\fi}}
\ctr@ld@f\def\bcl@rcircTD{\relax%
    \ifnum\s@mme<\p@rtent\advance\s@mme\@ne%
    \advance\v@leur\delt@\c@lptellP{\C@nt@r}{-11}{-12}\Figptpr@j-5:/-3/%
    \advance\v@leur\delt@\c@lptellP{\C@nt@r}{-11}{-12}\Figptpr@j-6:/-3/%
    \advance\v@leur\delt@\c@lptellP{\C@nt@r}{-11}{-12}\Figptpr@j-10:/-3/%
    \figptscontrolDD-8[-9,-5,-6,-10]\BdingB@xfalse%
    \PSwrit@cmdS{-8}{}{\fwf@g}{\X@de}{\Y@de}\PSwrit@cmdS{-7}{}{\fwf@g}{\X@tr}{\Y@tr}%
    \BdingB@xtrue\PSwrit@cmdS{-10}{\c@mcurveto}{\fwf@g}{\X@qu}{\Y@qu}%
    \if@rrowratio\c@lcurvradDD0.5[-9,-8,-7,-10]\advance\mili@u\result@t\fi%
    \B@zierBB@x{1}{\Y@un}(\X@un,\X@de,\X@tr,\X@qu)%
    \B@zierBB@x{2}{\X@un}(\Y@un,\Y@de,\Y@tr,\Y@qu)%
    \edef\X@un{\X@qu}\edef\Y@un{\Y@qu}\figptcopyDD-9:/-10/\bcl@rcircTD\fi}
\ctr@ld@f\def\Pscirc@rrowhead#1;#2(#3,#4){{%
    \PSc@mment{pscirc@rrowhead Center=#1 ; Radius=#2 (Ang1=#3,Ang2=#4)}%
    \v@leur=#2\unit@\edef\s@glen{\repdecn@mb{\v@leur}}\v@lY=\z@\v@lX=\v@leur%
    \resetc@ntr@l{2}\Figv@ctCreg-3(\v@lX,\v@lY)\figpttraDD-5:=#1/1,-3/%
    \figptrotDD-5:=-5/#1,#4/%
    \figvectPDD-3[#1,-5]\Figg@tXY{-3}\v@leur=\v@lX%
    \ifdim#3pt<#4pt\v@lX=\v@lY\v@lY=-\v@leur\else\v@lX=-\v@lY\v@lY=\v@leur\fi%
    \Figv@ctCreg-3(\v@lX,\v@lY)\vecunit@{-3}{-3}%
    \if@rrowratio\v@leur=#4pt\advance\v@leur-#3pt\maxim@m{\mili@u}{-\v@leur}{\v@leur}%
    \mili@u=\degT@rd\mili@u\v@leur=\s@glen\mili@u\edef\s@glen{\repdecn@mb{\v@leur}}%
    \mili@u=#2\mili@u\mili@u=\@rrowheadratio\mili@u\else\mili@u=\@rrowheadlength pt\fi%
    \figpttraDD-6:=-5/\s@glen,-3/\v@leur=#2pt\v@leur=2\v@leur%
    \invers@{\v@leur}{\v@leur}\c@rre=\repdecn@mb{\v@leur}\mili@u
    \mili@u=\c@rre\mili@u=\repdecn@mb{\c@rre}\mili@u%
    \v@leur=\p@\advance\v@leur-\mili@u
    \invers@{\mili@u}{2\v@leur}\delt@=\c@rre\delt@=\repdecn@mb{\mili@u}\delt@%
    \xdef\sDcc@ngle{\repdecn@mb{\delt@}}
    \sqrt@{\mili@u}{\v@leur}\arct@n\v@leur(\mili@u,\c@rre)%
    \v@leur=\rdT@deg\v@leur
    \ifdim#3pt<#4pt\v@leur=-\v@leur\fi%
    \if@rrowhout\v@leur=-\v@leur\fi\edef\cor@ngle{\repdecn@mb{\v@leur}}%
    \figptrotDD-6:=-6/-5,\cor@ngle/\psarrowheadDD[-6,-5]%
    \PSc@mment{End pscirc@rrowhead}}}
\ctr@ln@m\psarrowcircP
\ctr@ld@f\def\psarrowcircPDD#1;#2[#3,#4]{{\ifcurr@ntPS\ifps@cri%
    \PSc@mment{psarrowcircPDD Center=#1; Radius=#2, [P1=#3,P2=#4]}%
    \s@uvc@ntr@l\et@tpsarrowcircPDD\Ps@ngleparam#1;#2[#3,#4]%
    \ifdim\v@leur>\z@\ifdim\v@lmin>\v@lmax\advance\v@lmax\DePI@deg\fi%
    \else\ifdim\v@lmin<\v@lmax\advance\v@lmin\DePI@deg\fi\fi%
    \edef\@ngdeb{\repdecn@mb{\v@lmin}}\edef\@ngfin{\repdecn@mb{\v@lmax}}%
    \psarrowcirc#1;\r@dius(\@ngdeb,\@ngfin)%
    \PSc@mment{End psarrowcircPDD}\resetc@ntr@l\et@tpsarrowcircPDD\fi\fi}}
\ctr@ld@f\def\psarrowcircPTD#1;#2[#3,#4,#5]{{\ifcurr@ntPS\ifps@cri\s@uvc@ntr@l\et@tpsarrowcircPTD%
    \PSc@mment{psarrowcircPTD Center=#1; Radius=#2, [P1=#3,P2=#4,P3=#5]}%
    \figgetangleTD\@ngfin[#1,#3,#4,#5]\v@leur=#2pt%
    \maxim@m{\mili@u}{-\v@leur}{\v@leur}\edef\r@dius{\repdecn@mb{\mili@u}}%
    \ifdim\v@leur<\z@\v@lmax=\@ngfin pt\advance\v@lmax-\DePI@deg%
    \edef\@ngfin{\repdecn@mb{\v@lmax}}\fi\psarrowcircTD#1,#3,#5;\r@dius(0,\@ngfin)%
    \PSc@mment{End psarrowcircPTD}\resetc@ntr@l\et@tpsarrowcircPTD\fi\fi}}
\ctr@ld@f\def\psaxes#1(#2){{\ifcurr@ntPS\ifps@cri\s@uvc@ntr@l\et@tpsaxes%
    \PSc@mment{psaxes Origin=#1 Range=(#2)}\an@lys@xes#2,:\resetc@ntr@l{2}%
    \ifx\t@xt@\empty\ifTr@isDim\ps@xes#1(0,#2,0,#2,0,#2)\else\ps@xes#1(0,#2,0,#2)\fi%
    \else\ps@xes#1(#2)\fi\PSc@mment{End psaxes}\resetc@ntr@l\et@tpsaxes\fi\fi}}
\ctr@ld@f\def\an@lys@xes#1,#2:{\def\t@xt@{#2}}
\ctr@ln@m\ps@xes
\ctr@ld@f\def\ps@xesDD#1(#2,#3,#4,#5){%
    \figpttraC-5:=#1/#2,0/\figpttraC-6:=#1/#3,0/\psarrowDD[-5,-6]%
    \figpttraC-5:=#1/0,#4/\figpttraC-6:=#1/0,#5/\psarrowDD[-5,-6]}
\ctr@ld@f\def\ps@xesTD#1(#2,#3,#4,#5,#6,#7){%
    \figpttraC-7:=#1/#2,0,0/\figpttraC-8:=#1/#3,0,0/\psarrowTD[-7,-8]%
    \figpttraC-7:=#1/0,#4,0/\figpttraC-8:=#1/0,#5,0/\psarrowTD[-7,-8]%
    \figpttraC-7:=#1/0,0,#6/\figpttraC-8:=#1/0,0,#7/\psarrowTD[-7,-8]}
\ctr@ln@m\newGr@FN
\ctr@ld@f\def\newGr@FNPDF#1{\s@mme=\Gr@FNb\advance\s@mme\@ne\xdef\Gr@FNb{\number\s@mme}}
\ctr@ld@f\def\newGr@FNDVI#1{\newGr@FNPDF{}\xdef#1{\jobname GI\Gr@FNb.anx}}
\ctr@ld@f\def\psbeginfig#1{\newGr@FN\DefGIfilen@me\gdef\@utoFN{0}%
    \def\t@xt@{#1}\relax\ifx\t@xt@\empty\psupdatem@detrue%
    \gdef\@utoFN{1}\Psb@ginfig\DefGIfilen@me\else\expandafter\Psb@ginfigNu@#1 :\fi}
\ctr@ld@f\def\Psb@ginfigNu@#1 #2:{\def\t@xt@{#1}\relax\ifx\t@xt@\empty\def\t@xt@{#2}%
    \ifx\t@xt@\empty\psupdatem@detrue\gdef\@utoFN{1}\Psb@ginfig\DefGIfilen@me%
    \else\Psb@ginfigNu@#2:\fi\else\Psb@ginfig{#1}\fi}
\ctr@ln@m\PSfilen@me \ctr@ln@m\auxfilen@me
\ctr@ld@f\def\Psb@ginfig#1{\ifcurr@ntPS\else%
    \edef\PSfilen@me{#1}\edef\auxfilen@me{\jobname.anx}%
    \ifpsupdatem@de\ps@critrue\else\openin\frf@g=\PSfilen@me\relax%
    \ifeof\frf@g\ps@critrue\else\ps@crifalse\fi\closein\frf@g\fi%
    \curr@ntPStrue\c@ldefproj\expandafter\setupd@te\defaultupdate:%
    \ifps@cri\initb@undb@x%
    \immediate\openout\fwf@g=\auxfilen@me\initpss@ttings\fi%
    \fi}
\ctr@ld@f\def\Gr@FNb{0}
\ctr@ld@f\def\figforTeXFileno{\Gr@FNb}
\ctr@ld@f\def\figforTeXFigno{0 }
\ctr@ld@f\def\figforTeXnextFigno{1 }
\ctr@ld@f\edef\DefGIfilen@me{\jobname GI.anx}
\ctr@ld@f\def\initpss@ttings{\psreset{arrowhead,curve,first,flowchart,mesh,second,third}%
    \Use@llipsefalse}
\ctr@ld@f\def\B@zierBB@x#1#2(#3,#4,#5,#6){{\c@rre=\t@n\epsil@n
    \v@lmax=#4\advance\v@lmax-#5\v@lmax=\thr@@\v@lmax\advance\v@lmax#6\advance\v@lmax-#3%
    \mili@u=#4\mili@u=-\tw@\mili@u\advance\mili@u#3\advance\mili@u#5%
    \v@lmin=#4\advance\v@lmin-#3\maxim@m{\v@leur}{-\v@lmax}{\v@lmax}%
    \maxim@m{\delt@}{-\mili@u}{\mili@u}\maxim@m{\v@leur}{\v@leur}{\delt@}%
    \maxim@m{\delt@}{-\v@lmin}{\v@lmin}\maxim@m{\v@leur}{\v@leur}{\delt@}%
    \ifdim\v@leur>\c@rre\invers@{\v@leur}{\v@leur}\edef\Uns@rM@x{\repdecn@mb{\v@leur}}%
    \v@lmax=\Uns@rM@x\v@lmax\mili@u=\Uns@rM@x\mili@u\v@lmin=\Uns@rM@x\v@lmin%
    \maxim@m{\v@leur}{-\v@lmax}{\v@lmax}\ifdim\v@leur<\c@rre%
    \maxim@m{\v@leur}{-\mili@u}{\mili@u}\ifdim\v@leur<\c@rre\else%
    \invers@{\mili@u}{\mili@u}\v@leur=-0.5\v@lmin%
    \v@leur=\repdecn@mb{\mili@u}\v@leur\m@jBBB@x{\v@leur}{#1}{#2}(#3,#4,#5,#6)\fi%
    \else\delt@=\repdecn@mb{\mili@u}\mili@u\v@leur=\repdecn@mb{\v@lmax}\v@lmin%
    \advance\delt@-\v@leur\ifdim\delt@<\z@\else\invers@{\v@lmax}{\v@lmax}%
    \edef\Uns@rAp{\repdecn@mb{\v@lmax}}\sqrt@{\delt@}{\delt@}%
    \v@leur=-\mili@u\advance\v@leur\delt@\v@leur=\Uns@rAp\v@leur%
    \m@jBBB@x{\v@leur}{#1}{#2}(#3,#4,#5,#6)%
    \v@leur=-\mili@u\advance\v@leur-\delt@\v@leur=\Uns@rAp\v@leur%
    \m@jBBB@x{\v@leur}{#1}{#2}(#3,#4,#5,#6)\fi\fi\fi}}
\ctr@ld@f\def\m@jBBB@x#1#2#3(#4,#5,#6,#7){{\relax\ifdim#1>\z@\ifdim#1<\p@%
    \edef\T@{\repdecn@mb{#1}}\v@lX=\p@\advance\v@lX-#1\edef\UNmT@{\repdecn@mb{\v@lX}}%
    \v@lX=#4\v@lY=#5\v@lZ=#6\v@lXa=#7\v@lX=\UNmT@\v@lX\advance\v@lX\T@\v@lY%
    \v@lY=\UNmT@\v@lY\advance\v@lY\T@\v@lZ\v@lZ=\UNmT@\v@lZ\advance\v@lZ\T@\v@lXa%
    \v@lX=\UNmT@\v@lX\advance\v@lX\T@\v@lY\v@lY=\UNmT@\v@lY\advance\v@lY\T@\v@lZ%
    \v@lX=\UNmT@\v@lX\advance\v@lX\T@\v@lY%
    \ifcase#2\or\v@lY=#3\or\v@lY=\v@lX\v@lX=#3\fi\b@undb@x{\v@lX}{\v@lY}\fi\fi}}
\ctr@ld@f\def\PsB@zier#1[#2]{{\f@gnewpath%
    \s@mme=\z@\def\list@num{#2,0}\extrairelepremi@r\p@int\de\list@num%
    \PSwrit@cmdS{\p@int}{\c@mmoveto}{\fwf@g}{\X@un}{\Y@un}\p@rtent=#1\bclB@zier}}
\ctr@ld@f\def\bclB@zier{\relax%
    \ifnum\s@mme<\p@rtent\advance\s@mme\@ne\BdingB@xfalse%
    \extrairelepremi@r\p@int\de\list@num\PSwrit@cmdS{\p@int}{}{\fwf@g}{\X@de}{\Y@de}%
    \extrairelepremi@r\p@int\de\list@num\PSwrit@cmdS{\p@int}{}{\fwf@g}{\X@tr}{\Y@tr}%
    \BdingB@xtrue%
    \extrairelepremi@r\p@int\de\list@num\PSwrit@cmdS{\p@int}{\c@mcurveto}{\fwf@g}{\X@qu}{\Y@qu}%
    \B@zierBB@x{1}{\Y@un}(\X@un,\X@de,\X@tr,\X@qu)%
    \B@zierBB@x{2}{\X@un}(\Y@un,\Y@de,\Y@tr,\Y@qu)%
    \edef\X@un{\X@qu}\edef\Y@un{\Y@qu}\bclB@zier\fi}
\ctr@ln@m\psBezier
\ctr@ld@f\def\psBezierDD#1[#2]{\ifcurr@ntPS\ifps@cri%
    \PSc@mment{psBezierDD N arcs=#1, Control points=#2}%
    \iffillm@de\PsB@zier#1[#2]%
    \f@gfill%
    \else\PsB@zier#1[#2]\f@gstroke\fi%
    \PSc@mment{End psBezierDD}\fi\fi}
\ctr@ln@m\et@tpsBezierTD
\ctr@ld@f\def\psBezierTD#1[#2]{\ifcurr@ntPS\ifps@cri\s@uvc@ntr@l\et@tpsBezierTD%
    \PSc@mment{psBezierTD N arcs=#1, Control points=#2}%
    \iffillm@de\PsB@zierTD#1[#2]%
    \f@gfill%
    \else\PsB@zierTD#1[#2]\f@gstroke\fi%
    \PSc@mment{End psBezierTD}\resetc@ntr@l\et@tpsBezierTD\fi\fi}
\ctr@ld@f\def\PsB@zierTD#1[#2]{\ifnum\curr@ntproj<\tw@\PsB@zier#1[#2]\else\PsB@zier@TD#1[#2]\fi}
\ctr@ld@f\def\PsB@zier@TD#1[#2]{{\f@gnewpath%
    \s@mme=\z@\def\list@num{#2,0}\extrairelepremi@r\p@int\de\list@num%
    \let\c@lprojSP=\relax\setc@ntr@l{2}\Figptpr@j-7:/\p@int/%
    \PSwrit@cmd{-7}{\c@mmoveto}{\fwf@g}%
    \loop\ifnum\s@mme<#1\advance\s@mme\@ne\extrairelepremi@r\p@intun\de\list@num%
    \extrairelepremi@r\p@intde\de\list@num\extrairelepremi@r\p@inttr\de\list@num%
    \subB@zierTD\NBz@rcs[\p@int,\p@intun,\p@intde,\p@inttr]\edef\p@int{\p@inttr}\repeat}}
\ctr@ld@f\def\subB@zierTD#1[#2,#3,#4,#5]{\delt@=\p@\divide\delt@\NBz@rcs\v@lmin=\z@%
    {\Figg@tXY{-7}\edef\X@un{\the\v@lX}\edef\Y@un{\the\v@lY}%
    \s@mme=\z@\loop\ifnum\s@mme<#1\advance\s@mme\@ne%
    \v@leur=\v@lmin\advance\v@leur0.33333 \delt@\edef\unti@rs{\repdecn@mb{\v@leur}}%
    \v@leur=\v@lmin\advance\v@leur0.66666 \delt@\edef\deti@rs{\repdecn@mb{\v@leur}}%
    \advance\v@lmin\delt@\edef\trti@rs{\repdecn@mb{\v@lmin}}%
    \figptBezierTD-8::\trti@rs[#2,#3,#4,#5]\Figptpr@j-8:/-8/%
    \c@lsubBzarc\unti@rs,\deti@rs[#2,#3,#4,#5]\BdingB@xfalse%
    \PSwrit@cmdS{-4}{}{\fwf@g}{\X@de}{\Y@de}\PSwrit@cmdS{-3}{}{\fwf@g}{\X@tr}{\Y@tr}%
    \BdingB@xtrue\PSwrit@cmdS{-8}{\c@mcurveto}{\fwf@g}{\X@qu}{\Y@qu}%
    \B@zierBB@x{1}{\Y@un}(\X@un,\X@de,\X@tr,\X@qu)%
    \B@zierBB@x{2}{\X@un}(\Y@un,\Y@de,\Y@tr,\Y@qu)%
    \edef\X@un{\X@qu}\edef\Y@un{\Y@qu}\figptcopyDD-7:/-8/\repeat}}
\ctr@ld@f\def\NBz@rcs{2}
\ctr@ld@f\def\c@lsubBzarc#1,#2[#3,#4,#5,#6]{\figptBezierTD-5::#1[#3,#4,#5,#6]%
    \figptBezierTD-6::#2[#3,#4,#5,#6]\Figptpr@j-4:/-5/\Figptpr@j-5:/-6/%
    \figptscontrolDD-4[-7,-4,-5,-8]}
\ctr@ln@m\pscirc
\ctr@ld@f\def\pscircDD#1(#2){\ifcurr@ntPS\ifps@cri\PSc@mment{pscircDD Center=#1 (Radius=#2)}%
    \psarccircDD#1;#2(0,360)\PSc@mment{End pscircDD}\fi\fi}
\ctr@ld@f\def\pscircTD#1,#2,#3(#4){\ifcurr@ntPS\ifps@cri%
    \PSc@mment{pscircTD Center=#1,P1=#2,P2=#3 (Radius=#4)}%
    \psarccircTD#1,#2,#3;#4(0,360)\PSc@mment{End pscircTD}\fi\fi}
\ctr@ln@m\p@urcent
{\catcode`\%=12\gdef\p@urcent{
\ctr@ld@f\def\PSc@mment#1{\ifpsdebugmode\immediate\write\fwf@g{\p@urcent\space#1}\fi}
\ctr@ln@m\acc@louv \ctr@ln@m\acc@lfer
{\catcode`\[=1\catcode`\{=12\gdef\acc@louv[{}}
{\catcode`\]=2\catcode`\}=12\gdef\acc@lfer{}]]
\ctr@ld@f\def\PSdict@{\ifUse@llipse%
    \immediate\write\fwf@g{/ellipsedict 9 dict def ellipsedict /mtrx matrix put}%
    \immediate\write\fwf@g{/ellipse \acc@louv ellipsedict begin}%
    \immediate\write\fwf@g{ /endangle exch def /startangle exch def}%
    \immediate\write\fwf@g{ /yrad exch def /xrad exch def}%
    \immediate\write\fwf@g{ /rotangle exch def /y exch def /x exch def}%
    \immediate\write\fwf@g{ /savematrix mtrx currentmatrix def}%
    \immediate\write\fwf@g{ x y translate rotangle rotate xrad yrad scale}%
    \immediate\write\fwf@g{ 0 0 1 startangle endangle arc}%
    \immediate\write\fwf@g{ savematrix setmatrix end\acc@lfer def}%
    \fi\PShe@der{EndProlog}}
\ctr@ld@f\def\Pssetc@rve#1=#2|{\keln@mun#1|%
    \def\n@mref{r}\ifx\l@debut\n@mref\pssetroundness{#2}\else
    \immediate\write16{*** Unknown attribute: \BS@ psset curve(..., #1=...)}%
    \fi}
\ctr@ln@m\curv@roundness
\ctr@ld@f\def\pssetroundness#1{\edef\curv@roundness{#1}}
\ctr@ld@f\def\defaultroundness{0.2} 
\ctr@ln@m\pscurve
\ctr@ld@f\def\pscurveDD[#1]{{\ifcurr@ntPS\ifps@cri\PSc@mment{pscurveDD Points=#1}%
    \s@uvc@ntr@l\et@tpscurveDD%
    \iffillm@de\Psc@rveDD\curv@roundness[#1]%
    \f@gfill%
    \else\Psc@rveDD\curv@roundness[#1]\f@gstroke\fi%
    \PSc@mment{End pscurveDD}\resetc@ntr@l\et@tpscurveDD\fi\fi}}
\ctr@ld@f\def\pscurveTD[#1]{{\ifcurr@ntPS\ifps@cri%
    \PSc@mment{pscurveTD Points=#1}\s@uvc@ntr@l\et@tpscurveTD\let\c@lprojSP=\relax%
    \iffillm@de\Psc@rveTD\curv@roundness[#1]%
    \f@gfill%
    \else\Psc@rveTD\curv@roundness[#1]\f@gstroke\fi%
    \PSc@mment{End pscurveTD}\resetc@ntr@l\et@tpscurveTD\fi\fi}}
\ctr@ld@f\def\Psc@rveDD#1[#2]{%
    \def\list@num{#2}\extrairelepremi@r\Ak@\de\list@num%
    \extrairelepremi@r\Ai@\de\list@num\extrairelepremi@r\Aj@\de\list@num%
    \f@gnewpath\PSwrit@cmdS{\Ai@}{\c@mmoveto}{\fwf@g}{\X@un}{\Y@un}%
    \setc@ntr@l{2}\figvectPDD -1[\Ak@,\Aj@]%
    \@ecfor\Ak@:=\list@num\do{\figpttraDD-2:=\Ai@/#1,-1/\BdingB@xfalse%
       \PSwrit@cmdS{-2}{}{\fwf@g}{\X@de}{\Y@de}%
       \figvectPDD -1[\Ai@,\Ak@]\figpttraDD-2:=\Aj@/-#1,-1/%
       \PSwrit@cmdS{-2}{}{\fwf@g}{\X@tr}{\Y@tr}\BdingB@xtrue%
       \PSwrit@cmdS{\Aj@}{\c@mcurveto}{\fwf@g}{\X@qu}{\Y@qu}%
       \B@zierBB@x{1}{\Y@un}(\X@un,\X@de,\X@tr,\X@qu)%
       \B@zierBB@x{2}{\X@un}(\Y@un,\Y@de,\Y@tr,\Y@qu)%
       \edef\X@un{\X@qu}\edef\Y@un{\Y@qu}\edef\Ai@{\Aj@}\edef\Aj@{\Ak@}}}
\ctr@ld@f\def\Psc@rveTD#1[#2]{\ifnum\curr@ntproj<\tw@\Psc@rvePPTD#1[#2]\else\Psc@rveCPTD#1[#2]\fi}
\ctr@ld@f\def\Psc@rvePPTD#1[#2]{\setc@ntr@l{2}%
    \def\list@num{#2}\extrairelepremi@r\Ak@\de\list@num\Figptpr@j-5:/\Ak@/%
    \extrairelepremi@r\Ai@\de\list@num\Figptpr@j-3:/\Ai@/%
    \extrairelepremi@r\Aj@\de\list@num\Figptpr@j-4:/\Aj@/%
    \f@gnewpath\PSwrit@cmdS{-3}{\c@mmoveto}{\fwf@g}{\X@un}{\Y@un}%
    \figvectPDD -1[-5,-4]%
    \@ecfor\Ak@:=\list@num\do{\Figptpr@j-5:/\Ak@/\figpttraDD-2:=-3/#1,-1/%
       \BdingB@xfalse\PSwrit@cmdS{-2}{}{\fwf@g}{\X@de}{\Y@de}%
       \figvectPDD -1[-3,-5]\figpttraDD-2:=-4/-#1,-1/%
       \PSwrit@cmdS{-2}{}{\fwf@g}{\X@tr}{\Y@tr}\BdingB@xtrue%
       \PSwrit@cmdS{-4}{\c@mcurveto}{\fwf@g}{\X@qu}{\Y@qu}%
       \B@zierBB@x{1}{\Y@un}(\X@un,\X@de,\X@tr,\X@qu)%
       \B@zierBB@x{2}{\X@un}(\Y@un,\Y@de,\Y@tr,\Y@qu)%
       \edef\X@un{\X@qu}\edef\Y@un{\Y@qu}\figptcopyDD-3:/-4/\figptcopyDD-4:/-5/}}
\ctr@ld@f\def\Psc@rveCPTD#1[#2]{\setc@ntr@l{2}%
    \def\list@num{#2}\extrairelepremi@r\Ak@\de\list@num%
    \extrairelepremi@r\Ai@\de\list@num\extrairelepremi@r\Aj@\de\list@num%
    \Figptpr@j-7:/\Ai@/%
    \f@gnewpath\PSwrit@cmd{-7}{\c@mmoveto}{\fwf@g}%
    \figvectPTD -9[\Ak@,\Aj@]%
    \@ecfor\Ak@:=\list@num\do{\figpttraTD-10:=\Ai@/#1,-9/%
       \figvectPTD -9[\Ai@,\Ak@]\figpttraTD-11:=\Aj@/-#1,-9/%
       \subB@zierTD\NBz@rcs[\Ai@,-10,-11,\Aj@]\edef\Ai@{\Aj@}\edef\Aj@{\Ak@}}}
\ctr@ld@f\def\psendfig{\ifcurr@ntPS\ifps@cri\immediate\closeout\fwf@g%
    \immediate\openout\fwf@g=\PSfilen@me\relax%
    \ifPDFm@ke\PSBdingB@x\else%
    \immediate\write\fwf@g{\p@urcent\string!PS-Adobe-2.0 EPSF-2.0}%
    \PShe@der{Creator\string: TeX (fig4tex.tex)}%
    \PShe@der{Title\string: \PSfilen@me}%
    \PShe@der{CreationDate\string: \the\day/\the\month/\the\year}%
    \PSBdingB@x%
    \PShe@der{EndComments}\PSdict@\fi%
    \immediate\write\fwf@g{\c@mgsave}%
    \openin\frf@g=\auxfilen@me\c@pypsfile\fwf@g\frf@g\closein\frf@g%
    \immediate\write\fwf@g{\c@mgrestore}%
    \PSc@mment{End of file.}\immediate\closeout\fwf@g%
    \immediate\openout\fwf@g=\auxfilen@me\immediate\closeout\fwf@g%
    \immediate\write16{File \PSfilen@me\space created.}\fi\fi\curr@ntPSfalse\ps@critrue}
\ctr@ld@f\def\PShe@der#1{\immediate\write\fwf@g{\p@urcent\p@urcent#1}}
\ctr@ld@f\def\PSBdingB@x{{\v@lX=\ptT@ptps\c@@rdXmin\v@lY=\ptT@ptps\c@@rdYmin%
     \v@lXa=\ptT@ptps\c@@rdXmax\v@lYa=\ptT@ptps\c@@rdYmax%
     \PShe@der{BoundingBox\string: \repdecn@mb{\v@lX}\space\repdecn@mb{\v@lY}%
     \space\repdecn@mb{\v@lXa}\space\repdecn@mb{\v@lYa}}}}
\ctr@ld@f\def\psfcconnect[#1]{{\ifcurr@ntPS\ifps@cri\PSc@mment{psfcconnect Points=#1}%
    \pssetfillmode{no}\s@uvc@ntr@l\et@tpsfcconnect\resetc@ntr@l{2}%
    \fcc@nnect@[#1]\resetc@ntr@l\et@tpsfcconnect\PSc@mment{End psfcconnect}\fi\fi}}
\ctr@ld@f\def\fcc@nnect@[#1]{\let\N@rm=\n@rmeucDD\def\list@num{#1}%
    \extrairelepremi@r\Ai@\de\list@num\edef\pr@m{\Ai@}\v@leur=\z@\p@rtent=\@ne\c@llgtot%
    \ifcase\fclin@typ@\edef\list@num{[\pr@m,#1,\Ai@}\expandafter\pscurve\list@num]%
    \else\ifdim\fclin@r@d\p@>\z@\Pslin@conge[#1]\else\psline[#1]\fi\fi%
    \v@leur=\@rrowp@s\v@leur\edef\list@num{#1,\Ai@,0}%
    \extrairelepremi@r\Ai@\de\list@num\mili@u=\epsil@n\c@llgpart%
    \advance\mili@u-\epsil@n\advance\mili@u-\delt@\advance\v@leur-\mili@u%
    \ifcase\fclin@typ@\invers@\mili@u\delt@%
    \ifnum\@rrowr@fpt>\z@\advance\delt@-\v@leur\v@leur=\delt@\fi%
    \v@leur=\repdecn@mb\v@leur\mili@u\edef\v@lt{\repdecn@mb\v@leur}%
    \extrairelepremi@r\Ak@\de\list@num%
    \figvectPDD-1[\pr@m,\Aj@]\figpttraDD-6:=\Ai@/\curv@roundness,-1/%
    \figvectPDD-1[\Ak@,\Ai@]\figpttraDD-7:=\Aj@/\curv@roundness,-1/%
    \delt@=\@rrowheadlength\p@\delt@=\C@AHANG\delt@\edef\R@dius{\repdecn@mb{\delt@}}%
    \ifcase\@rrowr@fpt%
    \FigptintercircB@zDD-8::\v@lt,\R@dius[\Ai@,-6,-7,\Aj@]\psarrowheadDD[-5,-8]\else%
    \FigptintercircB@zDD-8::\v@lt,\R@dius[\Aj@,-7,-6,\Ai@]\psarrowheadDD[-8,-5]\fi%
    \else\advance\delt@-\v@leur%
    \p@rtentiere{\p@rtent}{\delt@}\edef\C@efun{\the\p@rtent}%
    \p@rtentiere{\p@rtent}{\v@leur}\edef\C@efde{\the\p@rtent}%
    \figptbaryDD-5:[\Ai@,\Aj@;\C@efun,\C@efde]\ifcase\@rrowr@fpt%
    \delt@=\@rrowheadlength\unit@\delt@=\C@AHANG\delt@\edef\t@ille{\repdecn@mb{\delt@}}%
    \figvectPDD-2[\Ai@,\Aj@]\vecunit@{-2}{-2}\figpttraDD-5:=-5/\t@ille,-2/\fi%
    \psarrowheadDD[\Ai@,-5]\fi}
\ctr@ld@f\def\c@llgtot{\@ecfor\Aj@:=\list@num\do{\figvectP-1[\Ai@,\Aj@]\N@rm\delt@{-1}%
    \advance\v@leur\delt@\advance\p@rtent\@ne\edef\Ai@{\Aj@}}}
\ctr@ld@f\def\c@llgpart{\extrairelepremi@r\Aj@\de\list@num\figvectP-1[\Ai@,\Aj@]\N@rm\delt@{-1}%
    \advance\mili@u\delt@\ifdim\mili@u<\v@leur\edef\pr@m{\Ai@}\edef\Ai@{\Aj@}\c@llgpart\fi}
\ctr@ld@f\def\Pslin@conge[#1]{\ifnum\p@rtent>\tw@{\def\list@num{#1}%
    \extrairelepremi@r\Ai@\de\list@num\extrairelepremi@r\Aj@\de\list@num%
    \figptcopy-6:/\Ai@/\figvectP-3[\Ai@,\Aj@]\vecunit@{-3}{-3}\v@lmax=\result@t%
    \@ecfor\Ak@:=\list@num\do{\figvectP-4[\Aj@,\Ak@]\vecunit@{-4}{-4}%
    \minim@m\v@lmin\v@lmax\result@t\v@lmax=\result@t%
    \det@rm\delt@[-3,-4]\maxim@m\mili@u{\delt@}{-\delt@}\ifdim\mili@u>\Cepsil@n%
    \ifdim\delt@>\z@\figgetangleDD\Angl@[\Aj@,\Ak@,\Ai@]\else%
    \figgetangleDD\Angl@[\Aj@,\Ai@,\Ak@]\fi%
    \v@leur=\PI@deg\advance\v@leur-\Angl@\p@\divide\v@leur\tw@%
    \edef\Angl@{\repdecn@mb\v@leur}\c@ssin{\C@}{\S@}{\Angl@}\v@leur=\fclin@r@d\unit@%
    \v@leur=\S@\v@leur\mili@u=\C@\p@\invers@\mili@u\mili@u%
    \v@leur=\repdecn@mb{\mili@u}\v@leur%
    \minim@m\v@leur\v@leur\v@lmin\edef\t@ille{\repdecn@mb{\v@leur}}%
    \figpttra-5:=\Aj@/-\t@ille,-3/\psline[-6,-5]\figpttra-6:=\Aj@/\t@ille,-4/%
    \figvectNVDD-3[-3]\figvectNVDD-8[-4]\inters@cDD-7:[-5,-3;-6,-8]%
    \ifdim\delt@>\z@\psarccircP-7;\fclin@r@d[-5,-6]\else\psarccircP-7;\fclin@r@d[-6,-5]\fi%
    \else\psline[-6,\Aj@]\figptcopy-6:/\Aj@/\fi
    \edef\Ai@{\Aj@}\edef\Aj@{\Ak@}\figptcopy-3:/-4/}\psline[-6,\Aj@]}\else\psline[#1]\fi}
\ctr@ld@f\def\psfcnode[#1]#2{{\ifcurr@ntPS\ifps@cri\PSc@mment{psfcnode Points=#1}%
    \s@uvc@ntr@l\et@tpsfcnode\resetc@ntr@l{2}%
    \def\t@xt@{#2}\ifx\t@xt@\empty\def\g@tt@xt{\setbox\Gb@x=\hbox{\Figg@tT{\p@int}}}%
    \else\def\g@tt@xt{\setbox\Gb@x=\hbox{#2}}\fi%
    \v@lmin=\h@rdfcXp@dd\advance\v@lmin\Xp@dd\unit@\multiply\v@lmin\tw@%
    \v@lmax=\h@rdfcYp@dd\advance\v@lmax\Yp@dd\unit@\multiply\v@lmax\tw@%
    \Figv@ctCreg-8(\unit@,-\unit@)\def\list@num{#1}%
    \delt@=\curr@ntwidth bp\divide\delt@\tw@%
    \fcn@de\PSc@mment{End psfcnode}\resetc@ntr@l\et@tpsfcnode\fi\fi}}
\ctr@ld@f\def\d@butn@de{\g@tt@xt\v@lX=\wd\Gb@x%
    \v@lY=\ht\Gb@x\advance\v@lY\dp\Gb@x\advance\v@lX\v@lmin\advance\v@lY\v@lmax}
\ctr@ld@f\def\fcn@deE{%
    \@ecfor\p@int:=\list@num\do{\d@butn@de\v@lX=\unssqrttw@\v@lX\v@lY=\unssqrttw@\v@lY%
    \ifdim\thickn@ss\p@>\z@
    \v@lXa=\v@lX\advance\v@lXa\delt@\v@lXa=\ptT@unit@\v@lXa\edef\XR@d{\repdecn@mb\v@lXa}%
    \v@lYa=\v@lY\advance\v@lYa\delt@\v@lYa=\ptT@unit@\v@lYa\edef\YR@d{\repdecn@mb\v@lYa}%
    \arct@n\v@leur(\v@lXa,\v@lYa)\v@leur=\rdT@deg\v@leur\edef\@nglde{\repdecn@mb\v@leur}%
    {\c@lptellDD-2::\p@int;\XR@d,\YR@d(\@nglde)}
    \advance\v@leur-\PI@deg\edef\@nglun{\repdecn@mb\v@leur}%
    {\c@lptellDD-3::\p@int;\XR@d,\YR@d(\@nglun)}%
    \figptstra-6=-3,-2,\p@int/\thickn@ss,-8/\pssetfillmode{yes}\us@secondC@lor%
    \psline[-2,-3,-6,-5]\psarcell-4;\XR@d,\YR@d(\@nglun,\@nglde,0)\fi
    \v@lX=\ptT@unit@\v@lX\v@lY=\ptT@unit@\v@lY%
    \edef\XR@d{\repdecn@mb\v@lX}\edef\YR@d{\repdecn@mb\v@lY}%
    \pssetfillmode{yes}\us@thirdC@lor\psarcell\p@int;\XR@d,\YR@d(0,360,0)%
    \pssetfillmode{no}\us@primarC@lor\psarcell\p@int;\XR@d,\YR@d(0,360,0)}}
\ctr@ld@f\def\fcn@deL{\delt@=\ptT@unit@\delt@\edef\t@ille{\repdecn@mb\delt@}%
    \@ecfor\p@int:=\list@num\do{\Figg@tXYa{\p@int}\d@butn@de%
    \ifdim\v@lX>\v@lY\itis@Ktrue\else\itis@Kfalse\fi%
    \advance\v@lXa-\v@lX\Figp@intreg-1:(\v@lXa,\v@lYa)%
    \advance\v@lXa\v@lX\advance\v@lYa-\v@lY\Figp@intreg-2:(\v@lXa,\v@lYa)%
    \advance\v@lXa\v@lX\advance\v@lYa\v@lY\Figp@intreg-3:(\v@lXa,\v@lYa)%
    \advance\v@lXa-\v@lX\advance\v@lYa\v@lY\Figp@intreg-4:(\v@lXa,\v@lYa)%
    \ifdim\thickn@ss\p@>\z@\Figg@tXYa{\p@int}\pssetfillmode{yes}\us@secondC@lor
    \c@lpt@xt{-1}{-4}\c@lpt@xt@\v@lXa\v@lYa\v@lX\v@lY\c@rre\delt@%
    \Figp@intregDD-9:(\v@lZ,\v@lYa)\Figp@intregDD-11:(\v@lZa,\v@lYa)%
    \c@lpt@xt{-4}{-3}\c@lpt@xt@\v@lYa\v@lXa\v@lY\v@lX\delt@\c@rre%
    \Figp@intregDD-12:(\v@lXa,\v@lZ)\Figp@intregDD-10:(\v@lXa,\v@lZa)%
    \ifitis@K\figptstra-7=-9,-10,-11/\thickn@ss,-8/\psline[-9,-11,-5,-6,-7]\else%
    \figptstra-7=-10,-11,-12/\thickn@ss,-8/\psline[-10,-12,-5,-6,-7]\fi\fi
    \pssetfillmode{yes}\us@thirdC@lor\psline[-1,-2,-3,-4]%
    \pssetfillmode{no}\us@primarC@lor\psline[-1,-2,-3,-4,-1]}}
\ctr@ld@f\def\c@lpt@xt#1#2{\figvectN-7[#1,#2]\vecunit@{-7}{-7}\figpttra-5:=#1/\t@ille,-7/%
    \figvectP-7[#1,#2]\Figg@tXY{-7}\c@rre=\v@lX\delt@=\v@lY\Figg@tXY{-5}}
\ctr@ld@f\def\c@lpt@xt@#1#2#3#4#5#6{\v@lZ=#6\invers@{\v@lZ}{\v@lZ}\v@leur=\repdecn@mb{#5}\v@lZ%
    \v@lZ=#2\advance\v@lZ-#4\mili@u=\repdecn@mb{\v@leur}\v@lZ%
    \v@lZ=#3\advance\v@lZ\mili@u\v@lZa=-\v@lZ\advance\v@lZa\tw@#1}
\ctr@ld@f\def\fcn@deR{\@ecfor\p@int:=\list@num\do{\Figg@tXYa{\p@int}\d@butn@de%
    \advance\v@lXa-0.5\v@lX\advance\v@lYa-0.5\v@lY\Figp@intreg-1:(\v@lXa,\v@lYa)%
    \advance\v@lXa\v@lX\Figp@intreg-2:(\v@lXa,\v@lYa)%
    \advance\v@lYa\v@lY\Figp@intreg-3:(\v@lXa,\v@lYa)%
    \advance\v@lXa-\v@lX\Figp@intreg-4:(\v@lXa,\v@lYa)%
    \ifdim\thickn@ss\p@>\z@\pssetfillmode{yes}\us@secondC@lor
    \Figv@ctCreg-5(-\delt@,-\delt@)\figpttra-9:=-1/1,-5/%
    \Figv@ctCreg-5(\delt@,-\delt@)\figpttra-10:=-2/1,-5/%
    \Figv@ctCreg-5(\delt@,\delt@)\figpttra-11:=-3/1,-5/%
    \figptstra-7=-9,-10,-11/\thickn@ss,-8/\psline[-9,-11,-5,-6,-7]\fi
    \pssetfillmode{yes}\us@thirdC@lor\psline[-1,-2,-3,-4]%
    \pssetfillmode{no}\us@primarC@lor\psline[-1,-2,-3,-4,-1]}}
\ctr@ln@m\@rrowp@s
\ctr@ln@m\Xp@dd     \ctr@ln@m\Yp@dd
\ctr@ln@m\fclin@r@d \ctr@ln@m\thickn@ss
\ctr@ld@f\def\Pssetfl@wchart#1=#2|{\keln@mtr#1|%
    \def\n@mref{arr}\ifx\l@debut\n@mref\expandafter\keln@mtr\l@suite|%
     \def\n@mref{owp}\ifx\l@debut\n@mref\edef\@rrowp@s{#2}\else
     \def\n@mref{owr}\ifx\l@debut\n@mref\setfcr@fpt#2|\else
     \immediate\write16{*** Unknown attribute: \BS@ psset flowchart(..., #1=...)}%
     \fi\fi\else%
    \def\n@mref{lin}\ifx\l@debut\n@mref\setfccurv@#2|\else
    \def\n@mref{pad}\ifx\l@debut\n@mref\edef\Xp@dd{#2}\edef\Yp@dd{#2}\else
    \def\n@mref{rad}\ifx\l@debut\n@mref\edef\fclin@r@d{#2}\else
    \def\n@mref{sha}\ifx\l@debut\n@mref\setfcshap@#2|\else
    \def\n@mref{thi}\ifx\l@debut\n@mref\edef\thickn@ss{#2}\else
    \def\n@mref{xpa}\ifx\l@debut\n@mref\edef\Xp@dd{#2}\else
    \def\n@mref{ypa}\ifx\l@debut\n@mref\edef\Yp@dd{#2}\else
    \immediate\write16{*** Unknown attribute: \BS@ psset flowchart(..., #1=...)}%
    \fi\fi\fi\fi\fi\fi\fi\fi}
\ctr@ln@m\@rrowr@fpt \ctr@ln@m\fclin@typ@
\ctr@ld@f\def\setfcr@fpt#1#2|{\if#1e\def\@rrowr@fpt{1}\else\def\@rrowr@fpt{0}\fi}
\ctr@ld@f\def\setfccurv@#1#2|{\if#1c\def\fclin@typ@{0}\else\def\fclin@typ@{1}\fi}
\ctr@ln@m\h@rdfcXp@dd \ctr@ln@m\h@rdfcYp@dd
\ctr@ln@m\fcn@de \ctr@ln@m\fcsh@pe
\ctr@ld@f\def\setfcshap@#1#2|{%
    \if#1e\let\fcn@de=\fcn@deE\def\h@rdfcXp@dd{4pt}\def\h@rdfcYp@dd{4pt}%
     \edef\fcsh@pe{ellipse}\else%
    \if#1l\let\fcn@de=\fcn@deL\def\h@rdfcXp@dd{4pt}\def\h@rdfcYp@dd{4pt}%
     \edef\fcsh@pe{lozenge}\else%
          \let\fcn@de=\fcn@deR\def\h@rdfcXp@dd{6pt}\def\h@rdfcYp@dd{6pt}%
     \edef\fcsh@pe{rectangle}\fi\fi}
\ctr@ld@f\def\psline[#1]{{\ifcurr@ntPS\ifps@cri\PSc@mment{psline Points=#1}%
    \let\pslign@=\Pslign@P\Pslin@{#1}\PSc@mment{End psline}\fi\fi}}
\ctr@ld@f\def\pslineF#1{{\ifcurr@ntPS\ifps@cri\PSc@mment{pslineF Filename=#1}%
    \let\pslign@=\Pslign@F\Pslin@{#1}\PSc@mment{End pslineF}\fi\fi}}
\ctr@ld@f\def\pslineC(#1){{\ifcurr@ntPS\ifps@cri\PSc@mment{pslineC}%
    \let\pslign@=\Pslign@C\Pslin@{#1}\PSc@mment{End pslineC}\fi\fi}}
\ctr@ld@f\def\Pslin@#1{\iffillm@de\pslign@{#1}%
    \f@gfill%
    \else\pslign@{#1}\ifx\derp@int\premp@int%
    \f@gclosestroke%
    \else\f@gstroke\fi\fi}
\ctr@ld@f\def\Pslign@P#1{\def\list@num{#1}\extrairelepremi@r\p@int\de\list@num%
    \edef\premp@int{\p@int}\f@gnewpath%
    \PSwrit@cmd{\p@int}{\c@mmoveto}{\fwf@g}%
    \@ecfor\p@int:=\list@num\do{\PSwrit@cmd{\p@int}{\c@mlineto}{\fwf@g}%
    \edef\derp@int{\p@int}}}
\ctr@ld@f\def\Pslign@F#1{\s@uvc@ntr@l\et@tPslign@F\setc@ntr@l{2}\openin\frf@g=#1\relax%
    \ifeof\frf@g\message{*** File #1 not found !}\end\else%
    \read\frf@g to\tr@c\edef\premp@int{\tr@c}\expandafter\extr@ctCF\tr@c:%
    \f@gnewpath\PSwrit@cmd{-1}{\c@mmoveto}{\fwf@g}%
    \loop\read\frf@g to\tr@c\ifeof\frf@g\mored@tafalse\else\mored@tatrue\fi%
    \ifmored@ta\expandafter\extr@ctCF\tr@c:\PSwrit@cmd{-1}{\c@mlineto}{\fwf@g}%
    \edef\derp@int{\tr@c}\repeat\fi\closein\frf@g\resetc@ntr@l\et@tPslign@F}
\ctr@ln@m\extr@ctCF
\ctr@ld@f\def\extr@ctCFDD#1 #2:{\v@lX=#1\unit@\v@lY=#2\unit@\Figp@intregDD-1:(\v@lX,\v@lY)}
\ctr@ld@f\def\extr@ctCFTD#1 #2 #3:{\v@lX=#1\unit@\v@lY=#2\unit@\v@lZ=#3\unit@%
    \Figp@intregTD-1:(\v@lX,\v@lY,\v@lZ)}
\ctr@ld@f\def\Pslign@C#1{\s@uvc@ntr@l\et@tPslign@C\setc@ntr@l{2}%
    \def\list@num{#1}\extrairelepremi@r\p@int\de\list@num%
    \edef\premp@int{\p@int}\f@gnewpath%
    \expandafter\Pslign@C@\p@int:\PSwrit@cmd{-1}{\c@mmoveto}{\fwf@g}%
    \@ecfor\p@int:=\list@num\do{\expandafter\Pslign@C@\p@int:%
    \PSwrit@cmd{-1}{\c@mlineto}{\fwf@g}\edef\derp@int{\p@int}}%
    \resetc@ntr@l\et@tPslign@C}
\ctr@ld@f\def\Pslign@C@#1 #2:{{\def\t@xt@{#1}\ifx\t@xt@\empty\Pslign@C@#2:
    \else\extr@ctCF#1 #2:\fi}}
\ctr@ln@m\c@ntrolmesh
\ctr@ld@f\def\Pssetm@sh#1=#2|{\keln@mun#1|%
    \def\n@mref{d}\ifx\l@debut\n@mref\pssetmeshdiag{#2}\else
    \immediate\write16{*** Unknown attribute: \BS@ psset mesh(..., #1=...)}%
    \fi}
\ctr@ld@f\def\pssetmeshdiag#1{\edef\c@ntrolmesh{#1}}
\ctr@ld@f\def\defaultmeshdiag{0}    
\ctr@ld@f\def\psmesh#1,#2[#3,#4,#5,#6]{{\ifcurr@ntPS\ifps@cri%
    \PSc@mment{psmesh N1=#1, N2=#2, Quadrangle=[#3,#4,#5,#6]}%
    \s@uvc@ntr@l\et@tpsmesh\Pss@tsecondSt\setc@ntr@l{2}%
    \ifnum#1>\@ne\Psmeshp@rt#1[#3,#4,#5,#6]\fi%
    \ifnum#2>\@ne\Psmeshp@rt#2[#4,#5,#6,#3]\fi%
    \ifnum\c@ntrolmesh>\z@\Psmeshdi@g#1,#2[#3,#4,#5,#6]\fi%
    \ifnum\c@ntrolmesh<\z@\Psmeshdi@g#2,#1[#4,#5,#6,#3]\fi\Psrest@reSt%
    \psline[#3,#4,#5,#6,#3]\PSc@mment{End psmesh}\resetc@ntr@l\et@tpsmesh\fi\fi}}
\ctr@ld@f\def\Psmeshp@rt#1[#2,#3,#4,#5]{{\l@mbd@un=\@ne\l@mbd@de=#1\loop%
    \ifnum\l@mbd@un<#1\advance\l@mbd@de\m@ne\figptbary-1:[#2,#3;\l@mbd@de,\l@mbd@un]%
    \figptbary-2:[#5,#4;\l@mbd@de,\l@mbd@un]\psline[-1,-2]\advance\l@mbd@un\@ne\repeat}}
\ctr@ld@f\def\Psmeshdi@g#1,#2[#3,#4,#5,#6]{\figptcopy-2:/#3/\figptcopy-3:/#6/%
    \l@mbd@un=\z@\l@mbd@de=#1\loop\ifnum\l@mbd@un<#1%
    \advance\l@mbd@un\@ne\advance\l@mbd@de\m@ne\figptcopy-1:/-2/\figptcopy-4:/-3/%
    \figptbary-2:[#3,#4;\l@mbd@de,\l@mbd@un]%
    \figptbary-3:[#6,#5;\l@mbd@de,\l@mbd@un]\Psmeshdi@gp@rt#2[-1,-2,-3,-4]\repeat}
\ctr@ld@f\def\Psmeshdi@gp@rt#1[#2,#3,#4,#5]{{\l@mbd@un=\z@\l@mbd@de=#1\loop%
    \ifnum\l@mbd@un<#1\figptbary-5:[#2,#5;\l@mbd@de,\l@mbd@un]%
    \advance\l@mbd@de\m@ne\advance\l@mbd@un\@ne%
    \figptbary-6:[#3,#4;\l@mbd@de,\l@mbd@un]\psline[-5,-6]\repeat}}
\ctr@ln@m\psnormal
\ctr@ld@f\def\psnormalDD#1,#2[#3,#4]{{\ifcurr@ntPS\ifps@cri%
    \PSc@mment{psnormal Length=#1, Lambda=#2 [Pt1,Pt2]=[#3,#4]}%
    \s@uvc@ntr@l\et@tpsnormal\resetc@ntr@l{2}\figptendnormal-6::#1,#2[#3,#4]%
    \figptcopyDD-5:/-1/\psarrow[-5,-6]%
    \PSc@mment{End psnormal}\resetc@ntr@l\et@tpsnormal\fi\fi}}
\ctr@ld@f\def\psreset#1{\trtlis@rg{#1}{\Psreset@}}
\ctr@ld@f\def\Psreset@#1|{\keln@mde#1|%
    \def\n@mref{ar}\ifx\l@debut\n@mref\psresetarrowhead\else
    \def\n@mref{cu}\ifx\l@debut\n@mref\psset curve(roundness=\defaultroundness)\else
    \def\n@mref{fi}\ifx\l@debut\n@mref\psset (color=\defaultcolor,dash=\defaultdash,%
         fill=\defaultfill,join=\defaultjoin,width=\defaultwidth)\else
    \def\n@mref{fl}\ifx\l@debut\n@mref\psset flowchart(arrowp=\defaultfcarrowposition,%
	arrowr=\defaultfcarrowrefpt,line=\defaultfcline,xpadd=\defaultfcxpadding,%
	ypadd=\defaultfcypadding,radius=\defaultfcradius,shape=\defaultfcshape,%
	thick=\defaultfcthickness)\else
    \def\n@mref{me}\ifx\l@debut\n@mref\psset mesh(diag=\defaultmeshdiag)\else
    \def\n@mref{se}\ifx\l@debut\n@mref\psresetsecondsettings\else
    \def\n@mref{th}\ifx\l@debut\n@mref\psset third(color=\defaultthirdcolor)\else
    \immediate\write16{*** Unknown keyword #1 (\BS@ psreset).}%
    \fi\fi\fi\fi\fi\fi\fi}
\ctr@ld@f\def\psset#1(#2){\def\t@xt@{#1}\ifx\t@xt@\empty\trtlis@rg{#2}{\Pssetf@rst}
    \else\keln@mde#1|%
    \def\n@mref{ar}\ifx\l@debut\n@mref\trtlis@rg{#2}{\Psset@rrowhe@d}\else
    \def\n@mref{cu}\ifx\l@debut\n@mref\trtlis@rg{#2}{\Pssetc@rve}\else
    \def\n@mref{fi}\ifx\l@debut\n@mref\trtlis@rg{#2}{\Pssetf@rst}\else
    \def\n@mref{fl}\ifx\l@debut\n@mref\trtlis@rg{#2}{\Pssetfl@wchart}\else
    \def\n@mref{me}\ifx\l@debut\n@mref\trtlis@rg{#2}{\Pssetm@sh}\else
    \def\n@mref{se}\ifx\l@debut\n@mref\trtlis@rg{#2}{\Pssets@cond}\else
    \def\n@mref{th}\ifx\l@debut\n@mref\trtlis@rg{#2}{\Pssetth@rd}\else
    \immediate\write16{*** Unknown keyword: \BS@ psset #1(...)}%
    \fi\fi\fi\fi\fi\fi\fi\fi}
\ctr@ld@f\def\pssetdefault#1(#2){\ifcurr@ntPS\immediate\write16{*** \BS@ pssetdefault is ignored
    inside a \BS@ psbeginfig-\BS@ psendfig block.}%
    \immediate\write16{*** It must be called before \BS@ psbeginfig.}\else%
    \def\t@xt@{#1}\ifx\t@xt@\empty\trtlis@rg{#2}{\Pssd@f@rst}\else\keln@mde#1|%
    \def\n@mref{ar}\ifx\l@debut\n@mref\trtlis@rg{#2}{\Pssd@@rrowhe@d}\else
    \def\n@mref{cu}\ifx\l@debut\n@mref\trtlis@rg{#2}{\Pssd@c@rve}\else
    \def\n@mref{fi}\ifx\l@debut\n@mref\trtlis@rg{#2}{\Pssd@f@rst}\else
    \def\n@mref{fl}\ifx\l@debut\n@mref\trtlis@rg{#2}{\Pssd@fl@wchart}\else
    \def\n@mref{me}\ifx\l@debut\n@mref\trtlis@rg{#2}{\Pssd@m@sh}\else
    \def\n@mref{se}\ifx\l@debut\n@mref\trtlis@rg{#2}{\Pssd@s@cond}\else
    \def\n@mref{th}\ifx\l@debut\n@mref\trtlis@rg{#2}{\Pssd@th@rd}\else
    \immediate\write16{*** Unknown keyword: \BS@ pssetdefault #1(...)}%
    \fi\fi\fi\fi\fi\fi\fi\fi\initpss@ttings\fi}
\ctr@ld@f\def\Pssd@f@rst#1=#2|{\keln@mun#1|%
    \def\n@mref{c}\ifx\l@debut\n@mref\edef\defaultcolor{#2}\else
    \def\n@mref{d}\ifx\l@debut\n@mref\edef\defaultdash{#2}\else
    \def\n@mref{f}\ifx\l@debut\n@mref\edef\defaultfill{#2}\else
    \def\n@mref{j}\ifx\l@debut\n@mref\edef\defaultjoin{#2}\else
    \def\n@mref{u}\ifx\l@debut\n@mref\edef\defaultupdate{#2}\pssetupdate{#2}\else
    \def\n@mref{w}\ifx\l@debut\n@mref\edef\defaultwidth{#2}\else
    \immediate\write16{*** Unknown attribute: \BS@ pssetdefault (..., #1=...)}%
    \fi\fi\fi\fi\fi\fi}
\ctr@ld@f\def\Pssd@@rrowhe@d#1=#2|{\keln@mun#1|%
    \def\n@mref{a}\ifx\l@debut\n@mref\edef\defaultarrowheadangle{#2}\else
    \def\n@mref{f}\ifx\l@debut\n@mref\edef\defaultarrowheadangle{#2}\else
    \def\n@mref{l}\ifx\l@debut\n@mref\y@tiunit{#2}\ifunitpr@sent%
     \edef\defaulth@rdahlength{#2}\else\edef\defaulth@rdahlength{#2pt}%
     \message{*** \BS@ pssetdefault (..., #1=#2, ...) : unit is missing, pt is assumed.}%
     \fi\else
    \def\n@mref{o}\ifx\l@debut\n@mref\edef\defaultarrowheadout{#2}\else
    \def\n@mref{r}\ifx\l@debut\n@mref\edef\defaultarrowheadratio{#2}\else
    \immediate\write16{*** Unknown attribute: \BS@ pssetdefault arrowhead(..., #1=...)}%
    \fi\fi\fi\fi\fi}
\ctr@ld@f\def\Pssd@c@rve#1=#2|{\keln@mun#1|%
    \def\n@mref{r}\ifx\l@debut\n@mref\edef\defaultroundness{#2}\else%
    \immediate\write16{*** Unknown attribute: \BS@ pssetdefault curve(..., #1=...)}%
    \fi}
\ctr@ld@f\def\Pssd@fl@wchart#1=#2|{\keln@mtr#1|%
    \def\n@mref{arr}\ifx\l@debut\n@mref\expandafter\keln@mtr\l@suite|%
     \def\n@mref{owp}\ifx\l@debut\n@mref\edef\defaultfcarrowposition{#2}\else
     \def\n@mref{owr}\ifx\l@debut\n@mref\edef\defaultfcarrowrefpt{#2}\else
     \immediate\write16{*** Unknown attribute: \BS@ pssetdefault flowchart(..., #1=...)}%
     \fi\fi\else%
    \def\n@mref{lin}\ifx\l@debut\n@mref\edef\defaultfcline{#2}\else
    \def\n@mref{pad}\ifx\l@debut\n@mref\edef\defaultfcxpadding{#2}%
                    \edef\defaultfcypadding{#2}\else
    \def\n@mref{rad}\ifx\l@debut\n@mref\edef\defaultfcradius{#2}\else
    \def\n@mref{sha}\ifx\l@debut\n@mref\edef\defaultfcshape{#2}\else
    \def\n@mref{thi}\ifx\l@debut\n@mref\edef\defaultfcthickness{#2}\else
    \def\n@mref{xpa}\ifx\l@debut\n@mref\edef\defaultfcxpadding{#2}\else
    \def\n@mref{ypa}\ifx\l@debut\n@mref\edef\defaultfcypadding{#2}\else
    \immediate\write16{*** Unknown attribute: \BS@ pssetdefault flowchart(..., #1=...)}%
    \fi\fi\fi\fi\fi\fi\fi\fi}
\ctr@ld@f\def\defaultfcarrowposition{0.5}
\ctr@ld@f\def\defaultfcarrowrefpt{start}
\ctr@ld@f\def\defaultfcline{polygon}
\ctr@ld@f\def\defaultfcradius{0}
\ctr@ld@f\def\defaultfcshape{rectangle}
\ctr@ld@f\def\defaultfcthickness{0}
\ctr@ld@f\def\defaultfcxpadding{0}
\ctr@ld@f\def\defaultfcypadding{0}
\ctr@ld@f\def\Pssd@m@sh#1=#2|{\keln@mun#1|%
    \def\n@mref{d}\ifx\l@debut\n@mref\edef\defaultmeshdiag{#2}\else%
    \immediate\write16{*** Unknown attribute: \BS@ pssetdefault mesh(..., #1=...)}%
    \fi}
\ctr@ld@f\def\Pssd@s@cond#1=#2|{\keln@mun#1|%
    \def\n@mref{c}\ifx\l@debut\n@mref\edef\defaultsecondcolor{#2}\else%
    \def\n@mref{d}\ifx\l@debut\n@mref\edef\defaultseconddash{#2}\else%
    \def\n@mref{w}\ifx\l@debut\n@mref\edef\defaultsecondwidth{#2}\else%
    \immediate\write16{*** Unknown attribute: \BS@ pssetdefault second(..., #1=...)}%
    \fi\fi\fi}
\ctr@ld@f\def\Pssd@th@rd#1=#2|{\keln@mun#1|%
    \def\n@mref{c}\ifx\l@debut\n@mref\edef\defaultthirdcolor{#2}\else%
    \immediate\write16{*** Unknown attribute: \BS@ pssetdefault third(..., #1=...)}%
    \fi}
\ctr@ln@w{newif}\iffillm@de
\ctr@ld@f\def\pssetfillmode#1{\expandafter\setfillm@de#1:}
\ctr@ld@f\def\setfillm@de#1#2:{\if#1n\fillm@defalse\else\fillm@detrue\fi}
\ctr@ld@f\def\defaultfill{no}     
\ctr@ln@w{newif}\ifpsupdatem@de
\ctr@ld@f\def\pssetupdate#1{\ifcurr@ntPS\immediate\write16{*** \BS@ pssetupdate is ignored inside a
     \BS@ psbeginfig-\BS@ psendfig block.}%
    \immediate\write16{*** It must be called before \BS@ psbeginfig.}%
    \else\expandafter\setupd@te#1:\fi}
\ctr@ld@f\def\setupd@te#1#2:{\if#1n\psupdatem@defalse\else\psupdatem@detrue\fi}
\ctr@ld@f\def\defaultupdate{no}     
\ctr@ln@m\curr@ntcolor \ctr@ln@m\curr@ntcolorc@md
\ctr@ld@f\def\Pssetc@lor#1{\ifps@cri\result@tent=\@ne\expandafter\c@lnbV@l#1 :%
    \def\curr@ntcolor{}\def\curr@ntcolorc@md{}%
    \ifcase\result@tent\or\pssetgray{#1}\or\or\pssetrgb{#1}\or\pssetcmyk{#1}\fi\fi}
\ctr@ln@m\curr@ntcolorc@mdStroke
\ctr@ld@f\def\pssetcmyk#1{\ifps@cri\def\curr@ntcolor{#1}\def\curr@ntcolorc@md{\c@msetcmykcolor}%
    \def\curr@ntcolorc@mdStroke{\c@msetcmykcolorStroke}%
    \ifcurr@ntPS\PSc@mment{pssetcmyk Color=#1}\us@primarC@lor\fi\fi}
\ctr@ld@f\def\pssetrgb#1{\ifps@cri\def\curr@ntcolor{#1}\def\curr@ntcolorc@md{\c@msetrgbcolor}%
    \def\curr@ntcolorc@mdStroke{\c@msetrgbcolorStroke}%
    \ifcurr@ntPS\PSc@mment{pssetrgb Color=#1}\us@primarC@lor\fi\fi}
\ctr@ld@f\def\pssetgray#1{\ifps@cri\def\curr@ntcolor{#1}\def\curr@ntcolorc@md{\c@msetgray}%
    \def\curr@ntcolorc@mdStroke{\c@msetgrayStroke}%
    \ifcurr@ntPS\PSc@mment{pssetgray Gray level=#1}\us@primarC@lor\fi\fi}
\ctr@ln@m\fillc@md
\ctr@ld@f\def\us@primarC@lor{\immediate\write\fwf@g{\d@fprimarC@lor}%
    \let\fillc@md=\prfillc@md}
\ctr@ld@f\def\prfillc@md{\d@fprimarC@lor\space\c@mfill}
\ctr@ld@f\def\defaultcolor{0}       
\ctr@ld@f\def\c@lnbV@l#1 #2:{\def\t@xt@{#1}\relax\ifx\t@xt@\empty\c@lnbV@l#2:
    \else\c@lnbV@l@#1 #2:\fi}
\ctr@ld@f\def\c@lnbV@l@#1 #2:{\def\t@xt@{#2}\ifx\t@xt@\empty%
    \def\t@xt@{#1}\ifx\t@xt@\empty\advance\result@tent\m@ne\fi
    \else\advance\result@tent\@ne\c@lnbV@l@#2:\fi}
\ctr@ld@f\def\Blackcmyk{0 0 0 1}
\ctr@ld@f\def\Whitecmyk{0 0 0 0}
\ctr@ld@f\def\Cyancmyk{1 0 0 0}
\ctr@ld@f\def\Magentacmyk{0 1 0 0}
\ctr@ld@f\def\Yellowcmyk{0 0 1 0}
\ctr@ld@f\def\Redcmyk{0 1 1 0}
\ctr@ld@f\def\Greencmyk{1 0 1 0}
\ctr@ld@f\def\Bluecmyk{1 1 0 0}
\ctr@ld@f\def\Graycmyk{0 0 0 0.50}
\ctr@ld@f\def\BrickRedcmyk{0 0.89 0.94 0.28} 
\ctr@ld@f\def\Browncmyk{0 0.81 1 0.60} 
\ctr@ld@f\def\ForestGreencmyk{0.91 0 0.88 0.12} 
\ctr@ld@f\def\Goldenrodcmyk{ 0 0.10 0.84 0} 
\ctr@ld@f\def\Marooncmyk{0 0.87 0.68 0.32} 
\ctr@ld@f\def\Orangecmyk{0 0.61 0.87 0} 
\ctr@ld@f\def\Purplecmyk{0.45 0.86 0 0} 
\ctr@ld@f\def\RoyalBluecmyk{1. 0.50 0 0} 
\ctr@ld@f\def\Violetcmyk{0.79 0.88 0 0} 
\ctr@ld@f\def\Blackrgb{0 0 0}
\ctr@ld@f\def\Whitergb{1 1 1}
\ctr@ld@f\def\Redrgb{1 0 0}
\ctr@ld@f\def\Greenrgb{0 1 0}
\ctr@ld@f\def\Bluergb{0 0 1}
\ctr@ld@f\def\Cyanrgb{0 1 1}
\ctr@ld@f\def\Magentargb{1 0 1}
\ctr@ld@f\def\Yellowrgb{1 1 0}
\ctr@ld@f\def\Grayrgb{0.5 0.5 0.5}
\ctr@ld@f\def\Chocolatergb{0.824 0.412 0.118}
\ctr@ld@f\def\DarkGoldenrodrgb{0.722 0.525 0.043}
\ctr@ld@f\def\DarkOrangergb{1 0.549 0}
\ctr@ld@f\def\Firebrickrgb{0.698 0.133 0.133}
\ctr@ld@f\def\ForestGreenrgb{0.133 0.545 0.133}
\ctr@ld@f\def\Goldrgb{1 0.843 0}
\ctr@ld@f\def\HotPinkrgb{1 0.412 0.706}
\ctr@ld@f\def\Maroonrgb{0.690 0.188 0.376}
\ctr@ld@f\def\Pinkrgb{1 0.753 0.796}
\ctr@ld@f\def\RoyalBluergb{0.255 0.412 0.882}
\ctr@ld@f\def\Pssetf@rst#1=#2|{\keln@mun#1|%
    \def\n@mref{c}\ifx\l@debut\n@mref\Pssetc@lor{#2}\else
    \def\n@mref{d}\ifx\l@debut\n@mref\pssetdash{#2}\else
    \def\n@mref{f}\ifx\l@debut\n@mref\pssetfillmode{#2}\else
    \def\n@mref{j}\ifx\l@debut\n@mref\pssetjoin{#2}\else
    \def\n@mref{u}\ifx\l@debut\n@mref\pssetupdate{#2}\else
    \def\n@mref{w}\ifx\l@debut\n@mref\pssetwidth{#2}\else
    \immediate\write16{*** Unknown attribute: \BS@ psset (..., #1=...)}%
    \fi\fi\fi\fi\fi\fi}
\ctr@ln@m\curr@ntdash
\ctr@ld@f\def\s@uvdash#1{\edef#1{\curr@ntdash}}
\ctr@ld@f\def\defaultdash{1}        
\ctr@ld@f\def\pssetdash#1{\ifps@cri\edef\curr@ntdash{#1}\ifcurr@ntPS\expandafter\Pssetd@sh#1 :\fi\fi}
\ctr@ld@f\def\Pssetd@shI#1{\PSc@mment{pssetdash Index=#1}\ifcase#1%
    \or\immediate\write\fwf@g{[] 0 \c@msetdash}
    \or\immediate\write\fwf@g{[6 2] 0 \c@msetdash}
    \or\immediate\write\fwf@g{[4 2] 0 \c@msetdash}
    \or\immediate\write\fwf@g{[2 2] 0 \c@msetdash}
    \or\immediate\write\fwf@g{[1 2] 0 \c@msetdash}
    \or\immediate\write\fwf@g{[2 4] 0 \c@msetdash}
    \or\immediate\write\fwf@g{[3 5] 0 \c@msetdash}
    \or\immediate\write\fwf@g{[3 3] 0 \c@msetdash}
    \or\immediate\write\fwf@g{[3 5 1 5] 0 \c@msetdash}
    \or\immediate\write\fwf@g{[6 4 2 4] 0 \c@msetdash}
    \fi}
\ctr@ld@f\def\Pssetd@sh#1 #2:{{\def\t@xt@{#1}\ifx\t@xt@\empty\Pssetd@sh#2:
    \else\def\t@xt@{#2}\ifx\t@xt@\empty\Pssetd@shI{#1}\else\s@mme=\@ne\def\debutp@t{#1}%
    \an@lysd@sh#2:\ifodd\s@mme\edef\debutp@t{\debutp@t\space\finp@t}\def\finp@t{0}\fi%
    \PSc@mment{pssetdash Pattern=#1 #2}%
    \immediate\write\fwf@g{[\debutp@t] \finp@t\space\c@msetdash}\fi\fi}}
\ctr@ld@f\def\an@lysd@sh#1 #2:{\def\t@xt@{#2}\ifx\t@xt@\empty\def\finp@t{#1}\else%
    \edef\debutp@t{\debutp@t\space#1}\advance\s@mme\@ne\an@lysd@sh#2:\fi}
\ctr@ln@m\curr@ntwidth
\ctr@ld@f\def\s@uvwidth#1{\edef#1{\curr@ntwidth}}
\ctr@ld@f\def\defaultwidth{0.4}     
\ctr@ld@f\def\pssetwidth#1{\ifps@cri\edef\curr@ntwidth{#1}\ifcurr@ntPS%
    \PSc@mment{pssetwidth Width=#1}\immediate\write\fwf@g{#1 \c@msetlinewidth}\fi\fi}
\ctr@ln@m\curr@ntjoin
\ctr@ld@f\def\pssetjoin#1{\ifps@cri\edef\curr@ntjoin{#1}\ifcurr@ntPS\expandafter\Pssetj@in#1:\fi\fi}
\ctr@ld@f\def\Pssetj@in#1#2:{\PSc@mment{pssetjoin join=#1}%
    \if#1r\def\t@xt@{1}\else\if#1b\def\t@xt@{2}\else\def\t@xt@{0}\fi\fi%
    \immediate\write\fwf@g{\t@xt@\space\c@msetlinejoin}}
\ctr@ld@f\def\defaultjoin{miter}   
\ctr@ld@f\def\Pssets@cond#1=#2|{\keln@mun#1|%
    \def\n@mref{c}\ifx\l@debut\n@mref\Pssets@condcolor{#2}\else%
    \def\n@mref{d}\ifx\l@debut\n@mref\pssetseconddash{#2}\else%
    \def\n@mref{w}\ifx\l@debut\n@mref\pssetsecondwidth{#2}\else%
    \immediate\write16{*** Unknown attribute: \BS@ psset second(..., #1=...)}%
    \fi\fi\fi}
\ctr@ln@m\curr@ntseconddash
\ctr@ld@f\def\pssetseconddash#1{\edef\curr@ntseconddash{#1}}
\ctr@ld@f\def\defaultseconddash{4}  
\ctr@ln@m\curr@ntsecondwidth
\ctr@ld@f\def\pssetsecondwidth#1{\edef\curr@ntsecondwidth{#1}}
\ctr@ld@f\edef\defaultsecondwidth{\defaultwidth} 
\ctr@ld@f\def\psresetsecondsettings{%
    \pssetseconddash{\defaultseconddash}\pssetsecondwidth{\defaultsecondwidth}%
    \Pssets@condcolor{\defaultsecondcolor}}
\ctr@ln@m\sec@ndcolor \ctr@ln@m\sec@ndcolorc@md
\ctr@ld@f\def\Pssets@condcolor#1{\ifps@cri\result@tent=\@ne\expandafter\c@lnbV@l#1 :%
    \def\sec@ndcolor{}\def\sec@ndcolorc@md{}%
    \ifcase\result@tent\or\pssetsecondgray{#1}\or\or\pssetsecondrgb{#1}%
    \or\pssetsecondcmyk{#1}\fi\fi}
\ctr@ln@m\sec@ndcolorc@mdStroke
\ctr@ld@f\def\pssetsecondcmyk#1{\def\sec@ndcolor{#1}\def\sec@ndcolorc@md{\c@msetcmykcolor}%
    \def\sec@ndcolorc@mdStroke{\c@msetcmykcolorStroke}}
\ctr@ld@f\def\pssetsecondrgb#1{\def\sec@ndcolor{#1}\def\sec@ndcolorc@md{\c@msetrgbcolor}%
    \def\sec@ndcolorc@mdStroke{\c@msetrgbcolorStroke}}
\ctr@ld@f\def\pssetsecondgray#1{\def\sec@ndcolor{#1}\def\sec@ndcolorc@md{\c@msetgray}%
    \def\sec@ndcolorc@mdStroke{\c@msetgrayStroke}}
\ctr@ld@f\def\us@secondC@lor{\immediate\write\fwf@g{\d@fsecondC@lor}%
    \let\fillc@md=\sdfillc@md}
\ctr@ld@f\def\sdfillc@md{\d@fsecondC@lor\space\c@mfill}
\ctr@ld@f\edef\defaultsecondcolor{\defaultcolor} 
\ctr@ld@f\def\Pss@tsecondSt{%
    \s@uvdash{\typ@dash}\pssetdash{\curr@ntseconddash}%
    \s@uvwidth{\typ@width}\pssetwidth{\curr@ntsecondwidth}\us@secondC@lor}
\ctr@ld@f\def\Psrest@reSt{\pssetwidth{\typ@width}\pssetdash{\typ@dash}\us@primarC@lor}
\ctr@ld@f\def\Pssetth@rd#1=#2|{\keln@mun#1|%
    \def\n@mref{c}\ifx\l@debut\n@mref\Pssetth@rdcolor{#2}\else%
    \immediate\write16{*** Unknown attribute: \BS@ psset third(..., #1=...)}%
    \fi}
\ctr@ln@m\th@rdcolor \ctr@ln@m\th@rdcolorc@md
\ctr@ld@f\def\Pssetth@rdcolor#1{\ifps@cri\result@tent=\@ne\expandafter\c@lnbV@l#1 :%
    \def\th@rdcolor{}\def\th@rdcolorc@md{}%
    \ifcase\result@tent\or\Pssetth@rdgray{#1}\or\or\Pssetth@rdrgb{#1}%
    \or\Pssetth@rdcmyk{#1}\fi\fi}
\ctr@ln@m\th@rdcolorc@mdStroke
\ctr@ld@f\def\Pssetth@rdcmyk#1{\def\th@rdcolor{#1}\def\th@rdcolorc@md{\c@msetcmykcolor}%
    \def\th@rdcolorc@mdStroke{\c@msetcmykcolorStroke}}
\ctr@ld@f\def\Pssetth@rdrgb#1{\def\th@rdcolor{#1}\def\th@rdcolorc@md{\c@msetrgbcolor}%
    \def\th@rdcolorc@mdStroke{\c@msetrgbcolorStroke}}
\ctr@ld@f\def\Pssetth@rdgray#1{\def\th@rdcolor{#1}\def\th@rdcolorc@md{\c@msetgray}%
    \def\th@rdcolorc@mdStroke{\c@msetgrayStroke}}
\ctr@ld@f\def\us@thirdC@lor{\immediate\write\fwf@g{\d@fthirdC@lor}%
    \let\fillc@md=\thfillc@md}
\ctr@ld@f\def\thfillc@md{\d@fthirdC@lor\space\c@mfill}
\ctr@ld@f\def\defaultthirdcolor{1}  
\ctr@ld@f\def\pstrimesh#1[#2,#3,#4]{{\ifcurr@ntPS\ifps@cri%
    \PSc@mment{pstrimesh Type=#1, Triangle=[#2,#3,#4]}%
    \s@uvc@ntr@l\et@tpstrimesh\ifnum#1>\@ne\Pss@tsecondSt\setc@ntr@l{2}%
    \Pstrimeshp@rt#1[#2,#3,#4]\Pstrimeshp@rt#1[#3,#4,#2]%
    \Pstrimeshp@rt#1[#4,#2,#3]\Psrest@reSt\fi\psline[#2,#3,#4,#2]%
    \PSc@mment{End pstrimesh}\resetc@ntr@l\et@tpstrimesh\fi\fi}}
\ctr@ld@f\def\Pstrimeshp@rt#1[#2,#3,#4]{{\l@mbd@un=\@ne\l@mbd@de=#1\loop\ifnum\l@mbd@de>\@ne%
    \advance\l@mbd@de\m@ne\figptbary-1:[#2,#3;\l@mbd@de,\l@mbd@un]%
    \figptbary-2:[#2,#4;\l@mbd@de,\l@mbd@un]\psline[-1,-2]%
    \advance\l@mbd@un\@ne\repeat}}
\initpr@lim\initpss@ttings\initPDF@rDVI
\ctr@ln@w{newbox}\figBoxA
\ctr@ln@w{newbox}\figBoxB
\ctr@ln@w{newbox}\figBoxC
\catcode`\@=12

\pssetdefault(update=yes)
\newbox\figbox

\begin{abstract}
The plane waveguides with corners considered here are infinite V-shaped strips with constant thickness.
They are parametrized by their sole opening angle.
We study the eigenpairs of the Dirichlet Laplacian in such domains when this angle tends to $0$. We provide multi-scale asymptotics for eigenpairs associated with the lowest eigenvalues. For this, we investigate the eigenpairs of a one-dimensional model which can be viewed as their Born-Oppenheimer approximation. We also investigate the Dirichlet Laplacian on triangles with sharp angles. The eigenvalue asymptotics involve powers of the cube root of the angle, while the eigenvector asymptotics include simultaneously two scales in the triangular part, and one scale in the straight part of the guides.
\end{abstract}

\maketitle

\section{Introduction and main results}

\subsection{Motivations}
Quantum waveguides refer to meso- or nanoscale wires (or thin sheets) inside electronic devices. They can be modelled by one-electron Schr\"o\-dinger operators with potentials having high contrast in their values. In many situations, such Schr\"odinger operators can be approximated by a simple Laplace operator with Dirichlet conditions on the boundary of the wires \cite{Duclos95}. The presence of bound states is an undesirable effect which is nevertheless frequent and useful to predict. The same Laplace-Dirichlet problems arise for TE (transverse electric) modes in electromagnetic waveguides \cite{CLMM93}.

This is a well-known fact, from the papers \cite{ExSe89, Duclos95, Carron04, Duclos05}, that curvature makes discrete spectrum to appear in waveguides. Moreover the analysis of this spectrum can be accurately performed in the thin tube limit (in dimension $2$ and $3$, see \cite[Section 5]{Duclos95}). In fact, this asymptotical regime corresponds to a semiclassical limit so that the standard techniques of \cite{Hel88} could have been used to investigate that problem. 

Since curvature induces discrete spectrum, this is then a natural question to ask what happens in dimension $2$ when there is a corner (which corresponds to infinite curvature): Does discrete spectrum always exist in this case? This question is investigated with the $L$-shape waveguide in \cite{Exner89} where the existence of discrete spectrum is proved. For an arbitrary angle too, this existence is proved in \cite{ABGM91} and an asymptotic study of the ground energy is done when $\theta$ goes to $\frac{\pi}{2}$ (where $\theta$ is the semi-opening of the waveguide). Another question which arises is the estimate of the lowest eigenvalues in the regime $\theta\to 0$. This problem is analyzed in \cite{CLMM93} through matched asymptotic expansions and electromagnetic experiments. This is precisely the question we tackle in this paper: We are going to prove rigorously complete asymptotic expansions for the eigenpairs in plane waveguides with corner (also called ``broken strips") as $\theta$ tends to $0$. We have provided in \cite{DaLafRa11} numerical experiments by the finite element method for this situation too.

For the case of dimension 3, we can cite the paper \cite{ExTa10} which deals with the Dirichlet Laplacian in a conical layer. In this case, there is an infinite number of eigenvalues below the essential spectrum. The other initial motivation for the present investigation is our previous work \cite{BDPR} in which we study the Neumann realization on ${\R^2_{+}=\{(s,t)\in\R^2 : t>0\}}$ of the Schr\"odinger operator $-\dr_{s}^2-\dr_{t}^2+(t\cos\theta-s\sin\theta)^2$ in the regime $\theta\to 0$ (see also \cite{LuPan00a,HelMo02}). It turns out that the lowest eigenfunctions of this operator are concentrated near the cancellation line of the potential, which also enlighten the link between a confining electric potential and a strip with Dirichlet boundary conditions.

In our way towards the analysis of plane waveguides with corners, a natural step turns out to be the study of the Dirichlet problem on isosceles triangles with small angle. This subject is already dealt with in \cite[Theorem 1]{Fre07} where four-term asymptotics is proved for the first eigenvalue, whereas a three-term asymptotics for the second eigenvalue is provided in \cite[Section 2]{Fre07}. In fact the spectral analysis of triangles with small angles is not the sole way to succeed in the study of waveguides. Nevertheless, as  just mentioned, this problem has a particular interest on its own and permits to enlighten the presentation of the proofs.

Finally, in the same vein, we can mention the papers \cite{FriSolo08, FriSolo09} whose results provide two-term asymptotics for the thin rhombi and also \cite{BoFre09} which deals with a regular case (thin ellipse for instance), see also \cite{BoFre10}.

\subsection{The Dirichlet Laplacian on the broken guide}
Here we introduce the family of broken guides $\Omega_\theta$, parametrized by the angle $\theta$, and give basic properties of the spectrum of the positive Laplacian with Dirichlet condition in $\Omega_\theta$. Then we state our main result related to the behavior as $\theta\to0$ of the lowest eigenvalues of these operators.
\subsubsection{Basic properties}
Let us denote by $(x_1,x_2)$ the Cartesian coordinates of the plane and by ${\bf0}=(0,0)$ the origin. The positive Laplace operator is given by $-\dr_{1}^2-\dr_{2}^2$. The domains of interest are the ``broken waveguides'' which are infinite V-shaped open sets: For any angle $\theta\in\left(0,\frac{\pi}{2}\right)$ we introduce
\begin{equation}
\label{E:Omtheta}
   \Omega_{\theta}=\left\{(x_{1},x_{2})\in\mathbb{R}^2 : x_{1}\tan\theta<|x_{2}|<\left(x_{1}+\frac{\pi}{\sin\theta}\right)\tan\theta\right\}.
\end{equation}
Note that its width is independent from $\theta$, normalized to $\pi$, and $\theta$ represents the (half) opening of the V, see Fig.\ \ref{F:1}. The limit case where $\theta=\frac\pi2$ corresponds to the straight strip $(-\pi,0)\times\R$.
The aim of this paper is the investigation of the lowest eigenvalues of the {\em positive} Dirichlet Laplacian $\Delta^\Dir_{\Omega_{\theta}}$ in the small angle limit $\theta\to0$.

\begin{figure}[ht]
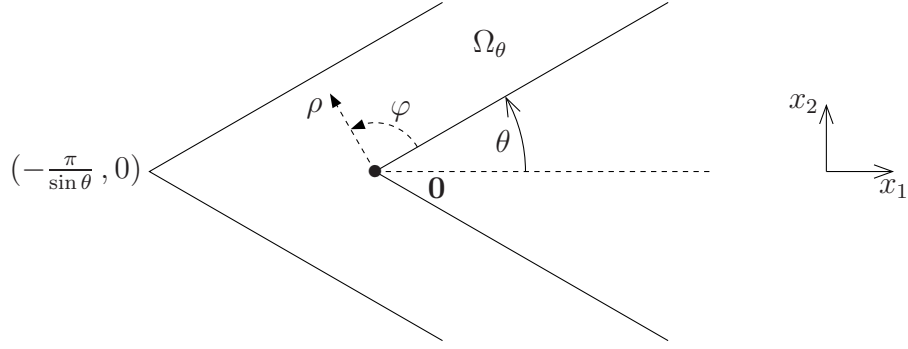

    \figinit{1mm}
    \figpt 1:(  0,  0)
    \figpt 2:( 45,  0)
    \figptrot 3: = 2 /1, 30/
    \figpt 4:(-30,  0)
    \figpt 5:(12, 0)
    \figptrot 6: = 5 /1, 120/
    \figvectP 101 [1,4]
    \figptstra 14 = 3/1, 101/
    \figptssym 23 = 3, 14 /1, 2/
    \figpt 10:(60,  0)

    \def\MyPSfile{F1.pf}
    \psbeginfig{\MyPSfile}
    \psaxes 10(9)
    \figptsaxes 11 : 10(9)
    \psline[23,1,3]
    \psline[24,4,14]
    \psarrowcirc 1 ; 20 (0,30)
    \psset arrowhead (fillmode=yes, length=2)
    \psset (dash=4)
    \psarrow[1,6]
    \psarrowcirc 1 ; 6.5 (30,120)
    \psline[1,2]
  \psendfig

    \figvisu{\figbox}{}{%
    \figinsert{\MyPSfile}
    \figwrites 11 :{$x_1$}(1)
    \figwritew 12 :{$x_2$}(1)
    \figwritegcw 4 :{$(-\frac{\pi}{\sin\theta}\,,0)$}(1,0)
    \figwritegce 1 :{$\Omega_{\theta}$}(13,17)
    \figwritegce 1 :{$\varphi$}(2,8)
    \figwritegce 1 :{$\theta$}(16,4)
    \figwritegcw 6 :{$\rho$}(1,-2)
    \figsetmark{$\bullet$}
    \figwritegce 1 :{${\bf 0}$}(7,-2.1)
    }
\centerline{\box\figbox}
\caption{The broken guide $\Omega_{\theta}$ (here $\theta=\frac\pi6$). Cartesian and polar coordinates.}
\label{F:1}
\end{figure}

The operator $\Delta^\Dir_{\Omega_{\theta}}$ is a positive unbounded self-adjoint operator with domain
\[
   \Dom(\Delta^\Dir_{\Omega_{\theta}}) = \{\psi\in H^1_0(\Omega_\theta):\quad
   \Delta\psi\in L^2(\Omega_\theta)\}.
\]
When $\theta=\frac\pi2$, we simply have $\Dom(\Delta^\Dir_{\Omega_{\theta}}) = H^2\cap H^1_0(\Omega_\theta)$.
In contrast, when $\theta\in\left(0,\frac{\pi}{2}\right)$, the boundary of $\Omega_\theta$ is not smooth, it is polygonal. The presence of the non-convex corner with vertex ${\bf0}$ is the reason for the space $\Dom(\Delta^\Dir_{\Omega_{\theta}})$ to be distinct from $H^2\cap H^1_0(\Omega_\theta)$. Nevertheless this domain can be precisely characterized as follows. Let us introduce polar coordinates $(\rho,\varphi)$ centered at the origin, with $\varphi=0$ coinciding with the upper part $x_2=x_1\tan\theta$ of the boundary of $\Omega_\theta$. Let $\chi$ be a smooth radial cutoff function with support in the region $x_{1}\tan\theta<|x_{2}|$ and $\chi\equiv1$ in a neighborhood of the origin. We introduce the explicit \emph{singular function}
\begin{equation}
\label{eq:sing}
   \psi^{\theta}_\sing(x_1,x_2) = \chi(\rho)\, \rho^{\pi/\omega} \sin\frac{\pi\varphi}{\omega},\quad
   \mbox{with}\quad \omega = 2(\pi-\theta).
\end{equation}
Then there holds, see the classical references \cite{Kondratev67,Grisvard85}:
\begin{equation}
\label{eq:dom}
   \Dom(\Delta^\Dir_{\Omega_{\theta}}) = \left(H^2\cap H^1_0(\Omega_\theta)\right) \oplus [\psi^\theta_\sing]
\end{equation}
where $[\psi^\theta_\sing]$ denotes the space generated by $\psi^\theta_\sing$.

We denote by $\mu_{\Gui,n}(\theta)$ its $n$-th Rayleigh quotient, $n\ge1$ (here $\|\cdot\|$ is the $L^2$ norm on $\Omega_\theta$):
\begin{equation*}
   \mu_{\Gui,n}(\theta) = \inff_{
   \psi_1,\ldots,\psi_j\ \mbox{\footnotesize independent in}\ 
   H^1_0(\Omega_\theta)} 
   \ \ \sup_{\psi\in\mathrm{span}\{\psi_1,\ldots,\psi_j\}}
   \frac{\|\nabla\psi\|^2} {\|\psi\|^2} .
\end{equation*}
We gather in the following statement several important preliminary properties for the spectrum of $\Delta^\Dir_{\Omega_{\theta}}$. All these results are proved in the literature. We briefly indicate hereafter what are the main arguments of the proofs, and where details can be found.

\begin{prop}
\label{P:ess}
{\em (i)} If $\theta=\frac\pi2$, $\Delta^{\Dir}_{\Om_{\theta}}$ has no discrete spectrum. Its essential spectrum is the closed  interval $[1,+\infty)$. 

{\em (ii)} For any $\theta\in(0,\frac\pi2)$, the essential spectrum of  $\Delta^{\Dir}_{\Om_{\theta}}$ coincides with $[1,+\infty)$. 

{\em (iii)} For any $\theta\in(0,\frac\pi2)$, the discrete spectrum of $\Delta^{\Dir}_{\Om_{\theta}}$ is {\em nonempty} and {\em finite}. In other words, $\Delta^{\Dir}_{\Om_{\theta}}$ has at least one eigenvalue below $1$, but a finite number of them.

{\em (iv)} For any $\theta\in(0,\frac\pi2)$ and any eigenvalue in the discrete spectrum of $\Delta^{\Dir}_{\Om_{\theta}}$, the associated eigenvectors $\psi$ are {\em even} with respect to the horizontal axis: $\psi(x_1,-x_2)=\psi(x_1,x_2)$.

{\em (v)} For any $n\ge1$, the function $\theta\mapsto\mu_{\Gui,n}(\theta)$ is continuous and non decreasing on $(0,\frac\pi2)$.

{\em (vi)} For any $n\ge1$ and $\theta_0$ such that $\mu_{\Gui,n}(\theta_0)<1$, the function $\theta\mapsto\mu_{\Gui,n}(\theta)$ is strictly increasing on $(0,\theta_0]$.
\end{prop}

\begin{proof}
(i) is a clear consequence of the separation of variables in $\Omega_{\pi/2}=(-\pi,0)\times\R$.

(ii) is a consequence of the fact that outside a compact set, $\Omega_\theta$ is the union of two strips isometric to $(0,+\infty)\times(0,\pi)$. 

(iii) The fact that there are eigenvalues below the essential spectrum is known since \cite{ABGM91}. See also in \cite[\S4]{DaLafRa11} another proof based on a more general argument developed in \cite{Duclos95, Carron04, Duclos05} for waveguides with curvature. The fact that there is only a finite number of such eigenvalues is proved in \cite[\S5]{DaLafRa11} using a similar method as \cite[Theorem 2.1]{MoTr05}.

(iv) Since the domain and the operator are invariant by the symmetry $x_2\mapsto -x_2$, the eigenvectors are even of odd with respect to the horizontal axis. An argument of monotonicity for Dirichlet eigenvalues excludes the odd eigenvectors, see \cite[\S2.2]{DaLafRa11} for details.

(v) The Rayleigh quotients are non-decreasing functions of $\theta$ as a consequence of the previous point and a suitable change of variable which transform the operator $-\Delta$ in a domain depending on $\theta$ into an operator depending on $\theta$ on a fixed domain, see \cite[\S3]{DaLafRa11} for details.

(vi) If $\mu_{\Gui,n}(\theta_0)<1$, by points (v) and (ii), $\mu_{\Gui,n}(\theta)$ is an eigenvalue for all $\theta\in(0,\theta_0]$. The same proof as in point (v) then shows that $\mu_{\Gui,n}(\theta)$ depend in an analytic way from $\theta$ in $(0,\theta_0]$. In addition, anticipating the result of Theorem \ref{spectrumguide'}, we find that the function $\theta\mapsto\mu_{\Gui,n}(\theta)$ is not constant so that we deduce from (v) that it is strictly increasing where it is analytic.
\end{proof}

\subsubsection{Statement of the main result}
One of the main results of this paper is a complete asymptotic expansion\footnote{
By the notation $\lambda(\theta)\underset{\theta\to 0}{\sim}\sum_{j\ge0}c_{j}\theta^{j\rho}$ (with a positive $\rho$)
we mean that for any positive integer $J$ we have the estimate \\
\centerline{
$|\lambda(\theta)-\sum_{0\le j \le J}c_j\theta^{j\rho}|\le C_J \,\theta^{(J+1)\rho}$ \ for $\theta$ small enough.}}  of the eigenvalues  $\mu_{\Gui,n}(\theta)$ in powers of $\theta^{1/3}$. To state this result, we need the following notation: For $n\ge1$, let $z_{\A}(n)$ be the $n$-th zero of the inverse Airy function $\A(x)=\mathsf{Ai}(-x)$.

\begin{theorem}\label{spectrumguide'}
For all $N_{0}$, there exists $\theta_{0}>0$, such that for all $\theta\in(0,\theta_{0}]$, $\Delta^{\Dir}_{\Omega_{\theta}}$  has at least $N_{0}$ eigenvalues. These eigenvalues admit the expansions:
$$
   \mu_{\Gui,n}(\theta)\underset{\theta\to 0}{\sim}
   \sum_{j\ge0}\gamma^\Delta_{j,n}\theta^{j/3}
   \quad
   \mbox{with} \ \ \gamma^\Delta_{0,n}=\frac{1}{4}, \ \ \gamma^\Delta_{1,n}=0,
   \ \ \mbox{and}\ \ \gamma^\Delta_{2,n}=2(4\pi\sqrt{2})^{-2/3}z_{\A}(n)
$$
and the term of order $j=3$ is not zero. The corresponding eigenvectors have multi-scale expansions (see Section \ref{S631} for details).
\end{theorem}

\subsection{Related questions}\label{related}
In the small angle limit the vertical line $x_1=0$ appears as a right barrier for eigenmodes, cf.\ the computations in \cite[\S8]{DaLafRa11}. In a first approach, this can be explained by a one-dimensional approximation in the spirit of the Born-Oppenheimer approximation: It is obtained by replacing $-\dr^2_{x_2}$ in the expression of $\Delta^\Dir_{\Omega_\theta}$ by its lowest eigenvalue on each slice of $\Omega_\theta$ at fixed $x_1$. The effective potential thus obtained has a triangular well at $x_1=0$ (on the left) and a barrier on the right. That is why it is quite natural to study first a similar 1D model operator (see Section \ref{3}). The main interest is to exhibit for such a simple situation how the zeros of the Airy function come into play and how two distinct scales are necessary to describe eigenmodes. Moreover, as a by-product of our proofs, it turns out that the first two terms in the eigenvalue asymptotics for $\Delta^\Dir_{\Omega_\theta}$ and its 1D approximation coincide.

\begin{figure}[ht]
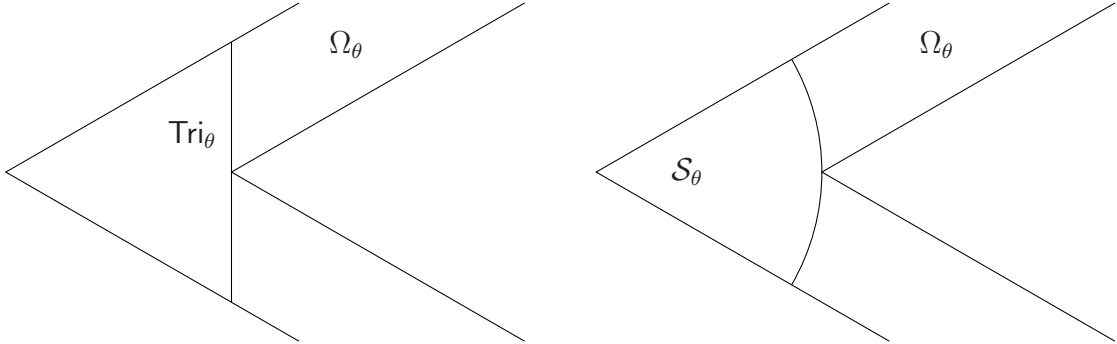

\begin{minipage}[c]{0.49\textwidth}
    \figinit{1mm}
    \figpt 101:(  0,  0)
    \figpt 102:( 45,  0)
    \figptrot 103: = 102 /101, 30/
    \figpt 104:(-30,  0)
    \figpt 105:(12, 0)
    \figpt 106:(0, 10)
    \figvectP 201 [101,104]
    \figptstra 114 = 103/1, 201/
    \figptssym 123 = 103, 114 /101, 102/
    \figvectP 206 [101,106]
    \figvectP 211 [104,114]
    \figvectP 221 [104,124]
    \figptinterlines 117:[101,206;104,211]
    \figptinterlines 127:[101,206;104,221]

    \def\MyPSfile{}
    \psbeginfig{}
    \psline[123,101,103]
    \psline[124,104,114]
    \psline[117,127]
  \psendfig

    \figvisu{\figbox}{}{%
    \figinsert{\MyPSfile}
    \figwritegce 101 :{$\Omega_{\theta}$}(13,17)
    \figwritegcw 101 :{$\Tri_{\theta}$}(2,5)
    }
\centerline{\box\figbox}
\end{minipage}
\begin{minipage}[c]{0.49\textwidth}
    \figinit{1mm}
    \figpt 101:(  0,  0)
    \figpt 102:( 45,  0)
    \figptrot 103: = 102 /101, 30/
    \figpt 104:(-30,  0)
    \figpt 105:(12, 0)
    \figpt 106:(0, 10)
    \figvectP 201 [101,104]
    \figptstra 114 = 103/1, 201/
    \figptssym 123 = 103, 114 /101, 102/
    \figvectP 206 [101,106]
    \figvectP 211 [104,114]
    \figvectP 221 [104,124]
    \figptinterlines 117:[101,206;104,211]
    \figptinterlines 127:[101,206;104,221]

    \def\MyPSfile{}
    \psbeginfig{}
    \psline[123,101,103]
    \psline[124,104,114]
    \psarccirc 104 ; 30 (-30,30)
  \psendfig

    \figvisu{\figbox}{}{%
    \figinsert{\MyPSfile}
    \figwritegce 104 :{$\mathcal{S}_\theta$}(10,0)
    \figwritegce 101 :{$\Omega_{\theta}$}(13,17)
    }
\centerline{\box\figbox}
\end{minipage}

\caption{Broken guide $\Omega_{\theta}$ with associated triangle $\Tri_\theta$ and sector $\mathcal{S}_\theta$.}
\label{F:1b}
\end{figure}

In the proof of Theorem \ref{spectrumguide'}, we will have to perform an accurate analysis of the spectral gap separating the eigenvalues. This gap, as stated in Theorem \ref{spectrumguide'}, is of order $\theta^{2/3}$ and is related to the difference between zeros of the Airy function (in other words the gap is determined as soon as we have proved a two-term expansion). In order to succeed in the investigation, we will have to estimate this gap by comparing with a simpler spectral problem. Here we have to choose between several possibilities: 
Either we could compare with the spectrum of the isosceles triangle $\Tri_\theta$ (with Dirichlet conditions), or we could compare with the spectrum of the sector $\mathcal{S}_\theta$ (with Dirichlet conditions), see Fig.~\ref{F:1b}. The case of the sector is well-known in the small angle limit (we find that the first two terms in the expansion of the eigenvalues coincide with that of $\mu_{\Gui,n}(\theta)$, see \cite{EL98, Fre07}). Nevertheless, we have preferred to analyze the problem of the triangle which is less known and which has an interest on its own (see \cite{Fre07}). In addition, as it will be seen in the analysis, the reduction to the triangle (through estimates of Agmon type) is slightly easier. A posteriori, the first two terms of the eigenvalues are the same as for the sector. Finally, the option to provide a full investigation of the triangle permits to divide difficulties inherent to each problem. This pedagogic perspective is also one of the motivations to study a 1D model operator which roughly describes the spectral behavior of the waveguide.

\subsection{Organization of the paper} 
In Section \ref{2}, we discuss the different reductions to simplified operators and introduce the main notation used in this paper. We state all our results related to eigenvalue asymptotics. In Section \ref{3} we investigate through a construction of quasimodes and an ODE analysis the one dimensional toy model $-\kappa^2\partial^2_z+W$ with the discontinuous triangular potential $W$ equal to $-z$ when $z\le0$ and $1$ when $z>0$. In Section \ref{4} we study a one dimensional approximation of the Dirichlet problem on a triangle with small angle. By Agmon estimates and a projection method, this leads in Section \ref{5} to results on triangles in the small angle limit.
Finally, in Section \ref{6}, we perform a construction of quasimodes adapted to  waveguides and introduce in particular Dirichlet-to-Neumann operators to solve a transmission problem; we complete the proof by comparing with the triangle case. We conclude our paper by discussing relations between the eigenvector asymptotics and the reentrant corner singularity. We also discuss the extension of our results on X-shaped waveguides (crossing straight wires).

\section{Reductions}\label{2}
This section is devoted to the introduction of reduced and simplified operators that we will consider throughout this paper. First we will use the symmetry of the waveguide to reduce the investigation to an half-guide. This first simplification makes a discontinuity in boundary conditions to appear at the origin $\bf 0$ (see Figure \ref{F:2}). In fact, as will be seen later, this jump in boundary conditions  traps the eigenfunctions, which are localized in the left part of the guide. Due to this localization, it makes sense to tackle the Dirichlet Laplacian on triangles $\Tri_\theta$. We also introduce a 1D approximation of Born-Oppenheimer type for the guides and the triangles. This helps to understand the concentration of eigenfunctions near the origin. Finally we state our results concerning eigenvalue asymptotics for all these model problems.

\subsection{Half-guide and triangles}

\subsubsection{The half-guide}
As a consequence of the parity properties of the eigenvectors of $\Delta^{\Dir}_{\Om_{\theta}}$, cf. point (iv) of Proposition \ref{P:ess}, we can reduce the spectral problem to the  half-guide
\begin{equation}
\label{E:Omthetaplus}
   \Omega_{\theta}^+=\left\{(x_{1},x_{2})\in \Omega_\theta : \ 
   x_{2}>0\right\}.
\end{equation}
We define the Dirichlet part of the boundary by
$\partial_\Dir\Omega^+_{\theta} = 
\partial\Omega_{\theta} \cap \partial\Omega^+_{\theta}$,
and the corresponding variational space (the form domain)
\[
   H^1_\Mix(\Omega^+_{\theta}) = \big\{\psi\in H^1(\Omega^+_{\theta}):\quad 
   \psi=0 \ \mbox{ on } \ \partial_\Dir\Omega^+_{\theta} \big\}.
\]
Then the new operator of interest, denoted by $\Delta^\Mix_{\Omega_{\theta}^+}$, is the Laplacian with mixed Dirichlet-Neumann conditions on $\Omega_{\theta}^+$. Its domain is:
\[
  \Dom(\Delta^\Mix_{\Omega_{\theta}^+}) =
  \big\{ \psi\in H^1_\Mix(\Omega_{\theta}^+) : \ \ \Delta\psi\in L^2(\Omega_{\theta}^+)
  \ \ \mbox{and}\ \  \partial_{2}\psi=0 \ \mbox{ on }\ x_{2}=0 \big\}.
\]
Then the operators $\Delta^{\Dir}_{\Om_{\theta}}$ and $\Delta^\Mix_{\Omega_{\theta}^+}$ have the same eigenvalues below $1$ and the eigenvectors of the latter are the restriction to $\Omega_{\theta}^+$ of the former.

\subsubsection{Rescaling of the half-guide}
In order to analyze the asymptotics $\theta\to0$, it is useful to rescale the integration domain and transfer the dependence on $\theta$ into the coefficients of the operator. For this reason, let us perform the following linear change of coordinates:
\begin{equation}
\label{E:xy}
   x=x_{1}\sqrt{2}\sin\theta, \quad y=x_{2}\sqrt{2}\cos\theta,
\end{equation}
which maps $\Omega^+_\theta$ onto $\Omega^+_{\pi/4}$ which will serve as reference domain, see Fig.\ \ref{F:2}. That is why we set for simplicity
\begin{equation}
\label{E:Omega}
   \Omega := \Omega^+_{\pi/4}\,, \ \ 
   \partial_\Dir\Omega = \partial_\Dir \Omega^+_{\pi/4}\,, \ \ \mbox{and}\ \ 
   H^1_\Mix(\Omega) = \big\{\psi\in H^1(\Omega): 
   \psi=0 \ \mbox{ on } \ \partial_\Dir\Omega \big\}.
\end{equation}

\begin{figure}[ht]
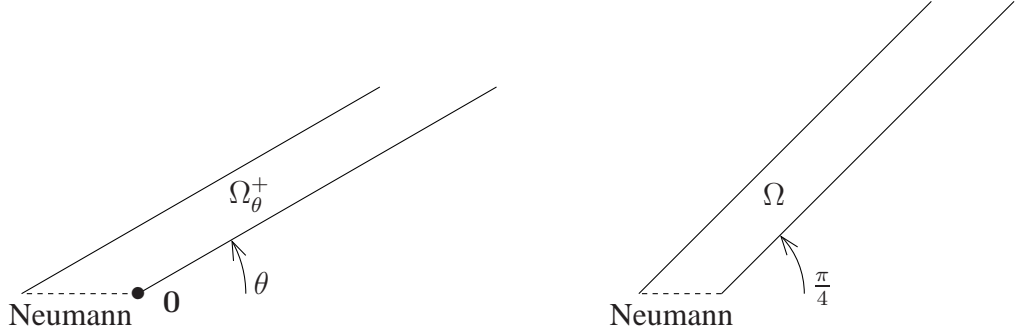

    \figinit{1.1mm}
    \figpt 1:(  0,  0)
    \figpt 2:( 50,  0)
    \figptrot 3: = 2 /1, 30/
    \figpt 4:(-14.1,  0)
    \figvectP 101 [1,4]
    \figptstra 14 = 3/1, 101/
    \figpt 31:(  0,  0)
    \figpt 32:( 50,  0)
    \figptrot 33: = 32 /31, 45/
    \figpt 34:(-10,  0)
    \figvectP 301 [31,34]
    \figptstra 44 = 33/1, 301/
    \figptstra 50 = 31, 33, 34, 44/-5, 101/

    \def\MyPSfile{}
    \psbeginfig{}
    \psline[1,3]
    \psline[4,14]
    \psline[50,51]
    \psline[52,53]
    \psarrowcirc 1 ; 13 (0,30)
    \psarrowcirc 50 ; 10 (0,45)
    \psset (dash=4)
    \psline[1,4]
    \psline[50,52]
  \psendfig

    \figvisu{\figbox}{}{%
    \figwritebe 1 :{$\theta$}(14)
    \figwritegce 1 :{$\Omega_{\theta}^+$}(11,12)
    \figwritebe 50 :{$\frac\pi4$}(11)
    \figwritegce 50 :{$\Omega$}(5,12)
    \figwritegce 4 :{Neumann}(-1.5,-2.5)
    \figwritegce 52 :{Neumann}(-3.5,-2.5)
    \figsetmark{$\bullet$}
    \figwritegce 1 :{${\bf 0}$}(3,-1.)
    }
\centerline{\box\figbox}
\caption{The half-guide $\Omega_{\theta}^+$ for $\theta=\frac\pi6$ and the reference domain $\Omega$.}
\label{F:2}
\end{figure}

Then, $\Delta^{\Mix}_{\Omega_{\theta}^+}$ is unitarily equivalent to the operator defined on $\Omega$ by:
\begin{equation}
\label{E:DGui}
 \D_\Gui(\theta) := -2\sin^2\!\theta\,\dr_{x}^2-2\cos^2\!\theta\,\dr_{y}^2,
\end{equation}
with Neumann condition on $y=0$ and Dirichlet everywhere else on the boundary of $\Om$. 
We let $h=\tan\theta$ ; after a division by $2\cos^2\theta$, we get the new operator:
\begin{equation}
\label{E:LGui}
   \L_{\Gui}(h) = -h^2\dr_{x}^2-\dr_{y}^2,
\end{equation}
with domain:
\[
  \Dom(\L_{\Gui}(h))=\big\{\psi\in H^1_\Mix(\Om) :  \ \ \L_{\Gui}(h)\psi\in L^2(\Om)
  \ \ \mbox{and}\ \  
  \dr_{y}\psi=0 \ \mbox{ on } \ y=0 \big\}.
\]

\subsubsection{The triangles}
\label{S133}
We will also need to introduce the triangular end of this waveguide:
\begin{equation}
\label{E:Tritheta}
   \Tri_{\theta}=\left\{(x_{1},x_{2})\in\Omega_\theta:\  x_1<0\right\}
\end{equation}
and the corresponding Dirichlet Laplacian denoted by $\Delta^{\Dir}_{\Tri_{\theta}}$.

Prior to the investigation of $\L_{\Gui}(h)$, we are to going to study $\L_{\Tri}(h)$ which denotes the same operator $-h^2\dr_{x}^2-\dr_{y}^2$ with Dirichlet conditions on the triangular end $\Tri$ of the model waveguide $\Omega_{\pi/4}$ 
\begin{equation}
\label{E:Tri}
   \Tri=\left\{(x,y)\in\R^2 : -\pi\sqrt{2}<x<0 \mbox{ and } |y|<x+\pi\sqrt{2} \right\} .
\end{equation}

\subsection{Born-Oppenheimer approximation and models}
As mentioned at the beginning of this section, we will use a projection method to analyze $\L_{\Gui}(h)$. This method is based on the original idea of Born and Oppenheimer (see \cite{BO27}) which was used to study the Hamiltonian of molecules (see \cite{CDS81, Martinez89, KMSW92}). By analogy with this situation, we can say that, in this paper, $x$ plays the role of the nuclei variables whereas $y$ plays the role of the electrons and where $h$ would represent a mass ratio. The variable $x$ is sometimes said to be the slow variable and $y$ the fast variable.
Therefore we will broaden the "molecular idea" to our waveguide situation.

\subsubsection{Schr\"{o}dinger operators in one dimension}
In the analysis of $\L_{\Tri}(h)$ and $\L_{\Gui}(h)$, we will see that its so-called Born-Oppenheimer approximation will play an important role:
\begin{subequations}
\label{HG}
\begin{equation}
\mathcal{H}_{\BO, \Gui}(h)=-h^2\dr_{x}^2+V(x),
\end{equation}
where 
\begin{equation}
V(x)=\left\{\begin{array}{cl}
\di\frac{\pi^2}{4(x+\pi\sqrt{2})^2}&\mbox{ when } x\in(-\pi\sqrt{2},0),\\[2.5ex]
\di\frac{1}{2}&\mbox{ when } x\geq 0.
\end{array}\right.\end{equation}
\end{subequations}

\begin{remark}
This \enquote{approximation} will be justified afterwards and will not be directly used in the investigation. Nevertheless it already gives a non trivial insight of some fine structures appearing in the analysis (such as the different scalings and the transmission problem between the left and right parts of the waveguide).
\end{remark}

This effective potential $V$ is obtained by replacing $-\dr^2_{y}$ in the expression of $\L_{\Gui}(h)$ by its lowest eigenvalue on each slice of $\Omega$ at fixed $x$.
When $h$ goes to zero, the behavior of the ground eigenpairs of $\mathcal{H}_{\BO, \Gui}(h)$ is driven by the structure of the potential near its minimum, attained at $x=0$: In a neighborhood of $x=0$, $V$ can be approximated by its left and right tangents, which provides the approximate potential $V_\app$ defined by 
\begin{equation}\label{Vapp}
V_\app(x)=\left\{\begin{array}{cl}
\di \frac18-\frac{1}{4\pi\sqrt{2}}\,x&\mbox{ when } x\in(-\pi\sqrt{2},0),\\[2.5ex]
\di\frac{1}{2}&\mbox{ when } x\geq 0.
\end{array}\right.
\end{equation}
After the change of variables $z=\sqrt{2}x/(3\pi)$ and the change of parameter $\kappa=4h/(3\pi\sqrt{3})$, we find the correspondence
\begin{equation}
\label{1toy}
   -h^2\dr_{x}^2+V_\app(x) \sim \frac38\mathcal{H}_{\toy}(\kappa)[z;\partial_z]+\frac18
\end{equation}
where the toy model operator $\mathcal{H}_{\toy}(\kappa)[z;\partial_z]$ is defined as:
\begin{equation}\label{toy}
   \mathcal{H}_{\toy}(\kappa)=-\kappa^2\dr_{z}^2+W(z)
   \quad\mbox{with}\quad
   W(z)=\left\{\begin{array}{cl}
  -z&\mbox{ when } z \le 0, \\
  1&\mbox{ when } z \geq 0.
\end{array}\right.
\end{equation}
This toy model invites us to recall the properties of the Airy operator.

\subsubsection{The Airy function and its zeros}
Let us recall the basic properties of the Airy operator, i.e.\ the Dirichlet realization on $L^2(\mathbb{R}_{-})$ of the operator $-\dr_{z}^2-z$. The electric potential tending to infinity when $z\to-\infty$, this positive operator has compact resolvent. Thus, its spectrum can be described as an increasing sequence of eigenvalues tending to $+\infty$. Let us use the traditional notation $\mathsf{Ai}$ for the Airy function. We recall that it satisfies: 
$$-\mathsf{Ai}''+z\mathsf{Ai}=0.$$
All along this paper, we will use $\A$ the reverse Airy function, i.e. $\A(z)=\mathsf{Ai}(-z)$. We recall that $\A$ does not vanish on $\mathbb{R}_{-}$, is exponentially decreasing when $z\to-\infty$ and that its zeros (which are simple) form an increasing sequence of positive numbers tending to $+\infty$.

\begin{notation}
 The $n$-th zero of $\A$ are denoted by $z_{\A}(n)$.
\end{notation}

\noindent If $(\la,\psi_\la)$ is an eigenpair of the Airy operator, we have
$-\psi_{\la}''-z\psi_{\la}=\la\psi_{\la}$, hence the equation $-\psi_{\la}''-(z+\la)\psi_{\la}=0$. We deduce that there exists a number $c(\la)$ so that:
$$\psi_{\la}(z)=c(\la)\A(z+\la).$$
With those remarks, we can see that the spectrum of the Airy operator is $\{z_{\A}(n),n\geq1\}$ and these eigenvalues are simple.

\subsubsection{Born-Oppenheimer approximation on the triangle} 
Finally, let us introduce the Dirichlet realization on $L^2((-\pi\sqrt{2},0))$ of:
\begin{equation}\label{HT}
\mathcal{H}_{\BO, \Tri}(h)=-h^2\dr_{x}^2+\frac{\pi^2}{4(x+\pi\sqrt{2})^2}\,.
\end{equation}
This operator is the Born-Oppenheimer \enquote{approximation} of the operator $\L_\Tri(h)$ on the triangle $\Tri$. The proof that it is actually an approximation will be done in Subsection \ref{Reduction-to-BO} through the Feshbach projection: Indeed the operator $\mathcal{H}_{\BO, \Tri}(h)$ has the same two-term eigenvalue asymptotics as the operator $\mathcal{L}_{\Tri}(h)$ on the triangle.

\subsection{Asymptotic expansions of eigenvalues}
We are now in position to state the results on eigenvalue expansion that we have proved in this paper.

\subsubsection{One-dimensional models} The lowest eigenvalues of the toy model \eqref{toy} admit analytic expansions with respect to $\kappa^{1/3}$ (when $\kappa$ is small enough):

\begin{theorem}\label{spectrumtoy}
For all $N_{0}\in\mathbb{N}$, there exists $\kappa_{0}>0$ such that, for $\kappa\in(0,\kappa_{0})$, there exists at least $N_{0}$ eigenvalues of $\mathcal{H}_{\toy}(\kappa)$ below $1$. Denoting by $\la_{\toy,n}(\kappa)$ the increasing sequence of these eigenvalues, we have the converging expansions for $1\leq n\leq N_{0}$ and $\kappa$ small enough:
$$
   \la_{\toy,n}(\kappa)=\kappa^{2/3}\sum_{j=0}^{+\infty}\alpha_{j,n}\kappa^{j/3}
   \quad
   \mbox{with first coefficient} \ \ \alpha_{0,n}=z_{\A}(n).
$$
The corresponding eigenvectors have expansions in powers of $\kappa^{1/3}$ with the scales $z/\kappa^{2/3}$ when $z<0$ and $z/h$ when $z>0$, see \eqref{E3:2b}-\eqref{E3:2c}.
\end{theorem}

As already mentioned we will meet in our investigation the Born-Oppenheimer approximations of $\mathcal{L}_{\Tri}(h)$ and $\mathcal{L}_{\Gui}(h)$. In order to compare the different asymptotics, let us state the result about the eigenvalues of $\mathcal{H}_{\BO, \Tri}(h)$: \begin{equation}\label{EspectrumBOT}
   \la_{\BO, \Tri,n}(h)\underset{h\to 0}{\sim}\sum_{j\geq 0}\hat{\beta}_{j,n}h^{2j/3}
   \ \mbox{with}\ \hat{\beta}_{0,n}=\frac{1}{8} \ \mbox{and}\ 
   \hat{\beta}_{1,n}=(4\pi\sqrt{2})^{-2/3}z_{\A}(n),
\end{equation}
and about the eigenvalues of $\mathcal{H}_{\BO, \Gui}(h)$:
\begin{equation}\label{spectrumBOG}
   \la_{\BO, \Gui,n}(h)\underset{h\to 0}{\sim}\sum_{j\ge0}\hat\gamma_{j,n}h^{j/3}
   \ \mbox{with} \ \hat\gamma_{0,n}=\frac{1}{8}, \ \hat\gamma_{1,n}=0,
   \ \mbox{and}\ \hat\gamma_{2,n}=(4\pi\sqrt{2})^{-2/3}z_{\A}(n).
\end{equation}
Let us point out that this latter estimate will not be used to prove our main theorem (see Theorem \ref{spectrumguide}) but somehow reflects that $\mathcal{H}_{\BO, \Gui}(h)$ is an approximation of $\mathcal{L}_{\Gui}(h)$.

\subsubsection{Triangles} 
The lowest eigenvalues of the triangle $\Tri_\theta$ admit expansions at any order in powers of $\theta^{1/3}$. We first state the result for the scaled operator $\mathcal{L}_{\Tri}(h)$ introduced in \S\ref{S133}:

\begin{theorem}\label{spectrumtriangle}
The eigenvalues of $\mathcal{L}_{\Tri}(h)$, denoted by $\la_{\Tri,n}(h)$, admit the expansions:
$$
   \la_{\Tri,n}(h)\underset{h\to 0}{\sim}\sum_{j\ge0}\beta_{j,n}h^{j/3}
   \quad
   \mbox{with} \ \ \beta_{0,n}=\frac{1}{8}, \ \ \beta_{1,n}=0,
   \ \ \mbox{and}\ \ \beta_{2,n}=(4\pi\sqrt{2})^{-2/3}z_{\A}(n),
$$
the terms of odd rank being zero for $j\leq 8$.
The corresponding eigenvectors have expansions in powers of $h^{1/3}$ with the two scales $x/h^{2/3}$ and $x/h$, see \eqref{quasi-eigenfunction-Tri}.
\end{theorem}

In terms of the physical domain $\Tri_\theta$, we deduce immediately from the previous theorem that the eigenvalues of $\Delta^{\Dir}_{\Tri_{\theta}}$, denoted by $\mu_{\Tri,n}(\theta)$, admit the expansions:
$$
   \mu_{\Tri,n}(\theta)\underset{h\to 0}{\sim}
   \sum_{j\ge0}\beta^\Delta_{j,n}\theta^{j/3}
   \quad
   \mbox{with} \ \ \beta^\Delta_{0,n}=\frac{1}{4}, \ \ \beta^\Delta_{1,n}=0,
   \ \ \mbox{and}\ \ \beta^\Delta_{2,n}=2(4\pi\sqrt{2})^{-2/3}z_{\A}(n),
$$
the coefficients $\beta^\Delta_{j,n}$ having the same properties as the $\beta_{j,n}$.
Performing the dilatation:
$$\tilde{x}_{1}=\sin2\theta\, x_{1} \quad \tilde{x}_{2}=\sin2\theta\, x_{2},$$
we transform $\Tri_\theta$ into a new isosceles triangle with angle $\alpha=2\theta$ and two sides with length  $c=2\pi$. Let us denote by $\mu_{\widetilde{\Tri},n}(\alpha)$ its Dirichlet eigenvalues. It is easy to see that the eigenvalues satisfy the relation:
$$\mu_{\Tri,n}(\theta)=(\sin\alpha)^2\mu_{\widetilde{\Tri},n}(\alpha),$$
so that we find back the result of \cite[Theorem 1]{Fre07}.

\begin{remark}
As it will be seen in the proof, the existence of a non-zero coefficient $\beta_{9,n}$ at the order $9$ in the expansion of $\la_{\Tri,n}(h)$ reduces to the evaluation of an integral, see \eqref{neq0}. If $\beta_{9,n}\neq0$, there is a nonzero odd term after $\OO(\alpha^{2/3})$ in the asymptotics of $\mu_{\widetilde{\Tri},1}(\alpha)$.
\end{remark}

\subsubsection{Broken guides}
Finally, we state the approximation result for the eigenvalues of the scaled operator $\mathcal{L}_{\Gui}(h)$ introduced in \eqref{E:LGui}:

\begin{theorem}\label{spectrumguide}
For all $N_{0}$, there exists $h_{0}>0$, such that for $h\in(0,h_{0})$ the $N_{0}$ first eigenvalues of $\mathcal{L}_{\Gui}(h)$ exist. These eigenvalues, denoted by $\la_{\Gui,n}(h)$, admit the expansions:
$$
   \la_{\Gui,n}(h)\underset{h\to 0}{\sim}\sum_{j\ge0}\gamma_{j,n}h^{j/3}
   \quad
   \mbox{with} \ \ \gamma_{0,n}=\frac{1}{8}, \ \ \gamma_{1,n}=0,
   \ \ \mbox{and}\ \ \gamma_{2,n}=(4\pi\sqrt{2})^{-2/3}z_{\A}(n)
$$
and the term of order $h$ is not zero. The corresponding eigenvectors have expansions in powers of $h^{1/3}$ with the scale $x/h$ when $x>0$, and both scales $x/h^{2/3}$ and $x/h$ when $x<0$, see \eqref{QuasiLgui}.
\end{theorem}

Deducing the eigenvalues in the waveguide $\Omega_\theta$ (Theorem \ref{spectrumguide'}) is an obvious consequence of this theorem.

\subsection{Notation and terminology} The $L^2$ norm will always be denoted by 
$\|\cdot\|$, in general without mention of the integration domain.
For a subset $S\subset\R$ and a point $p\in\R$, $\dist(S,p)$ is the distance between $S$ and $p$, i.e.\ $\inf_{s\in S}|s-p|$.

We denote by  $\gS(A)$ the spectrum of a self-adjoint operator $A$, by $\gS_\ess(A)$ its essential spectrum, and by $\gS_\dis(A)$ its discrete spectrum. An eigenmode (or eigenpair) of $A$ is a pair $(\lambda,\psi)$ with $\psi$ in the domain of $A$, such that $A\psi=\lambda\psi$; then $\lambda$ is the eigenvalue and $\psi$ the eigenvector. A quasimode for $A$ is a pair $(\tilde\lambda,\tilde\psi)$ such that $\|A \tilde\psi-\tilde\lambda \tilde\psi\|\le \varepsilon \|\tilde\psi\|$ with $\varepsilon$ small; $\lambda$ is the quasi-eigenvalue and $\psi$ the quasi-eigenvector. The spectral theorem implies that $\dist(\gS(A),\tilde\lambda)\le\varepsilon$.

\begin{table}[h]
\def\arraystretch{1.3}
\begin{tabular}{l|l|l|l}
Domain & Notation & Variables & Main operators \\
\hline
Scaled Triangle & $\Tri$ \ \eqref{E:Tri} &  $(x,y)$ \ \eqref{E:xy} 
   & $\L_\Tri(h)=-h^2\partial^2_x-\partial^2_y$ \\
Rectangle & $\Rec$ \ \eqref{E:Rec} & $(u,t)$ \ \eqref{Eut}& $\L_\Rec(h)$  \ \eqref{E:LRec}\\
Half-strip & $\Hst=\R_-\times(-1,1)$ & $(s,t)$ \ \eqref{E:ssigma2}
   & $\sum_{j} \L_{2j}h^{2j/3}$ \ \eqref{E:Lj} \\
   &  &  $(\sigma,t)$  \ \eqref{E:ssigma2}
   &  $\sum_{j} \cN_{3j}h^{j}$  \ \eqref{E:Nj} \\
\hline
Scaled half-guide   & $\Omega$ \ \eqref{E:Omega} & $(x,y)$  \ \eqref{E:xy} 
   & $\L_\Gui(h)=-h^2\partial^2_x-\partial^2_y$ \\
Left half-strip & $\Stlef=\R_-\times(0,1)$ &  $(s,t)$ 
   & $\sum_{j} \L_{2j}h^{2j/3}$ \ Notation \ref{6not}\\
   &  & $(\sigma,t)$ 
   & $\sum_{j} \cN^\lef_{3j}h^{j}$ \ Notation \ref{6not} \\
Right half-strip & $\Stri=\R_+\times(0,1)$ &  $(\sigma,\tau)$ \eqref{E:utau}
   & $\sum_{j} \cN^\ri_{3j}h^{j}$ \ Notation \ref{6not} \\
\hline
\end{tabular}
\bigskip
\caption{Main notation for domains, variables and operators.}
\end{table}

\section{Toy model in one dimension}
\label{3}
This subsection is devoted to the proof of Theorem \ref{spectrumtoy} devoted to the spectral asymptotics of the operator $\mathcal{H}_{\toy}(\kappa)$ defined in \eqref{toy}. This proof is divided into two steps. First, we construct quasimodes for $\mathcal{H}_{\toy}(\kappa)$, and second, we show that the lowest quasi-eigenvalues are the approximations of the lowest eigenvalues of $\mathcal{H}_{\toy}(\kappa)$ of the same rank.

\subsection{Construction of quasimodes}
In this section we prove in particular the following:

\begin{prop}\label{quasitoy}
For all $N_{0}\in\mathbb{N}^*$, there exists $\kappa_{0}>0$ and $C>0$ such that for $\kappa\in(0,\kappa_{0})$:
\begin{equation}
\label{E3:1}
   \dist \big(\gS_\dis(\mathcal{H}_{\toy}(\kappa)),\,
   \kappa^{2/3}z_{\A}(n) \big) \leq C\kappa, \quad n=1,\cdots N_0.
\end{equation}
\end{prop}

\begin{proof}
The basic tool for the proof is the construction of quasimodes and the application of the spectral theorem. Convenient quasimodes are given by power series in $\kappa^{1/3}$ of {\em profiles} at the scales
\begin{equation}
\label{E:ssigma1}
   s=\kappa^{-2/3}z\ \ \mbox{when} \ \  z\le0 \ \ \mbox{(left)}
   \quad\mbox{and}\quad \sigma=\kappa^{-1}z\ \ \mbox{when} \ \  z\ge0\ \ \mbox{(right)}.
\end{equation}
More precisely we look for quasi-eigenfunctions $\psi_{\kappa}$ in the form:
\begin{equation}
\label{E3:2}
   \psi_{\kappa}(z) \sim 
   \begin{cases}
   \ \sum_{j\ge0} \Psi_{\lef,j}(s)\, \kappa^{j/3}
   \ \ &\mbox{when} \ \  z\le0 \\[0.8ex]
   \ \sum_{j\ge0} \Phi_{\ri,j}(\sigma)\, \kappa^{j/3}
   \ \ &\mbox{when} \ \  z\ge0 \,,
   \end{cases}
\end{equation}
and quasi-eigenvalues in the form:
\begin{equation}
\label{E3:3}
   \alpha_\kappa\sim\kappa^{2/3}\sum_{j\ge0} \alpha_{j} \kappa^{j/3}
   \quad\mbox{as}\quad\kappa\to0.
\end{equation}
The continuity conditions at $z=0$ provide the formal identities:
\begin{equation}
\label{E3:4}
\begin{cases}
\begin{array}{ccc}
   \sum_{j\ge0} \Psi_{\lef,j}(0)\, \kappa^{j/3} & 
   = &\sum_{j\ge0} \Phi_{\ri,j}(0)\, \kappa^{j/3} 
   \\[0.8ex]
   \kappa^{-2/3}\sum_{j\ge0} \partial_s\Psi_{\lef,j}(0)\, \kappa^{j/3} & 
   = & \kappa^{-1}\sum_{j\ge0} \partial_\sigma\Phi_{\ri,j}(0)\, \kappa^{j/3},
\end{array}
\end{cases}
\end{equation}
and the formal eigen-equation is
\begin{equation}
\label{E3:5}
   -\kappa^2\psi''_\kappa(z) + W(z)\psi_\kappa(z) = \alpha_\kappa\psi_\kappa(z)
   \quad z\in\R.
\end{equation}

\paragraph{Determination of $\alpha_{0}$}
Collecting the terms in $\kappa^{2/3}$ in \eqref{E3:5} and using \eqref{E3:2}-\eqref{E3:4} we obtain:
\[
\begin{cases}
\begin{array}{cll}
   -\Phi''_{\ri,0}(\sigma)+\Phi_{\ri,0}(\sigma)=0 
   & \mbox{for }\ \sigma>0, 
   & \ \mbox{and} \quad \Phi_{\ri,0}'(0) = 0,
   \\[0.8ex]
   -\Psi_{\lef,0}''(s)-s\Psi_{\lef,0}(s)=\alpha_{0}\Psi_{\lef,0}(s) 
   & \mbox{for }\  s<0, 
   & \  \mbox{and}\quad \Psi_{\lef,0}(0)=\Phi_{\ri,0}(0).
\end{array}
\end{cases}
\]
We deduce first that $\Phi_{\ri,0}=0$ and thus $\Psi_{\lef,0}(0)=0$. This implies that $\alpha_0$ is a zero of the reverse Airy function $\A$. At this stage we can choose a positive integer $n$, take $\alpha_{0}=z_{\A}(n)$ and $\Psi_{\lef,0}$ as the corresponding normalized eigenfunction $g_{(n)}$.

\paragraph{Determination of $\alpha_{1}$}
Collecting the terms in $\kappa$, we get the equations:
\[
\begin{cases}
\begin{array}{cll}
   -\Phi''_{\ri,1}+\Phi_{\ri,1}=0 
   & \mbox{for } \sigma>0, 
   & \mbox{and }\ \Phi_{\ri,1}'(0) = \Psi_{\lef,0}'(0),
   \\[0.8ex]
   -\Psi_{\lef,1}''-s\Psi_{\lef,1}-\alpha_{0} \Psi_{\lef,1}
   =\alpha_{1}\Psi_{\lef,0} 
   & \mbox{for }  s<0, 
   & \mbox{and }\ \Psi_{\lef,1}(0)=\Phi_{\ri,1}(0).
\end{array}
\end{cases}
\]
We find first:
$$\Phi_{\ri,1}(\sigma)=-\Psi'_{\lef,0}(0)e^{-\sigma}.$$
Moreover we obtain the existence of a number $\alpha_1$ and of an exponentially decreasing $\Psi_{\lef,1}$ solution of the second equation with the help of the following lemma:
\begin{lem}\label{lem-Ai0}
Let $n\geq 1$. We denote by $g_{(n)}$ an eigenvector of the operator $-\dr_{s}^2-s$ associated with the eigenvalue $z_{\A}(n)$ and normalized in $L^2(\R_-)$. Let $f=f(s)$ be a real function with exponential decay and let $c\in\R$. Then there exists a unique $\alpha\in\R$ such that the problem:
$$\left(-\dr_{s}^2-s-z_{\A}(n)\right)g=f+\alpha g_{(n)}\ \ \mbox{in}\ \ \R_-, \mbox{ with } g(0)=c,$$
has a solution with exponential decay. There holds
\[
   \alpha = c\,g'_{(n)}(0) - \int_{-\infty}^0 f(s)\,g_{(n)}(s)\,ds.
\]
\end{lem}

\paragraph{Further terms}
A similar procedure can be reproduced at each step, providing the construction of $\Phi_{\ri,j}$, then $\alpha_j$ and $\Psi_{\lef,j}$, for any $j\ge2$.

\paragraph{Expressions for quasimodes} 
Relying on the previous iterative constructions we can set for all integer $J\ge0$
\begin{equation}
\label{E3:2b}
   \psi_{\kappa}^{[J]}(z) = 
   \begin{cases}\di
   \ \sum_{j=0}^{J+2} \Psi_{\lef,j}\Big(\frac{z}{\kappa^{2/3}}\Big)\, \kappa^{j/3}
   \ \ &\mbox{when} \ \  z\le0 \\[0.8ex] \di
   \ \sum_{j=0}^{J+2} \Phi_{\ri,j}\Big(\frac{z}{\kappa}\Big)\, \kappa^{j/3}
   + \Psi'_{\lef,J+2}(0)\,\kappa^{J/3}z\,\chi\Big(\frac{z}{\kappa}\Big)
   \ \ &\mbox{when} \ \  z\ge0 \,,
   \end{cases}
\end{equation}
where $\chi$ is a smooth cutoff function equal to $1$ near $0$. 
By construction, $\psi_{\kappa}^{[J]}$ and its first derivative are continuous in $z=0$. Moreover $\psi_{\kappa}^{[J]}$  is exponentially decreasing as $z\to\pm\infty$.
Therefore it belongs to the domain of $\mathcal{H}_{\toy}(\kappa)$.
With this remark and taking the error introduced by $\chi$ into account, we get for all $\kappa_0>0$:
\begin{equation}
\label{E3:2c}
   \Big\|\Big(\mathcal{H}_{\toy}(\kappa)-
   \kappa^{2/3}\big(z_\A(n)+\sum_{j=1}^{J+2}\alpha_j\kappa^{J/3}\big)\Big)
   \psi_{\kappa}^{[J]}\Big\|\leq 
   C({J,n,\kappa_0})\,\kappa^{1+J/3},\quad \forall\kappa\le\kappa_0.
\end{equation}
Hence
$$\left\|\left(\mathcal{H}_{\toy}(\kappa)-\kappa^{2/3}z_\A(n)\right)\psi_{\kappa}\right\|\leq C(n,\kappa_0)\,\kappa,\quad \forall\kappa\le\kappa_0,$$
and the spectral theorem applies. In particular, for $\kappa$ small enough, the discrete spectrum of $\mathcal{H}_{\toy}(\kappa)$ is not empty since $\gS_{\ess}(\mathcal{H}_{\toy}(\kappa))=[1,+\infty)$.
\end{proof}

\begin{remark}
We have proved in fact more than Proposition \ref{quasitoy}. The expression \eqref{E3:2b} of quasimodes and corresponding estimates \eqref{E3:2c} will provide an asymptotic expansion for the eigenvectors of $\mathcal{H}_{\toy}(\kappa)$, once one knows Proposition \ref{toyfirstterm} below.
\end{remark}

\subsection{Localization of the lowest eigenvalues}
We now want to refine Proposition \ref{quasitoy} by proving that the $\la_{\toy,n}(\kappa)$ are power series with respect to $\kappa^{1/3}$ and whose coefficients are given by (\ref{E3:3}). 
We begin to prove the following proposition:

\begin{prop}\label{toyfirstterm}
For all $N_{0}\in\mathbb{N}^*$, there exists $\kappa_{0}>0$ and $C>0$ such that for $\kappa\in(0,\kappa_{0})$:
\begin{equation}
\label{E3:1b}
   |\la_{\toy,n}(\kappa)-\kappa^{2/3}z_{\A}(n)| \leq C\kappa, \quad n=1,\cdots N_0.
\end{equation}
\end{prop}

\begin{proof}
Let $N_{0}\in\mathbb{N}^*$.
As a consequence of Proposition \ref{quasitoy}, we have in particular that, for all $\kappa\in(0,\kappa_{0})$, the first $N_{0}$ eigenvalues $\la_{\toy,n}(\kappa)$ (denoted by $\lambda_n$ for shortness) exist and  satisfy:
\begin{equation}
\label{E3:8}
   |\lambda_n |\leq C(N_{0})\,\kappa^{2/3},
   \quad\kappa\in(0,\kappa_{0}),\quad n=1,\cdots N_0.
\end{equation}
Let us denote by $\psi_{n}$ an eigenfunction associated with $\la_{n}$ so that $\langle \psi_{n},\psi_{m}\rangle=0$ if $n\neq m$. For $z<0$ we have:
$$-\kappa^2\psi_n''-z\psi_n=\la_n\psi_n.$$
Thus, there exists a coefficient $c_n(\kappa)$ such that:
\begin{equation}
\label{E:toylef}
   \psi_n(z)=c_n(\kappa)\A(\kappa^{-2/3}z+\kappa^{-2/3}\la_n),\quad z<0.
\end{equation}
For $z>0$ we have the equation $-\kappa^2\psi_n''=\la_n\psi_n$, hence the existence of $d_n(\kappa)$ such that:
\begin{equation}
\label{E:toyrig}
   \psi_n(z)=d_n(\kappa)e^{-\kappa^{-1}z\sqrt{1-\la_n}},\quad z>0.
\end{equation}
The transmission conditions at $z=0$ imply:
$$c_n(\kappa)\A(\kappa^{-2/3}\la_n)=d_n(\kappa),\quad c_n(\kappa)\kappa^{1/3}\A'(\kappa^{-2/3}\la_n)=-d_n(\kappa)\sqrt{1-\la_n}.$$
This implies:
\begin{equation}\label{rel}
\A(\kappa^{-2/3}\la_n)=-\frac{\kappa^{1/3}}{\sqrt{1-\la_n}}\,\A'(\kappa^{-2/3}\la_n).
\end{equation}
We infer:
$$|\A(\kappa^{-2/3}\la_n)|\leq C(N_{0})\,\kappa^{1/3}.$$
Since $\kappa^{-2/3}\la_n$ is bounded, see \eqref{E3:8}, and the zeros of the Airy function being isolated and simple, we deduce that for all $n\in\{1,\cdots,N_{0}\}$, there exists $p=p(n,\kappa)$ such that:
$$|\kappa^{-2/3}\la_n-z_{\A}(p)|\leq C(N_{0})\kappa^{1/3}.$$
Note that $p$ is bounded too.
It remains to prove that $p=n$ for $\kappa$ small enough.
In view of Proposition \ref{quasitoy}, it suffices now to prove than if $\kappa$ is small enough and $n\neq m$ (with $n$, $m\le N_0$), the integers $p(n,\kappa)$ and $p(m,\kappa)$ are distinct. Let us prove this by contradiction. Since the considered sets of integers $n$, $m$ and $p$ are finite, the negation of what we want to prove can be written as
\[
   \exists\, m,n,p\in\N, \quad\forall\kappa_1>0,
   \quad\exists \kappa\in(0,\kappa_1)\quad\mbox{such that}\quad
   p(m,\kappa)=p(n,\kappa)=p.
\]
The eigenfunctions can be taken in the form:
\begin{align*}
   \psi_j(z)=
   \begin{cases}
   \ \A(\kappa^{-2/3}z+\kappa^{-2/3}\la_j) 
   \ \ &\mbox{when} \ \  z\le0 \\[0.8ex]
   \ \A(\kappa^{-2/3}\la_j)\, e^{-\kappa^{-1}z\sqrt{1-\la_j}}
   \ \ &\mbox{when} \ \  z\ge0 \,,
   \end{cases}\quad \mbox{for}\quad j=m,n,
\end{align*}
and we have
$$\langle\psi_{n},\psi_{m}\rangle=\int_{z<0} \,\A(\kappa^{-2/3}z+\kappa^{-2/3}\la_{n}) \A(\kappa^{-2/3}z+\kappa^{-2/3}\la_{m}) \,dz+\OO(\kappa^{5/3})=0.$$
A rescaling leads to:
$$\left|\int_{z<0} \A(z+\kappa^{-2/3}\la_{n})\, \A(z+\kappa^{-2/3}\la_{m}) \,dz\right|\leq C(N_{0})\kappa.$$
By assumption, $\kappa^{-2/3}\la_{n}=z_{\A}(p)+\OO(\kappa^{1/3})$ and $\kappa^{-2/3}\la_{m}=z_{\A}(p)+\OO(\kappa^{1/3})$.
For $j=n,m$, $\A$ being Lipschitz on $(-\infty,M]$ for all $M$, there exists $D(N_{0})>0$ such that for all $z<0$:
$$|\A(z+\kappa^{-2/3}\la_j) - \A(z+z_{\A}(p))|\leq D(N_{0})\kappa^{1/3},
\quad \mbox{for}\quad j=m,n,$$
so that:
$$\left|\int_{z<0} \A(z+\kappa^{-2/3}\la_{n}) \A(z+\kappa^{-2/3}\la_{m}) \,dz-
\int_{z<0} \A^2(z+z_A(p)) \,dz\right|\leq \tilde{D}(N_{0})\kappa^{1/3}.$$
We deduce:
$$
   \forall\kappa_1>0,
   \quad\exists \kappa\in(0,\kappa_1)\quad\mbox{such that}\quad
   \left|\int_{z<0} \A^2(z+z_{\A}(p)) \,dz\right|\leq \tilde{D}(N_{0})\kappa^{1/3}
$$
which leads to a contradiction and ends the proof of Proposition \ref{toyfirstterm}.
\end{proof}

\subsection{Proof of Theorem  \ref{spectrumtoy}}
Let us observe that Proposition \ref{toyfirstterm} allows to separate the first $N_0$ eigenvalues when $\kappa<\kappa_0$. 
Let us write $\delta=\kappa^{1/3}$. We let:
$$\breve{\la}_{n}(\delta):=\delta^{-2}\la_{\toy,n}(\delta^3),$$
so that $\breve{\la}_{n}(\delta)$ is uniformly bounded for $n=1,\ldots,N_0$ and $\delta<\kappa_0^{1/3}$.

\begin{figure}[ht]
\begin{center}
\includegraphics[keepaspectratio=true,width=8.5cm]{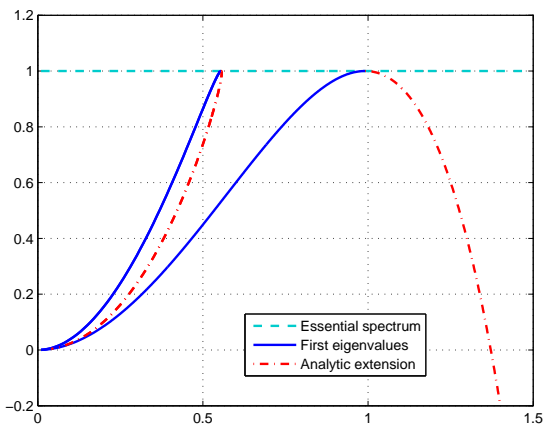}
\caption{The first two eigenvalues $\lambda_{\toy,1}$ and $\lambda_{\toy,2}$ as functions of $\delta=\kappa^{1/3}$. 
\label{Fmat1}}
\end{center}
\end{figure}

\begin{figure}[ht]
\begin{center}\hskip-4.em
\includegraphics[keepaspectratio=true,width=8.5cm]{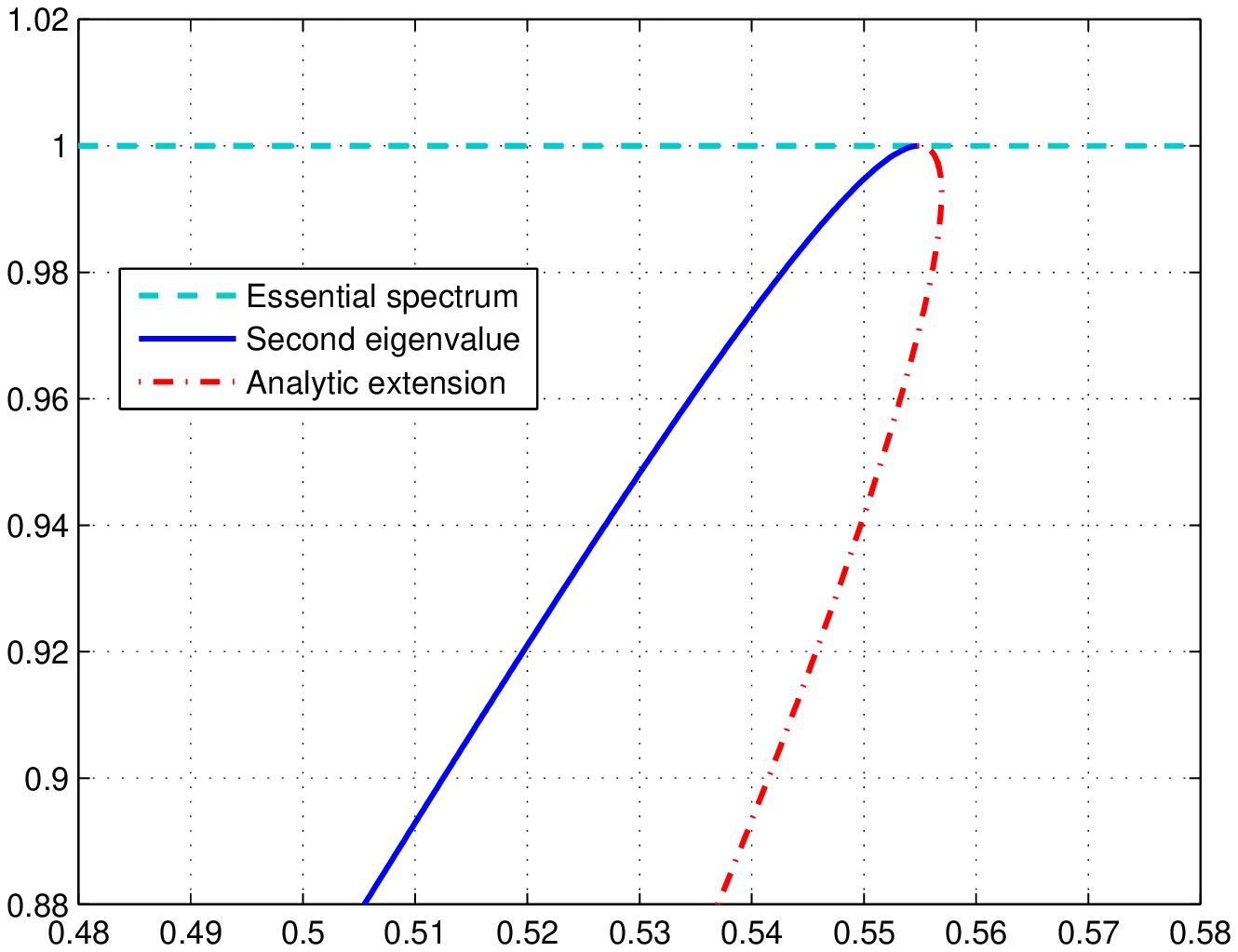} \hskip-1.7em
\includegraphics[keepaspectratio=true,width=8.5cm]{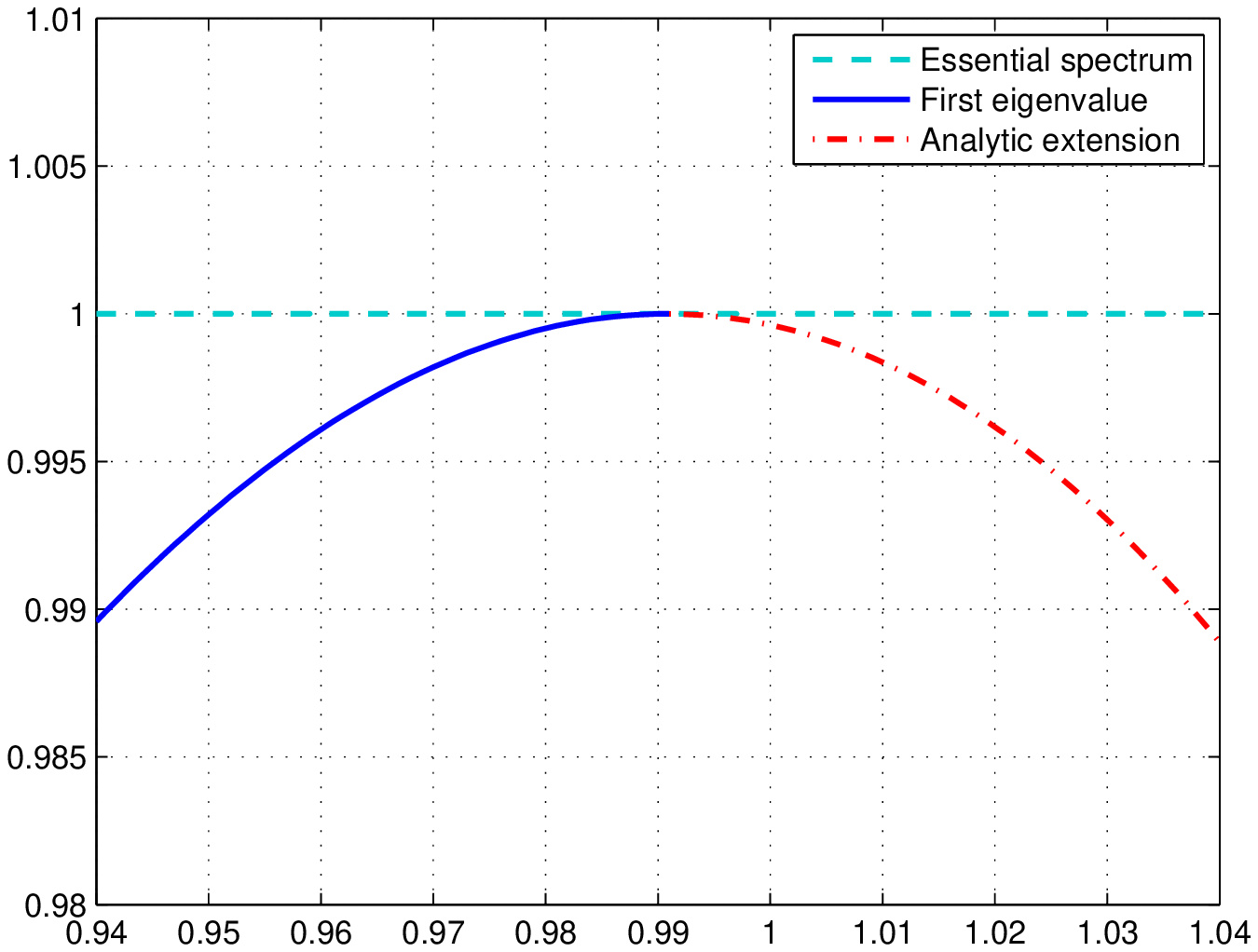}
\caption{The eigenvalues $\lambda_{\toy,1}$ (left) and $\lambda_{\toy,2}$ (right) as functions of $\delta=\kappa^{1/3}$, zoom near the bottom of the essential spectrum. 
\label{Fmat2}}
\end{center}
\end{figure}

\noindent
We rewrite \eqref{rel} in the form:
\begin{equation}\label{analyticdelta}
\A(\breve{\la}_{n}(\delta))=-\frac{\delta}{\sqrt{1-\delta^{2}\breve{\la}_{n}(\delta)}}\,\A'(\breve{\la}_{n}(\delta)).
\end{equation}
We know that $\A$ is analytic and, using again the simplicity of its zeros, we can apply the analytic implicit function theorem near $\delta=0$ and for all $n\in\{1,\cdots,N_{0}\}$, which, together with \eqref{E3:2b}-\eqref{E3:2c} and Proposition \ref{toyfirstterm}, ends the proof of Theorem \ref{spectrumtoy}.

\begin{remark}
From \eqref{analyticdelta}, we can deduce that the $\breve{\la}_{n}(\delta)$ are solutions of  the analytic equation:
\begin{equation}
\label{E:analyticdelta2}
(1-\delta^{2}\breve{\la})\A(\breve{\la})^2-\delta^2\,\A'(\breve{\la})^2=0
\end{equation}
This equation provides an analytic extension of the functions $\delta\mapsto\breve{\la}_{n}(\delta)$, hence of $\lambda_{\toy,n}=\delta^2\breve{\la}_{n}(\delta)$, in the sense of analytic curves. We represent in Figures \ref{Fmat1} and \ref{Fmat2} the first two eigenvalues and their analytic extensions. Taking the continuity and monotonicity of the eigenvalues with respect to $\delta$ into account, we can see that any branch which starts by $\delta\mapsto\lambda(\delta)=\delta^2 z_\A +  \OO(\delta^3)$ represents an eigenvalue while $\lambda(\delta)$ is less that $1$. Beyond $1$, the Rayleigh quotient stays $\equiv1$, but the curve $\lambda(\delta)$ has an analytic extension as a continuation of a branch of roots of  the equation \eqref{E:analyticdelta2}.
\end{remark}

\section{Born-Oppenheimer approximation for the triangle}
\label{4}
This section is devoted to the analysis of $\H_{\BO, \Tri}(h)$ defined in \eqref{HT}. We are going to prove:

\begin{theorem}\label{spectrumBOT}
The eigenvalues of $\mathcal{H}_{\BO, \Tri}(h)$, denoted by $\la_{\BO, \Tri,n}(h)$, admit the expansions:
$$\la_{\BO, \Tri,n}(h)\underset{h\to 0}{\sim}\sum_{j\geq 0}\hat{\beta}_{j,n}h^{2j/3},
\quad\mbox{with}\quad
\hat{\beta}_{0,n}=\frac{1}{8} \ \ \mbox{and}\ \ \hat{\beta}_{1,n}=(4\pi\sqrt{2})^{-2/3}z_{\A}(n).
$$
\end{theorem}

Again, the proof is essentially organized in two steps. The first step is the construction of quasimodes which proves that quasi-eigenvalues are close to true eigenvalues. The second step uses Agmon type exponential localization for true eigenvectors to prove that true eigenvalues are close to quasi-eigenvalues.

\subsection{Quasimodes}
In this subsection, we construct quasimodes and prove the proposition:
\begin{prop}\label{quasibotri}
For all $N_{0}\in\mathbb{N}^*$, there exists $h_{0}>0$ and $C>0$ such that for $h\in(0,h_{0})$:
\begin{equation}
   \dist \Big(\gS_\dis(\mathcal{H}_{\BO,\Tri}(h)),\,
   \frac{1}{8}+h^{2/3}(4\pi\sqrt{2})^{-2/3} z_{\A}(n) \Big) \leq Ch^{4/3}, \quad n=1,\cdots N_0.
\end{equation}
\end{prop}

\begin{proof}
The proper scale in $x$ is $h^{2/3}$ as can be seen by approximating the potential in $x=0$ by its tangent and recognizing the Airy operator. Thus, we will construct quasimodes  $\psi_h$ as functions of $s=h^{-2/3}x$: We look for quasimodes $(\lambda_h,\psi_h)$ in the form of series
$$
   \la_h\sim\sum_{j\geq 0}\hat\beta_{j}h^{2j/3}
   \quad\mbox{and}\quad
   \psi_h(x) \sim \sum_{j\geq 0}\Psi_{j}(s)h^{2j/3}
$$
in order to solve
$\H_{\BO, \Tri}(h)\psi_h=\la_h\psi_h$ in the sense of formal series.
A Taylor expansion at $x=0$ of the potential $V$ yields:
$$
   \H_{\BO, \Tri}(h) \sim -h^2\dr_{x}^2+\sum_{j\ge0} V_{j}x^j,
   \quad\mbox{with}\quad
   V_{0}=\frac{1}{8} \ \ \mbox{and}\ \ V_{1}=-\frac{1}{4\pi\sqrt{2}},
$$
which, in $s$ variable, becomes
\begin{equation}
\label{4E1}
   \H_{\BO, \Tri}(h) \sim \frac{1}{8} + h^{2/3}\big(-\partial^2_s+V_1s\big) + 
   \sum_{j\ge2}h^{2j/3} V_{j} s^j.
\end{equation}
The construction of the terms $\hat\beta_j$ and $\Psi_j$ is similar (even simpler) than for Proposition \ref{quasitoy}.

\noindent$\bullet$ \ The expansion \eqref{4E1} yields that $\hat\beta_0=\frac18$, and collecting the terms in $h^{2/3}$ and we obtain:
\begin{equation}
\label{4E1b}
   -\Psi''_{0}(s)-\frac{s}{4\pi\sqrt{2}}\,\Psi_{0}(s)=\hat\beta_1\Psi_{0}(s)
   \ \ \ \forall s<0 \quad\mbox{and}\quad \Psi_{0}(0)=0.
\end{equation}
Thus for any chosen positive integer $n$ we can take $\hat\beta_1=(4\pi\sqrt{2})^{-2/3}z_{\A}(n)$ together with
\begin{equation}
\label{E:Psi0}
   \Psi_{0}(s)=\A\big((4\pi\sqrt{2})^{-1/3}s+z_{\A}(n)\big).
\end{equation}

\noindent$\bullet$ \ Collecting the terms in $h^{4/3}$ we obtain
$$
   -\Psi''_{1}(s)+V_{1}s\Psi_{1}(s)-\hat\beta_1\Psi_{1}(s)=
   \hat\beta_2\Psi_{0}-V_{2}s^2\Psi_{0}
   \ \ \ \forall s<0 \quad\mbox{and}\quad \Psi_{1}(0)=0.
$$
The compatibility condition states that 
$\hat\beta_2\langle \Psi_{0},\Psi_{0}\rangle = V_{2}\langle s^2\Psi_{0},\Psi_{0}\rangle$.
This determines $\hat\beta_2$ and implies the existence of a unique solution $\Psi_{1}\in L^2(\R_-)$ such that $\langle{\Psi}_{1},\Psi_{0}\rangle=0$.

\noindent$\bullet$ \ This procedure can be continued at any order and determines $(\hat\beta_j,\Psi_{j})$ at each step. This construction depends on the choice of the integer $n$ and can be done for any positive integer $n$.

\noindent$\bullet$ \ To conclude, we consider a smooth non-negative cutoff function $\chi^\lef$ satisfying:
\begin{equation}\label{chi-lef}
\chi^\lef(x)=1\quad \mbox{ for } x\in\Big(-\frac{\pi}{\sqrt{2}},+\infty\Big)\quad\mbox{ and }\quad\chi^{\lef}(x)=0\quad \mbox{ for } x\leq-\pi,
\end{equation}
and introduce for any $J\geq 0$ the quasimode $(\beta_{h}^{[J]},\psi_{h}^{[J]})$ with:
\begin{equation}
\label{E:BOtriquasi}
   \beta_{h}^{[J]} = \sum_{j=0}^J \hat\beta_{j}h^{2j/3}\quad\mbox{and}\quad
   \psi_{h}^{[J]}(x)=\chi^\lef(x)\sum_{j=0}^{J}\Psi_{j}\Big(\frac{x}{h^{2/3}}\Big)h^{2j/3}.
\end{equation}
Thanks to this cut-off $\psi_{h}^{[J]}$ satisfies Dirichlet condition in $-\pi\sqrt{2}$, and in $0$ by construction. Using the exponential decay of $x\mapsto\Psi_{j}(h^{-2/3}x)$  and the definition of $\Psi_{j}$ and $\hat\beta_{j}$, we get for any $h_{0}>0$ the existence of $C(n,J,h_{0})>0$ such that:
\begin{equation}
\label{E:BOtriest}
   \Big\| \big(\mathcal{H}_{\BO, \Tri}(h)-
  \beta_{h}^{[J]} \big)\psi_{h}^{[J]} \Big\|\leq 
   C(n,J,h_{0})\, h^{2(J+1)/3},\quad\forall h\in(0,h_{0}).
\end{equation}
This proves the existence of quasimodes at any order and ends the proof of Proposition \ref{quasibotri}.
\end{proof}

\subsection{Agmon estimates}
In this subsection, we prove Agmon estimates (see \cite{Agmon82,Agmon85}) for the eigenfunctions of $\mathcal{H}_{\BO, \Tri}(h)$. The role of Agmon estimates is to replace an explicit knowledge of the solution at infinity like in \eqref{E:toylef}-\eqref{E:toyrig} by suboptimal exponential estimates.

Here we prove two kinds of estimates: near $x=0$ and near $x=-\pi\sqrt{2}$. In the analysis of triangles (cf. Section \ref{S:TriAgmon}), we will meet the same estimates.
Let us consider an eigenpair $(\la,\psi)$ of $\mathcal{H}_{\BO, \Tri}(h)$. The Agmon identity writes, for some Lipschitz function $\Phi$ to be determined:
\begin{equation}\label{Agmon1D}
\int_{-\pi\sqrt{2}}^{0}  h^2|\dr_{x} (e^{\Phi}\psi)|^2+V(x)|e^{\Phi}\psi|^2-h^2 |\Phi'e^{\Phi}|^2-\la  |(e^{\Phi}\psi)|^2\, dx=0.
\end{equation}
It is a consequence of Proposition \ref{quasibotri} that the lowest $N_0$ eigenvalues $\lambda$ of $\mathcal{H}_{\BO, \Tri}(h)$ satisfy:
\begin{equation}\label{small-e-v}
|\la-\tfrac{1}{8}|\leq \Gamma_0 \,h^{2/3},
\end{equation}
for some positive constant $\Gamma_0$ depending on $N_0$.

\subsubsection{Agmon estimates near $x=0$}
We use (\ref{Agmon1D}) and the convexity of $V$ to get the inequality:
$$\int_{-\pi\sqrt{2}}^{0}  h^2|\dr_{x} (e^{\Phi}\psi)|^2+\left(\frac{1}{8}-\frac{x}{4\pi\sqrt{2}}\right)|e^{\Phi}\psi|^2-h^2 |\Phi'e^{\Phi}|^2-\la  |(e^{\Phi}\psi)|^2\, dx\leq 0.$$
With (\ref{small-e-v}), we deduce:
$$\int_{-\pi\sqrt{2}}^{0}-\frac{x}{4\pi\sqrt{2}}|e^{\Phi}\psi|^2-h^2 |\Phi'e^{\Phi}|^2-Ch^{2/3}  |(e^{\Phi}\psi)|^2\, dx\leq 0.$$
This leads to the choice 
$$\Phi(x)=\eta \,h^{-1}|x|^{3/2},$$
for a number $\eta>0$ to be chosen small enough. We get:
$$
   \int_{-\pi\sqrt{2}}^{0}\left(\frac{|x|}{4\pi\sqrt{2}}
   -\frac94\eta^2|x|-Ch^{2/3}\right)|e^{\Phi}\psi|^2 \, dx\leq 0.$$
For $\eta$ small enough, we obtain the existence of $\tilde{\eta}>0$ such that:
$$\int_{-\pi\sqrt{2}}^{0}\left(\tilde{\eta}|x|-Ch^{2/3}\right)|e^{\Phi}\psi|^2 \, dx\leq 0.$$
Splitting the integral into the parts
$-\pi\sqrt{2}<x< -Dh^{2/3}$ (where $\Phi$ is unbounded) and
$-Dh^{2/3}<x<0$ (where $\Phi$ is bounded) 
with $\tilde{\eta}D-C=d>0$, we find:
\begin{align*}
   \int^{- Dh^{2/3}}_{-\pi\sqrt{2}} d\,h^{2/3}|e^{\Phi}\psi|^2 \, dx &
   \leq\int^{- Dh^{2/3}}_{-\pi\sqrt{2}}
   \left(\tilde{\eta}|x|-Ch^{2/3}\right)|e^{\Phi}\psi|^2 \, dx\\
   &\leq\int^0_{- Dh^{2/3}}\left(\tilde{\eta}|x|+Ch^{2/3}\right)|e^{\Phi}\psi|^2 \, dx
   \leq\tilde{C}h^{2/3}\int^0_{-Dh^{2/3}}|\psi|^2\,dx.
\end{align*}
We deduce the proposition:

\begin{prop}\label{Agmon1}
Let $\Gamma_0>0$. There exist $h_{0}>0$, $C_{0}>0$ and $\eta_{0}>0$ such that for $h\in(0,h_{0})$ and all eigenpair $(\la,\psi)$ of $\mathcal{H}_{\BO, \Tri}(h)$ satisfying $|\lambda-\frac18|\le\Gamma_0h^{2/3}$, we have:
$$
   \int_{-\pi\sqrt{2}}^0  e^{\eta_{0}h^{-1}|x|^{3/2}} 
   \Big(|\psi|^2 + |h^{2/3}\dr_{x}\psi|^2 \Big)\,dx\leq C_{0}\|\psi\|^2.
$$
\end{prop}

\subsubsection{Agmon estimates near $x=-\pi\sqrt{2}$}
We use again \eqref{Agmon1D} and \eqref{small-e-v}:
$$\int_{-\pi\sqrt{2}}^{0} h^2|\dr_{x} (e^{\Phi}\psi)|^2+\left(\frac{\pi^2}{4(x+\pi\sqrt{2})^2}-\frac{1}{8}\right)|e^{\Phi}\psi|^2-h^2 |\Phi' e^{\Phi}\psi|^2-Ch^{2/3}  |(e^{\Phi}\psi)|^2\, dx \leq 0.$$
We take:
$$\Phi(x)=-\rho h^{-1}\ln\big(D^{-1}(x+\pi\sqrt{2})\big),$$
where we choose $\rho\in(0,\frac\pi2)$ so that there holds:
$$
   \int_{-\pi\sqrt{2}}^{0}  \left(\Big(\frac{\pi^2}{4}-\rho^2\Big)(x+\pi\sqrt{2})^{-2}
   -\frac{1}{8}\right)|e^{\Phi}\psi|^2-Ch^{2/3}  |(e^{\Phi}\psi)|^2\, dx \leq 0,
$$
and $D>0$ large enough so that 
\[
   \Big(\frac{\pi^2}{4}-\rho^2\Big)D^2-\frac18>0.
\]
Then we split the integral into the parts 
$-\pi\sqrt{2}<x< -\pi\sqrt{2}+D$ (where $\Phi$ is unbounded) and
$-\pi\sqrt{2}+D<x<0$ (where $\Phi$ is bounded) and the same procedure as in the previous paragraph leads to the proposition:

\begin{prop}\label{Agmon2}
Let $\Gamma_0>0$ and $\rho_0\in(0,\frac\pi2)$. There exist $h_{0}>0$, $C_{0}>0$ such that for any $h\in(0,h_{0})$ and all eigenpair $(\la,\psi)$ of $\mathcal{H}_{\BO, \Tri}(h)$ satisfying $|\lambda-\frac18|\le\Gamma_0h^{2/3}$, we have:
$$
   \int_{-\pi\sqrt{2}}^0 (x+\pi\sqrt{2})^{-\rho_{0}/h}
   \Big( |\psi|^2 + |h\,\dr_{x}\psi|^2 \Big)\, dx\leq C_{0}\|\psi\|^2.$$
\end{prop}

\subsection{Proof of Theorem \ref{spectrumBOT}}
\label{BO-triangle-end}
Let us fix $N_{0}$ and consider the $N_{0}$ first eigenvalues of $\mathcal{H}_{\BO, \Tri}(h)$ denoted by $\la_{n}=\la_{\BO,\Tri,n}(h)$. For each $n\in\{1,\cdots N_{0}\}$, we choose a normalized $\psi_{n}$ in the eigenspace of $\la_{n}$ so that $\langle\psi_{n},\psi_{m}\rangle=0$ for $n\neq m$.
Let us introduce the space:
$$\EN=\mathrm{span}(\psi_{1},\ldots,\psi_{N_0}).$$
We recall that, for $h$ small enough, (\ref{small-e-v}) holds. We can write:
$$\mathcal{H}_{\BO, \Tri}(h)\psi_{n}=\la_{n}\psi_{n}$$
so that (the $\psi_{n}$ are orthogonal in $L^2$ and for the quadratic form), for all $\psi\in\EN$:
$$Q_{\BO,\Tri,h}(\psi)\leq\la_{N_{0}}\|\psi\|^2.$$
For $\eps_{0}$ small enough we introduce a smooth cutoff function $\chi$ being $0$ for $|x+\pi\sqrt{2}|\leq\eps_{0}$ and $1$ for $|x+\pi\sqrt{2}|\geq 2\eps_{0}$. Proposition \ref{Agmon2} implies that:
$$Q_{\BO,\Tri,h}(\chi\psi)\leq(\la_{N_{0}}+\OO(h^{\infty}))\|\chi\psi\|^2.$$
Then, the convexity of the potential allows to write:
$$\left\langle\left(-h^2\dr_{x}^2-\frac{1}{4\pi\sqrt{2}}\,x+\frac{1}{8}\right)\chi\psi,\chi\psi\right\rangle\leq (\la_{N_{0}}+\OO(h^{\infty}))\|\chi\psi\|^2,$$
where we have used the convexity. 
The dimension  of $\chi\EN$ is $N_{0}$ so that, with the properties of the Airy operator and the mini-max principle, we get:
$$\frac{1}{8}+(4\pi\sqrt{2})^{-2/3}z_{\A}(N_{0})\leq\la_{N_{0}}+\OO(h^{\infty}).$$
This is true for all fixed $N_{0}$ and, combined with Proposition \ref{quasibotri}, provides the separation of the lowest eigenvalues of $\mathcal{H}_{\BO, \Tri}(h)$:
\[
  \Big|\la_{\BO,\Tri,n}(h) - \Big( \frac{1}{8}+h^{2/3}(4\pi\sqrt{2})^{-2/3} z_{\A}(n) \Big) \Big|
  \leq Ch^{4/3} .
\]
Finally, with \eqref{E:BOtriquasi}-\eqref{E:BOtriest}, we obtain Theorem \ref{spectrumBOT}.

\subsection{Born-Oppenheimer approximation for the waveguide}\label{S:4.4}
Let us end this section by an informal analysis of the spectrum of the operator $\H_{\BO, \Gui}(h)$ defined in \eqref{HG}. This investigation is not necessary in our way to prove Theorem \ref{spectrumguide}, but it already gives a flavor of the ideas to analyze $\L_{\Gui}(h)$. We can obtain the asymptotic expansions of the eigenvalues of $\H_{\BO, \Gui}(h)$ by combining the analysis for $\H_{\toy}(h)$ and for $\H_{\BO, \Tri}(h)$. Indeed we can perform a quasimode construction like for $\H_{\toy}(h)$ and $\H_{\BO, \Tri}(h)$ by solving a transmission problem between the negative half-axis and the positive half-axis. 
For that purpose, we establish the following Agmon type estimate which states that the eigenfunctions of  $\H_{\BO, \Gui}(h)$ do not penetrate in the region $x>0$ more than at the scale $h$.
\begin{prop}
\label{AgmonBO}
Let $(\la,\psi)$ be an eigenpair of $\H_{\BO, \Gui}(h)$ such that $|\la-\frac{1}{8}|\leq Ch^{2/3}$. There exist $\alpha>0$, $h_{0}>0$ and $C>0$ such that for all $h\in(0,h_{0})$, we have:
$$\int_{x\geq 0} e^{\alpha h^{-1}x}
\Big(|\psi|^2 + |h\dr_{x}\psi|^2\Big)\, dxdy\leq C\|\psi\|^2 .$$
\end{prop}
Thanks to the latter estimate we cut off the part of the eigenfunctions living on $x>0$ modulo a remainder of order $O(h^{\infty})$. This allows the comparison of $\H_{\BO, \Gui}(h)$ with $\H_{\BO, \Tri}(h)$ and provides the proof that all the lowest eigenvalues of $\H_{\BO, \Gui}(h)$ are described by the quasimode construction. In Section \ref{6.2} a similar analysis will be done for the whole waveguide.

\section{Triangle with Dirichlet boundary condition}
\label{5}
The aim of this section is to prove Theorem \ref{spectrumtriangle}. As usual, the proof will be divided into two main steps: a construction of quasimodes and the use of the true eigenfunctions of $\L_{\Tri}(h)$ as quasimodes for the Born-Oppenheimer approximation in order to obtain a lower bound for the true eigenvalues.

We first perform a change of variables to transform the triangle into a rectangle:
\begin{equation}
\label{Eut}
   u=x\in (-\pi\sqrt{2},0) ,\quad t=\frac{y}{x+\pi\sqrt{2}} \in (-1,1).
\end{equation}
so that $\Tri$ is transformed into 
\begin{equation}
\label{E:Rec}
   \Rec=(-\pi\sqrt{2},0)\times(-1,1).
\end{equation}
The operator $\L_{\Tri}(h)$ becomes:
\begin{equation}
\label{E:LRec}
   \L_{\Rec}(h)(u,t;\partial_u,\partial_t) = -h^2\Big(\dr_{u}
   -\frac{t}{u+\pi\sqrt{2}}\,\dr_{t}\Big)^2-\frac{1}{(u+\pi\sqrt{2})^2}\,\dr_{t}^2,
\end{equation}
with Dirichlet boundary conditions on $\partial\Rec$. The equation $\L_{\Tri}(h)\psi_h=\beta_h\psi_h$  is transformed into the equation 
\[
   \L_{\Rec}(h)\hat\psi_h=\beta_h\hat\psi_h \quad\mbox{with}\quad 
   \hat\psi_h(u,t) = \psi_h(x,y).
\]

\subsection{Quasimodes}
This subsection is devoted to the proof of the following proposition.
\begin{prop}\label{quasitri}
There are sequences $(\beta_{j,n})_{j\ge0\,}$ for any integer $n\ge1$ so that there holds:\\ For all $N_{0}\in\mathbb{R}$ and $J\in\mathbb{N}$, there exists $h_{0}>0$ and $C>0$ such that for $h\in(0,h_{0})$
\begin{equation}
   \dist \Big(\gS_\dis\big(\mathcal{L}_{\Tri}(h)\big),\,
  \sum_{j=0}^{J}\beta_{j,n}h^{j/3} \Big) \leq Ch^{(J+1)/3}, \quad n=1,\cdots N_0.
\end{equation}
Moreover, we have: $\beta_{0,n}=\frac{1}{8}$, $\beta_{1,n}=0$, and $\beta_{2,n}=(4\pi\sqrt{2})^{-2/3}z_{\A}(n)$.
\end{prop}

\begin{proof}
We want to construct quasimodes $(\beta_h,\psi_h)$ for the operator $\L_{\Tri}(h)(\partial_x,\partial_y)$. It will be more convenient to work on the rectangle $\Rec$ with the operator $\L_{\Rec}(h)(u,t;\partial_u,\partial_t)$.
We introduce the new scales
\begin{equation}
\label{E:ssigma2}
   s=h^{-2/3}u\quad\mbox{and} \quad\sigma=h^{-1}u,
\end{equation}
and we look quasimodes $(\beta_h,\hat\psi_h)$ in the form of series 
\begin{equation}
\label{5EAn}
   \beta_h\sim\sum_{j\geq 0}\beta_{j}h^{j/3} 
   \quad\mbox{and}\quad
   \hat\psi_h(u,t) \sim 
   \sum_{j\geq 0}\big(\Psi_{j}(s,t)+\Phi_{j}(\sigma,t)\big) h^{j/3}
\end{equation}
in order to solve
$\L_{\Rec}(h)\hat\psi_h=\beta_h\hat\psi_h$ in the sense of formal series.
As will be seen hereafter, an Ansatz containing the scale $h^{-2/3}u$ alone (like for the Born-Oppenheimer operator $\H_{\BO, \Tri}(h)$) is not sufficient to construct quasimodes for $\L_{\Rec}(h)$.
Expanding  the operator in powers of $h^{2/3}$, we obtain the formal series:
\begin{equation}
\label{E:Lj}
   \L_{\Rec}(h)(h^{2/3}s,t;h^{-2/3}\partial_s,\partial_t) 
   \sim \sum_{j\geq 0} \L_{2j}h^{2j/3}
   \quad\mbox{with leading term}\quad \L_{0} = -\frac{1}{2\pi^2}\dr^2_{t}
\end{equation}
and in powers of $h$:
\begin{equation}
\label{E:Nj}
   \L_{\Rec}(h)(h\sigma,t;h^{-1}\partial_\sigma,\partial_t) \sim
   \sum_{j\geq 0} \cN_{3j}h^{j}
   \quad\mbox{with leading term}\quad \cN_{0} = -\dr_{\sigma}^2-\frac{1}{2\pi^2}\dr^2_{t}.
\end{equation}
In what follows, in order to finally ensure the Dirichlet conditions on the triangle $\Tri$, we will require for our Ansatz the boundary conditions, for any $j\in\N$:
\begin{gather}
\label{5Dir1}
   \Psi_{j}(0,t)+\Phi_{j}(0,t)=0,\quad -1\leq t\leq1 \\
\label{5Dir2}
   \Psi_{j}(s,\pm1)=0,\ \  s<0 \quad\mbox{and}\quad 
\Phi_{j}(\sigma,\pm1)=0,\ \   \sigma\leq 0.
\end{gather}

More specifically, we are interested in the ground energy $\lambda=\frac18$ of the Dirichlet problem for $\L_0$ on the interval $(-1,1)$. Thus we have to solve Dirichlet problems for the operators $\cN_0-\frac18$ and $\L_0-\frac18$ on the half-strip 
\begin{equation}
\label{E:Hst}
   \Hst = \R_-\times(-1,1),
\end{equation} 
and look for \emph{exponentially decreasing solutions}. The situation is similar to that encountered in thin structure asymptotics with Neumann boundary conditions. The following lemma shares common features with the Saint-Venant principle, see for example \cite[\S2]{DauGru98}.

\begin{lem}\label{lem-N0}
We denote the first normalized eigenvector of $\L_0$ on $H^1_0((-1,1))$ by $c_0$:
\[
   c_{0}(t)=\cos\left(\frac{\pi t}{2}\right).
\]
Let $F=F(\sigma,t)$ be a function in $L^2(\Hst)$ with exponential decay with respect to $\sigma$ and let $G\in H^{3/2}((-1,1))$ be a function of $t$ with $G(\pm1)=0$. 
Then there exists a unique $\gamma\in\R$ such that the problem
$$\left(\cN_{0}-\frac{1}{8}\right)\Phi=F\ \ \mbox{in}\ \ \Hst,\quad \Phi(\sigma,\pm1)=0,\quad  \Phi(0,t)=G(t)+\gamma c_{0}(t),$$
admits a (unique) solution in $H^2(\Hst)$ with exponential decay. There holds
\[
   \gamma =-\int_{-\infty}^{0} \int_{-1}^1 F(\sigma,t)\,\sigma c_0(t)\,d\sigma dt -
   \int_{-1}^1 G(t)\,c_0(t)\,dt.
\]
\end{lem}

The following two lemmas are consequences of the Fredholm alternative.
\begin{lem}\label{lem-L0}
Let $F=F(s,t)$ be a function in $L^2(\Hst)$ with exponential decay with respect to $s$. Then, there exist solution(s) $\Psi$ such that:
$$\left(\L_{0}-\frac{1}{8}\right)\Psi=F\ \ \mbox{in}\ \ \Hst, \quad \Psi(s,\pm1)=0 $$
if and only if $\big\langle F(s,\cdot),c_{0}\big\rangle_{t}=0$ for all $s<0$. In this case, $\Psi(s,t)=\Psi^\perp(s,t) + g(s)c_0(t)$ where $\Psi^\perp$ satisfies $\big\langle\Psi(s,\cdot) ,c_{0}\big\rangle_{t}\equiv0$ and has also an exponential decay.
\end{lem}

Then, we will also need a rescaled version of Lemma \ref{lem-Ai0}.

\begin{lem}\label{lem-Ai}
Let $n\geq 1$. We recall that $z_{\A}(n)$ is the n-th zero of the reverse Airy function, and we denote by 
$$g_{(n)}=\A\big((4\pi\sqrt{2})^{-1/3}s+z_{\A}(n)\big)$$ the eigenvector of the operator $-\dr_{s}^2-(4\pi\sqrt{2})^{-1}s$ with Dirichlet condition on $\R_-$ associated with the eigenvalue $(4\pi\sqrt{2})^{-2/3}z_{\A}(n)$. Let $f=f(s)$ be a function in $L^2(\R_-)$ with exponential decay and let $c\in\R$. Then there exists a unique $\beta\in\R$ such that the problem:
$$\left(-\dr_{s}^2-\frac{s}{4\pi\sqrt{2}}-(4\pi\sqrt{2})^{-2/3}z_{\A}(n)\right)g=f+\beta g_{(n)}\ \ \mbox{in}\ \ \R_-, \mbox{ with } g(0)=c,$$
has a solution in $H^2(\R_-)$ with exponential decay.
\end{lem}

Now we can start the construction of the terms of our Ansatz \eqref{5EAn}.

\paragraph{Terms in $h^0$}
The equations provided by the constant terms are:
$$
  \L_{0}\Psi_{0}=\beta_{0}\Psi_{0}(s,t),\quad   \cN_{0}\Phi_{0}=\beta_{0}\Phi_{0}(s,t)
$$
with boundary conditions \eqref{5Dir1}-\eqref{5Dir2} for $j=0$, so that we choose $\beta_{0}=\frac{1}{8}$ and $\Psi_{0}(s,t)=g_{0}(s)c_{0}(t)$. The boundary condition \eqref{5Dir1}  provides: $\Phi_{0}(0,t)=-g_{0}(0)c_{0}(t)$ so that, with Lemma \ref{lem-N0}, we get $g_{0}(0)=0$ and $\Phi_{0}=0$. The function $g_0(s)$ will be determined later.

\paragraph{Terms in $h^{1/3}$}
Collecting the terms of order $h^{1/3}$, we are led to:
$$(\L_{0}-\beta_{0})\Psi_{1}=\beta_{1}\Psi_{0}-\L_{1}\Psi_{1}=\beta_{1}\Psi_{0},\quad (\cN_{0}-\beta_{0})\Phi_{1}=\beta_{1}\Phi_{0}-\cN_{1}\Phi_{1}=0$$
with boundary conditions \eqref{5Dir1}-\eqref{5Dir2} for $j=1$.
Using Lemma \ref{lem-L0}, we find $\beta_{1}=0$, $\Psi_{1}(s,t)=g_{1}(s)c_{0}(t)$,  $g_{1}(0)=0$ and $\Phi_{1}=0$.

\paragraph{Terms in $h^{2/3}$}
We get:
$$(\L_{0}-\beta_{0})\Psi_{2}=\beta_{2}\Psi_{0}-\L_{2}\Psi_{0},\quad(\cN_{0}-\beta_{0})\Phi_{2}=0,$$
where $\L_{2}=-\dr_{s}^2+\frac{s}{\pi^3\sqrt{2}}\,\dr_{t}^2$ and with boundary conditions \eqref{5Dir1}-\eqref{5Dir2} for $j=2$. Lemma \ref{lem-L0} provides the equation in $s$ variable
\[
   \big\langle (\beta_{2}\Psi_{0}-\L_{2}\Psi_{0}(s,\cdot)),c_{0}\big\rangle_{t}=0, \quad s<0.
\]
Taking the formula $\Psi_0=g_{0}(s)c_{0}(t)$ into account this becomes
\[
   \beta_2 g_0(s) = \left(-\dr_{s}^2-\frac{s}{4\pi\sqrt{2}}\right)g_0(s).
\]
This equation leads to take $\beta_{2}=(4\pi\sqrt{2})^{-2/3}z_{\A}(n)$ and for $g_{0}$ the corresponding eigenfunction $g_{(n)}$. We deduce $(\L_{0}-\beta_{0})\Psi_{2}=0$, then
get $\Psi_{2}(s,t)=g_{2}(s)c_{0}(t)$ with $g_{2}(0)=0$ and $\Phi_{2}=0$.

\paragraph{Terms in $h^{3/3}$}
We get:
$$(\L_{0}-\beta_{0})\Psi_{3}=\beta_{3}\Psi_{0}+\beta_{2}\Psi_{1}-\L_{2}\Psi_{1},\quad(\cN_{0}-\beta_{0})\Phi_{3}=0,$$
with boundary conditions \eqref{5Dir1}-\eqref{5Dir2} for $j=3$.
The scalar product with $c_{0}$ (Lemma \ref{lem-L0}) and then the scalar product with $g_{0}$ (Lemma \ref{lem-Ai}) provide $\beta_{3}=0$ and $g_{1}=0$. We deduce: $\Psi_{3}(s,t)=g_{3}(s)c_{0}(t)$, and $g_{3}(0)=0$, $\Phi_{3}=0$.

\paragraph{Terms in $h^{4/3}$}
We get:
$$(\L_{0}-\beta_{0})\Psi_{4}=\beta_{4}\Psi_{0}+\beta_{2}\Psi_{2}-\L_{4}\Psi_{0}-\L_{2}\Psi_{2},\quad(\cN_{0}-\beta_{0})\Phi_{4}=0,$$
where 
$$\L_{4}=\frac{\sqrt{2}}{\pi} \,t\dr_{t}\dr_{s}-\frac{3}{4\pi^4}s^2\dr_{t}^2,$$
and with boundary conditions \eqref{5Dir1}-\eqref{5Dir2} for $j=4$.
The scalar product with $c_{0}$ provides an equation for $g_{2}$ and the scalar product with $g_{0}$ determines $\beta_{4}$. By Lemma \ref{lem-L0} this step determines $\Psi_{4}=\Psi_{4}^{\perp}+c_{0}(t)g_{4}(s)$ with a non-zero $\Psi_{4}^{\perp}$ and $g_{4}(0)=0$. Since by construction $\big\langle\Psi_{4}^{\perp}(0,\cdot),\,c_0\big\rangle_t=0$, Lemma \ref{lem-N0} yields a solution $\Phi_{4}$ with exponential decay. Note that it also satisfies $\big\langle\Phi_{4}(\sigma,\cdot),\,c_0\big\rangle_t=0$ for all $\sigma<0$.

\paragraph{Terms in $h^{5/3}$}
We get:
$$(\L_{0}-\beta_{0})\Psi_{5}=\beta_{5}\Psi_{0}+\beta_{2}\Psi_{3}-\L_{2}\Psi_{3},\quad(\cN_{0}-\beta_{0})\Phi_{5}=0,$$
and with boundary conditions \eqref{5Dir1}-\eqref{5Dir2} for $j=5$.
We find $\beta_{5}=0$, $g_{3}=0$, $\Psi_{5}=g_{5}(s)c_{0}(t)$, $g_{5}(0)=0$, $\Phi_{5}=0$.

\paragraph{Terms in $h^{6/3}$}
We get:
$$(\L_{0}-\beta_{0})\Psi_{6}=\beta_{6}\Psi_{0}+\beta_{4}\Psi_{2}+\beta_{2}\Psi_{4}-\L_{2}\Psi_{4}-\L_{4}\Psi_{2},\quad(\cN_{0}-\beta_{0})\Phi_{6}=\beta_2\Phi_4,$$
and with boundary conditions \eqref{5Dir1}-\eqref{5Dir2} for $j=6$.
This determines $\beta_{6}$, $g_{4}$, $\Psi_{6}=\Psi_{6}^{\perp}+c_{0}(t)g_{6}(s)$, $g_{6}(0)=0$, and $\Phi_{6}$ with exponential decay due to the orthogonality of $\Phi_4$ to $c_0$.

\paragraph{Terms in $h^{7/3}$}
We get:
$$(\L_{0}-\beta_{0})\Psi_{7}=\beta_{7}\Psi_{0}+\beta_{2}\Psi_{5}-\L_{2}\Psi_{5},\quad(\cN_{0}-\beta_{0})\Phi_{7}=-\cN_{3}\Phi_{4},$$
where 
$$\cN_{3}=\frac{2}{\pi\sqrt{2}}t\dr_{\sigma}\dr_{t}+\frac{\sigma}{\pi^3\sqrt{2}}\dr_{t}^2,$$
and with boundary conditions \eqref{5Dir1}-\eqref{5Dir2} for $j=7$.
We take $\beta_{7}=0$, $g_{5}=0$, $\Psi_{7}=g_{7}(s)c_{0}(t)$. 
Then, Lemma \ref{lem-N0} induces a value for the trace $g_{7}(0)$ so that there exists $\Phi_{7}$ with an exponential decay.
Note that if there holds:
\begin{equation}
\label{neq0}
   \int_\Hst (\cN_{3}\Phi_{4})(\sigma,t)\, \sigma c_0(t) \,d\sigma dt \neq 0,
\end{equation}
we would deduce by Lemma \ref{lem-N0} that $g_{7}(0)\neq 0$.

\paragraph{Terms in $h^{8/3}$}
We get:
\begin{align*}
   (\L_{0}-\beta_{0})\Psi_{8}&=\beta_{8}\Psi_{0}+\beta_{6}\Psi_{2}+
   \beta_{4}\Psi_{4}+\beta_{2}\Psi_{6}-\L_{8}\Psi_{0}-\L_{6}\Psi_{2}-\L_{4}\Psi_{4}-\L_{2}\Psi_{6},\\
   (\cN_{0}-\beta_{0})\Phi_{8}&=\beta_{4}\Phi_{4}+\beta_{2}\Phi_{6}.
\end{align*}
This determines $\beta_{8}$, $g_{6}$ and $\Psi_{8}=\Psi_{8}^{\perp}+c_{0}g_{8}$, the trace $g_{8}(0)$ and the exponentially decreasing solution $\Phi_{8}$.

\paragraph{Terms in $h^{9/3}$}
We get:
$$(\L_{0}-\beta_{0})\Psi_{9}=\beta_{9}\Psi_{0}+\beta_{2}\Psi_{7}-\L_{2}\Psi_{7},\quad(\cN_{0}-\beta_{0})\Phi_{9}=\beta_2\Phi_7-\cN_3\Phi_6.$$
We find $\beta_{9}$, $g_{7}$ and then $\Psi_{9}=\Psi_{9}^{\perp}+c_{0}g_{9}$ and $g_{9}(0)$, $\Phi_{9}$. Note that if $g_{7}(0)\neq 0$, i.e.\ if \eqref{neq0} holds, we would deduce that $\beta_{9}\neq 0$.

\paragraph{Continuation.} The construction of the further terms goes on along the same lines. 
 This leads to define the quasimodes for $\L_{\Tri}(h)$:
\begin{gather}
   \beta^{[J]}_h=\sum_{j=0}^J \beta_{j}h^{j/3},\\
\label{quasi-eigenfunction-Tri}
   \psi_{h}^{[J]} = \chi^\lef(x) \sum_{j=0}^J 
   \left(\Psi_{j}\Big(\frac{x}{h^{2/3}}\,,\frac{y}{x+\pi\sqrt{2}}\Big)
   +\Phi_{j}\Big(\frac{x}{h}\,,\frac{y}{x+\pi\sqrt{2}}\Big)\right) h^{j/3},
\end{gather}
where $\chi^{\lef}$ is defined in \eqref{chi-lef}.
The conclusion follows from the spectral theorem.
\end{proof}

\subsection{Agmon estimates}
\label{S:TriAgmon}
On our way to prove Theorem \ref{spectrumtriangle}, we now state Agmon estimates like for $\mathcal{H}_{\BO, \Tri}(h)$.
Let us first notice that, due to Proposition \ref{quasitri}, the lowest eigenvalues of $\mathcal{L}_{\Tri}(h)$ still satisfy an estimate like \eqref{small-e-v}.
It turns out that we have the following lower bound, for all $\psi\in\Dom(Q_{\Tri,h})$:
$$Q_{\Tri,h}(\psi)\geq\int_{\Tri} h^2|\dr_{x}\psi|^2+\frac{\pi^2}{4(x+\pi\sqrt{2})^2}|\psi|^2\,dxdy.$$
Thus, the analysis giving Propositions \ref{Agmon1} and \ref{Agmon2} applies exactly in the same way and we obtain:

\begin{prop}\label{Agmon1'}
Let $\Gamma_0>0$. There exist $h_{0}>0$, $C_{0}>0$ and $\eta_{0}>0$ such that for $h\in(0,h_{0})$ and all eigenpair $(\la,\psi)$ of $\mathcal{L}_{\Tri}(h)$ satisfying $|\lambda-\frac18|\le\Gamma_0h^{2/3}$, we have:
$$\int_{\Tri}  e^{\eta_{0}h^{-1}|x|^{3/2}}\Big(|\psi|^2+|h^{2/3}\dr_{x}\psi|^2\Big)\,dxdy
\leq C_{0}\|\psi\|^2.$$
\end{prop}

\begin{prop}\label{Agmon2'}
Let $\Gamma_0>0$. There exist $h_{0}>0$, $C_{0}>0$ and $\rho_{0}>0$ such that for $h\in(0,h_{0})$ and all eigenpair $(\la,\psi)$ of $\mathcal{L}_{\Tri}(h)$ satisfying $|\lambda-\frac18|\le\Gamma_0h^{2/3}$, we have:
$$\int_{\Tri} (x+\pi\sqrt{2})^{-\rho_{0}/h}\Big(|\psi|^2+|h\,\dr_{x}\psi|^2\Big)\, dxdy
\leq C_{0}\|\psi\|^2.$$
\end{prop}

\subsection{Approximation of the first eigenfunctions by tensor products}
In this subsection, we will work with the operator $\L_{\Rec}(h)$ rather than $\L_{\Tri}(h)$.
Let us consider the first $N_{0}$ eigenvalues of $\L_{\Rec}(h)$ (shortly denoted by $\la_{n}$). In each corresponding eigenspace, we choose a normalized eigenfunction $\hat\psi_{n}$ so that $\langle\hat\psi_{n},\hat\psi_{m}\rangle=0$ if $n\neq m$.
As in Section \ref{BO-triangle-end}, we introduce:
$$\EN=\mathrm{span}(\hat\psi_{1},\ldots,\hat\psi_{N_0}).$$
Let us define $Q_{\Rec}^0$ the following quadratic form:
$$Q_{\Rec}^0(\hat{\psi})=\int_{\Rec} \left(\frac{1}{2\pi^2}|\dr_{t}\hat{\psi}|^2-\frac{1}{8}|\hat{\psi}|^2\right) (u+\pi\sqrt{2})\, dudt,$$
associated with the operator $\L_{\Rec}^0=\Id_{u}\otimes\left(-\frac{1}{2\pi^2}\dr_{t}^2-\frac{1}{8}\right)$ on $L^2(\Rec,(u+\pi\sqrt{2})dudt)$.
We consider the projection on the eigenspace associated with the eigenvalue $0$ of $-\frac{1}{2\pi^2}\dr_{t}^2-\frac{1}{8}$:
\begin{equation}\label{Pi0}
   \Pi_{0}\hat{\psi}(u,t) = 
   \big\langle\hat{\psi}(u,\cdot),c_{0}\big\rangle_{t} \,c_{0}(t),
\end{equation}
where we recall that $c_{0}(t)=\cos\left(\frac{\pi}{2}t\right)$. We can now state a first approximation result:

\begin{prop}\label{approx1}
There exist $h_{0}>0$ and $C>0$ such that for $h\in(0,h_{0})$ and all ${\hat\psi\in\EN}$:
$$0\leq Q_{\Rec}^0(\hat\psi)\leq Ch^{2/3}\|\hat\psi\|^2$$
and
$$\|(\Id-\Pi_0)\hat\psi\| + \|\dr_{t}(\Id-\Pi_0)\hat\psi\| \leq  Ch^{1/3}\|\hat\psi\|.$$
Moreover, $\Pi_{0}$ : $\EN\to\Pi_{0}(\EN)$ is an isomorphism.
\end{prop}

\begin{proof}
If $\hat\psi=\hat\psi_{n}$, we have:
$$Q_{\Rec,h}(\hat\psi_{n})=\la_{n}\|\hat\psi_{n}\|^2.$$
From this we infer:
$$Q_{\Rec,h}(\hat\psi_{n})\leq\left(\frac{1}{8}+Ch^{2/3}\right)\|\hat\psi_{n}\|^2.$$
The orthogonality of the $\hat\psi_{n}$ (in $L^2$ and for the quadratic form) allows to extend this inequality to $\hat\psi\in\EN$:
$$Q_{\Rec,h}(\hat\psi)\leq\left(\frac{1}{8}+Ch^{2/3}\right)\|\hat\psi\|^2.$$
This clearly implies:
$$Q_{\Rec}^0(\hat\psi)\leq Ch^{2/3}\|\hat\psi\|^2.$$
$\Pi_{0}\hat\psi$ being in the kernel of $\L_{\Rec}^0$, we have:
$$Q_{\Rec}^0(\hat\psi)=Q_{\Rec}^0((\Id-\Pi_{0})\hat\psi).$$
If we denote by $\mu_{2}$ the second eigenvalue of the $1D$ operator $-\frac{1}{2\pi^2}\dr_{t}^2-\frac{1}{8}$, we get by the min-max principle:
$$Q_{\Rec}^0((\Id-\Pi_{0})\hat\psi)\geq\mu_{2}\|(\Id-\Pi_{0})\hat\psi\|^2.$$
Now the conclusions are standard.
\end{proof}

\subsection{Reduction to the Born-Oppenheimer approximation}\label{Reduction-to-BO}
In this section, we prove Theorem \ref{spectrumtriangle} by using the projections of the true eigenfunctions ($\Pi_{0}\psi_{n}$) as test functions for the Born-Oppenheimer approximation. 
Let us consider an eigenpair $(\la,\psi)$ of $\mathcal{L}_{\Tri}(h)$ such that (\ref{small-e-v}) holds. We let $\hat\psi(u,t)=\psi(x,y)$.  Then, $(\la,\hat\psi)$ satisfies:
\begin{align*}
   -h^2\left(\dr_{u}^2-\frac{2t\dr_{u}\dr_{t}}{u+\pi\sqrt{2}}
   +\frac{2t\dr_{t}}{(u+\pi\sqrt{2})^2}+\frac{t^2\dr_{t}^2}{(u+\pi\sqrt{2})^2}\right)
   \hat\psi
-\frac{1}{(u+\pi\sqrt{2})^2}\dr_{t}^2\hat\psi=\la\hat\psi.
\end{align*}
The main idea is to determine the (differential) equation satisfied by $\Pi_{0}\hat\psi$. In other words we will compute and control the commutator between the operator and the projection $\Pi_{0}$. For that purpose, a few lemmas will be necessary. The first one is an estimate established in the original coordinates $(x,y)$ in the triangle $\Tri$:

\begin{lem}\label{Lpr}
For all $k\in\mathbb{N}$, there exist $h_{0}>0$ and $C>0$ such that we have, for ${h\in(0,h_{0})}$:
$$\int_{\Tri}(x+\pi\sqrt{2})^{-k}|\dr_{y}\psi|^2 \,dxdy\leq C\|\psi\|^2.$$
\end{lem}

\begin{proof}
The equation satisfied by $\psi$ is:
$$(-h^2\dr_{x}^2-\dr_{y}^2)\psi=\la\psi.$$
Multiplying by $(x+\pi\sqrt{2})^{-k}$, taking the scalar product with $\psi$ and integrating by parts we find:
\[
   \int_{\Tri} (x+\pi\sqrt{2})^{-k}|\dr_{y}\psi|^2 \,dxdy\leq
   C\int_{\Tri} (x+\pi\sqrt{2})^{-k} \Big( |\psi|^2 
   +  h^2 (x+\pi\sqrt{2})^{-1} |\psi||\partial_x\psi|\Big)\,dxdy.
\]
Using the Agmon estimates of Proposition \ref{Agmon2'} with $\rho_0/h\ge k+1$ we deduce the lemma.
\end{proof}

We can now prove:

\begin{lem}\label{L1}
There exist $h_{0}>0$ and $C>0$ such that we have, for ${h\in(0,h_{0})}$:
$$\left\|\left\langle (u+\pi\sqrt{2})^{-1}t\dr_{u}\dr_{t}\hat\psi,\,
c_0(t)\right\rangle_{t}\right\|_{L^2(du)}\leq Ch^{-1}\|\hat\psi\|.$$
\end{lem}

\begin{proof}
Integrating by parts in $t$ for any fixed $u\in(-\pi\sqrt2,0)$, we find:
\begin{align*}
   \left|\left\langle (u+\pi\sqrt{2})^{-1}t\dr_{u}\dr_{t}\hat\psi,\,
   c_0(t)\right\rangle_{t}\right|
   &\leq C\int_{-1}^1(u+\pi\sqrt{2})^{-1}|\dr_{u}\hat\psi| \,dt\\
   &\leq C\left(\int_{-1}^1(u+\pi\sqrt{2})^{-2}|\dr_{u}\hat\psi|^2 \,dt\right)^{1/2}.
\end{align*}
To have the lemma, it remains to prove that
\[
   \int_{\Rec} (u+\pi\sqrt{2})^{-2}|\dr_{u}\hat\psi|^2 \,dudt \le C h^{-2}
   \int_{\Rec} |\hat\psi|^2 \,dudt.
\]
We have:
$$\int_{\Rec} (u+\pi\sqrt{2})^{-2}|\dr_{u}\hat\psi|^2 \,dudt
=\int_{\Tri} (x+\pi\sqrt{2})^{-3}\left|\left(\dr_{x}
+\frac{y\dr_{y}}{x+\pi\sqrt{2}}\right)\psi\right|^2 dxdy
$$
and we apply Lemma \ref{Lpr} to control the term in $\dr_{y}$.
We end the proof using the Agmon estimates of Proposition \ref{Agmon2'}.
\end{proof}

The same kind of computations yields:

\begin{lem}\label{L2}
There exist $h_{0}>0$ and $C>0$ such that we have, for ${h\in(0,h_{0})}$:
$$\left\|\left\langle (u+\pi\sqrt{2})^{-2}t\dr_{t}\hat\psi,\,
c_0(t)\right\rangle_{t}\right\|_{L^2(du)}
\leq C\|\hat\psi\|.$$
\end{lem}
Finally, we have:

\begin{lem}\label{L3}
There exist $h_{0}>0$ and $C>0$ such that we have, for ${h\in(0,h_{0})}$:
$$\left\|\left\langle (u+\pi\sqrt{2})^{-2}t^2\dr_{t}^2\hat\psi,\,
c_0(t)\right\rangle_{t}\right\|_{L^2(du)}
\leq C\|\hat\psi\|.$$
\end{lem}

From Lemmas \ref{L1}, \ref{L2} and \ref{L3}, and from Proposition \ref{approx1}, we infer:
\begin{prop}\label{almostBO}
Let $\Gamma_0>0$. 
There exist $h_{0}>0$ and $C>0$ such that for $h\in(0,h_{0})$ and all eigenpair $(\la,\psi)$ of $\mathcal{L}_{\Tri}(h)$ satisfying $|\lambda-\frac18|\le\Gamma_0h^{2/3}$, we have:
$$\left\|\left(-h^2\dr_{u}^2+\frac{\pi^2}{4(u+\pi\sqrt{2})^2}-\la\right)\Pi_{0}\hat\psi\right\|\leq Ch\|\Pi_{0}\hat\psi\|.$$
\end{prop}

\paragraph{Proof of Theorem \ref{spectrumtriangle}}
We deduce, from Proposition \ref{almostBO}, for all $n\in\{1,\cdots, N_{0}\}$:
$$\left\|\left(-h^2\dr_{u}^2+\frac{\pi^2}{4(u+\pi\sqrt{2})^2}\right)\Pi_{0}\hat\psi_{n}\right\|\leq (\la_{\Tri,N_{0}}(h)+Ch)\|\Pi_{0}\hat\psi_{n}\|.$$
From this inequality, we infer, for all $\psi\in\EN$:
$$\left\|\left(-h^2\dr_{u}^2+\frac{\pi^2}{4(u+\pi\sqrt{2})^2}\right)\Pi_{0}\hat\psi\right\|\leq (\la_{\Tri,N_{0}}(h)+Ch)\|\Pi_{0}\hat\psi\|$$
and thus:
$$Q_{\BO,\Tri,h}(\Pi_{0}\hat\psi)\leq(\la_{\Tri,N_{0}}(h)+Ch)\|\Pi_{0}\hat\psi\|.$$
It remains to apply the min-max principle to the $N_{0}$ dimensional space $\Pi_{0}\EN$ (see Proposition \ref{approx1}) and Theorem \ref{spectrumBOT} to get the separation of eigenvalues. Then, the conclusion follows from Proposition \ref{quasitri}.

\section{Eigenpair asymptotics for the waveguide}
\label{6}
In this section, we prove Theorem \ref{spectrumguide}. Firstly, we construct quasimodes and secondly we use Agmon estimates reduce to the triangle case.
On the left, $\L_{\Gui}(h)$ writes, in the coordinates $(u,t)$ defined in \eqref{Eut}:
\begin{equation}
\label{E:Llef}
   \L^\lef_{\Gui}(h)
   =-h^2\left(\dr_{u}-\frac{t}{u+\pi\sqrt{2}}\dr_{t}\right)^2-\frac{1}{(u+\pi\sqrt{2})^2}\dr_{t}^2
\end{equation}
and on the right, we let: 
\begin{equation}
\label{E:utau}
   u=x,\quad \tau=\frac{y-x}{\pi\sqrt{2}}
\end{equation}
and the operator writes:
\begin{equation}
\label{E:Lri}
   \L^\ri_{\Gui}(h)
   =-h^2\left(\dr_{u}-\frac{1}{\pi\sqrt{2}}\dr_{\tau}\right)^2-\frac{1}{2\pi^2}\dr_{\tau}^2.
\end{equation}
The integration domain is $(-\pi\sqrt{2},+\infty)\times(0,1)=\Omega_{\lef}\cup\Omega_{\ri}$ where:
$$\Omega_{\lef}=(-\pi\sqrt{2},0)\times(0,1)\mbox{ and }\Omega_{\ri}=(0,+\infty)\times(0,1).$$
The boundary conditions are Dirichlet on $(0,\infty)\times\{0\}\cup(-\pi\sqrt{2},\infty)\times\{1\}$ and Neumann on $(-\pi\sqrt{2},0)\times\{0\}$.

\subsection{Quasimodes}

The aim of this subsection is to prove the following proposition:

\begin{prop}\label{quasigui}
For any $n\ge1$, there exists a sequence $(\gamma_{j,n})$ such that, for all $N_{0}\in\mathbb{N}$ and $J\in\mathbb{N}$, there exists $h_{0}>0$ and $C>0$ such that for $h\in(0,h_{0})$:
\begin{equation}
   \dist \Big(\gS_\dis\big(\mathcal{L}_{\Gui}(h)\big),\,
  \sum_{j=0}^{J}\gamma_{j,n}h^{j/3} \Big) \leq Ch^{(J+1)/3}, \quad n=1,\cdots N_0.
\end{equation}
Moreover, we have: $\gamma_{0,n}=\frac{1}{8}$, $\gamma_{1,n}=0$ and $\gamma_{2,n}=(4\pi\sqrt{2})^{-2/3}z_{\A}(n)$.
\end{prop}

\subsubsection{Preliminaries}

\paragraph{Ansatz, boundary and transmission conditions}
In order to construct quasimodes for $\L_{\Gui}(h)$ of the form $(\gamma_h,\psi_h)$, we use the coordinates $(u,t)$ on the left and $(u,\tau)$ on the right and look for quasimodes $\hat\psi_h(u,t,\tau)=\psi_h(x,y)$.
Such quasimodes will have the form on the left:
\begin{equation}
\label{6Alef}
   \psi_{\lef}(u,t) \sim \sum_{j\geq 0}h^{j/3}
   \left(\Psi_{\lef,j}(h^{-2/3}u,t)+ \Phi_{\lef,j}(h^{-1}u,t)\right),
\end{equation}
and on the right:
\begin{equation}
\label{6Ari}
   \psi_{\ri}(u,\tau)\sim\sum_{j\geq 0} h^{j/3} \Phi_{\ri,j}(h^{-1}u,\tau)
\end{equation}
associated with quasi-eigenvalues:
$$\gamma_h\sim\sum_{j\geq 0}\gamma_{j} h^{j/3}.$$
We will denote $s=h^{-2/3}u$ and $\sigma=h^{-1}u$.
Since $\psi_h$ has no jump across the line $x=0$, we find that $\psi_{\lef}$ and $\psi_{\ri}$ should satisfy two transmission conditions on the line $u=0$:
$$
   \psi_{\lef}(0,t)=\psi_{\ri}(0,t)
   \quad\mbox{and}\quad
   \left(\dr_{u}-\frac{t}{\pi\sqrt{2}}\dr_{t}\right)\psi_{\lef}(0,t) = 
   \left(\partial_u-\frac{\partial_\tau}{\pi\sqrt2}\right)\psi_{\ri}(0,t),
$$
for all $t\in(0,1)$.
For the Ans\"atze \eqref{6Alef}-\eqref{6Ari} these conditions write for all $j\ge0$
\begin{align}
\label{transD}
  & \Psi_{\lef,j}(0,t)+\Phi_{\lef,j}(0,t)=\Phi_{\ri,j}(0,t) \\
\label{trans}
  & \dr_{\sigma}\Phi_{\lef,j}(0,t)+\dr_{s}\Psi_{\lef,j-1}(0,t)
   -\frac{t\dr_{t}}{\pi\sqrt{2}}\Phi_{\lef,j-3}(0,t)-\frac{t\dr_{t}}{\pi\sqrt{2}}\Psi_{\lef,j-3}(0,t)
  \\ \nonumber 
  & \hskip 18em =\dr_{\sigma}\Phi_{\ri,j}(0,t)
   -\frac{\partial_\tau}{\pi\sqrt2}\Phi_{\ri,j-3}(0,t),
\end{align}
where we understand that the terms associated with a negative index are $0$.

\begin{notation}
\label{6not}
We still set $s=h^{-2/3}u$ and $\sigma=h^{-1}u$.
Like in the case of the triangle $\Tri$, the operators $\L^\lef_\Gui$ and $\L^\ri_\Gui$, cf.\ \eqref{E:Llef}-\eqref{E:Lri}, written in variables $(s,t)$ and $(\sigma,t)$ expand in powers of $h^{2/3}$ and $h$, respectively. Now we have three operator series:
\begin{itemize}\itemsep=4pt
\item $\L^\lef_\Gui(h)(h^{2/3}s,t;h^{-2/3}\partial_s,\partial_t) 
   \sim \sum_{j\geq 0} \L_{2j}h^{2j/3}$. The operators are the same as for $\Tri$, but they are defined now on the half-strip $\Stlef:=(-\infty,0)\times(0,1)$.

\item $\L^\lef_\Gui(h)(h\sigma,t;h^{-1}\partial_\sigma,\partial_t) \sim
   \sum_{j\geq 0} \cN^\lef_{3j}h^{j}$ defined on $\Stlef$. 

\item $\L^\ri_{\Gui}(h)(h\sigma,\tau;h^{-1}\partial_\sigma,\partial_\tau) \sim
   \sum_{j\geq 0} \cN^\ri_{3j}h^{j}$ defined on $\Stri:=(0,\infty)\times(0,1)$. 
\end{itemize}
We agree to incorporate the boundary conditions on the horizontal sides of $\Stlef$ in the definition of the operators $\L_j$, $\cN^\lef_{j}$, and $\cN^\ri_{j}$:
\begin{itemize}\itemsep=3pt
\item Neumann-Dirichlet  $\partial_t\Psi(s,0)=0$ and $\Psi(s,1)=0$ $(s<0)$ for $\L_j$,
\item Neumann-Dirichlet  $\partial_t\Phi(\sigma,0)=0$ and $\Psi(\sigma,1)=0$ $(\sigma<0)$ for $\cN^\lef_{j}$,
\item Pure Dirichlet  $\Phi(\sigma,0)=0$ and $\Psi(\sigma,1)=0$ $(\sigma>0)$ for $\cN^\ri_{j}$.
\end{itemize}
Note that
\begin{equation}
\label{E:N0}
   \cN^\lef_0 = -\partial^2_\sigma - \frac{1}{2\pi^2}\partial^2_t \quad\mbox{and}\quad
   \cN^\ri_0 = -\partial^2_\sigma - \frac{1}{2\pi^2}\partial^2_\tau\,.
\end{equation}
\end{notation}

\paragraph{Dirichlet-to-Neumann operators}
Here we introduce the Dirichlet-to-Neumann operators $T^{\ri}$ and $T^{\lef}$ which we use to solve the problems in the variables $(\sigma,t)$. We denote by $I$ the interface $\{0\}\times(0,1)$ between $\Stri$ and $\Stlef$.

On the right, and with Notation \ref{6not}, we consider the problem:
\[
   \left(\cN^\ri_{0}-\frac{1}{8}\right)\Phi_\ri=0\ \ \mbox{in}\ \ \Stri
   \quad \mbox{and}\quad  \Phi_\ri(0,t)=G(t)
\]
where  $G\in H_{00}^{1/2}(I)$. Since the first eigenvalue of the transverse part of $\cN^\ri_{0}-\frac{1}{8}$ is positive, this problem has a unique exponentially decreasing solution $\Phi_\ri$. Its exterior normal derivative $-\dr_{\sigma}\Phi_{\ri}$ on the line $I$ is well defined in $H^{-1/2}(I)$. We define:
$$T^{\ri}G=\dr_{n}\Phi_{\ri}=-\dr_{\sigma}\Phi_{\ri}.$$
We have:
$$\langle T^\ri G,G\rangle=Q_{\ri}(\Phi_{\ri})\geq C\|G\|^2_{H^{1/2}_{00}(I)}.$$

On the left, we consider the problem:
\[
   \Big(\cN^\lef_{0}-\frac{1}{8} \Big)\Phi_\lef=0\ \ \mbox{in}\ \ \Stlef
   \quad \mbox{and}\quad  \Phi_\lef(0,t)=G(t)
\]
where $G\in H^{1/2}_{00}(I)$.

For all $G\in H_{00}^{1/2}(I)$ such that $\Pi_0G=0$ (where $\Pi_{0}$ is defined in \eqref{Pi0}), this problem has a unique exponentially decreasing solution $\Phi_\lef$. Its exterior normal derivative $\dr_{\sigma}\Phi_{\lef}$ on the line $I$ is well defined in $H^{-1/2}(I)$. We define:
$$T^{\lef}G=\dr_{n}\Phi_{\lef}=\dr_{\sigma}\Phi_{\lef}.$$
We have:
$$\langle T^\lef G,G\rangle=Q_{\lef}(\Phi_{\lef})\geq 0.$$

\begin{prop}
\label{6prop} 
The operator $T^\ri +T^\lef\Pi_{1}$ is coercive on $H^{1/2}_{00}(I)$ with $\Pi_{1}=\Id-\Pi_{0}$. In particular, it is invertible from $H^{1/2}_{00}(I)$ onto $H^{-1/2}(I)$.
\end{prop}
This proposition allows to prove the following lemma which is in the same spirit as Lemma \ref{lem-N0}, but now for transmission problems on $\Stlef\cup\Stri$ (we recall that $c_0(t) = \cos(\frac\pi2t)$):

\begin{lem}\label{lem-N0lefri}
Let $F_\lef=F_\lef(\sigma,t)$ and $F_\ri=F_\ri(\sigma,\tau)$ be real functions defined on $\Stlef$ and $\Stri$, respectively, with exponential decay with respect to $\sigma$. Let $G^0\in H^{1/2}_{00}(I)$ and $H\in H^{-1/2}(I)$ be data on the interface $I=\partial\Stlef\cap\partial\Stri$. 
Then there exists a unique coefficient $\zeta\in\R$ and a unique trace $G\in H^{1/2}_{00}(I)$ such that the transmission problem
\[
\begin{cases}
   \big(\cN^\lef_{0}-\frac{1}{8}\big)\Phi_\lef=F_\lef\quad\mbox{in}\ \ \Stlef,\quad  
   &\Phi_\lef(0,t)=G(t)+G^0(t)+\zeta c_{0}(t), \\[0.5ex]
   \big(\cN^\ri_{0}-\frac{1}{8}\big)\Phi_\ri=F_\ri\quad\mbox{in}\ \ \Stri,\quad  
   &\Phi_\ri(0,t)=G(t), \\[0.5ex]
   \partial_\sigma\Phi_\lef(0,t)-\partial_\sigma\Phi_\ri(0,t) = H(t)
   \quad\mbox{on}\ \  I,
\end{cases}
\]
admits a (unique) solution $(\Phi_\lef,\Phi_\ri)$ with exponential decay.
\end{lem}

\begin{proof}
Let $(\Phi_{\lef}^0,\zeta_0)$ be the solution provided by Lemma \ref{lem-N0} for the data $F=F_\lef$ and $G=0$. Let $\Phi_{\ri}^0$ be the unique exponentially decreasing solution of the problem
\[
   \Big(\cN^\ri_{0}-\frac{1}{8} \Big)\Phi^0_{\ri}=F_\ri\quad\mbox{in}\quad\Stri,\quad  
   \Phi^0_\ri(0,t)=0.
\]
Let $H^0$ be the jump $\partial_\sigma\Phi_{\ri}^0(0,t)-\partial_\sigma\Phi_{\lef}^0(0,t)$. If we define the new unknowns $\Phi_{\ri}^1=\Phi_{\ri}-\Phi_{\ri}^0$ and $\Phi_{\lef}^1=\Phi_{\lef}-\Phi_{\lef}^0$, the problem we want to solve becomes
\begin{align*}
   \Big(\cN^\lef_{0}-\frac{1}{8} \Big)\Phi_{\lef}^1&=0\quad\mbox{in}\quad\Stlef,\quad  
   \Phi_{\lef}^1(0,t)=G(t)+(\zeta-\zeta_0) c_{0}(t), \\
   \Big(\cN^\ri_{0}-\frac{1}{8} \Big)\Phi_{\ri}^1&=0\quad\mbox{in}\quad\Stri,\quad  
   \Phi_{\ri}^1(0,t)=G(t), \\
   \partial_\sigma\Phi_{\ri}^1(0,t)-\partial_\sigma\Phi_{\lef}^1(0,t) &= H(t)-H^0(t)
   \quad\mbox{on}\quad I.
\end{align*}
Using Proposition \ref{6prop} we can set $G=(T^\ri +T^\lef\Pi_{1})^{-1}(H-H_0)$, which ensures the solvability of the above problem. 
\end{proof}

\subsubsection{Construction of quasimodes}
\label{S:512}
\paragraph{Terms of order $h^0$}
Let us write the ``interior" equations:
\begin{align*}
\lef_{s} &:&\L_{0}\Psi_{\lef,0}&=\gamma_{0}\Psi_{\lef,0}\\ 
\lef_{\sigma} &:&\cN^\lef_{0}\Phi_{\lef,0}&=\gamma_{0}\Phi_{\lef,0}\\
\ri\ &:&\cN^\ri_{0}\Phi_{\ri,0}&=\gamma_{0}\Phi_{\ri,0}\,.\hskip5em
\end{align*}
The boundary conditions are:
$$\Psi_{\lef,0}(0,t)+\Phi_{\lef,0}(0,t)=\Phi_{\ri,0}(0,t),$$
$$ \dr_{\sigma}\Phi_{\lef,0}(0,t)=\dr_{\sigma}\Phi_{\ri,0}(0,t).$$
We get:
$$\gamma_{0}=\frac{1}{8},\quad \Psi_{\lef,0}=g_{0}(s)c_{0}(t).$$
We now apply Lemma \ref{lem-N0lefri} with $F_{\lef}=0$, $F_{\ri}=0$, $G_{0}=0$, $H=0$ to get 
$$G=0 \quad\mbox{and}\quad \zeta=0.$$
We deduce: $\Phi_{\lef,0}=0$, $\Phi_{\ri,0}=0$ and, since $\zeta=-g_{0}(0)$, $g_{0}(0)=0$. 
At this step, we do not have determined $g_{0}$ yet.

\paragraph{Terms of order $h^{1/3}$}
The interior equations read:
\begin{align*}
\lef_{s} &:&\L_{0}\Psi_{\lef,1}&=\gamma_{0}\Psi_{\lef,1}+\gamma_{1}\Psi_{\lef,0}\\ 
\lef_{\sigma} &:&\cN^{\lef}_{0}\Phi_{\lef,1}&=\gamma_{0}\Phi_{\lef,1}+\gamma_{1}\Phi_{\lef,0}\\
\ri &:&\cN^{\ri}_{0}\Phi_{\ri,1}&=\gamma_{0}\Phi_{\ri,1}+\gamma_{1}\Phi_{\ri,0}.\hskip5em
\end{align*}
Using Lemma \ref{lem-L0}, the first equation implies:
$$\gamma_{1}=0,\quad \Psi_{\lef,1}(s,t)=g_{1}(s)c_{0}(t).$$
The boundary conditions are:
$$g_{1}(0)c_{0}(t)+\Phi_{\lef,1}(0,t)=\Phi_{\ri,1}(0,t),$$
$$g_{0}'(0)c_{0}(t)+\dr_{\sigma}\Phi_{\lef,1}(0,t)=\dr_{\sigma}\Phi_{\ri,1}(0,t).$$
The system becomes:
\begin{align*}
\lef_{\sigma} &:& \Big(\cN_{0}^{\lef}-\frac{1}{8} \Big)\Phi_{\lef,1}&=0,\\
\ri\ &:& \Big(\cN_{0}^{\ri}-\frac{1}{8} \Big)\Phi_{\ri,1}&=0.\hskip5em
\end{align*}
We apply Lemma \ref{lem-N0lefri} with $F_{\lef}=0$, $F_{\ri}=0$, $G_{0}=0$, $H=-g_{0}'(0)c_{0}(t)$ to get:
$$G=-g_{0}'(0)(T^\ri +T^\lef\Pi_{1})^{-1}c_{0}.$$
Since $G=\Phi_{\ri,1}$ and $\zeta=-g_{1}(0)$, this determines $\Phi_{\lef,1}$, $\Phi_{\ri,1}$ and $g_{1}(0)$.

\paragraph{Terms of order $h^{2/3}$}
The interior equations write:
\begin{align*}
\lef_{s}&:&\L_{2}\Psi_{\lef,0}+\L_{0}\Psi_{\lef,2}&=\sum_{l+k=2}\gamma_{l}\Psi_{\lef,k}\\
\lef_{\sigma}&:&\cN_{0}^{\lef}\Phi_{\lef,2}&=\sum_{l+k=2}\gamma_{l}\Phi_{\lef,k}\\
\ri&:&\cN_{0}^{\ri}\Phi_{\ri,2}&=\frac{1}{8}\Phi_{\ri,2},\hskip5em
\end{align*}
with 
$$\L_{2}\Psi_{\lef,0}=-g_{0}''(s)c_{0}(t)+\frac{1}{\pi^3\sqrt{2}} sg_{0}(s)\dr_{t}^2(c_{0}).$$
Lemma \ref{lem-L0} and then Lemma \ref{lem-Ai} imply:
\begin{equation}
\label{E:g0}
   -g_{0}''-\frac{1}{4\pi\sqrt{2}} sg_{0}=\gamma_{2}g_{0}.
\end{equation}
Thus, $\gamma_{2}$ is one of the eigenvalues of the Airy operator  and $g_{0}$ an associated eigenfunction.
In particular, this determines the unknown functions of the previous steps. We are led to take:
$$\Psi_{\lef,2}(s,t)=\Psi_{\lef,2}^{\perp}+g_{2}(s)c_{0}(t), \mbox{ with }\Psi_{\lef,2}^{\perp}=0$$
and to the system:
\begin{align*}
\lef_{\sigma}&:&\Big(\cN_{0}^{\lef}-\frac{1}{8} \Big)\Phi_{\lef,2}&=0\\
\ri&:& \Big(\cN_{0}^{\ri}-\frac{1}{8} \Big)\Phi_{\ri,2}&=0.\hskip5em
\end{align*}
Using Lemma \ref{lem-N0lefri}, we find 
$$G=-g_{1}'(0)(T^\ri +T^\lef\Pi_{1})^{-1}c_{0}.$$
This determines $\Phi_{\ri,2}$,  $\Phi_{\lef,2}$ and $g_{2}(0)$. The function $g_{1}$ is still unknown at this step.

\paragraph{Further terms}
Let us assume that we can write $\Psi_{\lef,k}=\Psi^{\perp}_{\lef,k}+g_{k}(s)c_{0}(t)$ for $0\leq k\leq j$ and that $(g_{k})_{0\leq k\leq j-2}$ and $(\Psi^{\perp}_{\lef,k})_{0\leq k\leq j}$  are determined. Let us also assume that $g_{j-1}(0)$, $(\gamma_{k})_{0\leq k \leq j}$,  $(\Phi_{\ri,k})_{0\leq k\leq j-1}$, $(\Phi_{\lef,k})_{0\leq k\leq j-1}$ are already known. Finally, we assume that $g_{j}(0)$, $\Phi_{\lef,j}$, $\Phi_{\ri,j}$ are known once $g_{j-1}$ is determined and that all the functions have an exponential decay.

Let us collect the terms of order $h^{(j+1)/3}$.
The interior equations write:
\begin{align*}
\lef_{s}&:&\sum_{k=0}^{j+1} \L_{k}\Psi_{\lef,j+1-k}&=\sum_{k=0}^{j+1}\gamma_{k}\Psi_{\lef,j+1-k}\\
\lef_{\sigma}&:&\sum_{k=0}^{j+1}\cN_{k}^{\lef}\Phi_{\lef,j+1-k}&=\sum_{k=0}^{j+1}\gamma_{k}\Phi_{\lef,j+1-k}\\
\ri&:&\sum_{k=0}^{j+1}\cN_{k}^{\ri}\Phi_{\ri,j+1-k}&=\sum_{k=0}^{j+1}\gamma_{k}\Phi_{\ri,j+1-k},\hskip5em
\end{align*}
We examine the first equation and notice that $\L_{1}=0$ and $\gamma_{1}=0$ so that $\Psi_{\lef,j}$ does not appear. We can write this equation in the form:
\begin{multline*}
   \left(\L_{0}-\frac{1}{8}\right)\Psi_{\lef,j+1}=
   -\L_{2}\Psi_{\lef,j-1}-\gamma_{2}\Psi_{\lef,j-1}-\gamma_{j+1}\Psi_{\lef,0}\\[-1.2ex]
   -\sum_{k=4}^{j+1} \L_{k}\Psi_{\lef,j+1-k}
   -\sum_{k=3}^{j}\gamma_{k}\Psi_{\lef,j+1-k}.
\end{multline*}
We apply Lemma \ref{lem-L0} and we obtain an equation in the form:
$$-g_{j-1}''-\frac{1}{4\pi\sqrt{2}} sg_{j-1}-\gamma_{2}g_{j-1}=f+\gamma_{j+1}g_{0},$$
where $f$ and $g_{j-1}(0)$ are known. Then, Lemma \ref{lem-Ai} applies and provides a unique value of $\gamma_{j+1}$ such that $g_{j-1}$ has an exponential decay. 
From the recursion assumption, we deduce that $g_{j}(0)$, $\Phi_{\lef,j}$, $\Phi_{\ri,j}$ are now determined.
Lemma \ref{lem-L0} uniquely determines $\Psi_{\lef,j+1}^{\perp}$ such that:
$$\Psi_{\lef,j+1}=\Psi_{\lef,j+1}^{\perp}+g_{j+1}(s)c_{0}(t).$$
We can now write the system in the form:
\begin{align*}
\lef_{\sigma}&:&\Big(\cN_{0}^{\lef}-\frac{1}{8} \Big)\Phi_{\lef,j+1}&=F_{\lef}\\
\ri&:& \Big(\cN_{0}^{\ri}-\frac{1}{8} \Big)\Phi_{\ri,j+1}&=F_{\ri},\hskip5em
\end{align*}
where $F_{\lef}, F_{\ri}$ have an exponential decay. 
The transmission conditions are, cf.\ \eqref{transD}--\eqref{trans}:    
\begin{align*}
   \Phi_{\lef,j+1}(0,t) &=\Phi_{\ri,j+1}(0,t)-\Psi_{\lef,j+1}(0,t)\\
   &=\Phi_{\ri,j+1}(0,t)-\Psi_{\lef,j+1}^{\perp}(0,t)-g_{j+1}(0)c_{0}(t)
\end{align*}
and
$$ \partial_\sigma\Phi_{\lef,j+1}(0,t)-\partial_\sigma\Phi_{\ri,j+1}(0,t) = H(t) =-g'_{j}(0)c_{0}(t)+\tilde{H}(t),$$
where $\tilde{H}$ is known.
We can apply Lemma $\ref{lem-N0lefri}$ which determines $\Phi_{\ri,j+1}$, $\Phi_{\lef,j+1}$ (with an exponential decay) and $g_{j+1}(0)$ once $g_{j}$ is known.

\paragraph{Quasimodes}
The previous construction leads to introduce:
\begin{subequations}
\label{6quasi}
\begin{equation}
  \hat{\psi}_{h}^{[J]}(u,\cdot) = 
   \begin{cases}\di
   \ \sum_{j=0}^{J+2} \left(\Psi_{\lef,j}\Big(\frac{u}{h^{2/3}},t\Big)+\Phi_{\lef,j}\Big(\frac{u}{h},t\Big)\right) h^{j/3}
   \ \ &\mbox{when} \ \  u\le0 \\[0.8ex] \di
   \ \sum_{j=0}^{J+2} \Phi_{\ri,j}\Big(\frac{u}{h},\tau\Big)\, h^{j/3}+ u\,
   \chi^\ri\left(\frac{u}{h}\right)R_{J,h}(\tau)
   \ \ &\mbox{when} \ \  u\ge0 \,,
   \end{cases}
\end{equation}
where the correction term
\begin{align}
   R_{J,h}(\tau) &= \dr_{s}\Psi_{\lef,J+2}(0,\tau) h^{J/3}\\ \nonumber
   &-\sum_{j=J}^{J+2}\left(\frac{t\dr_{t}}{\pi\sqrt{2}}
   \Big(\Psi_{\lef,j}(0,\tau)+\Phi_{\lef,j}(0,\tau)\Big) \right)h^{j/3}
   +\sum_{j=J}^{J+2}\frac{\dr_{\tau}}{\pi\sqrt{2}}\Phi_{\ri,j}(0,\tau)\, h^{j/3}
\end{align}
is added to make $\hat{\psi}_{h}^{[J]}$ satisfy the transmission condition \eqref{trans}. Here we have used a smooth cutoff function $\chi^\ri$ being $1$ near $0$.
By construction, $\psi_{h}^{[J]}$ defined by the identity 
\begin{equation}
\label{}
   \psi_{h}^{[J]}(x,y)=\chi^{\lef}(u)\,\hat{\psi}_{h}^{[J]}(u,\cdot)
\end{equation} 
\end{subequations}
belongs to the domain of $\L_{\Gui}(h)$.
Using the exponential decays, for all $J\in\N$ we get the existence of $h_{0}>0$, $C(J,h_{0})>0$ such that for $h\in(0,h_{0})$:
$$\Big\|\Big(\L_{\Gui}(h)-\sum_{j=0}^{J+2}\gamma_{j} h^{j/3} \Big)
\psi_{h}^{[J]}\Big\|\leq C(J,h_{0}) \,h^{1+J/3}.$$

\subsection{Agmon estimates and consequences}\label{6.2}
In this last subsection, we prove Theorem \ref{spectrumguide}. For that purpose, we first state Agmon estimates to show that the first eigenfunctions are essentially living in the triangle $\Tri$ so that we can compare the problem in the whole guide with the triangle (see also Section \ref{S:4.4} where this idea was explained in the one-dimensional setting).
\begin{prop}
\label{6Agmon}
Let $(\la,\psi)$ be an eigenpair of $\L_{\Gui}(h)$ such that $|\la-\frac{1}{8}|\leq Ch^{2/3}$. There exist $\alpha>0$, $h_{0}>0$ and $C>0$ such that for all $h\in(0,h_{0})$, we have:
$$\int_{x\geq 0} e^{\alpha h^{-1}x}
\Big(|\psi|^2 + |h\dr_{x}\psi|^2\Big)\, dxdy\leq C\|\psi\|^2 .$$
\end{prop}
\begin{proof}
The proof is left to the reader, the main ingredients being the IMS formula and the fact that $\mathcal{H}_{\BO, \Gui}$ is a lower bound of $\L_{\Gui}(h)$ in the sense of quadratic forms (see the analysis of Propositions \ref{Agmon1} and \ref{Agmon2}). See also \cite[Proposition 6.1]{DaLafRa11} for a more direct method.
\end{proof}
\paragraph{Proof of Theorem {\rm\ref{spectrumguide}}}
Let $\psi_{n}^h$ be an eigenfunction associated with $\la_{\Gui,n}(h)$ and assume that the $\psi_{n}^h$ are orthogonal in $L^2(\Omega)$, and thus for the bilinear form $B_{\Gui,h}$ associated with the operator $\L_{\Gui}(h)$.
 
We choose $\varepsilon\in(0,\frac13)$ and introduce a smooth cutoff $\chi^{h}$at the scale $h^{1-\eps}$ for positive $x$
\[
   \chi^h(x) =\chi(xh^{\varepsilon-1}) \quad\mbox{with}\quad 
   \chi\equiv1 \ \mbox{ if }\ x\le\tfrac12,\quad
   \chi\equiv0  \ \mbox{ if }\ x\ge1
\] 
and we consider the functions $\chi^{h}\psi_{n}^h$.
We denote:
$$\EN=\mathrm{span}(\chi^h\psi_{1}^h,\ldots,\chi^h\psi_{N_0}^h).$$
We have:
$$Q_{\Gui,h}(\psi_{n}^h)=\la_{\Gui,n}(h)\|\psi_{n}^h\|^2$$
and deduce by the Agmon estimates of Proposition \ref{6Agmon}:
$$Q_{\Gui,h}(\chi^{h}\psi_{n}^h)=
\big(\la_{\Gui,n}(h)+\OO(h^{\infty})\big)\|\chi^{h}\psi_{n}^h\|^2.$$
In the same way, we get the "almost"-orthogonality, for $n\neq m$:
$$B_{\Gui,h}(\chi^{h}\psi_{n}^h,\chi^{h}\psi_{m}^h)=\OO(h^{\infty}).$$
We deduce, for all $v\in\EN$:
$$Q_{\Gui,h}(v)\leq
\big(\la_{\Gui,N_{0}}(h)+\OO(h^{\infty})\big)\|v\|^2.$$
We can extend the elements of $\EN$ by zero so that $Q_{\Gui,h}(v)=Q_{\Tri_{\eps,h}}(v)$ for ${v\in\EN}$ where $\Tri_{\eps,h}$ is the triangle with vertices $(-\pi\sqrt{2},0)$, $(h^{1-\eps},0)$ and $(h^{1-\eps},h^{1-\eps}+\pi\sqrt{2})$. A dilation reduces us to:
$$\left(1+\frac{h^{1-\eps}}{\pi\sqrt{2}}\right)^{-2}(-h^2\dr_{\tilde{x}}^2-\dr_{\tilde{y}}^2)$$
on the triangle $\Tri$.
The lowest eigenvalues of this new operator admits the lower bounds $\frac{1}{8}+z_{\A}(n)h^{2/3}-Ch^{1-\eps}$ ; in particular, we deduce: 
$$\la_{\Gui,N_{0}}(h)\geq \frac{1}{8}+z_{\A}(N_{0})h^{2/3}-Ch^{1-\eps}.$$
This provides the separation of the eigenvalues and, joint with Proposition \ref{quasigui}, this implies Theorem \ref{spectrumguide}.

\pagebreak[4]

\subsection{Conclusion}
\subsubsection{Eigenfunction asymptotics}
\label{S631}
With Theorem \ref{spectrumguide}, we deduce that the lowest eigenvalues of $\mathcal{L}_{\Gui}(h)$ are simple as soon as $h$ is small enough. Then, through the spectral theorem, we infer that the quasimodes constructed in \eqref{6quasi} are approximations of the true eigenfunctions (see for instance \cite{Hel88}). As a consequence, with the coordinates $u, t, \tau$ defined in \eqref{Eut} and \eqref{E:utau}, the $n$-th normalized eigenfunction admits the following expansion:
\begin{equation}
\label{QuasiLgui}
  \hat{\psi}_{n,h}(u,\cdot)\sim
   \begin{cases}\do
   \ \sum_{j\geq 0} \left(\Psi_{n,\lef,j}\Big(\frac{u}{h^{2/3}},t\Big)+\Phi_{n,\lef,j}\Big(\frac{u}{h},t\Big)\right) h^{j/3}
   \ \ &\mbox{when} \ \  u\le 0 \\[0.8ex] \di
   \ \sum_{j\geq 0} \Phi_{n,\ri,j}\Big(\frac{u}{h},\tau\Big)\, h^{j/3}
   \ \ &\mbox{when} \ \  u\ge 0 \,,
   \end{cases}
\end{equation}
where the functions $\Psi_{n,\lef,j}, \Phi_{n,\lef,j}, \Phi_{n,\ri,j}$ were constructed in Section \ref{S:512} (the subscript $n$ emphasizes the dependence on the rank of the zero of the Airy function determined when solving Equation \eqref{E:g0}).

\subsubsection{Remark on the Born-Oppenheimer approximation}

At the very beginning of this paper we have introduced the operator $\mathcal{H}_{\BO, \Gui}(h)$ (see \eqref{HG}) and we have somehow suggested that it is an approximation of $\mathcal{L}_{\Gui}(h)$ in the limit $h\to 0$. It turns out that we have not used $\mathcal{H}_{\BO, \Gui}(h)$ to investigate the spectrum of the waveguide. 
In fact, our analysis proves that the two term asymptotic expansion of the eigenvalues of $\mathcal{H}_{\BO, \Tri}(h)$, $\mathcal{H}_{\BO, \Gui}(h)$, $\mathcal{L}_{\Tri}(h)$ and $\mathcal{L}_{\Gui}(h)$ are the same so that we can \textit{a posteriori} say that $\mathcal{H}_{\BO, \Gui}(h)$ approximates $\mathcal{L}_{\Gui}(h)$.

\subsubsection{Back to the physical coordinates}
The two-term asymptotics
\[
   \Psi^\lef_0(s,t) \mathds{1}_{s<0} 
   + h^{1/3}\Big(\Phi^\lef_1(\sigma,t)  \mathds{1}_{\sigma<0}+\Phi^\ri_1(\sigma,\tau)\mathds{1}_{\sigma>0}\Big)
\]
provides us with the leading behavior of the eigenvectors in the scaled half-guide $\Omega$. It is interesting to come back to the physical domain and to interpret this two-term asymptotics in the original variables $(x_1,x_2)$. We have to chain formulas \eqref{E:xy} giving $(x,y)$, \eqref{Eut} giving $(u,t)$,  \eqref{E:utau} giving $(u,\tau)$, and \eqref{E:ssigma2} giving $s$ and $\sigma$. We have also to take the relation $h=\tan\theta$ into account.

Returning to section \ref{S:512} and more particularly to \eqref{E:g0} --- and Lemma \ref{lem-Ai}, we find that
\[
   \Psi^\lef_0(s,t) = \A\Big((4\pi\sqrt{2})^{-1/3}s+z_{\A}(n)\Big) 
   \cos\Big(\frac{\pi t}{2}\Big).
\]
Coming back to physical variables $(x_1,x_2)$ we find that
\[
   \Psi^\lef_0(s,t) = \A\Big( \Big(\frac{\theta}{2\pi}\Big)^{1/3}x_1 + z_{\A}(n)\Big) 
   \cos\Big(\frac{x_2}{2} - \frac{\theta x_1}{2\pi} \Big) + \OO(\theta^2)\quad\mbox{as}\quad \theta\to0.
\]
As for the term $\Phi_1 := \Phi^\lef_1  \mathds{1}_{\sigma<0}+\Phi^\ri_1\mathds{1}_{\sigma>0}$, we find that there exists a profile $\check \Phi_1$ independent of $\theta$ such that
\[
   \Phi_1(\sigma,t\mathds{1}_{\sigma<0}+\tau\mathds{1}_{\sigma>0}) = \check \Phi_1(\check x_1,\check x_2) 
   + \OO(\theta^2)\quad\mbox{as}\quad \theta\to0.
\]
Here $\check x_1=x_1$ and
\[
   \check x_2 = \begin{cases}
   \di\frac{\pi x_2\cos\theta}{\pi+x_1\sin\theta}\quad &\mbox{if}\ \ x_1<0,\\[1em]
   x_2\cos\theta - x_1\sin\theta\quad &\mbox{if}\ \ x_1>0.
   \end{cases}
\]
This profile $\check \Phi_1$ is exponentially decreasing as $\check x_1\to\pm\infty$. It is solution of a transmission problem with smooth data for the Laplace operator on the infinite strip $\R\times (0,\pi)$ with mixed Neumann-Dirichlet conditions on the bottom side $\check x_2=0$, and Dirichlet on $\check x_2=\pi$. Hence, it is piecewise $H^2$ modulo the addition of a multiple of the singular function $\psi^{0}_\sing$, cf.\ \eqref{eq:sing}.

The consequence of this is that the coefficient of the singularity $\psi^{\theta}_\sing$ for a normalized eigenvector of $\Delta^\Dir_{\Omega_{\theta}}$ behaves as $\OO(\theta^{1/3})$ as $\theta\to0$.

\subsubsection{X-shaped waveguides}
Our results provide without any difficulty the structure of the eigenpairs of lowest energy in the small angle limit when the domain is formed by the union of two infinite strips of same width $\pi$ crossing with an angle $2\theta$ (this model appears in the physical literature, see \cite{BEPS02}). The two non-convex corners of this structure are at the distance $\frac\pi{\sin\theta}=\OO(\theta^{-1})$. This X-structure can be viewed as a double symmetric V-structure and the eigenmodes can be constructed from the V-structure eigenmodes since they interact very weakly (their lower scale is $\theta^{1/3}$). Nevertheless they do interact by an exponentially small tunnelling effect which would be interesting to investigate.

\bigskip
\paragraph{Acknowledgments}
The authors would like to thank Francis Nier for giving them the impulse to write this paper. They are also grateful to the referee whose comments have improved the presentation of the strategy and of the spirit developed throughout this paper.


\bigskip

\end{document}